\documentclass{amsart}
\usepackage{amsmath,amssymb}
\usepackage[dvipsnames]{xcolor}
\usepackage{pinlabel}

\DeclareSymbolFont{bbold}{U}{bbold}{m}{n}
\DeclareSymbolFontAlphabet{\mathbbb}{bbold}

\usepackage[shortlabels]{enumitem}
\usepackage{tikz-cd}
\usepackage[unicode,colorlinks=true,linktocpage=true,citecolor=Emerald,linkcolor=Fuchsia]{hyperref}
\usepackage{graphicx}
\usepackage{mathtools}
\usetikzlibrary{graphs}
\usepackage{mathrsfs}
\usepackage{microtype}

\newtheorem{theorem}{Theorem}
\numberwithin{theorem}{subsection}
\newtheorem{thm}[theorem]{Theorem}
\newtheorem{proposition}[theorem]{Proposition}
\newtheorem{propn}[theorem]{Proposition}

\newtheorem{corollary}[theorem]{Corollary}
\newtheorem{cor}[theorem]{Corollary}
\newtheorem{lemma}[theorem]{Lemma}
\theoremstyle{definition}
\newtheorem{definition}[theorem]{Definition}
\newtheorem{defn}[theorem]{Definition}
\newtheorem{conj}[theorem]{Conjecture}
\newtheorem{predef}[theorem]{Preliminary Definition}
\newtheorem{example}[theorem]{Example}
\newtheorem{examples}[theorem]{Examples}
\newtheorem{notation}[theorem]{Notation}
\newtheorem{remark}[theorem]{Remark}

\newtheorem{construction}[theorem]{Construction}
\newtheorem{question}[theorem]{Question}

\providecommand{\op}{\mathrm{op}}
\providecommand{\xel}{\mathrm{el}}
\providecommand{\xint}{\mathrm{int}}
\newcommand{\xHom}{\operatorname{Hom}}
\newcommand{\xFun}{\operatorname{Fun}}

\newcommand{\bbone}{\mathbbb{1}}

\DeclareMathOperator*{\colim}{colim}

\newcommand{\finset}{\operatorname{Fin}}
\newcommand{\pfinset}{\finset_*}
\newcommand{\finsetskel}{\mathbf{F}}
\newcommand{\pfinsetskel}{\finsetskel_*}

\newcommand{\xCat}{\operatorname{Cat}}
\newcommand{\xMap}{\operatorname{Map}}
\newcommand{\Map}{\operatorname{Map}}
\newcommand{\xAlg}{\operatorname{Alg}}
\newcommand{\xsSet}{\operatorname{sSet}}
\newcommand{\xE}{\mathcal{E}}
\newcommand{\xxO}{\mathcal{O}}
\newcommand{\xcc}{\mathcal{C}}
\newcommand{\xS}{\mathcal{S}}
\newcommand{\xU}{\mathcal{U}}
\newcommand{\xV}{\mathcal{V}}
\newcommand{\xW}{\mathcal{W}}
\newcommand{\olI}{\overline{I}}

\newcommand{\xfe}{\mathfrak{e}}
\newcommand{\xfc}{\mathfrak{c}}
\newcommand{\id}{\operatorname{id}}

\newcommand{\Pre}{\name{P}}
\newcommand{\PSU}{\Pre_{\mathbb{S}}(\xU)}

\newcommand*\cocolon{%
	\nobreak
	\mskip6mu plus1mu
	\mathpunct{}%
	\nonscript
	\mkern-\thinmuskip
	{:}%
	\mskip2mu
	\relax
}

\newcommand{\icat}{$\infty$-category}
\newcommand{\isoto}{\xrightarrow{\sim}}
\newcommand{\xto}[1]{\xrightarrow{#1}}
\newcommand{\csquare}[8]{ %
	\[ %
	\begin{tikzpicture} %
	\matrix (m) [matrix of math nodes,row sep=3em,column sep=2.5em,text height=1.5ex,text depth=0.25ex] %
	{ #1 \pgfmatrixnextcell #2 \\ %
		#3 \pgfmatrixnextcell #4 \\ }; %
	\path[->,font=\footnotesize] %
	(m-1-1) edge node[auto] {$#5$} (m-1-2)%
	(m-1-1) edge node[left] {$#6$} (m-2-1)%
	(m-1-2) edge node[auto] {$#7$} (m-2-2)%
	(m-2-1) edge node[below] {$#8$} (m-2-2);%
	\end{tikzpicture}%
	\]%
}

\newcommand{\nolabelcsquare}[4]{\csquare{#1}{#2}{#3}{#4}{}{}{}{}}

\newcommand{\name}[1]{\ensuremath{\text{\textup{#1}}}}

\newcommand{\simp}{\mathbf{\Delta}}
\newcommand{\bbO}{\mathbf{\Omega}}

\newcommand{\levelg}{\mathbf{L}}
\newcommand{\levelV}{\levelg^\xV}
\newcommand{\levelU}{\levelg^\xU}
\newcommand{\levelgconn}{\levelg_{\name{c}}}

\newcommand{\levelint}{\levelg_{\name{int}}}
\newcommand{\levelel}{\levelg_{\name{el}}}
\newcommand{\levelPSU}{\levelg^{\PSU}}
\newcommand{\levelcPSU}{\levelgconn^{\PSU}}

\newcommand{\levelcV}{\levelgconn^{\xV}}
\newcommand{\levelcVel}{\levelg_{\name{c},\name{el}}^{\xV}}
\newcommand{\levelcUint}{\levelg_{\name{c},\name{int}}^{\xU}}
\newcommand{\levelcUel}{\levelg_{\name{c},\name{el}}^{\xU}}
\newcommand{\levelcU}{\levelgconn^{\xU}}

\newcommand{\oplevelV}{\levelg^{\op,\xV}}
\newcommand{\oplevelcU}{\levelg^{\op,\xU}_{\name{c}}}
\newcommand{\oplevelcUel}{\levelg^{\op,\xU}_{\name{c},\xel}}
\newcommand{\oplevelcV}{\levelgconn^{\op,\xV}}
\newcommand{\oplevelcVel}{\levelg_{\name{c}, \name{el}}^{\op,\xV}}

\newcommand{\bbY}{\mathbf{G}}
\newcommand{\bbYout}{\bbY_{\name{out}}}
\newcommand{\bbYV}{\bbY^{\xV}}
\newcommand{\bbYVint}{\bbY_{\xint}^{\xV}}
\newcommand{\bbYVel}{\bbY_{\xel}^{\xV}}
\newcommand{\bbYU}{\bbY^\xU}
\newcommand{\bbYUint}{\bbY_{\xint}^{\xU}}
\newcommand{\bbYUel}{\bbY_{\xel}^{\xU}}
\newcommand{\bbYPSU}{\bbY^{\PSU}}

\newcommand{\opbbYU}{\bbY^{\op,\xU}}
\newcommand{\opbbYV}{\bbY^{\op,\xV}}
\newcommand{\opbbYVel}{\bbY^{\op,\xV}_{\name{el}}}
\newcommand{\opbbYUel}{\bbY^{\op,\xU}_{\name{el}}}

\newcommand{\hryint}{\bbY_{\name{int}}}
\newcommand{\hryact}{\bbY_{\name{act}}}

\newcommand{\hryGamma}{\mathbf{\Gamma}}
\newcommand{\segalGamma}{\mathbf{\Gamma}}

\newcommand{\gc}{\mathbf{\Xi}}
\newcommand{\gcint}{\gc_{\name{int}}}
\newcommand{\gcact}{\gc_{\name{act}}}
\newcommand{\gcel}{\gc_{\xel}}

\newcommand{\gcone}{\gc}
\newcommand{\gctwo}{\mathbf{\Upsilon}}
\newcommand{\gcVone}{\gcone^\xV}
\newcommand{\gcVtwo}{\gctwo^\xV}
\newcommand{\gcUone}{\gcone^\xU}
\newcommand{\gcUtwo}{\gctwo^\xU}

\newcommand{\gcU}{\gc^{\xU}}
\newcommand{\gcPU}{\gc^{\Pre(\xU)}}
\newcommand{\gcPSU}{\gc^{\PSU}}
\newcommand{\gcV}{\gc^{\xV}}
\newcommand{\opgcV}{\gc^{\op,\xV}}

\newcommand{\opgcVint}{\gc^{\op,\xV}_{\name{int}}}
\newcommand{\opgcVel}{\gc^{\op,\xV}_{\name{el}}}

\newcommand{\kockgraphs}{\mathbf{K}}
\newcommand{\kockint}{\kockgraphs_{\name{int}}}

\newcommand{\bartau}{\overline{\tau}}
\newcommand{\bartauint}{\bartau_{\name{int}}}
\newcommand{\Dext}{\partial_{\name{ext}}}

\newcommand{\barishriek}{\bar \imath_!}
\newcommand{\barinotshriek}{\bar \imath_\vartriangle}
\newcommand{\bariustar}{\bar \imath^*}

\newcommand{\Set}{\name{Set}}
\newcommand{\Seg}{\name{Seg}}
\newcommand{\Fun}{\name{Fun}}
\newcommand{\blank}{\text{\textendash}}
\newcommand{\PrL}{\name{Pr}^{\mathrm{L}}}
\newcommand{\Cat}{\name{Cat}}
\newcommand{\CatI}{\Cat_{\infty}}
\newcommand{\IFF}{if and only if}

\newcommand{\Alg}{\name{Alg}}
\newcommand{\AlgPrdSet}{\name{Alg}_{\mathcal{P}rpd/\Set}}

\newcommand{\lists}{\mathtt{Lst}}
\newcommand{\sub}{\mathtt{Sb}}
\newcommand{\out}{\mathtt{out}}
\newcommand{\inp}{\mathtt{in}}
\newcommand{\edge}{\mathtt{E}}
\newcommand{\vertex}{\mathtt{V}}

\makeatletter
\providecommand*{\cupdot}{%
  \mathbin{%
    \mathpalette\@cupdot{}%
  }%
}
\newcommand*{\@cupdot}[2]{%
  \ooalign{%
    $\m@th#1\cup$\cr
    \hidewidth$\m@th#1\cdot$\hidewidth
  }%
}
\makeatother

\makeatletter
\providecommand*{\capdot}{%
  \mathbin{%
    \mathpalette\@capdot{}%
  }%
}
\newcommand*{\@capdot}[2]{%
  \ooalign{%
    $\m@th#1\cap$\cr
    \hidewidth$\m@th#1\cdot$\hidewidth
  }%
}
\makeatother

\newcommand{\ordcup}{\cupdot}
\newcommand{\ordcap}{\capdot}
\newcommand{\strsub}{\sqsubset}
\newcommand{\strcup}{\mathbin{\tilde \cup}}

\newcommand{\Segrep}{\Seg^{\name{rep}}}
\newcommand{\SegS}{\Seg_{\mathbb{S}}}
\newcommand{\SegrepS}{\Seg^{\name{rep}}_{\mathbb{S}}}
\newcommand{\Segrepc}{\widehat{\Seg}\vphantom{\Seg}^{\name{rep}}}

\newcommand{\properads}{\name{Prpd}(\Set)}
\newcommand{\calproperad}{\mathcal{P}rpd}
\newcommand{\bfproperad}{\mathbf{Prpd}}
\newcommand{\nmproperad}{\name{Prpd}}

\newcommand{\nsqelt}[1]{\vec{#1}}

\newcommand{\scriptyell}{\mathscr L}

\title{On rectification and enrichment of infinity properads}
\author{Hongyi Chu}
\address{Max Planck Institute for Mathematics, Bonn, Germany}
\email{chu@mpim-bonn.mpg.de}
\author{Philip Hackney}
\address{Department of Mathematics, University of Louisiana at Lafayette, Lafayette, LA 70504-3568 USA}
\email{philip@phck.net} 
\urladdr{http://phck.net}
\thanks{
The first author thanks the Labex CEMPI (ANR-11-LABX-0007-01) and Max Planck Institute for Mathematics for their hospitality and financial support during the process of writing this article.
This material is based upon work supported by the National Science Foundation under Grant No.\ DMS-1440140, while the second author was in residence at the Mathematical Sciences Research Institute in Berkeley, California, during the Spring 2020 semester. 
The second author also acknowledges the support of Australian Research Council Discovery Project grant DP160101519.
}

\date{\today}
\begin{document}

\begin{abstract}
We develop a theory of infinity properads enriched in a general symmetric monoidal infinity category.
These are defined as presheaves, satisfying a Segal condition and a Rezk completeness condition, over certain categories of graphs.
In particular, we introduce a new category of level graphs which also allow us to give a framework for algebras over an enriched infinity properad.
We show that one can vary the category of graphs without changing the underlying theory.

We also show that infinity properads cannot always be rectified, indicating that a conjecture of the second author and Robertson is unlikely to hold.
This stands in stark contrast to the situation for infinity operads, and we further demarcate these situations by examining the cases of infinity dioperads and infinity output properads.
In both cases, we provide a rectification theorem that says that each up-to-homotopy object is equivalent to a strict one.
\end{abstract}

\maketitle

\tableofcontents

\section{Introduction}

Properads, first introduced by Vallette \cite{Vallette:KDP} in the context of Koszul duality theory for props, are an intermediate notion between operads and props that are capable of governing certain types of bialgebraic structures.
Morphisms in operads take the form $f \colon a_1, \dots, a_n \to b$, that is, they can be interpreted as many-to-one operations.
If $g \colon c_1, \dots, c_m \to d$ is some other operation, then we can form \[ g\circ_i f \colon c_1, \dots, c_{i-1}, a_1, \dots, a_n, c_{i+1}, \dots, c_m \to d\]  whenever $1\leq i \leq m$ and $b = c_i.$
Properads, or more precisely the many-colored variant of them discovered independently by Duncan in his thesis \cite[\S6.1]{Duncan:TQC} under the name ``compact symmetric polycategories,'' are an extension of operads that allows one to consider many-to-many operations \[f \colon a_1, \dots, a_n \to b_1, \dots, b_p.\]
Composition, rather than connecting one output with one input, is designed to connect several (meaning `at least one') outputs with inputs.
These should be regarded as props (in the sense of \cite[\S24]{MacLane:CA}) without horizontal composition.\footnote{See also \cite[p.80]{Duncan:TQC}.}

Infinity properads were introduced by the second author, Robertson, and Yau in \cite{hrybook}, in part as a potential structural framework for chain-level string topology operations. 
In the present work we study infinity properads enriched in an arbitrary (presentably symmetric monoidal $\infty$-)category $\xV$. 
The main examples the reader should keep in mind are when $\xV$ is the category of chain complexes over a field of characteristic zero (which is the original context for the properads in \cite{Vallette:KDP}) or the category of spaces (which can be regarded as the `unenriched' case of infinity properads). 
The method for $\xV$-enrichment presented here works more generally, a story that will be told in forthcoming work of the first author and Haugseng \cite{patterns3}.

The basic idea for $\xV$-enrichment is as follows.
There is a category of graphs equipped with a suitable amount of structure so that one may identify ordinary properads as certain set-valued presheaves over this graph category, namely those presheaves satisfying a ``Segal condition.'' 
Infinity properads are then space-valued presheaves satisfying a (homotopical) Segal condition (as well as a discreteness condition for objects). 
One then builds a new indexing $\infty$-category (Definition~\ref{def LV_YV}) of $\xV$-decorated graphs and then $\xV$-enriched infinity properads are a subclass of Segal presheaves on this new $\infty$-category (Definition~\ref{def cts Seg psh}, Definition~\ref{def completness}). 
When $\xV$ is just spaces, this recovers the usual unenriched notion of infinity properads.

The present paper concentrates on two major questions in the theory:
\begin{enumerate}
\item 
Given a $\xV$-enriched infinity properad $P$, what should be meant by the $\infty$-category of algebras of $P$? 
We address this question by showing that the $\infty$-category of $\xV$-enriched infinity properads, $\nmproperad_\infty^\xV$, is tensored over $\infty$-categories (Theorem~\ref{theo tensor} and Proposition~\ref{prop tensor}), so by adjunction one can produce the desired $\infty$-category of algebras (Corollary~\ref{cor Alg(-,-)} and Corollary~\ref{corollary algebras cotensors complete}).
\item Suppose $\xV$ is a symmetric monoidal $\infty$-category associated to a symmetric monoidal model category and suppose $P$ is a $\xV$-enriched infinity properad, is it possible to rectify it to a strict properad enriched in $\xV$? 
We explain why we expect the answer to this question to be negative in general (even when $\xV$ is spaces, see Theorem~\ref{theorem non-rect}), but give an affirmative answer in certain special cases such as chain complexes (Theorem~\ref{theorem rational rectification for properads}) and symmetric spectra (Remark~\ref{remark symm spectra}). 
We also give an affirmative answer in related situations, such as for dioperads (i.e., symmetric polycategories) and for output properads. 
In particular, we expect that Conjecture 4.14 of \cite{infproplec}, which occurs in the setting of model categories, is false as stated, but that analogues will be true for both dioperads and for output properads.
\end{enumerate}

The two preceding questions each require us to approach enriched infinity properads using a different base indexing category of graphs.
For the second question, it is most appropriate to use the properadic graphical category from \cite{hrybook}, which was further developed in \cite{Kock_Properads} and \cite{hry_factorizations}. 
This is a category of directed graphs with loose ends which are acyclic as directed graphs and connected as undirected graphs. 
In \cite{hrybook} it was shown that Segal set-valued presheaves on this category are precisely the ordinary properads.
In this paper we give a new, conceptual presentation of morphisms of this category as certain homomorphisms on the partially-ordered set of subgraphs. The new description of the properadic graphical category then reveals its tight connection to the operads governing properads which is in turn essential for our rectification procedure.

The first question requires a different approach. 
Namely, we introduce a new category of directed \emph{level} graphs, which admits a cartesian fibration to the simplex category. 
This extra structure allows for an alternative description of the Segal condition. 
This enhanced relationship between the level graph category and the simplex category induces a relationship between infinity properads and infinity categories.
Of course level graphs also played a prominent role in \cite{Vallette:KDP}.
Indeed, the `vertical' structure of our level graph category is already visible in the simplicial bar construction (see the second remark \cite[p.4920]{Vallette:KDP}), though we will need the full structure below.

In each instance, we have utilized a graph category especially suited to the task at hand. 
In the first question, we used the level graph category which bears a close relationship with the simplex category, while in the second we used the directed graph category which is closely related to operads for properads. 
A third question arises, which is whether we are really talking about the same kind of enriched infinity properads in both instances. 
This is indeed the case (Corollary~\ref{cor equ enr properads}), though the proof is rather involved (\S\ref{subsection comparison proof}).

In the second question above, we alluded to there being related developments for infinity dioperads and output (or input) properads.
For the most part, these developments are entirely parallel to the case of properads, and simply amount to restricting to full subcategories of various graph categories.

Dioperads are like properads, in that operations can have many inputs and outputs, but are also like operads, in that the only compositions we have connect one output with one input.
The name `dioperad' first appeared in work of Gan \cite{Gan:KDD}, again in the context of Koszul duality theory, but the many-colored version `polycategory' had appeared earlier (see \cite{Szabo:P} for the non-symmetric version and \cite[\S5.1]{Hyland:PTA} or \cite{Garner:PPDL} for the symmetric version). 
To treat this case we will restrict all of our graph categories to the full subcategory whose objects are graphs which, as topological spaces, have trivial fundamental groups.

Output properads are those properads having the property that if $f \colon a_1, \dots, a_n \to b_1, \dots, b_p$ is an operation, then $p$ is positive.
For input properads one instead requires that $n$ is positive.
Many interesting properads are, in fact, output (or input) properads.
For instance, there are several homological conformal field theories for string topology \cite{ChataurMenichi:STCS,CohenGodin:PVST,Godin,HepworthLahtinen:OSTCS}, all of which require at least some type of positive boundary condition (or even noncompactness condition).
In the closed part of the theory of \cite{Godin}, this amounts to the structure of an algebra over the output properad given by the homology of the moduli spaces of connected Riemann surfaces with at least one outgoing boundary (see \S1 of \cite{Tamanoi:SSOT}).
It is not possible to relax this condition in string topology situations, that is, to consider a full hcft with both units and counits as the value on a circle would be finite-dimensional, while the (co)homology of a free loop space is usually infinite dimensional (see \S1.5 of \cite{HepworthLahtinen:OSTCS} for further details).
Further examples of this phenomenon abound, e.g.\ the topological conformal field theories of \cite[\S1.1]{Costello:TCFTCYC} have a similar restriction.

Though the development of the theory of enriched infinity properads, dioperads and output/input properads follow the same path, we do not have the same rectification theorems in the first case.
The main difference is that properads, unlike dioperads or output properads, are not modeled by a $\Sigma$-cofibrant operad.
At the very least, this means that one does not have access to standard tools (such as \cite[Theorem 4.4]{BergerMoerdijk}, which even provides a Quillen equivalence of model categories) for rectifying homotopy properads.
We prove slightly more in Theorem~\ref{theorem non-rect}, showing that when working over simplicial sets and for a specific model of homotopy properads, that the standard comparison adjunction between strict properads and homotopy properads is not a Quillen equivalence.

We should compare the preceding paragraph to the classical setting of the operads $\mathcal{A}ss$ and $\mathcal{C}om$ in topological spaces: the former is $\Sigma$-free while the latter is not.
Connected $\mathcal{A}_\infty$-spaces can be rectified to topological monoids (see \cite{BoardmanVogt,May,Stasheff:HAHSI}), whereas it can be shown by vanishing of k-invariants or Dyer--Lashof operations for commutative topological monoids that it is not possible to rectify all $\mathcal{E}_\infty$-spaces to commutative monoids (see, e.g., \cite[p.203]{BoardmanVogt}).
By analogy, our interpretation of the non-rectification result is that infinity properads are more free than strict properads, and are the correct notion for homotopy theory.

\begin{remark}\label{remark homotopical setting}
The homotopical setting for this paper is that of $\infty$-categories, and we avoid Quillen model categories until the very end (\S\ref{subsec rect}).
This added flexibility is important when studying enriched properads and their algebras, and much of what we do has no obvious counterpart in the realm of model categories and Quillen functors.
Let us point to two concrete situations where it is clear that the rigidity of model categories and Quillen functors would be an impediment.
\begin{itemize}[leftmargin=*]
\item Let $p$ be a prime number and let $\mathbf{Ch}_k$ be the category of chain complexes over the field $k$ with $p$ elements.
As was observed in \cite[3.3.3]{BergerMoerdijk}, it is not possible to have a model structure on (unreduced, monochrome) operads in $\mathbf{Ch}_k$ so the forgetful functor to symmetric sequences simultaneously creates weak equivalences and is a right Quillen functor.
The key step in this argument is that in positive characteristic the free graded commutative algebra functor from chain complexes to the category of CDGAs can take an acyclic chain complex to a CDGA with nontrivial homology.
A similar consideration applies in the case of properads; related concerns for props manifest in \cite[4.10]{Fresse:PMCHIS}.
\item If $P$ is a (monochrome) properad in a symmetric monoidal category $\mathbf{V}$, then the forgetful functor from $P$-algebras to $\mathbf{V}$ often does not have a left adjoint.
In particular, when $\mathbf{V}$ is a monoidal model category, we cannot expect to have a model structure on $P$-algebras so that the forgetful functor is right Quillen, since this functor is not even a right adjoint for many choices of $\mathbf{V}$ and $P$ (and similarly for left adjoints).
To have the desired model structure, there seem to be substantial restrictions on at least one of $\mathbf{V}$ or $P$, such as requiring that $P$ be an operad \cite{BergerMoerdijk}, or that $\mathbf{V}$ be cartesian \cite[Theorem 1.4]{JohnsonYau:OHIACP}.
\end{itemize}
\end{remark}

\subsection{Notation and conventions}
We write $\simp$ for the usual simplicial indexing category and $[n]$, $n\geq 0$, for its objects.
The category of simplicial sets will be denoted by $\xsSet$, and the category of small categories by $\xCat$.

The category of all small properads and properad maps will be denoted by $\properads$. 
Objects are unenriched properads, that is properads with sets of morphisms $a_1, \dots, a_n \to b_1, \dots, b_p$. 
We won't give a formal definition here as it is somewhat involved (see \S6.1 \cite{Duncan:TQC} or Definition 11.25 or Definition 11.27 of \cite{YauJohnson:FPAM}), but one will appear much later in this paper as well, in Definition~\ref{def PrdS} (the operad governing $S$-colored properads) and Definition~\ref{def prpdV} (the category of all properads).

We let $\finset$ and $\pfinset$ denote the category of finite sets and pointed finite sets, respectively. 
An object in $\pfinset$ is denoted by $A_+=A\amalg \{*\}$ where $A\in \finset$ and $*$ is the base point.
We write $\finsetskel$ for a skeleton of the category $\finset$, spanned by objects $\mathbf{n} \coloneqq \{1,\ldots,n\}$. 
Similarly, we write $\pfinsetskel$ for a skeleton of the category $\pfinset$, spanned by $\langle n \rangle \coloneqq \{1,\ldots,n\}\amalg\{*\}$. 
We will often implicitly identify an object $K_{+}\in \finset$ such that $|K| = n$ with $\langle n \rangle\in\pfinsetskel$.

This paper is mostly written $\infty$-categorically, that is, using the quasicategorical formalism for $(\infty,1)$-categories presented in \cite{ht}.
We will regard $1$-categories as special $\infty$-categories and all categorical constructions such as taking (co)limits should be interpreted in the $\infty$-categorical setting.

We write $\xS$ for the $\infty$-category of spaces (or $\infty$-groupoids) and, for an $\infty$-category $\xcc$, we write $\Pre(\xcc)$ for the $\infty$-category $\xFun(\xcc^\op,\xS)$ of presheaves of spaces on $\xcc$.
We will use $\xCat_\infty$ to denote the $\infty$-category of $\infty$-categories.

Given a functor $f\colon \xcc \to \mathcal{D}$ between $\infty$-categories and an object $d\in \mathcal{D}$. We let $\xcc_{d/}$ denote the pullback $\xcc\times_{\mathcal{D}} \mathcal{D}_{d/}$ whose objects are pairs $(c, \alpha)$ where $c\in \xcc$ and $\alpha\colon d\to f(c)$. For such an object we will often write $c$ and leave $f$ implicit. We define $\xcc_{/d}$ analogously.

In this paper, $\xU$ will always refer to a small symmetric monoidal \icat{}, while $\xV$ and $\xW$ will denote arbitrary symmetric monoidal $\infty$-categories.
That said, often $\xV$ will denote a presentably symmetric monoidal \icat{}, (i.e., $\xV$ is a presentable \icat{} and the tensor product preserves colimits in each variable).
The results generally remain true for large symmetric monoidal $\infty$-categories by passage to a larger universe as in \cite[Remark 3.5.9]{ChuHaugseng} and \cite[Theorem 5.6.6]{enriched}, but we will not comment further on that here.

\subsection{Outline}
Section~\ref{section categories of directed graphs} is devoted to several categories of directed graphs without (directed) cycles.
Key among these is the new category $\levelg$ of level graphs, introduced in \S\ref{subsec levelg}.
In \S\ref{section subcategory bbY}, we give a new, conceptual presentation of the properadic graphical category from \cite{hrybook}; for narrative purposes, a proof of the relevant equivalence is postponed until Appendix~\ref{appendix comparison with HRY}.
We also introduce a functor from the full subcategory $\levelgconn$ of $\levelg$ on the connected graphs to the properadic graphical category in \S\ref{section level to acyclic}, which plays a key role in later comparison results.
Several other graph categories appear in \S\ref{section categories of directed graphs} as subcategories of the main two, which are useful for studying structures related to properads (trees and forests are for operads, directed graphs without undirected cycles are for dioperads, and so on).

In Section~\ref{sec segal presheaves}, we introduce the algebraic version of $\xV$-enriched $\infty$-properads and give the first results.
Section~\ref{sec algebras over inf properads} restricts attention to categories of level graphs, shows how to tensor by Segal simplicial spaces, and introduces categories of algebras.
At this point, we have two competing notions of $\xV$-enriched $\infty$-properads, one based on $\levelg$ and the other based on $\bbY$.
We show in Section~\ref{sec comparison} that these two approaches coincide.

We leave the algebraic world behind in Section~\ref{sec ffes}, where we introduce a completeness condition for enriched $\infty$-properads.
Finally, in Section~\ref{sec alg} we compare our notion of enriched $\infty$-properads to enriched ordinary properads.

\subsection{Further directions}
The goal of this paper is to build a foundational framework for enriched $\infty$-properads and their algebras. 
We now propose several interesting areas of exploration  based on the machinery developed here:
\begin{itemize}[leftmargin=*]
\item \emph{Enriched $\infty$-properads as monoids}: 
First of all, one should be able to describe enriched $\infty$-properads as monoids in bicollections in a similar way as Vallette first introduced properads in \cite{Vallette:KDP}.
More precisely, the goal would be to construct a Day convolution double $\infty$-category $\widetilde{\levelV}$ by using the double $\infty$-categorical structure of $\levelV$ (Remark~\ref{rem doublecatV}) introduced in Definition~\ref{def LV_YV}.
This naturally generalizes the constructions given by \cite{HaugsengSS} for the $\infty$-operadic setting. 
The universal property of $\widetilde{\levelV}$ should then show that the $\infty$-category of monoids in $\widetilde{\levelV}$ is equivalent to the $\infty$-category $\Seg(\levelV)$ of Segal objects and by restricting to the $\infty$-category $\name{Prpd}^\xV_\infty$ of $\xV$-enriched $\infty$-properads viewed as a full subcategory of $\Seg(\levelV)$ we obtain the desired description of enriched $\infty$-properads as monoids. 

One of the reasons for introducing $\levelV$ in this paper is the fact that in contrast to other $\infty$-categories governing enriched $\infty$-properads such as $\bbYV$ the $\infty$-category $\levelV$ admits a natural double $\infty$-categorical structure which is essential for the construction of $\widetilde{\levelV}$ mentioned above.

\item \emph{Algebras as modules:} 
As a first application of the previous item, for any enriched $\infty$-properad $P$ we wish to describe $P$-algebras as certain module in bicollections.
This would improve on the results about algebras from the present paper by not only proving the existence, but also giving an explicit formula of computing algebras.
\item \emph{Koszul duality or bar-cobar construction:} 
Built on the description of enriched $\infty$-properads as monoids in bicollections Lurie's theory of Koszul duality for associative algebras \cite[\S 5.2]{ha} then gives an adjunction between enriched $\infty$-properads and enriched $\infty$-coproperads which are coassociative coalgebras in bicollections.
We expect that this approach generalizes the bar-cobar construction in the setting of ordinary properads and by restricting our general construction we should then obtain the Koszul duality for enriched $\infty$-operads with the most interesting case being the enrichment over spectra.
As was observed by Ching and Harper \cite{ChingBar,ChingHarper}, the coalgebraic structures in spectra are difficult to work with using model categorical methods, enticing one to work in the $\infty$-categorical setting and to study Koszul duality occurring for any stable symmetric monoidal $\infty$-categories instead of spectra.
In \cite{FrancisGaitsgory} Francis and Gaitsgory have used the expected properties of enriched $\infty$-operads to obtain Koszul duality equivalences under certain finiteness hypotheses and also conjectured how this should generalize.
It would be interesting to compare our $\infty$-categorical construction with the approach suggested by Francis--Gaitsgory.
\item \emph{Tensor product for $\infty$-properads:} 
In Section~\ref{sec algebras over inf properads}, we use the level graph structure of object in $\levelV$ to define the tensor product of enriched $\infty$-properads with $\infty$-categories. 
Replacing $\levelV$ with $\levelg$, one can extend the construction from \S\ref{sec algebras over inf properads} to give a tensor product of unenriched $\infty$-properads generalizing the tensor product on strict properads from \cite[\S 4.2]{hrybook} and the Boardman--Vogt tensor product of $(\infty\text{-})$operads.
Although there is no known tensor product of $\xV$-enriched $\infty$-operads unless $\xV$ is cartesian, we expect that in presence of a cartesian symmetric monoidal $\infty$-category $\xV$ a natural generalization of our construction of the tensor product in Section~\ref{sec algebras over inf properads} then gives a closed symmetric monoidal structure on the $\infty$-category of $\xV$-enriched $\infty$-properads. 
\item \emph{Simplicial localization of $P$-algebras:} Let $P$ be a properad in the category of chain complexes in characteristic zero.
As mentioned in Remark~\ref{remark homotopical setting}, we do not expect the category of $P$-algebras to have a meaningful Quillen model structure as the ground category is not cartesian.
Nonetheless, one can consider the category of $P$-algebras as a relative category, where the weak equivalences are the quasi-isomorphisms.
At this stage, one can use Dwyer--Kan simplicial localization (as in \cite{DwyerKan:SLC}) to obtain an $\infty$-category (see \cite[6.10]{BarwickKan:RCAMHTHT}); it would be interesting to see how this compares with the construction of algebras from the present paper.
Homotopy-invariance properties of this construction were proved by Yalin in \cite{Yalin:SLHAOP} (for props, rather than properads), a remarkable result given the lack of a suitable model structure on $P$-algebras. 
As our construction of algebras is manifestly homotopy-invariant, it is to be expected that any comparison would be closely related to Yalin's theorem.
\end{itemize}

\subsection*{Acknowledgments}
The authors give warm thanks to Michael Batanin, Rune Haugseng, Joachim Kock, Steve Lack, Dmitri Pavlov, Marcy Robertson, David White, and Donald Yau for valuable discussions related to this paper.

\section{Categories of directed graphs}\label{section categories of directed graphs}

In this paper, we are concerned with (finite) graphs which are \emph{directed}, have \emph{loose ends}, and are \emph{acyclic}.
Each graph (with loose ends) $G$ is given by two sets $\edge(G)$ and $\vertex(G)$ together with incidence data.
Namely, each vertex $v\in \vertex(G)$ should come equipped with two subsets $\inp(v)$ and $\out(v)$ of $\edge(G)$ (so that no edge is the input or output of two different vertices), but a given edge need not be an input (or output) for any vertex of the graph.
This last bit is what gives the distinction of `loose ends.'
Let us give a convenient, short formalism for graphs having the first two properties, which we learned from \cite[1.1.1]{Kock_Properads}.
\begin{definition}\label{graph definition}
Let $\mathscr{G}$ denote the category
\[
\begin{tikzcd}
\mathbf{e} & \mathbf{i} \dar{p} \lar{s} \\
\mathbf{o} \uar{t} \rar{q} & \mathbf{v}.
\end{tikzcd}
\]
A \emph{graph} $G$ is a functor $\mathscr{G} \to \finset$ which sends $s$ and $t$ to monomorphisms, that is, a diagram of finite sets of the form
\[
	\edge \overset{s}\hookleftarrow \mathtt{I} \overset{p}\rightarrow \vertex \overset{q}\leftarrow \mathtt{O} \overset{t}\hookrightarrow \edge.
\]
\begin{itemize}
	\item The image of $\mathbf{e}$, denoted by $\edge$ or $\edge(G)$, is the set of \emph{edges}.
	\item The image $\mathbf{v}$, denoted by $\vertex$ or $\vertex(G)$, is the set of \emph{vertices}.
	\item If $v\in \vertex$, we write $\inp(v) = p^{-1}(v)$ and $\out(v) = q^{-1}(v)$.
	\item We write $\inp(G) = \edge \setminus t(\mathtt{O})$ and $\out(G) = \edge \setminus s(\mathtt{I})$.
\end{itemize}
\end{definition}

We will typically regard $\mathtt{I}\cong \coprod_{v\in \vertex} \inp(v)$ and $\mathtt{O} \cong \coprod_{v\in \vertex} \out(v)$ as actual subsets of $\edge$.
Thus the set $\edge$ admits two decompositions
\[
	\edge = \inp(G) \amalg \coprod_{v\in \vertex} \out(v) \qquad \edge = \out(G) \amalg \coprod_{v\in \vertex} \inp(v).
\]
A \emph{na\"ive morphism} of graphs is simply a natural transformation of functors.
In \cite[1.1.7]{Kock_Properads} these were called \emph{morphism of graphs} and defined a full subcategory $\mathbf{Gr}^+ \subseteq \xFun(\mathscr{G}, \finset)$.
We won't have too much use for na\"ive morphisms as such in the present work (the one exception being as an alternative characterization of `structured subgraph,' in Definition~\ref{def structured subgraphs} -- see the proof of Lemma~\ref{lemma: level subgraph}) but they are useful in discussing the concepts of \emph{connectedness} and \emph{acyclicity}, as in \cite[\S 1.2]{Kock_Properads}.

\begin{remark}
There are other possible definitions of directed graph with loose ends, and also of basic notions like connectedness and acyclicity.
For instance, in \cite[\S 2.1.2]{hrybook} the definition of `generalized graph' is given; this formalism was extensively developed in \cite{YauJohnson:FPAM}.
This Yau--Johnson formalism for directed graphs is nearly equivalent to the one from Definition~\ref{graph definition}, the only exception being that in that formalism graphs are allowed to have components that are vertex-free loops.
Since we will only be interested in acyclic directed graphs in what follows, this difference would not appear anyway.
The equivalence of these approaches can be chained together from \cite[1.1.12]{Kock_Properads}, \cite[Proposition 15.2]{batanin-berger}, and \cite[Proposition 15.6]{batanin-berger}, noting that the directionality is preserved across all of these bijections.
\end{remark}

\begin{definition}[\'Etale map]
\label{definition kockint}
A na\"ive morphism of graphs $G\to H$ is called \emph{\'etale} if the middle two squares in the commutative diagram
\[ \begin{tikzcd}[column sep=small, row sep=scriptsize]
\edge(G)  \dar & \mathtt{I}(G) \lar \rar \dar & \vertex(G) \dar & \lar \mathtt{O}(G) \rar \dar & \edge(G) \dar \\
\edge(H)  & \mathtt{I}(H) \lar \rar & \vertex(H) & \lar \mathtt{O}(H) \rar & \edge(H)
\end{tikzcd} \]
are pullbacks.
We write $\kockint$ for the category whose objects are isomorphism classes of graphs which are both connected and acyclic, and whose morphisms are the \'etale maps.
\end{definition}

The category $\kockint$ was called $\mathbf{Gr}$ in \cite{Kock_Properads}.
Kock also had a larger category of graphs $\widetilde{\mathbf{Gr}}$ whose morphisms are more complicated; this category is equivalent to $\kockgraphs$ from the following definition.

\begin{definition}
\label{definition kockgraphs}
Given a connected, acyclic graph $G$ which is equipped with a total ordering on each of the sets $\inp(v)$ and $\out(v)$, an associated $\edge(G)$-colored properad is defined in \cite[Definition 5.7]{hrybook}.
For each object in $\kockint$, make a choice of representative of the isomorphism class and a choice of total ordering on all of the sets $\inp(v)$ and $\out(v)$.
Let $\kockgraphs$ denote the full subcategory of the category of colored properads (in $\Set$, see \cite[Definition 3.5]{hrybook}) spanned by the objects of $\kockint$, considered as colored properads.
\end{definition}

\subsection{Level graphs}\label{subsec levelg}
A level graph is a graph whose vertices and edges are arranged in several distinct layers, so that each edge in a middle layer connects vertices in the adjacent layers.
More precisely, we have the following, which we will later package as Definition~\ref{def:level2}.
\begin{predef}\label{def:level1}
	A \emph{level graph of height $n$} is a directed graph $G$ together with an assignment of an integer in $[0,n] = \{0,1,\dots, n\}$ to each edge and an assignment of a number in $[1,n]=\{1, \dots, n\}$ to each vertex.
	The functions $h_E\colon \edge(G) \to [0,n]$ and $h_V\colon \vertex(G) \to [1,n]$ should satisfy
	\[
		h_E(e) = \begin{cases}
			0 & \text{if } e\in \inp(G), \\
			n & \text{if } e\in \out(G), \\
			h_V(v) - 1 & \text{if $e\in \inp(v)$, and} \\
			h_V(v) & \text{if } e\in \out(v). \\
		\end{cases}
	\]
\end{predef}
See Figure \ref{figure level graph ex} for an example for a level graph of height $4$.
In general, the extremal layers will be edge layers, whose edges are connected to vertices only at one side, allowing us to glue graphs together.
We want to think of the vertices as `functions' or `processes' and the edges as `inputs' or `outputs' of these, and gluing corresponds to composition of total functions.

Notice that any directed graph which admits this extra structure is automatically acyclic, i.e., wheel-free.
On the other hand, there are acyclic graphs that do not admit any level structure, for instance, the graph in Example~\ref{example three vertex} below.

\begin{remark}\label{remark: loose edge}
Suppose $G$ is a level graph of height $n$ and suppose there exists an edge $e\in \edge$ which is not attached to any vertex in $G$.
Then the decompositions $\edge = \inp(G) \amalg \coprod_{v\in \vertex} \out(v)$ and $ \edge = \out(G) \amalg \coprod_{v\in \vertex} \inp(v)$ imply that $e\in \inp(G)\cap \out(G)$, and the condition on $h_E$ implies that $0=h_E(e)=n$.
Hence, the underlying graph of a level graph containing a loose edge is of height zero, that is, is a finite collection of loose edges.
\end{remark}

\begin{figure}
	\includegraphics[width=0.6\textwidth]{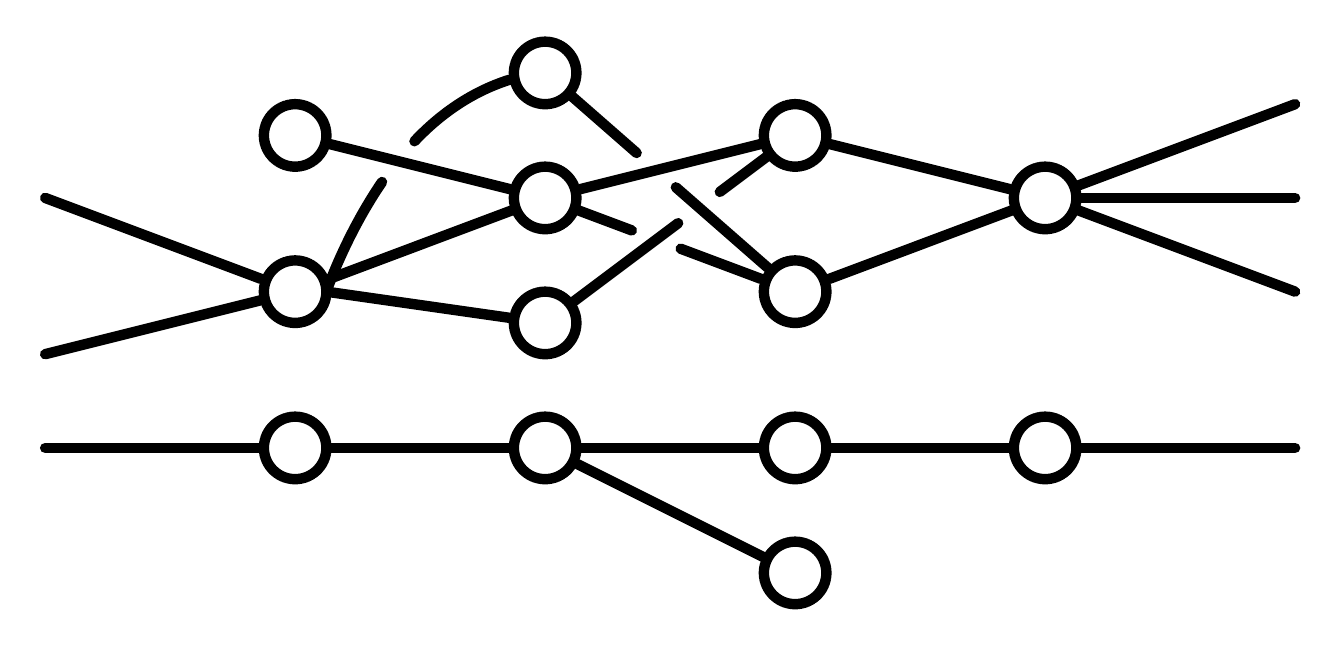}
	\caption{A level graph of height 4}\label{figure level graph ex}
\end{figure}

\begin{example}[Elementary level graphs]\label{example: elementary graphs}
The following level graphs will be called \emph{elementary}:
	\begin{itemize}
		\item If $p,q \geq 0$, the level graph $\xfc_{p,q}$ has height $1$, $p$ edges in level $0$, $q$ edges in level $1$, and a single vertex. We call such a level graph a \emph{corolla} and we write $\xfc$ for it if we do not want to emphasize the numbers of input and output edges.
		\item The level graph $\xfe$ which has height $0$ (hence no vertices) and a single edge.
	\end{itemize}
\end{example}

\begin{remark}
\label{remark many heights}
If $G$ is a connected graph with $\inp(G) \neq \varnothing \neq \out(G)$, then $G$ is a level graph in at most one way.
In particular, the height of $G$ is uniquely determined.
This is not the case when either $\inp(G)$ or $\out(G)$ is empty.
For example, consider the graph $G$ with a single vertex and no edges.
If $n\geq 1$, then each of the $n$ functions $h_V\colon\ast \simeq \vertex(G) \to [1,n]$ exhibits $G$ as a level graph of height $n$.
Conversely, if $G$ is a connected level graph of height $n$ that admits no other level graph structures, then both $\inp(G)$ and $\out(G)$ are nonempty sets.
\end{remark}

In the following we want to give an equivalent definition using certain category $\scriptyell_0^n$ (Definition~\ref{def twnzero}) to the category of finite sets.
This will allow us to compare (connected) level graphs with objects in the Hackney--Robertson--Yau category $\hryGamma$ (Theorem~\ref{theorem upsilon equivalent definitions} and Corollary~\ref{corollary functor l to y}).
For this purpose we introduce the following definitions.

\begin{definition}[Twisted arrow category]
\label{def twisted arrow}
If $\xcc$ is a category, then the \emph{twisted arrow category}
	$\name{Tw}(\xcc)$ has as objects the morphisms of $\xcc$, and morphisms $f \to f'$ in $\name{Tw}(\xcc)$ are given by commutative squares of the form
	\[
		\begin{tikzcd}
			\bullet \dar{f} & \bullet \lar \dar{f'} \\
			\bullet \rar & \bullet
		\end{tikzcd}
	\]
	in $\xcc$.
	Let $\epsilon\colon \simp\to \simp$ denote the functor $[n]\mapsto [n]^\op \star [n]$.
	Then the twisted arrow category is the restriction (to $\xCat$) of the functor $\epsilon^*\colon \xsSet\to \xsSet$ given by precomposition with $\epsilon$.
\end{definition}

Let $n\geq 0$ and consider the twisted arrow category $\name{Tw}(\Delta^n)$.
This category is, in fact, a partially-ordered set, and we will give it the alternative name $\scriptyell^{n} \cong \name{Tw}(\Delta^n)$ when we have identified the morphism $i \to j$ of $\Delta^n$ with the pair $(i,j)$.
In other words, objects of $\scriptyell^{n}$ are pairs $(i,j)$ with $0\leq i \leq j \leq n$ and there is a unique morphism $(i,j) \to (i',j')$ exactly when $i' \leq i$ and $j \leq j'$.
As an example, here is the category $\scriptyell^3$.
\[ \begin{tikzcd}[column sep=tiny, row sep=tiny]
0,0 \arrow[dr] & & 1,1 \arrow[dr]\arrow[dl] & & 2, 2 \arrow[dr]\arrow[dl] & & 3,3 \arrow[dl] \\
& 0,1 \arrow[dr]& &  1,2  \arrow[dr]\arrow[dl] & & 2,3\arrow[dl] \\
& & 0,2\arrow[dr] & & 1,3\arrow[dl] \\
& & & 0,3
\end{tikzcd} \]
Notice that every square in $\scriptyell^{n}$ is both a pushout and a pullback.
The opposite category of Definition~\ref{def twnzero}\eqref{ellzero} appears as \cite[Definition 2.2.7]{HaugsengSS}.
\begin{definition}\label{def twnzero}
Let $n\geq 0$.
\begin{enumerate}
	\item The category $\scriptyell^{n}$ has objects pairs $(i,j)$ with $0\leq i \leq j \leq n$ and a unique morphism $(i,j) \to (i',j')$ exactly when $0 \leq i' \leq i \leq j \leq j' \leq n$.
	\item Let $\scriptyell_0^n$ denote the full subcategory of $\scriptyell^n$ spanned by the objects $(i,j)$ where $j-i\leq 1$. \label{ellzero}
\end{enumerate}
\end{definition}

In other words, $\scriptyell_0^n$ is the full subcategory of $\scriptyell^{n}$ consisting of all objects of the form $(i,i)$ and $(i,i+1)$, that is,
\[
\scriptyell^{n}_0 = \left(
 \begin{tikzcd}[row sep=tiny, column sep=tiny]
0,0 \arrow[dr] & & 1,1 \arrow[dl]
& \cdots & n{-}1,n{-}1  \arrow[dr] & & n,n \arrow[dl]
 \\
& 0,1 & & \cdots & & n{-}1,n
\end{tikzcd} \right). \]

\begin{notation}\label{level graph subscript notation}
If $G\colon \scriptyell^n \to \finset$ is a functor, we denote its value at an object $(i,j) \in \scriptyell^n$ by $G_{i,j}$, and similarly for functors from $\scriptyell^n_0$.
\end{notation}

\begin{definition}\label{def:level2}
A \emph{level graph of height $n$} is a functor $G\colon \scriptyell^{n}_0\to \finset$.
\end{definition}

Note that natural transformations of functors correspond to na\"ive morphisms of graphs.
We will cut out a more appropriate class of level-preserving morphisms of height $n$ level graphs in Definition~\ref{def M}.
\begin{lemma}\label{lemma: level graph}
	Preliminary Definition~\ref{def:level1} is equivalent to Definition~\ref{def:level2}.
\end{lemma}
\begin{proof}
Given a functor $G\colon \scriptyell^{n}_0\to \finset$, the correspondence is realized by $h_E^{-1}(i) = G_{i,i}$ for $0\leq i \leq n$ and $h_V^{-1}(i) = G_{i-1,i}$ for $1\leq i \leq n$.
That is, the underlying directed graph (as in Definition~\ref{graph definition}) is given by
\[
\coprod_{i=0}^n G_{i,i} \overset{s}\hookleftarrow \coprod_{i=0}^{n-1} G_{i,i} \overset{p}\rightarrow \coprod_{i=1}^n G_{i-1,i} \overset{q}\leftarrow \coprod_{i=1}^n G_{i,i} \overset{t}\hookrightarrow \coprod_{i=0}^n G_{i,i}
\]
where $p$ restricts to $G_{i,i} \to G_{i,i+1}$ and $q$ restricts to $G_{i,i} \to G_{i-1,i}$.
\end{proof}

We instantly see that level graphs of height $1$ are just cospans of finite sets.
But cospans assemble into a (weak) double category, as in \cite[\S 5]{spancospan}, which we could use as a starting point to define a category $\levelg$ of level graphs.
We take an alternative approach, noting that a functor $\scriptyell^{n}_0 \to \finset$ is essentially the same thing as a \emph{pushout-preserving} functor $\scriptyell^{n} \to \finset$.
This observation is fruitful, as there is a faithful functor $\simp \to \xCat$ sending $[n]$ to $\scriptyell^{n}$ (since $\name{Tw}$ from Definition~\ref{def twisted arrow} is a functor), but no such functor which sends $[n]$ to $\scriptyell^{n}_0$.

\begin{definition}\label{definition special and tilde M}
	We say a functor $\scriptyell^{n}\to \xcc$ is \emph{special} if it is a left Kan extension of its restriction to $\scriptyell^{n}_0$.
	We write $\widetilde{M}\colon \simp^\op \to \xCat$ for the functor which takes $[n]$ to the full subcategory $\widetilde{M}_n$ of $\xFun(\scriptyell^{n},\finset)$ spanned by the special functors.
\end{definition}

Since every functor $\scriptyell^{n}_0 \to \finset$ admits a left Kan extension $\scriptyell^{n} \to \finset$, we see that $\widetilde{M}_n$ contains all of the level graphs of height $n$.
Notice that a functor $\scriptyell^{n}$ is special if and only if it is pushout-preserving; this is equivalent to saying that \emph{every} square in $\scriptyell^{n}$ is sent to a pushout.
Given any $\alpha\colon [n] \to [m]$, the functor $\scriptyell^\alpha \colon \scriptyell^n \to \scriptyell^m$ is automatically pushout-preserving since every square in both categories is a pushout square.

\begin{remark}\label{remark quotients}
If $F \colon \scriptyell^n \to \finset$ is a special functor, then $F_{i,j}$ is a quotient of
\begin{equation*}
\left( \coprod_{k=i}^j F_{k,k} \right) \amalg \left( \coprod_{k=i}^{j-1} F_{k,k+1} \right).
\end{equation*}
One can see this via induction on $j-i$.
The base cases don't utilize the assumption at all: for $j-i = 0$ the statement is clear.
For $j-i = 1$, elements of $F_{i,i}$ and $F_{j,j} = F_{i+1,i+1}$ are identified with their images in $F_{i,i+1} = F_{i,j}$.
For higher values, we have that $F_{i,j}$ is a pushout of $F_{i,j-1} \leftarrow F_{i+1,j-1} \rightarrow F_{i+1,j}$, so the result follows.
\end{remark}

\begin{lemma}\label{lemma F0n connected}
Suppose that $F\colon \scriptyell^n \to \finset$ is a special functor and let $H$ be the directed graph associated to $F|_{\scriptyell^{n}_0}$ as in Lemma~\ref{lemma: level graph}.
The graph $H$ is connected if and only if $F_{0,n}$ is a one-element set.
\end{lemma}
\begin{proof}
If $F_{0,n} = A_1 \amalg A_2$, then $F = F^1 \amalg F^2$ where $F^k_{i,j} = (F_{i,j} \to F_{0,n})^{-1}(A_k) \subseteq F_{i,j}$ for $k=1,2$.
Each $F^k$ is a left Kan extension of $F^k|_{\scriptyell^{n}_0}$; further, if $A_k \neq \varnothing$, then $F^k|_{\scriptyell^{n}_0}$ is not the trivial functor (which sends each objects to $\varnothing$) by Remark~\ref{remark quotients}.
Thus if $A_1$ and $A_2$ are both nonempty, we get a nontrivial coproduct decomposition $F|_{\scriptyell^{n}_0} = F^1|_{\scriptyell^{n}_0} \amalg F^2|_{\scriptyell^{n}_0}$.
The correspondence from Lemma~\ref{lemma: level graph} preserves coproducts, which tells us that $H$ is not connected.

In the other direction, suppose that $H$ is not connected.
As the correspondence preserves coproducts, we obtain a nontrivial decomposition $F|_{\scriptyell^{n}_0} = G^1 \amalg G^2$.
Letting $F^k$ be a left Kan extension of $G^k$ ($k=1,2$), we then have $F \cong F^1 \amalg F^2$.
As each $G^k$ is nonempty, so is $F^k_{0,n}$ by Remark~\ref{remark quotients}.
Thus we have $F_{0,n} \cong F^1_{0,n} \amalg F^2_{0,n}$ as a decomposition into disjoint nonempty sets.
\end{proof}

We now endeavor (in Remark~\ref{remark level sugraphs}) to isolate and generalize the construction from the proof of this lemma.

\begin{definition}[Level subgraphs]\label{definition level subgraph}
Suppose that $F\colon \scriptyell^n \to \finset$ is a special functor.
Elements of $F_{i,j}$ will be called $(i,j)$-\emph{level subgraphs} of the level graph $F|_{\scriptyell^{n}_0}$.
\end{definition}

Figure~\ref{figure height four level graph} provides a graphical representation of this concept, 
where each element of the set $F_{i,j}$ is depicted as a connected level graph.

\begin{definition}\label{partial scriptyell}
If $0 \leq i \leq j \leq n$, write $\scriptyell^{n}_{i,j}$ for the full subcategory consisting of those objects $(k,\ell)$ so that $i \leq k \leq \ell \leq j$.
In other words, if $\alpha \colon [j-i] \to [n]$ is given by $\alpha(t) = t+i$, then $\scriptyell^{n}_{i,j}$ is the image of the functor $\scriptyell^{\alpha} \colon \scriptyell^{j-i} \to \scriptyell^n$.
\end{definition}

\begin{figure}
\includegraphics[width=0.8\textwidth]{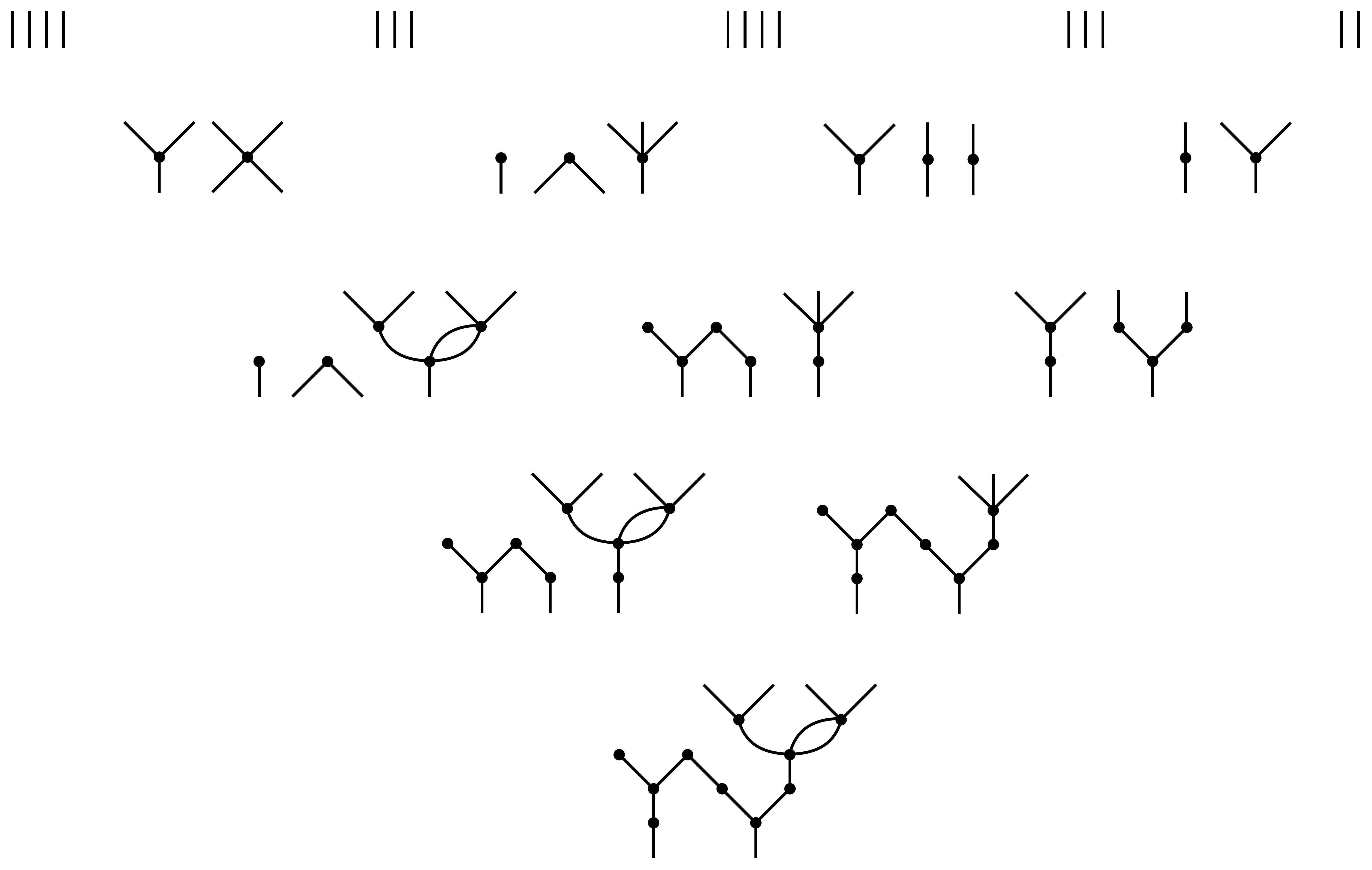}
\caption{Special functor associated to a level graph of height four.}
\label{figure height four level graph}
\end{figure}

\begin{remark}\label{remark level sugraphs}
Each $(i,j)$-level subgraph of $G$ determines a connected, height $(j-i)$-level graph.
Specifically, given a special functor $F\colon \scriptyell^{n} \to \finset$ and an element $x\in F_{i,j}$, define for $i \leq k \leq \ell \leq j$ a set $\widetilde F_{k,\ell}$ as the pullback
\[
\begin{tikzcd}
\widetilde F_{k,\ell} \rar \dar & F_{k,\ell} \dar \\
\{ x \} \rar & F_{i,j}.
\end{tikzcd}
\]
This determines a functor $K\colon \scriptyell^{j-i} \overset\cong\to \scriptyell^n_{i,j} \to \finset$ (see Definition~\ref{partial scriptyell}) by $K_{a,b} = \widetilde F_{a+i,b+i}$.
The functor $K$ is special because the pullback functor $\finset_{/ F_{i,j}} \to \finset_{/ \{ x \}}$ preserves colimits. 
The underlying level graph $K|_{\scriptyell^{j-i}_0}$ is connected by Lemma~\ref{lemma F0n connected} since $K_{0,j-i} = \widetilde F_{i,j} \cong \{x\}$.
\end{remark}

We now give a crucial definition of this section, which is inspired by the $\Phi$-sequences of \cite{bar} (though we are not actually considering a category of $\Phi$-sequences; see Remark~\ref{remark level not phi seq} below).
It will tell us that any morphism of level graphs which \emph{fixes levels} is a monomorphism on edge and vertex sets \eqref{M def mono}, and preserves sets of edges incident to a given vertex \eqref{M def cartesian}.

\begin{definition}\label{def M}
	Define a functor $M\colon \simp^\op\to \xCat$ by declaring that $M_n$ is the wide subcategory of $\widetilde{M}_n$ containing those morphisms $F\to F'$ which satisfy the following two properties:
	\begin{enumerate}
		\item For every $i\leq j$, the map $F_{i,j}\to F'_{i,j}$ is a monomorphism in $\finset$. \label{M def mono}
		\item For every $0\leq k\leq i\leq j\leq \ell\leq n$, the naturality square
		\begin{equation*}
		\begin{tikzcd}
		F_{i,j}\ar{r}\ar{d}& F'_{i,j}\ar {d}\\
		F_{k,\ell}\ar{r}& F'_{k,\ell}
		\end{tikzcd}
		\end{equation*}
		is cartesian in $\finset$. \label{M def cartesian}
	\end{enumerate}
\end{definition}

\begin{definition}[Category of level graphs]\label{def levelg}
Let $\xE\to \simp$ denote the Grothendieck fibration associated to the functor $M$.
We write $\levelg \to \simp$ for the Grothendieck fibration\footnote{Since $\simp$ is a skeletal category, the composite $\levelg \to \xE \to \simp$ is again a Grothendieck fibration.} obtained by choosing a skeleton $\levelg$ of $\xE$.
We call $\levelg$ the \emph{category of level graphs}.
Write $\levelg_n$ for the fiber over $[n] \in \simp$.
\end{definition}

\begin{remark}\label{rem doublecat}
It follows from the previous definition that $\levelg\to \simp$ is associated to a categorical object in $\Cat$, i.e.\ $\levelg_n\simeq \levelg_1\times_{\levelg_0}\ldots\times_{\levelg_0} \levelg_1$.
In other words, it induces a double categorical structure.
\end{remark}

\begin{notation}\label{not: check}
	For each height $n$ level graph $G$, there is a unique special functor $F \colon \scriptyell^n \to \finset$ which is an object of $\levelg_n$ so that $F|_{\scriptyell^n_0}$ is isomorphic to $G$.
	Further, every object of $\levelg$ arises in this way.
	We therefore simplify matters and identify every level graph with its corresponding object in $\levelg$, and call objects in $\levelg$ level graphs as well.
	Henceforth, we generally use $G$ to denote a special functor $\scriptyell^{n} \to \finset$ since it is essentially the same thing as the level graph $G|_{\scriptyell^n_0}$.
\end{notation}

Let us unravel Definition~\ref{def levelg} explicitly.
Suppose that $G$ and $H$ are two level graphs, of height $n$ and $m$ respectively.
Then a morphism from $G$ to $H$ consists of two pieces of data:
	\begin{itemize}
		\item A map $ \alpha \colon [n]\to [m]$ in $\simp$ and
		\item A natural transformation
		\[ \begin{tikzcd}
		\scriptyell^n  \arrow[rr, "\scriptyell^\alpha"] \arrow[dr, bend right, "G" swap, ""{name=B
		}]  & & \scriptyell^m \arrow[Rightarrow, from=B, "\eta" description] \arrow[dl, bend left, "H"]\\
		 & \finset
		\end{tikzcd} \]
		from $G$ to $H \circ \scriptyell^{\alpha}$.
	\end{itemize}
These should satisfy the following two conditions:
\begin{enumerate}
 	\item For $0 \leq i \leq j \leq n$, the map $\eta_{i,j} \colon G_{i,j} \to H_{\alpha(i),\alpha(j)}$ is a monomorphism.
 	\item For every $0\leq k\leq i\leq j\leq \ell\leq n$, the naturality square
		\begin{equation*}
		\begin{tikzcd}
		G_{i,j}\ar[r,"\eta_{i,j}"] \ar{d}& H_{\alpha(i),\alpha(j)}\ar {d}\\
		G_{k,\ell}\ar[r,"\eta_{k,\ell}"] & H_{\alpha(k),\alpha(\ell)}
		\end{tikzcd}
		\end{equation*}
		is a pullback.
 \end{enumerate}

Given a 1-category $\xcc$, a \emph{factorization system} (or \emph{orthogonal} factorization system) consists of a pair of subcategories $(\xcc^L, \xcc^R)$, each containing all isomorphisms of $\xcc$, so that each morphism $f$ admits a factorization $f = r\circ \ell$ where $r\in \xcc^R$ and $\ell \in \xcc^L$, and this factorization is unique up to \emph{unique} isomorphism (cf., \cite[Proposition 14.7]{AdamekHerrlichStrecker:ACCJC}).
This notion is subsumed by Definition~\ref{definition inf cat factorization system} in the $\infty$-categorical context.

\begin{example}\label{ex: bbG}
The category $\pfinset$ of finite pointed sets coincides with the opposite category of Segal's category $\segalGamma$ introduced in \cite[Definition 1.1]{SegalCatCohlgy}. 
In particular, a map $f\colon A_+\to B_+$ in $\pfinset$ can be identified with a partial map from $A$ to $B$. 
The category $\pfinset$ has an inert-active factorization system (see, for instance, \cite[Remark 2.1.2.2]{ha}), where a map $f\colon A_+\to B_+$ is \emph{inert} if $|f^{-1}(b)|=1$ for every $b\in B$, and \emph{active} if $f^{-1}(*)= \{*\}$. 
This factorization system restricts to one on the skeleton, $\pfinsetskel$, of $\pfinset$.
\end{example}

\begin{example}\label{ex simpfactsys}
Let $\alpha\colon [m]\to [n]$ be a morphism in $\simp$. We call it \emph{active} if it is boundary preserving, i.e.\ $\alpha(0)=0$ and $\alpha(m)=n$, and \emph{inert} if there is a constant $c_\alpha$ so that $\alpha(t) = c_\alpha + t$ for all $t$. 
It is standard that this constitutes an active-inert factorization system on $\simp$.
\end{example}

\begin{definition}\label{definition: int and el}
Let $(\alpha, \eta) \colon G \to H$ be a morphism in $\levelg$ where $\alpha\colon [m]\to [n]$.
\begin{itemize}
\item The map is called $\emph{inert}$ if $\alpha$ is inert in $\simp$.
\item The map is called \emph{active} if $\alpha$ is active in $\simp$ and $\eta_{i,j} \colon G_{i,j} \to H_{\alpha(i), \alpha(j)}$ is an isomorphism for every $0\leq i\leq j\leq m$.
\end{itemize} 
We have the following three subcategories of $\levelg$:
\begin{itemize}
\item Write $\levelint$ for the wide subcategory of $\levelg$ consisting of the inert maps.
\item Write $\levelel$ for the full subcategory of $\levelint$ spanned by the elementary graphs from Example~\ref{example: elementary graphs}.
\item Write $\levelg_{\name{act}}$ for the wide subcategory of $\levelg$ containing only active morphisms.
\end{itemize}
\end{definition}

\begin{remark}\label{rem underlying graph}
Inert maps are, in particular, monomorphisms.
Suppose $(\alpha, \eta) \colon G \to H$ is an inert map in $\levelg$. 
As $\alpha$ is inert, it is of the form $\alpha(i) = i + t$.
By Definition~\ref{def M}\eqref{M def mono}, $G_{i, j} \to H_{\alpha(i),\alpha(j)} = H_{i+t,j+t}$ is a monomorphism.
It follows that 
\begin{align*} 
\edge(G) = \coprod G_{i,i} &\to \coprod H_{k,k} = \edge(H)  & \text{and} \\
\vertex(G) = \coprod G_{i-1,i} &\to \coprod H_{k-1,k} = \vertex(H)
\end{align*} are monomorphisms.  
Together with Lemma~\ref{lemma: level graph}, we see that every inert morphism in $\levelg$ determines an inclusion of a subgraph.
\end{remark}

\begin{remark}[Weaker condition for active maps]\label{remark weaker active maps}
To show that a map as in Definition~\ref{definition: int and el} is active, it is enough that $\alpha$ be active and $\eta_{0,m} \colon G_{0,m} \to H_{0,n}$ be a bijection.
Indeed, for every $0\leq i \leq j \leq m$, the diagram
\[ \begin{tikzcd}
G_{i,j} \rar{\eta_{i,j}} \dar & H_{\alpha(i),\alpha(j)} \dar \\
G_{0,m} \rar["\eta_{0,m}", "\cong" swap] & H_{\alpha(0),\alpha(m)}
\end{tikzcd} \]
is a pullback by \eqref{M def cartesian} of Definition~\ref{def M}, which implies that $\eta_{i,j}$ is an isomorphism.
\end{remark}

\begin{lemma}\label{lem Lfs}
The pair of subcategories $(\levelg_{\name{act}},\levelint)$ constitute an orthogonal factorization system on $\levelg$.
\end{lemma}
\begin{proof}
Let us construct a factorization of a morphism $(\alpha, \eta)\colon (n,G) \to (m,H)$ in $\levelg$.
For $0\leq i \leq j \leq \alpha(n) - \alpha(0) = p$, let $K_{i,j}$ be the pullback
\[
\begin{tikzcd}
K_{i,j} \rar \dar & H_{i+\alpha(0),j+\alpha(0)} \dar \\
G_{0,n} \rar[hook, "\eta_{0,n}"] & H_{\alpha(0),\alpha(n)}.
\end{tikzcd}
\]
As in Remark~\ref{remark level sugraphs}, the functor $K\colon \scriptyell^p \to \finset$ is a special functor.
Letting $\beta \colon [p] \to [m]$ be $ \beta(i) = i + \alpha(0)$ and $\gamma \colon [n] \to [p]$ be $\gamma(i) = \alpha(i) - \alpha(0)$, we have a factorization
\[ \begin{tikzcd}
G \rar \dar[dotted,no head] & K \rar \dar[dotted,no head] & \beta^*H \rar \dar[dotted,no head] & H \dar[dotted,no head] \\
{[n]} \rar["\gamma"] & {[p]} \rar["\id"] & {[p]} \rar["\beta"] & {[m]}
\end{tikzcd} \]
lying above the usual active-inert factorization in $\simp$.
The morphism $(\gamma, \varepsilon)\colon G \to K$ is defined via the diagram
\[ \begin{tikzcd}[column sep=small, row sep=small]
G_{i,j} \arrow[dr, dashed, "\varepsilon_{i,j}"]
\arrow[drr, bend left, "\eta_{i,j}" swap]
\arrow[ddr, bend right]
& &[1.8em]
 \\
&
K_{\gamma(i),\gamma(j)} \rar\dar & H_{\alpha(i),\alpha(j)}\dar \\[1.8em]
&
G_{0,n} \rar & H_{\alpha(0),\alpha(n)};
\end{tikzcd} \]
since both the inner and outer squares are pullbacks, $\varepsilon_{i,j}$ is a bijection.
It follows that $(\gamma, \varepsilon)$ is active.
The map $K \to H$ is inert since $\beta$ is.

Suppose we have some other factorization
\[ \begin{tikzcd}
G \rar{(\gamma', \varepsilon')} & K' \rar{(\beta', \mu')} & H
\end{tikzcd} \]
into an active map followed by an inert map.
By uniqueness of factorizations in $\simp$, we know that $\gamma' = \gamma$ and $\beta' = \beta$.
Since $\gamma$ is active, for each $0 \leq i \leq j \leq p$, there exists $0\leq k \leq \ell \leq n$ with $0\leq \gamma(k) \leq i \leq j \leq \gamma(j) \leq p$.
Since $(\gamma, \varepsilon')$ is active, we know that $\varepsilon'_{k,\ell} \colon G_{k,\ell} \to K'_{\gamma(k),\gamma(\ell)}$ is a bijection, hence we have an isomorphism $\varepsilon'_{k,\ell} \varepsilon^{-1}_{k,\ell}  \colon K_{\gamma(k),\gamma(\ell)} \cong K'_{\gamma(k),\gamma(\ell)}$.
If $t=\beta(0)$, then \eqref{M def cartesian} of Definition~\ref{def M} gives that both squares in
\[ \begin{tikzcd}
K_{i,j} \arrow[r, hook] \arrow[r, hook] \dar& H_{i+t,j+t} \dar &  K'_{i,j} \arrow[l, hook'] \dar \\
K_{\gamma(k),\gamma(\ell)} \arrow[r, hook] \arrow[rr, bend right=20, "\cong"] & H_{\alpha(k), \alpha(\ell)} & K'_{\gamma(k),\gamma(\ell)} \arrow[l, hook']
\end{tikzcd} \]
are pullbacks, exhibiting an isomorphism so that the following diagram commutes.
\[ \begin{tikzcd}[row sep=small]
& K \arrow[dr]  
\ar[dd, "\cong"]
\\
G \arrow[ur] 
\arrow[dr]
& & H \\
& K' \ar[ur]
\end{tikzcd} \]
As inert maps are monomorphisms (Remark~\ref{rem underlying graph}), there is at most one such isomorphism $K\to K'$ making the right triangle commute.
Thus every morphism in $\levelg$ factors as an active map followed by an inert map, and this factorization is unique up to unique isomorphism.
\end{proof}

We now want to define a vertex functor $\vertex_{\levelg}\colon \levelg \to \pfinset^\op$ which takes every level graph to the set of its vertices with added base point.
Recall that if $H$ is a height $m$ level graph, then we have a decomposition $\vertex(H) = \coprod_{k=0}^{m-1} H_{k,k+1}$.

\begin{definition}[The functor $\levelg\to\pfinset^\op$] \label{def vertex levelg}
Suppose that $G\colon \scriptyell^n \to \finset$ and $H\colon \scriptyell^m \to \finset$ are level graphs and $f\colon G \to H$ is a morphism of $\levelg$ lying over $\alpha \colon [n] \to [m]$.
Define 
\[
	\vertex_{\levelg} (f) \colon \vertex(H)_+ \rightarrow \vertex(G)_+
\]
by specifying that $\vertex_{\levelg} (f) (v) = w$ just when there is a commutative square
\[ \begin{tikzcd}
\ast \rar{v} \dar{w} & H_{k,k+1} \dar \\
G_{i,i+1} \rar{f} & H_{\alpha(i), \alpha(i+1)},
\end{tikzcd} \]
and otherwise $\vertex_{\levelg} (f) (v) = \ast$.
\end{definition}

In other words, $\vertex_{\levelg} (f) (v) = w$ if $v$ is in the level subgraph associated to $f(w)$ and $\vertex_{\levelg} (f) (v)$ is the base point otherwise.
The following proposition shows that this rule defines a functor.
\begin{proposition}
The function $\vertex_{\levelg} (f)$ from the previous definition is well-defined.
Further, this assignment is compatible with composition.
\end{proposition}
\begin{proof}
First, note that for any $k$ there is at most one $i$ so that
\begin{equation}\label{map in scriptyellm}
(k,k+1) \to (\alpha(i),\alpha(i+1))
\end{equation}
exists in $\scriptyell^m$.
Suppose that there are $i < i'$ so that \eqref{map in scriptyellm} exists.
The existence of \eqref{map in scriptyellm} for $i$ means that $k+1 \leq \alpha(i+1)$, while \eqref{map in scriptyellm} for $i'$ implies $\alpha(i') \leq k$.
But $i+1 \leq i'$ and so we have $k+1 \leq \alpha(i+1) \leq \alpha(i') \leq k$, a contradiction.

If $k$ is such that \eqref{map in scriptyellm} does not exist for any $i$, then $\vertex_\levelg(f)$ takes every vertex in $H_{k,k+1}$ to the base point.
Otherwise, form the pullback
\[
\begin{tikzcd}
P \rar[hook] \dar & H_{k,k+1} \dar \\
G_{i,i+1} \rar[hook] & H_{\alpha(i),\alpha(i+1)};
\end{tikzcd}
\]
since the bottom map is a monomorphism, so is the top. Hence it gives us a morphism $\vertex (H_{k,k+1})_+ \rightarrow \vertex(P)_+$ which takes every vertex in $\vertex (H_{k,k+1})$ that does not lie in the image of the monomorphism to the base point.
Composing with $\vertex(P)_+ \to \vertex(G_{i,i+1})_+ \hookrightarrow \vertex(G)_+$, we attain our desired map $\vertex_\levelg(f) \colon \vertex(H)_+ \rightarrow\vertex(G)_+$.

We now show that $\vertex_{\levelg}$ is a functor.
It clearly sends identities to identities.

Consider maps $f \colon G \to H$ and $f' \colon H \to K$ (lying over $\alpha$ and $\beta$, respectively).
We wish to show that $\vertex_{\levelg}(f'f) = \vertex_{\levelg}(f)\vertex_{\levelg}(f')$. 
Let $S\subseteq \vertex(K)$ be the preimage of $\vertex(G)$ under $\vertex_{\levelg}(f'f)$, and let $T\subseteq \vertex(K)$ be the preimage of $\vertex(G)$ under $\vertex_{\levelg}(f)\vertex_{\levelg}(f')$.
Our goal is to show that $S=T$ and that the two functions agree on this subset.
To that end, let $x\in K_{\ell, \ell+1}$ and consider the following diagram.
\[ \begin{tikzcd}
\ast \arrow[rr, "x"] \arrow[dr, "w"] \arrow[dd, "u"] & & K_{\ell, \ell+1} \dar \\
& H_{k,k+1} \rar{f'} \dar & K_{\beta(k),\beta(k+1)} \dar \\
G_{i,i+1} \rar{f} & H_{\alpha(i),\alpha(i+1)} \rar{f'} & K_{\beta\alpha(i), \beta\alpha(i+1)}
\end{tikzcd} \]
It instantly shows that if $\vertex_{\levelg}(f') (x) = w$ and $\vertex_{\levelg}(f) (w) = u$, then $\vertex_{\levelg}(f'f)(x) = u = \vertex_{\levelg}(f) (\vertex_{\levelg}(f') (x))$.
Hence $T\subseteq S$, and the two functions agree on $T$.

To see that $\vertex_{\levelg}(f'f) = \vertex_{\levelg}(f)\vertex_{\levelg}(f')$, it remains to show that $S\subseteq T$.
Suppose that $x \in K_{\ell, \ell+1}$ is such that $\vertex_{\levelg}(f'f)(x) = u \in \vertex(G)$, that is, so that the outer rectangle exists.
Since $\beta \alpha(i) \leq \ell < \ell+1 \leq \beta\alpha(i+1)$, there exits a unique $k$ so that $\alpha(i) \leq k < \alpha(i+1)$, $\beta(k) \leq \ell$, and $\beta(k+1) > \ell$.
It follows that the lower-right rectangle exists in the diagram; it is a pullback since $f'$ is a morphism of $\levelg$.
Thus the indicated $w$ exists and we have that $S\subseteq T$.
\end{proof}

\begin{proposition}\label{prop active inert bbG}
	The functor $\vertex_{\levelg}\colon \levelg\to \pfinset^\op$ preserves the active-inert factorization systems.
\end{proposition}
\begin{proof}
If $(\alpha,\eta)\colon G\to H$ is an active map in $\levelg$, then $\alpha\colon [n]\to [m]$ is boundary preserving.
Therefore, if a vertex $v\in \vertex(H)$ lies in $H_{k, k+1}$ for some $0\leq k\leq m$, then by the previous proposition there exists a unique $0\leq i\leq n$ such that $\alpha(i)\leq k$ and $k+1\leq \alpha(i+1)$.
Since $\eta_i\colon G_{i,i+1}\to H_{\alpha(i), \alpha(i+1)}$ is an isomorphism, the composite $H_{k,k+1}\to H_{\alpha(i), \alpha(i+1)}\to G_{i,i+1}$ exists. It induces a map $\vertex_\levelg(f) \colon \vertex(H)_+ \rightarrow \vertex(G)_+$ by the construction of Definition~\ref{def vertex levelg} which takes only the base point to the base point, hence, $\vertex_{\levelg}$ preserves active morphisms.

Now suppose $(\alpha,\eta)\colon G\to H$ is inert and $v\in \vertex(H)$ lies in $H_{k, k+1}$ for some $k$.
Since $\alpha$ is inert, i.e.\ an interval inclusion, the construction of $\vertex_\levelg(f) \colon \vertex(H)_+ \rightarrow \vertex(G)_+$ shows that the preimage of every element in $\vertex(G)\subseteq \vertex(G)_+$ has cardinality $1$.
This proves that the functor $\vertex_{\levelg}$ also preserves inert maps.
\end{proof}
\begin{notation}
		We will also write $\vertex_\levelg$ for a composite $\levelg\to \pfinset^\op\simeq \pfinsetskel^\op$.
\end{notation}

We now discuss several important subcategories of $\levelg$.
Recall that the any level graph has an underlying directed graph (see Lemma~\ref{lemma: level graph}).
See Figure~\ref{figure lattice of subcategories} for a schematic.

\begin{definition}[Other categories of level graphs]
\label{def subcategories of levelg}
We define several full subcategories of $\levelg$.
\begin{itemize}
\item The most important is $\levelgconn$, which is the full subcategory on the \emph{connected} level graphs.
By Lemma~\ref{lemma F0n connected}, these are characterized by the property that the associated special functor $F \colon \scriptyell^n \to \finset$ preserves terminal objects.
\item The full subcategory $\levelg_{\name{sc}}$ consists of the \emph{simply-connected} level graphs.
\item There is also the larger subcategory $\levelg_{\name{0-type}}$, consisting of possibly disconnected graphs where each connected component is simply-connected.
\item The full subcategory $\levelg_{\name{out}}$ consists of those level graphs $G$ so that each vertex has at least one output. That is, the functions $G_{k,k} \to G_{k-1,k}$ are surjective.
\item The full subcategory $\levelg_{\name{out,c}}$ consists of those level graphs $G$ which are both connected and so that each vertex has at least one output.
\item Likewise there is a subcategory consisting of level graphs $G$ which so that each $G_{k,k} \to G_{k,k+1}$ surjective, and a subcategory of that where additionally the graphs are connected.
By symmetry, all theorems concerning $\levelg_{\name{out}}$ and $\levelg_{\name{out,c}}$ also hold for these nonempty input versions, so we will not mention them again.
\item There is a subcategory of $\levelg$ consisting of the level \emph{forests}, which are those graphs level $G$ so that $G_{k,k} \to G_{k-1,k}$ is bijective.
Following Proposition~\ref{prop DF} below we identify this category with the category $\simp_\finsetskel$ of $\finset$-sequences defined in \cite{bar}. 
Likewise, there is subcategory of level \emph{trees}, which are the level forests $G$ so that $G_{n,n}$ is a point. 
Again by Proposition~\ref{prop DF} this category is equivalent to the category $\simp^1_\finsetskel$ defined in \cite{ChuHaugsengHeuts}.
\item The simplex category $\simp$ can be identified with a full subcategory of all of the above.
It is spanned by the linear graphs, i.e.\ those graphs $G$ such that $G_{k,k}$ is a point for all $k$.
\end{itemize}
\end{definition}

\begin{remark}
Restricting the Grothendieck fibration $\levelg \to \simp$ to any of the subcategories $\levelg_{\name{0-type}}$, $\levelg_{\name{out}}$, $\simp_\finsetskel$, or $\simp$ again gives a Grothendieck fibration.
The same is not true for $\levelgconn$, $\levelg_{\name{sc}}$, $\levelg_{\name{out,c}}$, and $\simp^1_\finsetskel$, as the domain of a cartesian lift in $\levelg$ over an inert morphism is sometimes disconnected, even when the codomain is connected.
\end{remark}

\begin{remark}[Simply-connected graphs]
Most of the types of level graphs above are characterized by clear properties of $G \colon \scriptyell^n_0 \to \finset$ or of the associated special functor -- certain maps are surjections or bijections, or certain sets like $G_{0,n}$ or $G_{n,n}$ are terminal.
The exceptions are $\levelg_{\name{sc}}$ and $\levelg_{\name{0-type}}$.
Though it is possible to give a criterion for a level graph $G$ to be in one of these subcategories strictly in terms of the associated special functor, the only criterion we know is somewhat unsatisfactory, as it simply reformulates the notion of `undirected path' in the underlying directed graph.
It does not seem worthwhile to do so here, as we have developed other tools in this paper that do the job just as well.
Namely, in \S\ref{section level to acyclic} we define a functor $\tau \colon \levelgconn \to \bbY$ which takes a level graph to its underlying directed graph.
The codomain of $\tau$ is the properadic graphical category from \cite{hrybook}, so we may use theorems therein.
For example, we have the following proposition, which says that $\levelg_{\name{sc}} \subseteq \levelgconn$ and $\levelg_{\name{0-type}} \subseteq \levelg$ are both sieves. 
We emphasize that there is no logical circularity here: nothing in \S\ref{section subcategory bbY} and \S\ref{section level to acyclic} (including Appendix~\ref{appendix comparison with HRY} and Appendix~\ref{appendix proof l to y}) depends on this result.
\end{remark}

\begin{lemma}\label{lemma simply connected sieves}
Suppose that $G \to H$ is a map in $\levelg$ with $H\in \levelg_{\name{0-type}}$. Then $G$ is also an object of $\levelg_{\name{0-type}}$.
In particular, if $G$ is connected then $G\in \levelg_{\name{sc}}$.
\end{lemma}
\begin{proof}
Let $G'$ be an arbitrary connected component of $G$, and let $H'$ be the unique connected component of $H$ so that the following factorization exists
\[ \begin{tikzcd}
G' \rar[dashed] \dar & H'\dar  \\
G \rar & H.
\end{tikzcd} \]
Apply $\tau$ from \S\ref{section level to acyclic} to the morphism $G' \to H'$ of $\levelgconn$.
By assumption $\tau H'$ is simply-connected, so \cite[Proposition 5.2.8]{hrybook} gives that the graph $\tau G'$ is also simply-connected.
Each component of $G$ is simply-connected, hence $G\in \levelg_{\name{0-type}}$.
\end{proof}

We will show that several other of these subcategory inclusions are sieves in Lemma~\ref{lem comparison between levelgraph cats} below.

\begin{lemma}\label{lemma general level graph factorization systems}
The active-inert factorization system on $\levelg$ restricts to one on each of the subcategories from Definition~\ref{def subcategories of levelg}.
\end{lemma}
\begin{proof}
Let $G\to K\to H$ be the active-inert factorization of a morphism in $\levelg$, where $G$ has height $m$ and $K$ has height $n$. 
If $G$ is connected, then Lemma~\ref{lemma F0n connected} implies that $G_{0,m}\cong \{*\}$. 
By definition of active morphisms we have $G_{0,m}\isoto K_{0,n}$ where $n$ is the height of $K$, hence, $K$ is connected as well. 
It follows that the active-inert factorization system on $\levelg$ restricts to one on $\levelgconn$.

Now suppose that $G$ is connected and $H$ is simply-connected.
By the previous paragraph, we know that $K$ is connected,
and by Lemma~\ref{lemma simply connected sieves}, $K$ is simply-connected as well.
Thus the active-inert factorization system on $\levelgconn$ further restricts to one on $\levelg_{\name{sc}}$.

The remaining cases are proved in a similar manner, with $K$ inheriting properties from $G$ or $H$.
\end{proof}

\begin{proposition}\label{prop DF}
	There is an equivalence between the full subcategory of $\levelg$ consisting of the forests and the category $\simp_{\finsetskel}$ defined in \cite{bar}.
\end{proposition}
\begin{proof}
	It follows from \cite[Proposition 4.3.2.15]{ht} or classical results about Kan extension that $\widetilde{M}_n \simeq \xFun(\scriptyell^n_0,\finset)$ for every $n\geq 0$ (see Definition~\ref{definition special and tilde M}).
	This equivalence shows that the full subcategory in $\widetilde{M}_n$ spanned by forests can be identified with $\xFun([n],\finset)$. The two conditions in Definition~\ref{def M} implies that $M_n$ is equivalent to the subcategory of $\finset$-sequences in $\xFun([n],\finset)$ in the sense of \cite[Definition 2.4]{bar}. Hence, we have an equivalence of the corresponding cartesian fibrations, which implies the result.
\end{proof}

\begin{remark}
In \cite{bar}, Barwick uses the full subcategory $\simp_{\finsetskel}$ of $\levelg$ to construct `complete Segal operads' and proves that they are equivalent to Lurie's $\infty$-operads. 
In \cite{ChuHaugsengHeuts} the authors further show that Barwick's approach to $\infty$-operads is equivalent to dendroidal Segal spaces, which are in turn equivalent to simplicial operads by \cite{CisinkiMoerdijk2}.
\end{remark}

\begin{figure}
\[ \begin{tikzcd}[column sep=small]
& \levelg \arrow[d, no head] \arrow[dr, no head] \arrow[dl, no head] \\
\levelg_{\name{0-type}} \arrow[d, no head] & \levelgconn \arrow[dl, no head] \arrow[dr, no head] & \levelg_{\name{out}} \arrow[d, no head] \arrow[dl, no head, crossing over] \\
\levelg_{\name{sc}} \arrow[dr, no head] & \simp_{\finsetskel} \arrow[d, no head] \arrow[ul, no head, crossing over] & \levelg_{\name{out,c}} \arrow[dl, no head] \\
& \simp_{\finsetskel}^1 \arrow[d, no head] \\
& \simp
\end{tikzcd} \]
\caption{The lattice of subcategories from Definition~\ref{def subcategories of levelg}} \label{figure lattice of subcategories}
\end{figure}
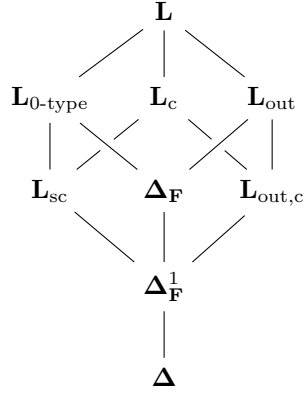

\begin{remark}\label{remark level not phi seq}
In contrast to the level forest situation, the category $\levelg$ does not arise from an operator category $\Phi$ in the sense of \cite[Definition 1.2]{bar}. 
Indeed, if $\levelg$ was the category of $\Phi$-sequences (\cite[Definition 2.4]{bar}) for some operator category $\Phi$, then
\begin{itemize}
\item height 0 level graphs would correspond to functors $[0] \to \Phi$, and
\item height 1 level graphs would correspond to functors $[1] \to \Phi$.
\end{itemize}
It would then follow that the objects of $\Phi$ are finite sets and the morphisms of $\Phi$ are cospans of finite sets.
In particular, each $\hom(X,Y)$ is infinite, so neither (1.2.1) nor (1.2.3) from \cite{bar} is satisfied.
\end{remark}

\begin{lemma}\label{lem comparison between levelgraph cats}
Suppose that $\mathbf{A} \subseteq \mathbf{B}$ is one of the fully-faithful inclusions appearing in the following diagram
\[ \begin{tikzcd}
\simp \rar & \simp_\finsetskel^1 \rar \dar & \levelg_{\name{out,c}} \dar & \simp_\finsetskel \rar \dar & \levelg_{\name{out}} \dar  \\
& \levelg_{\name{sc}} \rar[dashed, color=red] & \levelgconn & \levelg_{\name{0-type}}\rar[dashed, color=red]  & \levelg
\end{tikzcd} \]
including the two {\color{red} red dashed} maps.
\begin{enumerate}[label=(\arabic*), ref={\arabic*}]
	\item If $f \colon G \to H$ is a morphism in $\mathbf{B}$ and $H \in \mathbf{A}$, then $G\in \mathbf{A}$.\label{item one comparison between levelgraph cats}
\end{enumerate}
\noindent
Now suppose that $\mathbf{A} \subseteq \mathbf{B}$ is one of the solid black arrows.
\begin{enumerate}[resume*]
	\item If $G\in \mathbf{B}$ is such that for each corolla $\xfc$ of $\mathbf{B}$ and each inert map $\xfc \to G$, we have that $\xfc \in \mathbf{A}$, then $G\in \mathbf{A}$.\label{item two comparison between levelgraph cats}
\end{enumerate}
\end{lemma}
\begin{proof}
Item \eqref{item one comparison between levelgraph cats} for the two {\color{red} red dashed} maps is exactly the content of Lemma~\ref{lemma simply connected sieves}.
We thus concentrate on the cases where $\mathbf{A} \subseteq \mathbf{B}$ is one of the solid black arrows. 
In each case, we can characterize objects $F$ in $\mathbf{A}$ among those in $\mathbf{B}$ by the following conditions on the functions comprising $F\colon \scriptyell^{n}_0 \to \finset$
\[ \begin{tikzcd}[sep=huge]
\simp \rar["\color{blue}F_{k,k} \cong F_{k+1,k}"] & \simp_\finsetskel^1 \rar["\color{blue}F_{k,k} \cong F_{k-1,k}"] \dar["\color{blue}F_{k,k} \cong F_{k-1,k}" description] & \levelg_{\name{out,c}} \dar["\color{blue}F_{k,k} \twoheadrightarrow F_{k-1,k}" description]  \\
& \levelg_{\name{sc}} & \levelgconn
\end{tikzcd} \]
and likewise for the other solid black arrow maps.
As bijections and surjections are preserved by pushouts, we get extended conditions when working with special functors $F\colon \scriptyell^{n} \to \finset$.
For example, a special functor $F\colon \scriptyell^{n} \to \finset$ with $F_{0,n} = \ast$ (that is, a connected graph) is in $\levelg_{\name{out,c}}$ if and only if $F_{k,\ell} \to F_{k-1, \ell}$ is a surjection whenever $0 < k \leq \ell \leq n$.
It follows that if $H\in \levelg_{\name{out,c}}$ and $G\to H$ is a morphism, then we have a pullback square
		\begin{equation*}
		\begin{tikzcd}
		G_{i,i}\ar[r] \dar  \arrow[dr, phantom, "\lrcorner", very near start] & H_{\alpha(i),\alpha(i)} \dar[two heads] \\
		G_{i-1,i}\ar[r] & H_{\alpha(i-1),\alpha(i)}
		\end{tikzcd}
		\end{equation*}
which implies $G \in \levelg_{\name{out,c}}$ as well.
The other six black arrow cases of \eqref{item one comparison between levelgraph cats} follow similarly.

For \eqref{item two comparison between levelgraph cats} again consider the case $\levelg_{\name{out,c}} \to \levelgconn$, as the others are similar.
Fixing $k$, for each $x \in G_{k-1,k}$ there is a corresponding inert map $\xfc_x \to G$, which assembles into an inert map $\coprod_x \xfc_x \to G$ in $\levelg$.
This map is a bijection on the level $k-1$ and $k$ edges and the level $k$ vertices, and we see that $G_{k,k} \to G_{k-1,k}$ is surjective if and only if each $\xfc_x$ is in $\levelg_{\name{out,c}}$.
\end{proof}

\subsection{The category \texorpdfstring{$\bbY$}{G} of connected, acyclic graphs}\label{section subcategory bbY}
In \cite{hrybook} the authors introduced $\infty$-properads as presheaves on an indexing category $\hryGamma$, called the `graphical category' or `properadic graphical category.'
In this section, we give a new presentation for a category $\bbY$ (see Definition~\ref{definition complete morphism} and Definition~\ref{def Y}), and in Appendix~\ref{appendix comparison with HRY} we show that our $\bbY$ is indeed equivalent to the category $\hryGamma$.
Accordingly, going forward we will use whichever description of the properadic graphical category that is most convenient at the time, and to avoid duplication we will frequently use results from \cite{hrybook}.

Let $G$ and $H$ be directed graphs, and suppose that we are given a vertex $v\in G$ as well as bijections
$b_i \colon \inp(v) \cong \inp(H)$ and
$b_o \colon \out(v) \cong \out(H)$.
Then we can define a new graph $G(H)$, the \emph{graph substitution} of $H$ into $G$, where the vertex $v$ has been replaced by the graph $H$.
We call the quintuple $(G,H,v,b_i,b_o)$ \emph{graph substitution data}.
See Figure~\ref{figure graph substitution example} for an example (where all edges are oriented by gravity and the bijections are left implicit).
Basic facts about graph substitution (including associativity, unitality, and so on) may be found in \cite{YauJohnson:FPAM}.

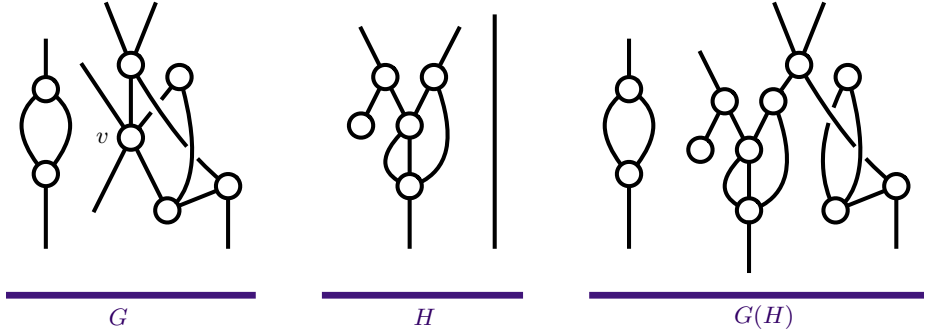
\begin{figure}[htb]
\labellist
\small\hair 2pt
 \pinlabel {$v$} [ ] at 84 133
 \pinlabel {$\color[HTML]{330066}G$} [ ] at 96 0
 \pinlabel {$\color[HTML]{330066}H$} [ ] at 323 0
 \pinlabel {$\color[HTML]{330066}G(H)$} [ ] at 575 0
\endlabellist
\centering
\includegraphics[width=\textwidth]{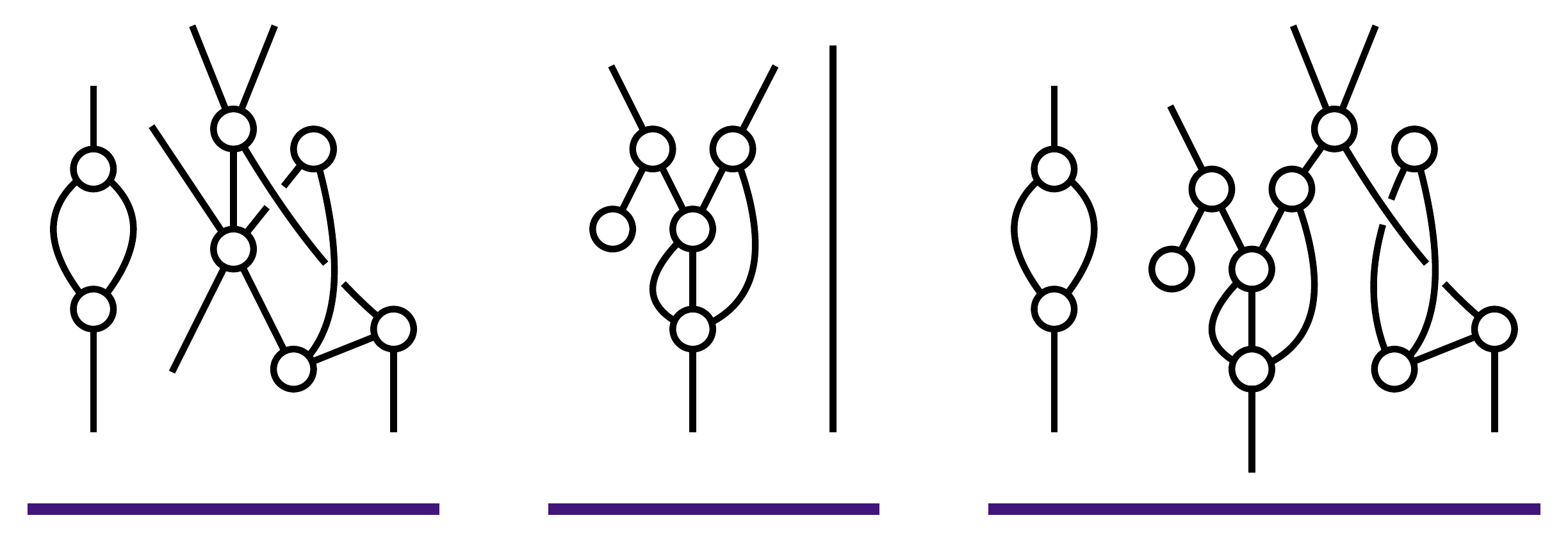}
\caption{Substitution of the graph $H$ into the vertex $v$ of $G$.}
\label{figure graph substitution example}
\end{figure}

\begin{definition}
Suppose that $K$ is a directed graph.
\begin{itemize}
\item An \emph{ordinary subgraph} $H$ of $K$ consists of a pair of subsets $\edge(H) \subseteq \edge(K)$ and $\vertex(H) \subseteq \vertex(K)$, with incidence data inherited from $K$.
That is, if $v\in \vertex(H)$, then we define $\inp^H(v)$ to be $\inp^K(v) \cap \edge(H)$, and likewise for $\out(v)$.
\item An \emph{open} subgraph $H$ of $K$ is an ordinary subgraph so that $\inp^H(v) = \inp^K(v)$ and $\out^H(v) = \out^K(v)$ for all $v\in \vertex(H)$.
\item If $G$ and $H$ are both open subgraphs of $K$, then write $G \ordcap H$ and $G\ordcup H$ for the open subgraphs with
\begin{align*}
	\edge(G \ordcap H) &= \edge(G) \cap \edge(H) & \vertex(G \ordcap H) &= \vertex(G) \cap \vertex(H) \\
	\edge(G \ordcup H) &= \edge(G) \cup \edge(H) & \vertex(G \ordcup H) &= \vertex(G) \cup \vertex(H).
\end{align*}
\end{itemize}
\end{definition}

Suppose that we are given graph substitution data so that $G(H)$ is well-defined.
Then $H$ is an open subgraph of $G(H)$.
We are primarily interested in acyclic graphs (that is, graphs without directed cycles), and in that case the reverse implication does not hold.

\begin{definition}\label{def structured subgraphs}
Suppose that $K$ is a connected acyclic directed graph.
\begin{itemize}
\item An ordinary subgraph $H$ of $K$ is said to be a \emph{structured subgraph} if
\begin{itemize}
	\item $H$ is connected, and
	\item there exists a connected acyclic directed graph $G$, graph substitution data $(G,H,v,b_i,b_o)$, and an isomorphism $G(H) \cong K$ which is the identity on $H$.
\end{itemize}
\item If $H$ is a structured subgraph of $K$, we write $H \strsub K$.
\item Write $\sub(K)$ for the set of structured subgraphs of $K$.
\end{itemize}
\end{definition}

Note that the data which guarantees that $H$ is a structured subgraph of $K$ is unique up to isomorphism, and we generally disregard it.
Also note that structured subgraphs are automatically open subgraphs.

We note that any vertex $v$ of $K$ determines a subgraph $C_v$, which consists of $v$ and all edges incident to $v$.
Indeed, $K(C_v)$ is canonically isomorphic to $K$ (unitality of graph substitution). Likewise, every edge $e$ of $K$ determines a subgraph $\downarrow_e$.

\begin{example}\label{example three vertex}
Consider the following three-vertex graph.
\begin{center}
\begin{tikzpicture}[new set=import nodes]
\begin{scope}[nodes={set=import nodes}] 
\node (i) at (0,2) {};
\node [circle,draw,very thick] (u) at (0,1) {$u$};
\node [circle,draw,very thick] (v) at (1,.5) {$v$};
\node [circle,draw,very thick] (w) at (0,0) {$w$};
\node (o) at (0,-1) {};
\end{scope}
\graph [edges={thick}]
{ (import nodes);
i -> u -> v -> w -> o; u -> w }; \end{tikzpicture}
\end{center}
The open subgraph with vertices $u$ and $w$ is \emph{not} a structured subgraph.
The open subgraph with vertices $u$ and $v$ \emph{is} a structured subgraph. Informally, collapsing the first subgraph down to a vertex yields a graph that contains a directed cycle, whereas collapsing the second graph down to a vertex yields a graph with no directed cycles.
\end{example}

\begin{remark}\label{rem convex}
	In \cite{hrybook}, structured subgraphs were simply called \emph{subgraphs}.
	In \cite{Kock_Properads}, structured subgraphs are called \emph{convex open subgraphs} and by reformulating \cite[1.6.5]{Kock_Properads} and \cite[1.6.10]{Kock_Properads} we obtain the following characterization of structured subgraphs:
	
	Suppose $G$ is an open subgraph of $K$.
	Then $G$ is a structured subgraph of $K$ if and only if the associated inclusion $i\colon G\to K$ has the right lifting property with respect to all na\"ive morphisms of graphs $\downarrow_0 \amalg \downarrow_m \to L$, where $L$ is a linear graph (see Definition~\ref{def subcategories of levelg}) of height $m\geq 0$ and 
	which take $\downarrow_0$ and $\downarrow_m$ to the unique edges in $\inp(L)$ and $\out(L)$, respectively. In other words, every commutative square
\[
\begin{tikzcd}	
\downarrow_0 \amalg \downarrow_m \rar \dar & G \dar{i} \\
L \rar \arrow[ur, dotted, "k" description] & K
\end{tikzcd}
\]
of na\"ive morphisms has a lift $k$.
\end{remark}

Notice in the case when $m=0$, the map $\downarrow_0 \amalg \downarrow_m \to L$ from the remark is just the 2-fold cover of $\downarrow$.

\begin{proposition}\label{strsub and ordsub coincide}
	Suppose that $G \strsub K$ is a structured subgraph of $K$ and $H$ is an ordinary subgraph of $G$.
	Then $H \strsub K$ if and only if $H \strsub G$.
\end{proposition}
\begin{proof}
According to the previous remark being a structured subgraph is equivalent to being right orthogonal to all na\"ive maps of inclusions of endpoints into linear graphs $L$.
If $H$ is a structured subgraph of $K$, then we can extend any commutative square 
\[
\begin{tikzcd}	
\downarrow_0 \amalg \downarrow_m \rar{a} \dar & H \dar{i_i} \\
L \rar{b} & G
\end{tikzcd}
\]
 to a commutative diagram
\[
\begin{tikzcd}
	\downarrow_0 \amalg \downarrow_m  \arrow[dd] \rar["a"] & H \dar{i_1}  \\
	& G \dar{i_2} \\
	L \rar["i_2b" swap] \arrow[ur, "b" description] \arrow[color=purple, dashed, uur, "c" description] & K
\end{tikzcd}
\]
The assumption $H \strsub K$ gives the existence of the map $c$ such that $cf=a$ and $i_2i_1c=i_2b$. 
It follows from the definition of ordinary subgraphs that the morphism $i_2$ is an monomorphism of graphs, hence, $i_1c=b$ and the `only if' direction is established.
The `if' direction is just closure under composition of maps in a right orthogonality class; alternatively, one can use associativity of graph substitution $K \cong G'(G) \cong G'(H'(H)) \cong ((G'(H'))(H)$.
\end{proof}

\begin{remark}
	The set $\sub(K)$ is a partially-ordered set, ordered by inclusion of subgraphs.
	In light of the previous proposition, it does not matter whether we take $H \leq G$ to mean `$H$ is an ordinary subgraph of $G$,' the stronger `$H$ is an open subgraph of $G$,' or the strongest `$H$ is a structured subgraph of $G$.'
	This partially-ordered set has a unique maximal element, $K$.
	Each edge of $K$ is a minimal element.
\end{remark}

\begin{definition}
Suppose that $G, H \in \sub(K)$.
If $G \ordcup H$ is a structured subgraph of $K$, we write $G \strcup H \coloneqq G \ordcup H \in \sub(K)$.
If not, we say that $G \strcup H$ does not exist.
\end{definition}

The following is immediate.

\begin{proposition}
Suppose that $G, H \in \sub(K)$.
If $G \ordcup H$ is a structured subgraph of $K$, then it is the least upper bound of $G$ and $H$. \qed
\end{proposition}

It is possible to have least upper bounds in $\sub(K)$ which are not of this form, and we are entirely uninterested in those.
For example, in the graph
\begin{center}
	\tikz \graph [math nodes, nodes={circle, draw, very thick}, edges=thick] { / [white] -> u -> v -> w -> / [white]};
\end{center}
the structured subgraphs $C_u$ and $C_w$ have a least upper bound, namely the entire graph.

\begin{example}
It may be that $H_1 \ordcup H_2$ is connected, but still is not a structured subgraph.
This should be immediate since $\vertex(K) \hookrightarrow \sub(K)$ is a monomorphism, but there can exist open connected subgraphs that are not structured subgraphs.
For example, in the graph from Example~\ref{example three vertex}, the open subgraph $C_u \ordcup C_w$ is connected, but is not a structured subgraph.
\end{example}

\begin{definition}
	If $K$ is a connected acyclic directed graph, then there are two functions $\inp$ and $\out$ that go from $\sub(K)$ to the set $\lists (\edge(K))$ of finite unordered lists of elements of $\edge(K)$.
	The first takes a subgraph $H$ to its set of input edges, while the second takes $H$ to its set of output edges.
\end{definition}

The following definition was inspired both by a question from Steve Lack and by Definition 1.12 of \cite{HackneyRobertsonYau:HCO}, which concerned morphisms in a category of undirected trees.

\begin{definition}\label{definition complete morphism}
Let $G$ and $K$ be connected acyclic directed graphs.
A \emph{morphism} $f \colon G\to K$ consists of two functions $f_0 \colon \edge(G) \to \edge(K)$ and $f_1 \colon \sub(G) \to \sub(K)$.
These should satisfy the following:
\begin{enumerate}
	\item The diagram
\[
\begin{tikzcd}
\lists (\edge(G))  \dar["\lists(f_0)"] &\sub(G)  \dar["f_1"] \lar["\inp" swap] \rar["\out"]& \lists (\edge(G))  \dar["\lists(f_0)"] \\
\lists (\edge(K))  & \sub(K) \lar["\inp" swap] \rar["\out"] & \lists (\edge(K))
\end{tikzcd}
\]
	commutes. \label{definition complete morphism preservation of ports}
	\item Suppose that $H_1, H_2 \in \sub(G)$ and $H_1 \strcup H_2$ exists, then $f_1(H_1 \strcup H_2) = f_1 (H_1) \strcup f_1 (H_2).$ \label{definition complete morphism intersection union}

\end{enumerate}
\end{definition}

This definition of morphism turns out to be equivalent to the (properadic) graphical maps from \cite[Definition 6.46]{hrybook}.
We show this in Theorem~\ref{theorem upsilon equivalent definitions}.
This implies that a morphism $f \colon G\to K$ is completely determined by $f_0$ by \cite[Corollary 6.62]{hrybook}.

The reader familiar with \cite[Definition 1.12]{HackneyRobertsonYau:HCO} may wonder why Definition~\ref{definition complete morphism} does not mention intersections of subgraphs.
This may be understood with an example.

\begin{example}\label{example no intersections}
Properadic graphical maps in the sense of \cite{hrybook} need not preserve intersections of subgraphs, hence this is not part of the definition of morphism.
Indeed, consider the following situation.
\begin{center}
\begin{tikzpicture}[new set=import nodes]
\begin{scope}[nodes={set=import nodes}] 
\node (i) at (0,2) {};
\node [circle,draw,very thick] (u) at (0,1) {$u$};
\node [circle,draw,very thick] (v) at (1,.5) {$v$};
\node [circle,draw,very thick] (w) at (0,0) {$w$};
\node (o) at (0,-1) {};
\node [circle,draw,very thick] (a) at (3,1) {$a$};
\node [circle,draw,very thick] (b) at (3,0) {$b$};
\node (i') at (3,2) {};
\node (o') at (3,-1) {};
\end{scope}
\graph [edges={thick}]
{ (import nodes);
i -> u -> v -> w -> o; u -> [bend right] w;
i' -> a -> [bend right] b; a -> [bend left, "$e$"] b -> o';
 }; \end{tikzpicture}
\end{center}
There is a properadic graphical map $f$ which sends $C_u$ to $C_a$, $C_w$ to $C_b$, and $C_v$ to $\downarrow_e$.
But $C_u \ordcap C_w$ consists of a single edge.
On the other hand, $f(C_u) \ordcap f(C_w) = C_a \ordcap C_b$ consists of \emph{two} edges, and is not a structured subgraph.
Morphisms of this type are essential in establishing the nerve theorem (see \cite[Theorem 7.42]{hrybook}) in this setting.
\end{example}

\begin{example}\label{example of etale but not mono}
Consider the following two graphs $G$ and $K$.
\begin{center}
\begin{tikzpicture}[new set=import nodes]
\begin{scope}[nodes={set=import nodes}] 
\node [circle,draw,very thick] (u0) at (0,2) {$u_0$};
\node [circle,draw,very thick] (u1) at (2,2) {$u_1$};
\node [circle,draw,very thick] (v0) at (0,0) {$v_0$};
\node [circle,draw,very thick] (v1) at (2,0) {$v_1$};
\node [circle,draw,very thick] (u) at (5,2) {$u$};
\node [circle,draw,very thick] (v) at (5,0) {$v$};
\end{scope}
\graph [edges={thick}]
{ (import nodes);
u0 -> {v0,v1};
u1 -- [line width=5pt, white] v0;
u1 -> {v0,v1};
u -> [bend right] v; u -> [bend left] v;
}; \end{tikzpicture}
\end{center}
There is no morphism $f \colon G \to K$.
Such a morphism would necessarily have $f_1(C_{u_0}) = f_1(C_{u_1}) = C_u$ as there is exactly one structured subgraph of $K$ with two outputs; likewise $f_1(C_{v_0}) = f_1(C_{v_1}) = C_v$.
But $H = C_{u_1} \strcup C_{v_1} \in \sub(G)$, so is required that $f_1(H) = f_1(C_{u_1}) \strcup f_1(C_{v_1}) = C_u \strcup C_v = K$.
But then $H$ has one input and one output, whereas $K$ has an empty set of inputs and empty set of outputs.
Likewise, there is no morphism $g \colon K \to G$.
If one existed, we would have $g_1(C_u) = C_{u_i}$ and $g_1(C_v) = C_{v_j}$ for some $i,j$.
But then $C_{u_i} \strcup C_{v_j}$ is a structured subgraph with one input and one output, whereas $C_u \strcup C_v = K$ has no inputs or outputs.
\end{example}

\begin{definition}[Properadic Graphical Category]\label{def Y}
We let $\bbY$ denote the category whose objects are (isomorphism classes of) connected acyclic directed graphs together with the morphisms of Definition~\ref{definition complete morphism}.
Composition is given by composition of pairs of functions. Similar to Definition~\ref{def subcategories of levelg} we write
\begin{itemize}
	\item $\bbYout$ for the full subcategory of $\bbY$ spanned by graphs whose vertices have at least one output and 
	\item $\bbY_{\name{sc}}$ for the full subcategory of $\bbY$ spanned by simply-connected graphs.
\end{itemize}
The dendroidal category, $\bbO$, from \cite{MoerdijkWeiss} is the full subcategory of $\bbYout$ so that each vertex has precisely one output.
\end{definition}

Under the equivalence of Theorem~\ref{theorem upsilon equivalent definitions}, the subcategory $\bbY_{\name{sc}}$ was called `$\Theta$' in \cite{hrybook}.
The dendroidal category $\bbO$ is a full subcategory of both $\bbYout$ and $\bbY_{\name{sc}}$.
See Figure~\ref{figure lattice of subcategories for bbY}.

\begin{figure}
\[ \begin{tikzcd}[column sep=small]
& \bbY \arrow[dl, no head] \arrow[dr, no head] & \\
\bbY_{\name{sc}} \arrow[dr, no head] &  & \bbYout \arrow[dl, no head] \\
& \bbO \arrow[d, no head] \\
& \simp
\end{tikzcd} \]
\caption{Lattice of subcategories (see Definition~\ref{def Y})} \label{figure lattice of subcategories for bbY}
\end{figure}
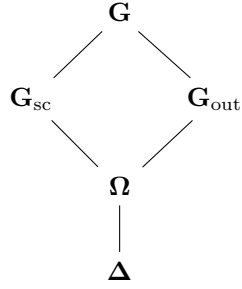

\begin{example}\label{ex inert}
	If $H \in \sub(K)$, then there is a canonical morphism $H \to K$ given by inclusion of subsets $\edge(H) \subseteq \edge(K)$ and $\sub(H) \subseteq \sub(K)$.
\end{example}

\begin{remark}[Alternative characterization of morphisms]\label{remark alternative characterization of morphisms}
Definition~\ref{definition complete morphism} is not the original definition of graphical map that appeared in \cite{hrybook}.
Let us give an alternative definition more closely aligned with the original.
Suppose that $f\colon G \to K$ is a morphism. 
We can identify each vertex with a corolla in $\sub(G)$, which induces a function $\vertex(G) \to \sub(K)$ so that the diagram
\[
\begin{tikzcd}
\lists (\edge(G))  \dar["\lists(f_0)"] &\vertex(G)  \dar["f_1"] \lar["\inp" swap] \rar["\out"]& \lists (\edge(G))  \dar["\lists(f_0)"] \\
\lists (\edge(K))  & \sub(K) \lar["\inp" swap] \rar["\out"] & \lists (\edge(K))
\end{tikzcd}
\]
commutes.
Condition \eqref{definition complete morphism intersection union} no longer makes sense, as taking the union of two corollas will yield a graph that is not a corolla.
Instead, we have the following two potential conditions (essentially borrowed from \cite{hrybook}), where we write $H_v \in \sub(K)$ for the image of $C_v$ under $f_1$.
\begin{enumerate}[start=3]
	\item For each $J\in \sub(G)$, the induced \'etale map $J\{ H_v \}_{v\in \vertex(J)} \to K$ is convex open. \label{alternate def subgraphs}
	\item The induced \'etale map $G\{H_v\}_{v\in \vertex(J)} \to K$ is convex open. \label{alternate def whole graph}
\end{enumerate}
Combining one of these two conditions with \eqref{definition complete morphism preservation of ports} yields a definition equivalent to the one from Definition~\ref{definition complete morphism}.
This equivalence follows from Theorem~\ref{theorem upsilon equivalent definitions}.
\end{remark}

\begin{definition}\label{def intact Y}
Let $f\colon G \to K$ be a morphism.
\begin{itemize}
	\item If $f_1(G) = K$, then $f$ is said to be \emph{active}.
	\item If $f$ is isomorphic to a morphism of the form of Example~\ref{ex inert}, then $f$ is said to be \emph{inert}.
\end{itemize}
\end{definition}

Notice that active maps are automatically bijective on boundaries.
The reverse implication also holds.

\begin{lemma}
If $f \colon G \to K$ is a morphism of $\bbY$ inducing a bijection between the inputs / outputs of $G$ and $K$, then $f$ is active.
\end{lemma}
\begin{proof}
There is at most one subgraph possessing a given boundary by \cite[Lemma 6.39]{hrybook}.
Our assumption is that the inputs / outputs of $f_1(G)$ are the same as those of $K$, so it follows that $f_1(G) = K$.
\end{proof}

Kock described a weak factorization system on $\kockgraphs$ (see Definition~\ref{definition kockgraphs}) with $\kockint$ as the right class.
As observed in \cite[2.4.14]{Kock_Properads}, this restricts to an orthogonal factorization system $(\hryact, \hryint)$ on $\bbY$.

\begin{theorem}[Kock]\label{theo Yfs}
The pair $(\hryact, \hryint)$ is an orthogonal factorization system on $\bbY$. \qed
\end{theorem}

\begin{remark}\label{rem inert=convex_open}
	The equivalence $\bbY \simeq \hryGamma$ of Theorem~\ref{theorem upsilon equivalent definitions} and \cite[2.4.14]{Kock_Properads} show that $G\to H$ is inert in $\bbY$ if and only if $G$ is a convex open subgraph of $H$.
	Therefore, if $H$ lies in $\bbYout$, $\bbY_{\name{sc}}$, $\bbO$, or $\simp$, so does $G$. 
	In particular, the active-inert factorization system on $\bbY$ restricts to each of these subcategories.
	This restricted factorization system on $\simp$ coincides with the one described in Example~\ref{ex simpfactsys} and with the restriction of the factorization system on $\levelg$ from Lemma~\ref{lemma general level graph factorization systems}.
\end{remark}

We also have the following analogue of Lemma~\ref{lem comparison between levelgraph cats}\eqref{item one comparison between levelgraph cats}. 
The corresponding statement for item \eqref{item two comparison between levelgraph cats} of Lemma~\ref{lem comparison between levelgraph cats} is immediate from Definition~\ref{def Y}.

\begin{lemma}\label{lem comparison between dirgraph cats}
Suppose that $\mathbf{A} \subseteq \mathbf{B}$ is one of the fully-faithful inclusions appearing in the following diagram
\[ \begin{tikzcd}
\simp \rar & \bbO \rar \dar & \bbYout \dar  \\
& \bbY_{\name{sc}} \rar[dashed, color=red] & \bbY. 
\end{tikzcd} \]
If $f \colon G \to H$ is a morphism in $\mathbf{B}$ and $H \in \mathbf{A}$, then $G\in \mathbf{A}$.
\end{lemma}
\begin{proof}
The statement for the {\color{red} dashed} map is \cite[Proposition 5.2.8]{hrybook}.
We thus concentrate on the cases where $\mathbf{A} \subseteq \mathbf{B}$ is one of the solid black arrows.
In each case, we may distinguish graphs in $\mathbf{A}$ among those in $\mathbf{B}$ by a corresponding property not just for vertices, but for subgraphs.
For instance, 
\begin{itemize}
	\item a graph $G\in \bbY$ is in $\bbYout$ if and only if each structured subgraph of $G$ has at least one output,
	\item a graph $G\in \bbY_{\name{sc}}$ is in $\bbO$ if and only if each structured subgraph has exactly one output,
\end{itemize}
and so on.
We have the active-inert factorization $G\to G\{K_v\} \to H$ from \cite[Lemma 6.42]{hrybook}.
As each $K_v$ is a structured subgraph of $H$, we have $K_v \in \mathbf{A}$.
But $K_v$ has the same number of inputs and outputs as the vertex $v\in G$, hence $G\in \mathbf{A}$ as well.
\end{proof}

We now introduce the functor $\vertex_\bbY\colon\bbY \to \pfinset^\op$.
\begin{definition}\label{definition bbY to Fin}
	We define the functor $\vertex_\bbY\colon\bbY \to \pfinset^\op$ by requiring that it takes 
	\begin{itemize}
		\item an object $G\in \bbY$ to the set $\vertex(G)_+$ of its vertices together with a base point, 
		\item a morphism $f\colon G\to K$ to the based map
		\[
		\vertex_\bbY(f)\colon\vertex(K)_+ \rightarrow \vertex(G)_+
		\]
		defined by the rule
		\begin{equation}\label{definition of partial map}
		\vertex_\bbY(f)(v) = w \qquad \text{if and only if} \qquad v \in \vertex(f_1(w))
		\end{equation}
		and otherwise $\vertex_\bbY(f)(v)=*$.
	\end{itemize}
	We will also write $\vertex_\bbY$ for the composite $\bbY\to \pfinset^\op\simeq \pfinsetskel^\op$.
\end{definition}
We now want to check that $\vertex_\bbY$ is indeed a functor. First note that since $f$ is a graphical map, there is \emph{at most one} $w$ so that $v \in \vertex({f_1(w)})$, hence this is a well-defined map.

Suppose that $g\colon H \to G$ is another graphical map; let us verify that $\vertex_\bbY(fg) = \vertex_\bbY(g) \vertex_\bbY(f)$.
We have $\vertex_\bbY(fg)(v) = w$ if and only if $v \in \vertex((fg)_1(w))$.
By definition of composition in $\bbY$,
\[
	\vertex((fg)_1(w)) = \coprod_{x \in \vertex(g_1(w))} \vertex(f_1(x));
\]
it follows that $v \in \vertex((fg)_1(w))$ if and only if $v\in \vertex(f_1(x))$ for some (unique) $x \in \vertex(g_1(w))$.
This of course happens if and only if $\vertex_\bbY(f)(v) = x$ and $\vertex_\bbY(g)(x) = w$.
Thus if $\vertex_\bbY(fg)(v)$ is in $\vertex(H)$, then so is $\vertex_\bbY(g)(\vertex_\bbY(f)(v))$, and we have the equality
\[
	\vertex_\bbY(fg)(v) = \vertex_\bbY(g)(\vertex_\bbY(f)(v)).
\]
To finish showing that $\vertex_\bbY(fg) = \vertex_\bbY(g)\vertex_\bbY(f)$, simply observe that $\vertex_\bbY(g)(\vertex_\bbY(f)(v))$ is in $\vertex(H)$ if and only if $\vertex_\bbY(f)(v)$ is in $\vertex(G)$ and $\vertex_\bbY(g)(\vertex_\bbY(f)(v))$ is in $\vertex(H)$, which in turn implies that $\vertex_\bbY(fg)(v)$ is in $\vertex(H)$.

\begin{proposition}\label{prop active inert bbY}
	The functor $\vertex_\bbY\colon \bbY\to \pfinset^\op$ preserves the active-inert factorization systems (see Example~\ref{ex: bbG}).
\end{proposition}
\begin{proof}
	If $f\colon G\to K$ is an active map in $\bbY$ then $f(G) = K$.
	Therefore, we have
	\[
		\vertex(K) = \vertex(f(G)) = \coprod_{w\in \vertex(G)} \vertex(f_1(w))
	\]
	and $\vertex_\bbY(f)$ is active by construction.
	 If $f$ is inert then it easily follows from the definition of $\vertex_\bbY(f)$ that for a subgraph we have $|\vertex(f_1(w))| = 1$ for every $w\in \vertex(G)$.
	In other words, the functor $\vertex_\bbY\colon \bbY\to \pfinset^\op$ also preserves inert morphisms.
\end{proof}

\subsection{From connected level graphs to connected acyclic graphs}\label{section level to acyclic}

Let $\levelgconn$ denote the full subcategory of $\levelg$ spanned by the connected level graphs.
The present goal is to define a functor $\tau\colon \levelgconn \to \bbY$.
On objects, $\tau$ simply forgets the level structure, that is, $\tau(G)$ is the directed graph from Lemma~\ref{lemma: level graph}.

\begin{lemma}\label{lemma: level subgraph}
If $G\in \levelgconn$, then any level subgraph (Definition~\ref{definition level subgraph}) is a structured subgraph (Definition~\ref{def structured subgraphs}) of $\tau(G)$.
\end{lemma}
\begin{proof}
Suppose that $x\in G_{i,j}$ is a level subgraph.
As in Remark~\ref{remark level sugraphs}, let $\widetilde H\colon\scriptyell^n_{i,j} \to \finset$ be given by pullbacks
\begin{equation}\label{pullback square for lemma level sub}
\begin{tikzcd}
\widetilde H_{k,\ell} \rar \dar &  G_{k,\ell} \dar \\
\{x\} \rar &  G_{i,j}
\end{tikzcd}
\end{equation}
and let $H$ be the composite $\scriptyell^{j-i} \overset\cong\to \scriptyell^n_{i,j} \to \finset$ with $\widetilde H_{k,\ell} = H_{k-i,\ell-i}$.
As this is defined by pullbacks, $\tau(H)$ is automatically an open subgraph of $\tau(G)$.
Remark~\ref{rem convex} implies that it suffices to construct a lift $k$ for any commutative square
\[
\begin{tikzcd}
\downarrow_0 \amalg \downarrow_m \rar{a} \dar & \tau(H) \dar \\
L \rar{f} \arrow[ur, dotted, "k" description] & \tau(G)
\end{tikzcd}
\]
of na\"ive morphisms of graphs (that is, in $\xFun(\mathscr{G}, \finset)$), where $L$ is a linear graph of height $m\geq 0$ with $\edge(L) = \{ \downarrow_0, \dots, \downarrow_m \}$ and $\vertex(L) = \{ \bullet_1, \dots, \bullet_m \}$.
When $m=0$ this is automatic, as $\tau(H) \to \tau(G)$ is a monomorphism.

There is a $p$ with $f(\downarrow_w) \in G_{p+w, p+w}$, $f(\bullet_w) \in G_{p+w-1, p+w}$, 
and $i\leq p$, $p+m \leq j$.
In $G$, we have commutative diagrams
\[ \begin{tikzcd}[sep=small]
\{ f(\downarrow_w) \}  \rar \dar[hook] & \{ f(\bullet_{w+1}) \} \dar[hook] & & \{f(\bullet_u)\} \dar[hook] & \{f(\downarrow_u)\} \dar[hook] \lar \\
G_{p+w, p+w} \rar & G_{p+w,p+w+1} \rar & G_{i,j} & G_{p+u-1,p+u} \lar & G_{p+u,p+u} \lar
\end{tikzcd} \]
for $0\leq w \leq m-1$ and $1\leq u \leq m$. 
Since $f(\downarrow_0)$ maps to $x\in G_{i,j}$ by assumption, it follows that all of $f(\downarrow_w)$ and $f(\bullet_w)$ map to $x\in G_{i,j}$.
The vertical maps thus factor through the pullbacks from \eqref{pullback square for lemma level sub}, yielding the dashed maps in the following diagram.
\[
\begin{tikzcd}[column sep=tiny]
\{f(\downarrow_0)\} \arrow[dr] \dar[dashed] & & \{f(\downarrow_1)\} \arrow[dl]\dar[dashed] 
&[-1.6em] \cdots & {} \arrow[dr] & & \{f(\downarrow_m)\} \arrow[dl]\dar[dashed] 
 \\
H_{p-i,p-i} \arrow[dr]  \dar &  \{ f(\bullet_1) \}\dar[dashed]  & H_{p-i+1,p-i+1} \arrow[dl] \dar 
 & & {} \arrow[dr] & \{f(\bullet_m)\} \dar[dashed] & H_{p-i+m,p-i+m} \dar \arrow[dl] 
 \\
G_{p,p} \arrow[dr] & H_{p-i,p+1-i} \dar & G_{p+1,p+1} \arrow[dl]
&  \cdots & {}  \arrow[dr] & H_{p-i+m-1,p-i+m} \dar & G_{p+m,p+m} \arrow[dl]
 \\
& G_{p,p+1} & &  & & G_{p+m-1,p+m}
\end{tikzcd}
\]
This determines the na\"ive morphism $k \colon L \to \tau(H)$. 
By uniqueness of the map to the pullback, we have $k(\downarrow_0) = a(\downarrow_0)$ and $k(\downarrow_m) = a(\downarrow_m)$, hence $k$ is the desired lift.
\end{proof}

We now want to construct a functor $\tau\colon \levelgconn\to \bbY$ which is on objects is taking the underlying directed graph (see Lemma~\ref{lemma: level graph}).
To define $\tau$ on morphisms, note that a morphism in $\levelgconn$ lying over $\alpha: [n] \to [m]$ consists of the following data:
\begin{enumerate}
 	\item For each $0 \leq i \leq n$, a monomorphism $G_{i,i} \to H_{\alpha(i), \alpha(i)}$; these assemble into a function $\edge(\tau(G)) \to \edge(\tau(H))$.
 	\item If $1 \leq i \leq n$, a monomorphism $G_{i-1,i} \to  H_{\alpha(i-1), \alpha(i)}$.
In light of the previous lemma, these assemble into a function $\vertex(\tau(G)) \to \sub(\tau(H))$.
\end{enumerate}
As a provisional definition, we would like $\tau(G\to H)$ to be specified by the above data.
Let us first check that this is plausible.
From the definition of morphism in $\levelg$, the following diagrams are pullbacks
\[
\begin{tikzcd}
G_{i,i} \rar[hook] \dar & H_{\alpha(i),\alpha(i)} \dar & G_{i,i} \rar[hook] \dar & H_{\alpha(i),\alpha(i)} \dar \\
G_{i,i+1} \rar[hook] &  H_{\alpha(i),\alpha(i+1)} & G_{i-1,i} \rar[hook] &  H_{\alpha(i-1),\alpha(i)}
\end{tikzcd}
\]
whenever $i$ is in the appropriate range.
In particular, for every $v\in\vertex(G)$, there exists $i$ such that $\inp(v)\simeq \{v\}\times_{G_{i,i+1}} G_{i,i} \simeq \{v\}\times_{H_{\alpha(i),\alpha(i+1)}} H_{\alpha(i),\alpha(i)}$ and similarly for $\out(v)$. Hence, the diagram
\[
\begin{tikzcd}
\lists (\edge(G))  \dar["\lists(-)"] &\vertex(G)  \dar \lar["\inp" swap] \rar["\out"]& \lists (\edge(G))  \dar["\lists(-)"] \\
\lists (\edge(H))  & \sub(H) \lar["\inp" swap] \rar["\out"] & \lists (\edge(H))
\end{tikzcd}
\]
from Remark~\ref{remark alternative characterization of morphisms} commutes.

\begin{proposition}
\label{proposition functor l to y}
If $f\colon G\to H$ is a morphism in $\levelgconn$ lying over $\alpha\colon [n] \to [m]$, then the pair $\edge(\tau(G)) \to \edge(\tau(H))$, $\vertex(\tau(G)) \to \sub(\tau(H))$ from above indeed constitutes a morphism in $\bbY$.
\end{proposition}

The proof this proposition is rather involved, utilizing the equivalence of $\bbY$ with the graphical category $\hryGamma$ from \cite{hrybook} (see Theorem~\ref{theorem upsilon equivalent definitions}).
As the methods used are rather different than what we are dealing with otherwise, we have separated this proof out into Appendix~\ref{appendix proof l to y}.

\begin{corollary}
\label{corollary functor l to y}
The assignment $\tau \colon \levelgconn \to \bbY$ is a functor.
\end{corollary}
\begin{proof}
We wish to show that $\tau(f)\tau(g) = \tau(fg)$ whenever $f$ and $g$ are composable morphisms in $\levelgconn$.
By Corollary 6.62 of \cite{hrybook}, the functor $\edge \colon \bbY \to \Set$ is faithful, so it is enough to show that $\edge(\tau(fg))$ is equal to $\edge(\tau(f)\tau(g)) = \edge(\tau(f)) \edge(\tau(g))$.
This follows because the assignment on objects $G \mapsto \edge(\tau(G))$ constitutes a functor $\levelgconn \to \Set$, so we have $\edge\tau(fg) = \edge\tau(f) \edge\tau(g)$.
\end{proof}

\begin{lemma}\label{lem factorization int el}
The functor $\tau\colon \levelgconn \to \bbY$ is compatible with the active-inert factorization systems.
Further, $\tau$ restricts to an isomorphism $\levelel \cong \bbY_\xel$.
\end{lemma}
\begin{proof}
	We first show that $\tau$ preserves active-inert factorization systems.
	Let $f\colon G\to H$ be an active morphism in $\levelgconn$. We need to show that $\tau(f)_1(G)=H$. This is clear if $H$ consists of just one edge. For a non-trivial graph $H$ the connectivity implies that $\tau(f)_1(G)\neq H$ can only happen if there is a vertex $w\in \vertex(H)$ such that there is no $v\in \vertex(G)$ with $w\in \vertex(\tau(f)_1(v))$. But this case cannot occur due to the fact that $\vertex_{\levelg}(f)$ is active in $\pfinset^\op$ by Lemma~\ref{prop active inert bbG}.
	
	Suppose $f\colon G\to H$ is inert in $\levelgconn$ lying over an interval inclusion $\alpha\colon [m]\to [n]$.
	Then the monomorphisms $G_{i,j}\hookrightarrow H_{\alpha(i), \alpha(i)-i+j}= H_{\alpha(i), \alpha(j)}$ for every $0\leq i\leq j\leq m$ and the connectivity of $G$ show that $G$ is a level subgraph of $H$. By Lemma~\ref{lemma: level subgraph}, $\tau(f)\colon \tau(G)\to \tau(H)$ is the inclusion of a structured subgraph and inert by definition.
	
	The definition of $\tau$ implies that the restriction $\levelel = \levelg_{\name{c,el}} \isoto \bbY_\xel$ is an equivalence.
It is an isomorphism as both $\levelg$ and $\bbY$ are skeletal categories.
\end{proof}

\section{The Segal condition and an algebraic version of enriched \texorpdfstring{$\infty$}{∞}-properads}\label{sec segal presheaves}

In this section, we give a preliminary version of the notion of enriched $\infty$-properad.
We first recall in \S\ref{subsection hcas} a general framework for Segal objects.
This is applied in \S\ref{subsec DFV} to give and compare definitions for `algebraic' $\xV$-enriched $\infty$-properads.
For many purposes (in particular, for existence of certain adjoints), it is important to work with small symmetric monoidal $\infty$-categories $\xU$ rather than general symmetric monoidal $\infty$-categories $\xV$.
In \S\ref{subsec psh} we make precise how one can work in the small, rather than presentable, setting.

There is some overlap between the material in this section and that which will appear in the forthcoming \cite{patterns3}, but many of the constructions and results below depend on special properties of the category of level graphs $\levelg$. 
These will be important in the next section, where we take up the question of algebras over $\infty$-properads.

\subsection{Algebraic patterns and homotopy-coherent algebraic structures}\label{subsection hcas}
We have already encountered (Example~\ref{ex: bbG}, Example~\ref{ex simpfactsys}, Lemma~\ref{lem Lfs}, Theorem~\ref{theo Yfs}) inert-active factorization systems on several 1-categories
\[
\pfinset \simeq \pfinsetskel, \simp^\op, \levelg^\op, \bbY^\op
\]
as well as several restrictions of the latter two (Lemma~\ref{lemma general level graph factorization systems} and Remark~\ref{rem inert=convex_open}).
By declaring certain objects to be `elementary objects,' these categories and factorization systems allow us to define Segal objects.
As the remainder of the paper deals with $\infty$-categories, we recall the definition of a factorization system in that context.

\begin{definition}\label{definition inf cat factorization system}
	Let $\xcc$ be an $\infty$-category and let $(\xcc^{L}, \xcc^{R})$ be a pair of wide subcategories of $\xcc$. Suppose $\Fun_{L,R}(\Delta^{2},\xcc)$ denotes the full subcategory of $\Fun(\Delta^{2},\xcc)$ spanned by those diagrams $\sigma$ such that $\sigma|_{\Delta^{\{0,1\}}}$ is in $\xcc^{L}$ and $\sigma|_{\Delta^{\{1,2\}}}$ is in $\xcc^{R}$. Then we say $\xcc$ has a \emph{factorization system} if the functor
	\[ \Fun_{L,R}(\Delta^{2},\xcc) \to
	\Fun(\Delta^{1},\xcc) \] given by restriction to $\Delta^{\{0,2\}}$ is an equivalence.
\end{definition}

The following is our first important example of factorization system that is not 1-categorical in nature.

\begin{example}[Symmetric monoidal $\infty$-categories]\label{example fact system smc}
Recall that a symmetric monoidal \icat{} is a cocartesian fibration $\xV^\otimes \to \pfinsetskel$ so that $\prod \rho^i_! \colon \xV^{\otimes}_n \to (\xV^{\otimes}_1)^{\times n}$ is an equivalence, where $\rho^i_!\colon \xV^\otimes_n\to \xV^\otimes_1$ denotes the cocartesian pushout along the inert map $\rho^i\colon \langle n\rangle \to \langle 1\rangle$ determined by $\rho^i(i)=1$.
Any symmetric monoidal $\infty$-category $\xV^{\otimes} \to \pfinsetskel$ has a canonical inert-active factorization system by \cite[Proposition 2.1.2.4]{ha}. 
Here a map in $\xV^\otimes$ is called inert if it is cocartesian and lies over an inert map in $\pfinsetskel$, and active if it lies over an active map in $\pfinsetskel$. 
\end{example}

The categories from Section~\ref{section categories of directed graphs}, along with several others, fit into a general framework developed by the first author with Haugseng in \cite{patterns1}.

\begin{definition}\label{def algpatt}
Let $\mathcal Q$ be an $\infty$-category with an active-inert factorization system, and let $\mathcal Q^\op$ be its opposite, which comes with an inert-active factorization system.
Fix a full subcategory $\mathcal Q^\op_\xel$ of $\mathcal Q^\op_{\name{int}}$ whose objects are called elementary objects. 
In \cite[Definition 2.1]{patterns1} this data is called an \emph{algebraic pattern}.
Following \cite[Definition 2.7]{patterns1} we can define \emph{Segal $\mathcal{Q}$-spaces} to be functors $F\colon \mathcal Q^\op \to \xS$ such that the canonical map 
	\[F(X)\to \lim_{E\in (\mathcal Q^{\op}_{\xel})_{X/}}	F(E)\]
	is an equivalence for each $X\in \mathcal Q^\op$. 
We write $\Seg(\mathcal Q)$ for the full subcategory of $\Pre(\mathcal Q) = \xFun(\mathcal Q^\op, \xS)$ spanned by Segal $\mathcal Q$-spaces (this is denoted by $\Seg_{\mathcal{Q}^\op}(\xS)$ in \cite{patterns1}).
\end{definition}

	In other words, $F$ is a Segal $\mathcal Q$-space if and only if the restriction $F|_{\mathcal Q^\op_{\name{int}}}$ is the right Kan extension of $F|_{\mathcal Q^\op_\xel}$ along the inclusion $\mathcal Q^{\op}_{\xel}\hookrightarrow \mathcal Q^{\op}_{\xint}$ (see \cite[Lemma 2.9]{patterns1}). 
	We write $\Seg(\mathcal Q)$ for the full subcategory of $\Pre(\mathcal Q) = \xFun(\mathcal Q^\op, \xS)$ spanned by Segal $\mathcal Q$-spaces (this is denoted by $\Seg_{\mathcal{Q}^\op}(\xS)$ in \cite{patterns1}).
	The description of Segal $\mathcal Q$-spaces as right Kan extensions and \cite[Proposition 4.3.2.15]{ht} imply that $\Seg(\mathcal Q_{\xint})\simeq \Pre(\mathcal Q_{\xel})$.
	
	The most prominent example of an algebraic pattern is $\mathcal Q=\simp$ which has an active-inert factorization system where the inert morphisms are interval inclusions and active morphisms are boundary preserving maps (see Example~\ref{ex simpfactsys}). 
	By choosing the elementary objects to be $[1]$ and $[0]$, we see that $\simp^\op$ admits the structure of an algebraic pattern and Segal $\simp$-spaces are exactly Segal spaces in the sense of \cite{Rezk}, which turn out to model $\infty$-categories. 
	The basic idea of the construction of Segal $\mathcal Q$-presheaves is that the elementary objects play the role of building blocks of an algebraic structure while the Segal condition, i.e.\ the requirement that the canonical maps $F(X)\to \lim_{E\in (\mathcal Q^\op_{\xel})_{X/}} F(E)$ are equivalences, says that every space $F(X)$ is given by gluing these building blocks along inert morphisms. 
	The algebraic operations such as compositions are induced by active morphisms. 
	In general, we observe that every $\infty$-category with an inert-active factorization system and elementary objects defines a kind of homotopy-coherent algebra. 
	As we will see in the next section this construction recovers the notion of $\infty$-properads.
	
	Following the idea that the algebraic structure of objects in $\Seg(\mathcal Q)$ are controlled by inert/active morphisms in $\mathcal Q^\op$, it is natural that a functor $f\colon \mathcal Q\to \mathcal R$ induces a functor $\Seg(\mathcal R)\to \Seg(\mathcal Q)$ (by precomposition) when $f$ is compatible with the additional data.
	A precise characterization can be found in \cite[Lemma 4.5]{patterns1}, which includes that $f$ must preserve both the factorization system and the elementary objects.

\subsection{Segal presheaves and decorated graph categories}\label{subsec DFV}
In this section, we introduce `$\xV$-decorated graph categories' and use the framework from \S\ref{subsection hcas} to give a first model for enriched $\infty$-properads (Definition~\ref{def cts Seg psh}). 
We also provide another model in terms of algebras over categories of graphs whose edges are decorated by elements of a space (Definition~\ref{def alg}); these two approaches are equivalent (Theorem~\ref{thm:PCSisAlgLT}).

It is a theorem of the first author, Robertson, and Yau that the full subcategory of $\xFun(\bbY^{\op},\Set)$ on the Segal objects is equivalent to the category $\properads$ (see \cite{hrybook}).
We thus regard a Segal object in $\Pre(\bbY) = \xFun(\bbY^{\op},\xS)$ to be, at least as a first approximation, a reasonable notion of an $\infty$-properad.

We will shortly unravel the following definition (see Remark~\ref{remark equiv segal and local}).

\begin{definition}
\label{def segal upsilon and ell}
	We now introduce a few algebraic structures given by the construction described in Definition~\ref{def algpatt}.
	\begin{enumerate}
		\item 
		We write $\bbY^{\op}_{\xel}$ for the full subcategory of $\hryint^{\op}$ of Theorem~\ref{theo Yfs} spanned by the corollas and the edge $\downarrow$.
		This yields 
		the $\infty$-category $\Seg(\bbY) \subseteq \xFun(\bbY^\op, \xS)$ of Segal $\bbY$-spaces.
		\item We write $\levelg^{\op}_{\xel}$ (resp.\ $\levelg^{\op}_{\name{c},\xel}$) for the full subcategory of $\levelg^{\op}_{\xint}$ (resp.\ $\levelg^{\op}_{\name{c},\xint}$) spanned by the elementary level graphs (Example~\ref{example: elementary graphs}).
		Note that $\levelg^{\op}_{\xel} = \levelg^{\op}_{\name{c},\xel}$.
		We let $\Seg(\levelg)$ and $\Seg(\levelgconn)$ denote the $\infty$-category of Segal $\levelg$-spaces and Segal $\levelgconn$-spaces.
	\end{enumerate}
\end{definition}

As we will see in the next section, the $\infty$-category of Segal $\bbY$-spaces is equivalent to that of Segal $\levelg$-spaces.

\begin{notation}[Graph categories {$\gc$}]\label{notation graph categories}
So far, we have seen several examples of categories of directed graphs which we will return to again and again.
In what follows, we will use a generic symbol $\gc$ for $\levelg$, $\bbY$, or any of the variations from Definition~\ref{def subcategories of levelg} and Definition~\ref{def Y}.
We refer generically to these types of categories as \emph{graph categories}.
The common features that we will utilize are the following:
\begin{itemize}
\item A factorization system $(\gcact, \gcint)$. See Lemma~\ref{lem Lfs}, Lemma~\ref{lemma general level graph factorization systems}, and Theorem~\ref{theo Yfs}.
\item A full subcategory $\gcel \subseteq \gcint$ of elementary graphs, whose objects are the corollas and the edge. We generically denote these by $\xfc_{p,q}$ and $\xfe$, respectively.
As an example, when $\gc = \bbY$, we take $\xfc_{p,q} \coloneqq C_{p,q}$ and $\xfe \coloneqq {\downarrow}$.
\item A vertex functor $\vertex_\gc \colon \gc \to \pfinset^\op \simeq \pfinsetskel^\op$ preserving the factorization systems. See Proposition~\ref{prop active inert bbG} and Proposition~\ref{prop active inert bbY}.
\item A canonical inclusion $\simp \to \gc$ as the linear level graphs. (This will become relevant only in Section~\ref{sec ffes}.)
\end{itemize}
The subcategories $\gcel$ are all isomorphic for $\gc \in \{\levelg, \levelgconn, \levelg_{\name{0-type}}, \levelg_{\name{sc}}, \bbY, \bbY_{\name{sc}} \}$; likewise, the subcategories $\gcel$ are all isomorphic for $\gc \in \{ \levelg_{\name{out}}, \levelg_{\name{out,c}}, \bbYout \}$.
Note that $\bbY_{\name{out,el}}$ is missing the objects $\xfc_{p,0}$ which are present in $\bbY_\xel$.
\end{notation}

\begin{remark}\label{rem algpattvariation} \quad
	\begin{enumerate}[label=(\roman*), ref={\roman*}]
		\item All the opposites of these graph categories are in particular algebraic patterns, but they come equipped with more structure.
		In particular, for each $\gc$ we have an \icat{} $\Seg(\gc) \subseteq \xFun(\gc^\op, \xS)$. 
		\label{rem algpattvariation more structure}
	\item The category $\kockgraphs$ from Definition~\ref{definition kockgraphs} is unfortunately not a graph category in this sense: it has only a weak factorization system, rather than an orthogonal one, and it does not admit a vertex functor (see Remark~\ref{remark no vertex map Gr} below).
	\item Corollas of the form $C_{0,n}$ or $C_{m,0}$ in $\bbY$ will admit many level graph structures (as in Remark~\ref{remark many heights}), though only one of those will be elementary in the sense of Example~\ref{example: elementary graphs}.
	It thus seems harmless to use $\xfc_{p,q}$ as a common notation for corollas.
	\end{enumerate}
\end{remark}

We do not pursue an abstract version of Notation~\ref{notation graph categories} here.

\begin{remark}[Spans of graph categories]\label{remark three contexts}
In \cite{ChuHaugseng}, the graph categories $\simp_{\finsetskel}^1, \simp_{\finsetskel},$ and $\bbO$ (all of whose objects are trees or forests) played a primary role.
Each of these three indexing categories is suitable to define (enriched) $\infty$-operads.
Indeed, these fit into a span $\simp_{\finsetskel} \leftarrow \simp_{\finsetskel}^1 \rightarrow \bbO$ of algebraic patterns respecting the extra structure specified in Notation~\ref{notation graph categories}.
Both of these functors induce equivalences at the level of Segal objects.

At a high level, much of the present paper is about extending constructions of \cite{ChuHaugseng} to the corresponding span $\levelg \leftarrow \levelgconn \rightarrow \bbY$ of graph categories.
The graph categories $\levelg$ and $\bbY$ each give a different approach to (enriched) $\infty$-properads, and this zig-zag allows us to compare them.
Likewise, we will utilize the spans $\levelg_{\name{0-type}} \leftarrow \levelg_{\name{sc}} \rightarrow \bbY_{\name{sc}}$ (for $\infty$-dioperads) and $\levelg_{\name{out}} \leftarrow \levelg_{\name{out,c}} \rightarrow \bbY_{\name{out}}$ (for $\infty$-output-properads), as well as the inclusions among these four spans.
For notational reasons we generally deemphasize the final two contexts.
\end{remark}

\begin{definition}[Segal cores]\label{segal cores}
Let $\gc$ be a graph category.
We will not distinguish between a graph $G\in \gc$ and its image in $\xFun(\gc^{\op},\xS)$ under the Yoneda embedding.
\begin{itemize}
\item If $G\in \gc$ is a directed graph, define the \emph{Segal core} to be \[ G_{\Seg} \coloneqq \colim_{H\in (\gc^{\op}_{\xel})_{G/}} H\] in $\xFun(\gc^{\op},\xS)$.
The map $G_{\Seg} \to G$ will be called a \emph{Segal core inclusion}.
\item A level graph $L\in \levelg$ will be called \emph{short} if the height of $L$ is $0$ or $1$.
A \emph{short Segal core inclusion} is just a Segal core inclusion $L_{\Seg} \to L$ where $L$ is a short level graph.
\end{itemize}
\end{definition}

\begin{remark}\label{remark equiv segal and local}
	It follows directly from the previous definition and the Yoneda embedding that an object $F\in \xFun(\gc^{\op},\xS)$ lies in $\Seg(\gc)$ if and only if $F$ is local with respect to the Segal core inclusions.
	The main advantage of working with $\levelg^\op$ instead of $\bbY^\op$ is that its ``rigid'' structure allows us to reformulate this description of Segal $\levelg$-spaces by rewriting the colimits $L_{\Seg}$ in various ways. 
	In Section~\ref{sec tensor product} we will use an alternative characterization of Segal $\levelg$-spaces, given in Proposition~\ref{propn:SegDFcond}, to construct the tensor product between $\infty$-properads and Segal spaces.
\end{remark}

\begin{example}
We emphasize that representable $\bbY$-presheaves do not necessarily possess the Segal property.
We saw, in Example~\ref{example of etale but not mono}, a pair of graphs $G$ and $K$ so that $\hom(G,K)$ is empty.
But in this same example, the set $\hom(G_{\Seg},K)$ is inhabited.
Indeed, there is a map
$C_{u_0} \amalg C_{u_1} \amalg C_{v_0} \amalg C_{v_1} \to K$
which factors through $G_{\Seg}$.
Hence $K$ is not local with respect to all Segal core inclusions. 
A similar phenomenon occurs in other settings where there is a distinction between the representable presheaf on a graph and the nerve of the free object generated by the graph (see, for example, \cite[Remark 5.10]{HackneyRobertsonYau:HCO}).
\end{example}

\begin{definition}[Segmentation map]\label{def segmentation map}
Suppose $L \colon \scriptyell^n \to \finset$ is a height $n$ level graph. 
For $1\leq i \leq n$, let $L^{\{i-1,i\}}$ be the composite 
\[ \scriptyell^1 \simeq \scriptyell^n_{i-1,i} \xrightarrow{L} \finset, \]
(see Definition~\ref{partial scriptyell}) that is, the restriction of $L$ to height $i$ vertices and their adjacent edges, and similarly for $L^{\{i\}} \colon \scriptyell^0 \simeq \scriptyell^n_{i,i} \xrightarrow{L} \finset$ ($0\leq i \leq n$).
Each of these objects admits an evident inert map in $\levelg$ with codomain $L$, and we define the \emph{segmentation map} associated to $L$ to be the morphism
\[
L|_{\Delta^n_{\Seg}} \coloneqq L^{\{0,1\}} \amalg_{L^{\{1\}}} L^{\{1,2\}} \amalg_{L^{\{2\}}} \cdots
			\amalg_{L^{\{n-1\}}} L^{\{n-1,n\}} \to L
\]
in $\xFun(\levelg^{\op},\xS)$.
In particular, when $n\leq 1$, the segmentation map is the identity.
\end{definition}

Observe that we do not define segmentation maps in the setting of $\levelgconn$-presheaves.
Indeed, if $L\in \levelgconn$, then the restricted graphs $L^{\{i-1,i\}}$ and $L^{\{i\}}$ are usually not connected (they will all be connected if and only if $L$ is a linear graph).
This problem disappears when working with the categories of disconnected level graphs $\levelg_{\name{0-type}}$ and $\levelg_{\name{out}}$ (Definition~\ref{def subcategories of levelg}).
In particular, a similar statement to the following holds when $\levelg$ is replaced by $\levelg_{\name{0-type}}$ or $\levelg_{\name{out}}$.

\begin{propn}\label{propn:SegDFcond}
Suppose $F \in \xFun(\levelg^{\op},\xS)$.
The following are equivalent:
\begin{enumerate}
\item $F$ is a Segal $\levelg$-space.\label{SegDFcond segal}
\item $F$ is local with respect to all Segal core inclusions $L_{\Seg} \to L$ (Definition~\ref{segal cores}).\label{segDFcond segal core}
\item $F$ is local with respect to the short Segal core inclusions (Definition~\ref{segal cores}) and the segmentation maps (Definition~\ref{def segmentation map}).\label{segDFcond short segmentation}
\end{enumerate}
\end{propn}
\begin{proof}
The first two are equivalent by Remark~\ref{remark equiv segal and local}.
If $L$ is a height $n$ level graph, then the Segal core inclusion factors as the following composite.
\begin{equation*}
\begin{tikzcd}
L_{\Seg} \rar{\simeq} \dar &  (L^{\{0,1\}})_{\Seg} \amalg_{(L^{\{1\}})_{\Seg}} \cdots \amalg_{(L^{\{n-1\}})_{\Seg}} (L^{\{n-1,n\}})_{\Seg} \dar \\
L & L^{\{0,1\}} \amalg_{L^{\{1\}}} \cdots
\amalg_{L^{\{n-1\}}} L^{\{n-1,n\}} \lar 
\end{tikzcd}
\end{equation*}
If \eqref{segDFcond segal core} is satisfied, then $F$ is local with respect to the vertical maps in the diagram. Hence, the 2-of-3 property implies that $F$ is local with respect to the bottom horizontal segmentation map. 
Further, it is automatic that if $F$ satisfies \eqref{segDFcond segal core}, then $F$ is local with respect to the short Segal core inclusions.
Hence \eqref{segDFcond segal core} implies \eqref{segDFcond short segmentation}.

On the other hand, the right vertical map of the diagram above is a pushout of short Segal core inclusions, so if $F$ is local with respect to short Segal core inclusions then it is local with respect to this map.
As the bottom map is a segmentation map, we see that \eqref{segDFcond short segmentation} implies \eqref{segDFcond segal core}.
\end{proof}

In the $1$-categorical setting, the concept of properads is a generalization of that of operads, which in turn is a generalization of categories. 
Following Rezk \cite{Rezk}, \cite{CisinkiMoerdijk2} and the previous definition in the $\infty$-categorical setting these algebraic structures can be described as presheaves satisfying Segal conditions. 
Using the terminology of Definition~\ref{def algpatt}, $\infty$-categories, $\infty$-operads, and $\infty$-properads are modeled by objects in $\Seg(\simp)$, $\Seg(\bbO)$, and $\Seg(\bbY)$, respectively, and the two generalization steps are induced by embeddings $\simp^{\op} \overset{i}{\hookrightarrow} \bbO^\op \overset{j}{\hookrightarrow} \bbY^\op $ of the corresponding indexing categories which respect both the inert-active factorization systems and the elementary objects. 
More precisely, the precompositions with $i$ and $j$ induce functors (see \cite[Lemma 4.5]{patterns1})
\[i^*\colon \Seg(\bbO)\to \Seg(\simp)\text{ and } j^*\colon \Seg(\bbY)\to \Seg(\bbO),\]
where $j^*$ takes any $\infty$-properad to its underlying $\infty$-operad while $i^*$ associates to any $\infty$-operad its underlying $\infty$-category
(see Proposition~\ref{proposition restriction restricts to segal stuff} below).

By writing $\simp_{\finsetskel}^{1, \op}$ for the full subcategory of $\simp_{\finsetskel}^{\op}$ (see Proposition~\ref{prop DF}) spanned by connected objects, that is, trees instead of forests, we then obtain a commutative diagram
\[
\begin{tikzcd}
\simp_{\finsetskel}^{1,\op} \arrow{r}\arrow{d} & \bbO ^\op \arrow{d}\\
\levelgconn^\op \arrow {r}{\tau} & \bbY^\op,
\end{tikzcd}
\]
where the bottom horizontal map is the morphism $\tau$ of Lemma~\ref{lem factorization int el}.
In \cite{ChuHaugseng}, the upper horizontal map was first extended to a functor $\simp_{\finsetskel}^{1,\op, \xV}\to \bbO^{\op,\xV}$ for any symmetric monoidal $\infty$-category $\xV$, where the objects of $\simp_{\finsetskel}^{1,\op, \xV}$ and $\bbO^{\op,\xV}$ are trees with each vertex decorated by an object of $\xV$.
Then it is was shown that this functor induces an equivalence of two corresponding models for $\xV$-enriched $\infty$-operads. 
We want to generalize this idea to the $\infty$-properadic setting and extend $\tau$ to $\bartau \colon \oplevelcV \to \opbbYV$ (these $\infty$-categories are defined just below). 
Our focus lies on studying the various properties of the associated $\infty$-categories of Segal spaces. 
Finally, by a careful examination of the map $\bartau$ we will prove in Section~\ref{sec comparison} that $\oplevelcV$ and $\opbbYV$ describe $\xV$-enriched $\infty$-properads.
The reason for introducing $\levelg^\op$ and $\levelgconn^\op$ will be clear in Section~\ref{sec tensor product}, where we use the particular structure of $\levelgconn^\op$ to prove that $\infty$-properads are tensored over Segal spaces, which then gives us the notion of algebras by adjunction.

We now introduce the $\infty$-categories used to define enriched $\infty$-properads:
\begin{definition}\label{def LV_YV}
Let $\gc$ be a graph category (Notation~\ref{notation graph categories}) and let $q \colon \xV^\otimes \to \pfinsetskel$ be a symmetric monoidal $\infty$-category.
Write $q' \colon \xV^{\op,\otimes} \to \pfinsetskel$ for the opposite symmetric monoidal $\infty$-category.
More precisely, 
$q'$ is the cocartesian fibration associated to the composite
\[
	\pfinsetskel \to \xCat_\infty \xrightarrow{\op} \xCat_\infty
\]
whose first map classifies the cocartesian fibration $q$.\footnote{The reader wishing an explicit description of these dualities is encouraged to consult \cite{BarwickGlasmanNardin}.} 
We let $\gc^{\op, \xV}$ be given by the pullback
\[ \begin{tikzcd}
\gc^{\op, \xV} \rar \dar &  \xV^{\op,\otimes} \dar{q'} \\
\gc^\op \rar{\vertex^\op_{\gc}} & \pfinsetskel.
\end{tikzcd} \]
\end{definition}
\begin{remark}\label{rem doublecatV}
	Similar to Remark~\ref{rem doublecat} it is easy to see that $\levelV$ admits a double $\infty$-categorical structure.
\end{remark}

\begin{notation}\label{notation: bar G}
By definition an object in ${\gc^{\op, \xV}}$ is given by a pair $(G, (v_c)_{c\in \vertex_{\gc}(G)})$, consisting of an object of $\gc^\op$ and an object of $\xV^{\op,\otimes}$.
We think of this as a graph $G\in \gc^\op$ whose vertices are labeled by objects of $\xV$. 
We will write $\overline G$ for the object $(G, (v_c)_{c\in \vertex_{\gc}(G)})$ when we do not wish to emphasize the labeling.
When we \emph{do} wish to emphasize the labeling, we will write such an object as $G(v_c)_{c\in \vertex_{\gc}(G)}$.
Since $\opgcV_{\xfe}\simeq \{*\}$ we also use $\xfe$ to denote the edge in $\opgcV_{\mathfrak e} \subseteq \opgcV$.
\end{notation}

\begin{remark}\label{rem gcalgpatt}
Every morphism in $\gc^{\op, \xV}$ is a pair $(f,g)$ where $f$ and $g$ are morphisms in $\gc^\op$ and $\xV^{\op,\otimes}$, respectively. 
Hence, the inert-active factorization systems on $\gc^\op$ and $\xV^{\op,\otimes}$ induce one on $\gc^{\op, \xV}$. 
More explicitly, a morphism in $\gc^{\op, \xV}$ is inert if and only if its images in $\gc^\op$ and $\xV^{\op,\otimes}$ are both inert. 
Active maps and elementary objects are defined similarly, but we can be more specific in these cases.
\begin{itemize}
\item 
A morphism $(f,g)$ of $\gc^{\op, \xV}$ is active if and only if $f$ is active in $\gc^\op$.
To see this, suppose that $f$ is active. 
Then its image in $\pfinsetskel$ is active, so by Example~\ref{example fact system smc}, $g$ is active as well. 
\item
Each elementary object of $\gc^\op$ maps to $\langle 0\rangle$ or $\langle 1 \rangle$ in $\pfinsetskel$.
It follows that elementary objects in $\gc^{\op, \xV}$ are those of the form $\xfc(v)$ where $\xfc$ is a corolla and $v\in \xV$, and also $\xfe$.
\end{itemize}
Following Definition~\ref{def algpatt} we can define the $\infty$-category $\Seg(\gcV)$ of Segal $\gcV$-spaces. 
\end{remark}

\begin{definition}\label{def bar Segal cores}
Let  $\overline G$ denote an object of ${\opgcV}$ lying over the graph $G$.
\begin{itemize}
\item The \emph{Segal core inclusion} of $\overline G$ is 
\[
	\overline G_{\Seg} \coloneqq \colim_{\overline E\in (\opgcV_{\xel})_{\overline G/}} \overline E\to \overline G
\]
in $\xFun(\opgcV,\xS)$,
where $\overline G_{\Seg}$ is called the \emph{Segal core} of $\overline G$.
\item If $L$ is a level graph, a Segal core inclusion $\overline L_{\Seg} \to \overline L$ is called a \emph{short} Segal core inclusion just when the underlying level graph $L$ is short (Definition~\ref{segal cores}).
\item Suppose that $L \in \levelg$ has height $n$ and $\overline L \in \levelV$.
Defining $\overline L^{\{i\}}$ and $\overline L^{\{i-1,i\}}$ analogously to Definition~\ref{def segmentation map}, the \emph{segmentation map} associated to $\overline L$ is
\[
\overline L|_{\Delta^n_{\Seg}} \coloneqq \overline L^{\{0,1\}} \amalg_{\overline L^{\{1\}}}  \cdots
\amalg_{\overline L^{\{n-1\}}} \overline L^{\{n-1,n\}} \to \overline L
\]
in $\xFun(\levelg^{\op,\xV},\xS)$.
\end{itemize}
\end{definition}

Notice that if $p \colon \levelV \to \levelg \to \simp$ denotes the canonical cartesian fibration, then $\overline L|_{\Delta^n_{\Seg}}$ fits into a pullback
\[ \begin{tikzcd}
\overline L|_{\Delta^n_{\Seg}} \rar \dar & \overline L \dar \\
p^*(\Delta^n_{\Seg}) \rar & p^*(\Delta^n)
\end{tikzcd} \]
in the $\infty$-topos $\Pre(\levelV)$.

As in Remark~\ref{remark equiv segal and local} the next proposition easily follows from the Yoneda embedding.
\begin{proposition}\label{proposition segal local gcv}
	An object $F \in \xFun(\opgcV,\xS)$ is a Segal $\gcV$-space if and only if it is local with respect to all Segal core inclusions $\overline G_\Seg \to \overline G$. \qed
\end{proposition}

\begin{remark}
	It follows from the definition that the cocartesian fibration ${\opgcV}\to \gc^{\op}$ restricts to a cocartesian fibration $\opgcVint \to \gcint^\op$. 
	By applying (the dual of) \cite[Lemma 2.3.13]{ChuHaugseng} to this cocartesian fibration and the full subcategory $\gcel^\op$, we see that the $\infty$-categories $(\opgcVel)_{\overline G/}$ and $(\gcel^\op)_{G/}$ are equivalent. 
	In particular, $\overline G_{\Seg}= \colim_{(\opgcVel)_{\overline G/}} \overline E$ can be identified with $\colim_{E\in (\gcel^\op)_{G/}} \overline E$ where $\overline G \to \overline E$ are cocartesian lifts.
\end{remark}

In the special case $\gc=\levelg$, we can give the following improvement, which can be proven in the same manner as Proposition~\ref{propn:SegDFcond}.
A similar statement holds when $\levelg$ is replaced by $\levelg_{\name{0-type}}$ or $\levelg_{\name{out}}$.

\begin{propn}\label{propn:SegDFVcond}
Suppose $F \in \xFun(\levelg^{\op,\xV},\xS)$.
The following are equivalent:
\begin{enumerate}
\item $F$ is a Segal $\levelV$-space. \label{SegDFVcond segal}
\item $F$ is local with respect to all Segal core inclusions $\overline L_{\Seg} \to \overline L$. \label{segDFVcond segal core}
\item $F$ is local with respect to the short Segal core inclusions and the segmentation maps from Definition~\ref{def bar Segal cores}. \label{segDFVcond short segmentation} 
\qed
\end{enumerate}
\end{propn}

\begin{definition}[Fibrewise representability]\label{def cts Seg psh}
Let $\xV$ be a presentably symmetric monoidal \icat{} and $p\colon \opgcV\to \gc^\op$ be the cocartesian fibration constructed in Definition~\ref{def LV_YV}.
\begin{itemize}
\item If $F \in \Seg(\gcV)$ and $\xfc_{m,n}$ is a corolla in $\gcel$, write 
\[ F(\xfc_{m, n} (\blank))\colon \xV^\op \simeq (\opgcV)_{\xfc_{m,n}} \to \xS_{/F(\xfe)^{m+n}}\simeq \xFun(F(\xfe)^{m+n}, \xS)\]
for the functor induced by the $p$-cocartesian lifts of the $m+n$ morphisms $\xfc_{m, n}\to \xfe$ in $\gcel^\op$.
\item We say that $F\in \Seg(\gcV)$ is \emph{fibrewise representable} if for each corolla $\xfc_{m,n} \in \gcel$ and each object $\underline{xy} = (x_1,\dots,x_m;y_1,\dots,y_n) \in F(\mathfrak e)^{m+n}$, the composite functor
\[ \begin{tikzcd}[column sep=huge]
F(\xfc_{m, n} (\blank, \underline{xy})) \colon
\xV^\op \rar{F(\xfc_{m, n} (\blank))} & \xFun(F(\xfe)^{m+n}, \xS) \rar{\name{ev}} & \xS
\end{tikzcd} \]
(with `$\name{ev}$' being evaluation at $\underline{xy}$) is representable.
In this case we let $\xMap_F(x_1,\ldots, x_m;y_1, \ldots, y_n)$ denote the object in $\xV$ representing the composite $F(\xfc_{m, n} (\blank, \underline{xy})) = \name{ev} \circ F(\xfc_{m, n} (\blank))$.
	\item We write $\Segrep(\gcV)$ for the full subcategory of $\Seg(\gcV)$ spanned by the fibrewise representable Segal $\gcV$-spaces.
\end{itemize}
\end{definition}

Suppose that $F\in \Segrep(\gcV)$.
Unraveling the definition, we have, for each $v\in \xV$ and each $\underline{xy} = x_1,\ldots, x_m;y_1, \ldots, y_n$ a pullback 
\[ \begin{tikzcd}
\Map_{\xV}(v, \xMap_F(x_1,\ldots, x_m;y_1, \ldots, y_n)) \rar \dar  & F(\xfc_{m,n}(v)) \dar \\
\{ x_1,\ldots, x_m;y_1, \ldots, y_n \} \rar & F(\xfe)^{m+n}.
\end{tikzcd} \]

\begin{remark}
Informally, we occasionally refer to objects of the $\infty$-categories 
$\Segrep(\levelV)$, $\Segrep(\bbYV)$, and $\Segrep(\levelcV)$ as \emph{algebraic $\xV$-enriched $\infty$-properads}.
We later will see that the choice among the three graph categories $\levelg$, $\bbY$, and $\levelgconn$ give equivalent notions for this term.
This is a preliminary definition, and $\xV$-enriched $\infty$-properads will appear below in Definition~\ref{def completness}.
Likewise, we refer to fibrewise representable $\bbYout^\xV$-Segal spaces and $\levelg_{\name{out}}^\xV$-Segal spaces as algebraic $\xV$-enriched $\infty$-output-properads, and fibrewise representable $\bbY_{\name{sc}}^\xV$-Segal spaces and $\levelg_{\name{0-type}}^\xV$-Segal spaces as algebraic $\xV$-enriched $\infty$-dioperads.
\end{remark}

\begin{remark}
As we will see in Corollary~\ref{cor equ enr properads}, the $\infty$-categories $\Segrep(\levelV)$ and $\Segrep(\bbYV)$ are equivalent, while the interpretation of Corollary~\ref{cor AlgPrd=PCSbbY} says that the objects in these $\infty$-categories can be interpreted as enriched $\xV$-properads. 
In this picture the object $\xMap_F(x_1,\ldots, x_m;y_1, \ldots, y_n)\in \xV$ in Definition~\ref{def cts Seg psh} should be thought of as the mapping object of the $\xV$-enriched $\infty$-properad $F$.
\end{remark}

\begin{remark}\label{rem repcts}
In \cite{ChuHaugseng} the authors apply the construction of Definition~\ref{def LV_YV} to the subcategory $\simp_{\finsetskel}$ of $\levelg$ mentioned in Proposition~\ref{prop DF}. 
The resulting $\infty$-category $\simp^\xV_{\finsetskel}$ is then used in \cite[Definition 2.3.9]{ChuHaugseng} to define continuous Segal presheaves, which model enriched $\infty$-operads. 
Although the definition of fibrewise representable Segal $\levelV$-spaces from above naturally generalizes the continuous definition \cite[Definition 2.3.9]{ChuHaugseng} we prefer to use ``fibrewise representable'' instead of ``continuous'' as we believe that it describes the phenomenon better.
\end{remark}

We now give alternative characterizations of fibrewise representable Segal $\gcV$-spaces:

\begin{propn}\label{propn:CtsDFV}
	Let $\xV$ be a presentably symmetric monoidal $\infty$-category and let $\varnothing$ denote the initial object in $\xV$. For $F\in \Seg(\gcV)$ the following are equivalent:
	\begin{enumerate}
		\item $F$ is fibrewise representable.
	\label{propn:CtsDFV fibrep}
		\item For every corolla $\xfc_{m,n}\in \gcel$, the restriction \[ \xV^{\op}
		\simeq (\opgcV)_{\xfc_{m,n}} \xrightarrow{F(\xfc_{m,n}(\blank))} \xS_{/F(\xfe)^{m+n}} \to \xS\]
		preserves weakly contractible limits, and the natural map $F(\xfc_{m,n}(\varnothing)) \to
		F(\xfe)^{m+n}$ is an equivalence.
	\label{propn:CtsDFV wc limits}
		\item $F$ is local with respect to the maps $\coprod_{m+n}
		\xfe \to \xfc_{m,n}(\varnothing)$ and the maps $\colim_{\mathcal{I}}\xfc_{m,n}(\phi) \to
		\xfc_{m,n}(\colim_{\mathcal{I}} \phi)$ in $\xFun(\opgcV,\xS)$, where $\phi$ is a weakly contractible diagram in $\xV$.\label{propn:CtsDFV weakly contractible}
	\end{enumerate}
\end{propn}
\begin{proof}
	By definition $F$ is fibrewise representable if for every corolla $\xfc_{m,n}$ and every $\underline{xy}\in F(\xfe)^{m+n}$ the functor $\xV^\op\to \xS$ given by composite of $F(\xfc_{m,n}(\blank))$ and the evaluation map $\name{ev}_{\underline{xy}}$ is representable. 
	The assumption that $\xV$ is presentable implies that a functor $\xV^\op\to \xS$ is representable if and only if it preserves all limits. Since limits in $\xFun(F(\xfe)^{m+n},\xS)$ are computed objectwise, \eqref{propn:CtsDFV fibrep} is equivalent to saying that the functor $F(\xfc_{m,n}(\blank)) \colon \xV^{\op}\to \xS_{/F(\xfe)^{m+n}}\simeq \xFun(F(\xfe)^{m+n},\xS)$ preserves all limits. 
	This is equivalent to saying that $F(\xfc_{m,n}(\blank))$ preserves terminal objects and weakly contractible limits by \cite[Lemma 2.2.7]{GepnerHaugsengKock:IOAM}. 
	The terminal object in $\xS_{/F(\xfe)^{m+n}}$ is $F(\xfe)^{m+n}$ and the forgetful functor $\xS_{/F(\xfe)^{m+n}}\to \xS$ preserves and detects weakly contractible limits. Hence, \eqref{propn:CtsDFV fibrep} holds if and only $F(\xfc_{m,n}(\blank))$ takes the terminal object $\varnothing$ to $F(\xfe)^{m+n}$ and preserves weakly contractible limits, which is equivalent to \eqref{propn:CtsDFV wc limits}.
	The equivalence between \eqref{propn:CtsDFV wc limits} and \eqref{propn:CtsDFV weakly contractible} follows from the Yoneda lemma. 
\end{proof}

\begin{definition}\label{def set s}
Let $\gc$ be a graph category (Notation~\ref{notation graph categories}).
	\begin{enumerate}
		\item A map $X \to X'$ in $\Pre(\gc) = \xFun(\gc^\op,\xS)$ (resp.\ $\Pre(\gcV)$) is called a \emph{Segal equivalence} if $\Map(X', F) \to \Map(X, F)$ is an equivalence for every $F\in \Seg(\gc)$ (resp.\ $F\in \Seg(\gcV)$),
		\item A map $X \to X'$ in $\Pre(\gcV)$ is called a \emph{$\xV$-Segal equivalence} if $\Map(X', F) \to \Map(X, F)$ is an equivalence for every $F\in \Segrep(\gcV)$.
	\end{enumerate}
\end{definition}
Note that each Segal equivalence in $\Pre(\gcV)$ is also a $\xV$-Segal equivalence.
\begin{definition}\label{def strongly saturated}
	We call a class $S$ of morphisms in a cocomplete $\infty$-category $\xcc$ \emph{strongly saturated} if 
	\begin{enumerate}[(i)]
		\item it satisfies the $2$-of-$3$ property,
		\item it is stable under pushouts along any morphism in $\xcc$,
		\item the full subcategory of $\xFun(\Delta^1, \xcc)$ spanned by $S$ is stable under small colimits.
	\end{enumerate}
	We say that the strongly saturated class $S$ is \emph{strongly generated} by $\mathbb S$ if $S$ is the smallest strongly saturated class containing $\mathbb S$. In this case we call elements in $\mathbb S$ the \emph{generators} of $S$.
\end{definition}

\begin{remark}
	If $\mathbb{S}$ is a proper set of morphisms in a presentable $\infty$-category $\xcc$, then \cite[Proposition 5.5.4.15]{ht} implies that the strongly saturated class generated by $\mathbb{S}$ coincides with the class of morphisms $f$ such that $\xMap_\xcc(f, X)$ is an equivalence for every $\mathbb{S}$-local object $X$ in $\xcc$.
\end{remark}

\begin{example}\label{example seg equivs strongly saturated}
By \cite[Lemma 5.5.4.11]{ht} the class of Segal equivalences is strongly saturated, and likewise for $\xV$-Segal equivalences.
The previous remark and Remark~\ref{remark equiv segal and local} show that the class of Segal equivalences in $\Pre(\gc)$ is generated by Segal core inclusions.
If $\xU$ is a small symmetric monoidal \icat{}, then by Proposition~\ref{proposition segal local gcv} the class of Segal equivalences in $\Pre(\gc^\xU)$ is generated by Segal core inclusions.
In addition, Proposition~\ref{propn:SegDFcond} (resp.\ Proposition~\ref{propn:SegDFVcond}) gives that the class of Segal equivalences in $\Pre(\levelg)$ (resp.\ $\Pre(\levelU)$) can instead be generated by the set of short Segal core inclusions together with the segmentation maps.
\end{example}

Even in the case of a small $\xU$, one cannot immediately produce generators for $\xU$-Segal equivalences since, by Proposition~\ref{propn:CtsDFV}, one would want to index some of these generators by the proper class of all weakly contractible diagrams in $\xU$.
In the present paper we will never need explicit generators for $\xU$ or $\xV$-Segal equivalences.

One could ask about compatibility between the above constructions for various graph categories.
Along these lines, we have the following.

\begin{proposition}\label{proposition restriction restricts to segal stuff}
Suppose that $i \colon \gcone \to \gctwo$ is one of the fully-faithful inclusions of graph categories appearing in Figure~\ref{figure lattice of subcategories} or Figure~\ref{figure lattice of subcategories for bbY}.
Then $i^* \colon \Pre(\gctwo) \to \Pre(\gcone)$ restricts to $i^* \colon \Seg(\gctwo) \to \Seg(\gcone)$.
Likewise, we have restrictions $\bariustar \colon \Seg(\gcVtwo) \to \Seg(\gcVone)$ and $\bariustar \colon \Segrep(\gcVtwo) \to \Segrep(\gcVone)$.
\end{proposition}
\begin{proof}
	By declaring the edge and each corolla to be elementary objects, by Remark~\ref{rem algpattvariation}\eqref{rem algpattvariation more structure} we have the opposites of all categories appearing in Figure~\ref{figure lattice of subcategories} and Figure~\ref{figure lattice of subcategories for bbY} are algebraic patterns.
	According to \cite[Lemma 4.5]{patterns1} and \cite[Remark 4.4]{patterns1} the functor $i^* \colon \Pre(\gctwo) \to \Pre(\gcone)$ given by the precomposition with $i^\op\colon \gcone^\op \to \gctwo^\op$ restricts to $\Seg(\gctwo) \to \Seg(\gcone)$ if for every object $G\in \gcone^\op$ the functor $i$ induces an equivalence $\gcone_{\xel,G/}^{\op} \isoto \gctwo_{\xel,i(G)/}^{\op}$. 
	Note that the existence of an inert morphism $ i(G)\to E$ with $E\in \gctwo^{\op}_{\xel}$ implies that $E\simeq i(E')$ for some $E'\in \gcone^{\op}_{\xel}$. 
	The fully faithfulness of $i$ then implies that $\gcone_{\xel,G/}^{\op} \to \gctwo_{\xel,i(G)/}^{\op}$ is essentially surjective as well as fully faithful.
	
	Suppose $\overline G = G(v_c)_{c\in \vertex_{\gc}(G)}$, then it follows from Remark~\ref{rem gcalgpatt} that $\gcone_{\xel,\overline G/}^{\xV, \op}\simeq \gcone_{\xel,G/}^{\op}\times_{\finsetskel_{*, \vertex_{\gc}(G)/}} \xV^{\op, \otimes}_{(v_c)_{c}/}$. 
	Hence the equivalence $\gcone_{\xel,G/}^{\op} \isoto \gctwo_{\xel,i(G)/}^{\op}$ induces an equivalence $\gcone_{\xel,\overline G/}^{\xV, \op}\isoto \gctwo_{\xel,\overline G/}^{\xV, \op}$ which shows that $\Pre(\gctwo^\xV) \to \Pre(\gcone^\xV)$ restricts to $\Seg(\gcVtwo) \to \Seg(\gcVone)$ by \cite[Lemma 4.5]{patterns1}. 
	Finally, it further restricts to $\Segrep(\gcVtwo) \to \Segrep(\gcVone)$ as being fibrewise representable is a property that only concerns elementary objects and the map $\gcVone\to \gcVtwo$ preserves elementary objects.
\end{proof}

\begin{lemma}\label{lem actlift}
Let $i \colon \levelgconn \to \levelg$ be the inclusion.
The functor $i_*\colon \Pre(\levelgconn)\to \Pre(\levelg)$ given by the right Kan extension along $i^\op$ restricts to a functor $\Seg(\levelgconn)\to \Seg(\levelg)$.
Similarly, if $\xU$ is a small symmetric monoidal $\infty$-category, then the right Kan extension restricts to Segal objects:
\[ \begin{tikzcd}
\Seg(\levelcU) \rar \dar[dashed] & \Pre(\levelcU) \dar
\\
\Seg(\levelU) \rar & \Pre(\levelU).
\end{tikzcd} \]
In both cases, the restricted functors are right adjoints to the precomposition functors appearing in Proposition~\ref{proposition restriction restricts to segal stuff}.
Similar statements hold when the pair $(\levelgconn, \levelg)$ is replaced by $(\levelg_{\name{sc}},\levelg_{\name{0-type}})$ or $(\levelg_{\name{out,c}},\levelg_{\name{out}})$.
\end{lemma}
\begin{proof}
We write $i\colon \gc_{\name{c}}\to \gc$ for any of the three inclusions under consideration.
For the first part, it suffices to show that both adjoint functors $i^*\colon \Pre(\gc)\rightleftarrows  \Pre(\gc_{\name{c}})\cocolon i_*$ restrict to functors between the corresponding $\infty$-category of Segal objects. 
By the previous proposition this is true for $i^*$. 
By \cite[Proposition 6.3]{patterns1} the right adjoint $i_*$ restricts to $\Seg(\gc_{\name{c}})\to \Seg(\gc)$ if $i^\op$ has unique lifting of active morphisms. 
Suppose we have an active map $f\colon H\to i(G)$ in $\gc^\op$ lying over some active morphism $[n]\to [m]$ in $\simp^\op$, then we have $H_{0,n}\isoto G_{0,m}$. 
By Lemma~\ref{lemma F0n connected}, $G$ is an object of $\gc_{\name{c}}$ if and only if $G_{0,m}\simeq\{*\}$. 
The same lemma then implies that $H\in \gc_{\name{c}}$.
Hence $f$ lifts to an active map of $\gc_{\name{c}}^\op$, and since $i$ is fully-faithful the unique lifting property holds.

Since a morphism in $\gcU$ is active if and only if its projection to $\gc$ is active (see Remark~\ref{rem gcalgpatt}), the same argument shows that the inclusion $\gc_{\name{c}}^\xU\to \gc^\xU$ induces a right adjoint $\Seg(\gc_\name{c}^\xU)\to \Seg(\gc^\xU)$ given by restriction.
\end{proof}

\begin{proposition}\label{proposition LcL}
The inclusion $i\colon\levelgconn \hookrightarrow \levelg$  induces an equivalence $i^*\colon \Seg(\levelg)\isoto\Seg(\levelgconn)$. 
For every small symmetric monoidal $\infty$-category $\xU$, the inclusion $\levelcU \hookrightarrow \levelU$ induces an equivalence $\Seg(\levelU)\isoto\Seg(\levelcU)$.
Similar statements hold when the pair $(\levelgconn, \levelg)$ is replaced by $(\levelg_{\name{sc}},\levelg_{\name{0-type}})$ or $(\levelg_{\name{out,c}},\levelg_{\name{out}})$.
\end{proposition}
\begin{proof}
We only give a proof for $(\levelgconn, \levelg)$, as the other two situations are entirely analogous. 
By Proposition~\ref{proposition restriction restricts to segal stuff} and Lemma~\ref{lem actlift} the adjunction
$i^*\colon \Pre(\levelg)\rightleftarrows \Pre(\levelgconn)\colon i_*,$ where the right adjoint $i_*$ is given by right Kan extension, restricts to an adjunction
\[i^*\colon \Seg(\levelg)\rightleftarrows \Seg(\levelgconn)\cocolon i_*\]
Since $i$ is fully faithful, the right adjoint $i_*$ is fully faithful and the counit is an equivalence. It only remains to show that the unit $\id \to i_*i^*$ is an equivalence after evaluating at any object $F\in \Seg(\levelg)$ and $I\in \levelg^\op$. 
The description of right Kan extension give $i_*i^*F(I)\simeq \lim_{J\in (\levelgconn^\op)_{I/}}F(iJ)$. Let us regard the set $\{I_j\}$ of objects $I\to I_j$ in $(\levelgconn^\op)_{I/}$ such that $I_j$ is a connected component of $I$. 
We now want to see that the inclusion of $\{I_j\}$ into $(\levelgconn^\op)_{I/}$ is final. 
According to \cite[Theorem 4.1.3.1]{ht} this is can be proven by verifying that for every $\phi\in (\levelgconn^\op)_{I/}$ the category $\{I_j\}_{/\phi}$ is weakly contractible. 
As the codomain of $\phi\colon I\to J$ in $\levelg^\op$ is connected, the morphism $\phi$ necessarily factors through a unique connected components $I_j$, and thus, the category $\{I_j\}_{/\phi}$ contains only the object $(I \to I_j) \to (I \to J)$.

This implies that $\lim_{J\in (\levelgconn^\op)_{I/}}F(iJ)\simeq \prod_{j}F(I_j)$ and the assumption that $F\in \Seg(\levelg)$ implies that $\prod_{j}F(I_j)\simeq \prod_{j}\lim_{E\in {(\levelg_{\xel}^\op})_{I_j/}}F(E)\simeq \lim_{E\in{(\levelg_{\xel}^\op)}_{I/}}F(E)$ where the last equivalence is given by writing ${(\levelg_{\xel}^\op)}_{I/}$ as a coproduct of ${(\levelg_{\xel}^\op)}_{I_j/}$.

The equivalence $\Seg(\levelU)\simeq \Seg(\levelcU)$ is proved in essentially the same way.
\end{proof}

In the $1$-categorical setting, taking the underlying set of colors gives a forgetful functor from the category of properads to the category of sets such that each fibre can be constructed as algebras over an operad (see Lemma~\ref{YJ Section 14.1} below). 
Theorem~\ref{thm:PCSisAlgLT} below can be interpreted as an $\infty$-categorical version of this classical fact where taking the underlying set corresponds to evaluation at the edge $\xfe$ in $\opgcV$. 
As in the classical case the fibres are certain algebras.

\begin{defn}[Edge decorations]
	Given a space $X\in \xS$, we can construct pictured right Kan extension
	\[ \begin{tikzcd}
	\{ \xfe \} \dar[swap]{X} \rar[hook] & \gc^\op \arrow[dl, dashed] \\
	\xS 
	\end{tikzcd} \]
	and we write $\gc^\op_X \to \gc^\op$ for the left fibration associated to $\gc^\op \to \xS$.
	We call a morphism in $\gc^\op_X$ 
\begin{itemize}
	\item \emph{inert} if it is cocartesian and lies over an inert morphism in $\gc^\op$, or
	\item \emph{active} if it lies over an active morphism in $\gc^\op$.
\end{itemize}
Likewise, we say that an object is elementary if its image in $\gc^\op$ is elementary.
This gives $\gc^\op_X$ the structure of an algebraic pattern.
	We let $\opgcV_X$ denote the pullback $\gc^\op_X \times_{\gc^\op}\opgcV \simeq \gc^\op_X \times_{\pfinsetskel} \xV^{\op,\otimes}$ whose inert/active morphisms as well as elementary objects are defined by the components.
\end{defn}
\begin{notation}\label{no Y_X}
	By unwinding the previous definition the fibre of $\gc^\op_X \to \gc^\op$ over a graph $G$ is given by the evaluation of the associated functor $\gc^\op \to \xS$ at $G$. 
	If $G$ has $n$ edges, then this functor takes $G$ to $X^n$. 
	Therefore, we can view an object in $\gc^\op_X$ as an object in $\gc^\op$ together with a labeling of its edges by elements of the space $X$. 
	We write $G(x_e)_{e\in \edge(G)}$ for an object in $\gc^\op_X$ and $G(v_c,x_e)_{c,e}$ for an object in $\opgcV_X$.
\end{notation}

\begin{defn}\label{def alg}
	Let $\xV^\otimes\to \pfinsetskel$ be a symmetric monoidal \icat{}.
	We define a \emph{$\gc^\op_X$-algebra} in $\xV$ to be a functor $\gc^\op_X \to \xV^\otimes$ which renders the diagram
\csquare{\gc^\op_X}{\xV^\otimes}{\gc^\op}{\pfinsetskel}{}{}{}{\vertex_\gc^\op}
	commutative and takes the inert morphisms lying over $\rho_{i}$ to inert morphisms in $\xV^\otimes$.
	We write $\Alg_{\gc^\op_X}(\xV)$ for the full subcategory of $\Fun_{\pfinsetskel}(\gc^\op_X, \xV^\otimes)$ spanned by the $\gc^\op_X$-algebras.
	This construction is (contravariantly) functorial in $X$, and we write $\Alg_{\gc^\op/\xS}(\xV) \to \xS$ for the cartesian fibration associated to the functor $\xS^\op \to \CatI$ which sends $X$ to $\Alg_{\gc^\op_X}(\xV)$.
\end{defn}

In particular, when $X$ is a point, the \icat{} $\Alg_{\bbY^\op_X}(\xS)$ coincides with the \icat{} presented by the model structure for Segal properads given in Theorem 5.3 of \cite{hry_factorizations}.

The following analogue of Theorem 2.4.4 of \cite{ChuHaugseng} can be proved in a similar manner, after replacing $\simp_{\finsetskel}^{\op, \xV}$ with $\opgcV$.
Recall from Remark~\ref{rem repcts} that the continuous condition from \cite{ChuHaugseng} is called fibrewise representability here.
\begin{thm}\label{thm:PCSisAlgLT}
	Let $\xV$ be a presentably symmetric monoidal \icat{}. There is an equivalence 
	\[
 \begin{tikzcd}[column sep=tiny]
  \Segrep(\gcV)\arrow{dr} \arrow{rr}{\sim} &
    & \Alg_{\gc^\op/\xS}(\xV) \arrow{dl} \\
    & \xS
 \end{tikzcd}
 \]
of cartesian fibrations where the left diagonal functor is given by the evaluation at the edge $\xfe$.\qed
\end{thm}

The next two items, regarding base change along lax monoidal functors, constitute analogues of Proposition 2.4.11 and Corollary 2.4.12 of \cite{ChuHaugseng} and are proved similarly.

\begin{propn}
Suppose $\xV$ and $\xW$ are presentably symmetric monoidal $\infty$-categories.
	\begin{enumerate}[(i)]
		\item If $F \colon \xV \to \xW$ is a lax monoidal functor, then $F$ induces a functor \[F_{*} \colon \Alg_{\gc^\op_X}(\xV) \to \Alg_{\gc^\op_X}(\xW).\]
		\item If $F \colon \xV \to \xW$ is a symmetric monoidal left adjoint, with (lax monoidal) right adjoint $G$, then there is an adjunction
		\[ F_{*} \colon \Alg_{\gc^\op_X}(\xV) \rightleftarrows \Alg_{\gc^\op_X}(\xW)\cocolon G_{*}.\]
		\item If $L \colon \xV \to \xW$ is a symmetric monoidal localization with (lax monoidal) fully faithful right adjoint $i$, then the right adjoint $i_{*} \colon \Alg_{\gc^\op_X}(\xW) \to \Alg_{\gc^\op_X}(\xV)$ is fully faithful.
		The image consists of those algebras $A \colon \gc^\op_X \to \xV^\otimes$ so that $A$ takes each object of $\gc^\op_X$ lying over a corolla in $\gc^\op$ to an object in $i(\xW)$. \qed
	\end{enumerate}
\end{propn}

\begin{cor}\label{cor:laxmonftr}
Suppose $\xV$ and $\xW$ are presentably symmetric monoidal $\infty$-categories.
	\begin{enumerate}[(i)]
		\item If $F \colon \xV \to \xW$ is a lax monoidal functor, then $F$ induces a functor $F_{*} \colon \Segrep(\gcV) \to
		\Segrep(\gc^\xW)$.
		\item If $F \colon \xV \to \xW$ is a symmetric monoidal left adjoint, with (lax monoidal) right adjoint $G$, then there is an adjunction
		\[ F_{*} \colon \Segrep(\gcV) \rightleftarrows  \Segrep(\gc^\xW)\cocolon G_{*}.\]
		Moreover, the functor $G_{*}$ can be identified with $F^{*}$.
		\item 
		If $L \colon \xV \to \xW$ is a symmetric monoidal localization with (lax monoidal) fully faithful right adjoint $i$, then the right adjoint $i_{*} \simeq L^{*} \colon \Segrep(\gc^\xW) \to \Segrep(\gcV)$ is fully faithful. 
		The image consists of those $F \in \Segrep(\gcV)$ so that the functors \[ F(\xfc_{m,n}(\blank, \underline{xy}))\colon \xV^\op\to \xS \] from Definition~\ref{def cts Seg psh} are representable by objects in $\xW$ for every corolla $\xfc_{m,n}$ in $\gc$ and every $\underline{xy} = (x_1,\dots,x_m;y_1,\dots,y_n) \in F(\xfe)^{m+n}$.\qed
	\end{enumerate}
\end{cor}

\subsection{Enrichment in presheaves}\label{subsec psh}
In this subsection we recall some results which can be proven in a similar manner to their $\infty$-operadic counterparts in \cite[\textsection 2.5 and \textsection 2.6]{ChuHaugseng}, by replacing $\simp_{\finsetskel}$ with a graph category $\gc$.  

Let $\xU$ be a small symmetric monoidal \icat{}, then according to \cite[Corollary 4.8.1.12]{ha} the presheaf $\infty$-category $\Pre(\xU)$ has a unique symmetric monoidal structure (called `Day convolution') such that the tensor product preserves colimits in each variable and the Yoneda embedding $y\colon \xU \to \Pre(\xU)$ is symmetric monoidal.

\begin{thm}\label{thm enr psh}
For every small symmetric monoidal \icat{} $\xU$, the fully faithful functor $\gcU \to \gcPU$, induced by the symmetric monoidal Yoneda embedding $y\colon \xU \to \Pre(\xU)$, gives a fully faithful functor $y_{*} \colon \Seg(\gcU) \hookrightarrow \Seg(\gcPU)$.
The functor $y_*$ restricts to an equivalence 
\[ y_*\colon \Seg(\gcU) \isoto \Segrep(\gcPU). \]
\end{thm}
\begin{proof}
	As in \cite[Theorem 2.5.2]{ChuHaugseng}.
\end{proof}

Applying this theorem to the case $\gc=\bbY$ and $\xU = *$ yields the following.
\begin{cor}\label{cor hry equ}
Segal $\bbY$-spaces are equivalent to fibrewise representable $\bbY^{\xS}$-spaces. \qed
\end{cor}

By taking the full subcategory of $\Seg(\bbY)$ for those functors which send the edge to a point, we obtain the \icat{} presented by the model structure for Segal properads from \cite[Theorem 5.3]{hry_factorizations}.
Thus we may regard fibrewise representable $\bbY^{\xS}$-spaces as a good algebraic version of $\infty$-properads.

\begin{definition}\label{def S properads}
Suppose that $\xU$ is a small symmetric monoidal $\infty$-category $\xU$ and $\mathbb{S}$ is a set of morphisms in $\Pre(\xU)$ such that the strongly saturated class generated by $\mathbb{S}$ is closed under tensor products. 
Let $y\colon \xU\to \Pre(\xU)$ be the Yoneda embedding.
	\begin{enumerate}[(i)]
		\item We write $\PSU$ for the full subcategory of $\Pre(\xU)$ spanned by the $\mathbb{S}$-local objects. 
		By \cite[Proposition 2.2.1.9]{ha}, it inherits a symmetric monoidal structure such that the localization $\Pre(\xU) \to \PSU$ is symmetric monoidal.
		\item We let $\SegrepS(\gcPU)$ denote the full subcategory of $\Segrep(\gcPU)$ spanned by functors which are local with respect to the maps $\xfc(s)$ where $\xfc$ is a corolla in $\gc$ and $s$ is in $\mathbb{S}$.
	\item We write $\SegS(\gcU)$ for the full subcategory of $\Seg(\gcU)$ spanned by functors which are local with respect to the maps $y^*\xfc(s)$ where $\xfc$ is a corolla in $\gc$ and $s$ is in $\mathbb{S}$.
	\end{enumerate}
\end{definition}

\begin{remark}\label{rmk S-seg repr}
	It follows from the definition that an object $F \in \Seg(\gcU)$ lies in $\SegS(\gcU)$ \IFF{} for every corolla $\xfc_{m,n}$ the induced functor \[ 
	\begin{tikzcd}
	\xU^{\op} \rar{F(\xfc_{m,n}(\blank))} &[+1.6em]
	\xS_{/F(\xfe)^{m+n}} \rar &[-0.5em] \xS
	\end{tikzcd}
	\] is local with respect to all maps in $\mathbb S$.
	Analogously, $F \in \Seg(\gcPU)$ lies in $\SegrepS(\gcPU)$ \IFF{} for every corolla $\xfc$, the induced functor $F(\xfc(\blank, \underline{xy})) \colon
	\Pre(\xU)^{\op} \to \xS$ is representable, and the representing object in $\Pre(\xU)$ is local with respect to all maps in $\mathbb{S}$.
\end{remark}

Since $\levelgconn$ and $\levelg$ have the same set of corollas, Proposition~\ref{proposition LcL} and the first part of the previous remark imply the following.

\begin{proposition}
\label{prop LcL S}
If $\xU$ is a small symmetric monoidal $\infty$-category and $\mathbb{S}$ is a set of morphisms in $\Pre(\xU)$ so that the strongly saturated class generated by $\mathbb{S}$ is closed under tensor products, then $\SegS(\levelU)\to \SegS(\levelcU)$ is an equivalence. \qed
\end{proposition}
\noindent
Under the same hypotheses, and by the same reasoning, we also have $\SegS(\levelg_{\name{0-type}}^\xU) \simeq \SegS(\levelg_{\name{sc}}^\xU)$ and $\SegS(\levelg_{\name{out}}^\xU) \simeq \SegS(\levelg_{\name{out,c}}^\xU)$.

After replacing $\simp_{\finsetskel}$ by $\gc$ in the proof of \cite[Corollary 2.6.3]{ChuHaugseng}, we obtain the following result:
\begin{proposition}\label{prop PSeg=PCS}
	Let $\xU$ and $\mathbb{S}$ be as in Definition~\ref{def S properads}.
	\begin{enumerate}[(i)]
		\item The equivalence $\Seg(\gcU) \isoto \Segrep(\gcPU)$ of Theorem~\ref{thm enr psh} restricts to an equivalence
		\[ \SegS(\gcU) \isoto \SegrepS(\gcPU).\]
		\item Let $\gcPU \to \gcPSU$ be the functor induced by the symmetric monoidal localization $\Pre(\xU) \to \PSU$.
		Precomposition induces an equivalence
		\[ \Segrep(\gcPSU) \isoto \SegrepS(\gcPU).\] 
	\end{enumerate}
Combined, these imply that $\SegS(\gcU) \simeq \Segrep(\gcPSU)$.
\qed
\end{proposition}
As an immediate consequence we have:
\begin{corollary}
	The $\infty$-category $\Segrep(\gcPSU)$ is given by an accessible localization of $\Seg(\gcU)$. In particular, it is presentable.
\end{corollary}
\begin{proof}
	Lemma 2.11 of \cite{patterns1} shows that $\Seg(\gcU)$ is presentable. 
	By the previous proposition we can identify $\Segrep(\gcPSU)$ with $\SegS(\gcU)$.
	This latter \icat{} is given by an accessible localization of the presentable $\infty$-category $\Seg(\gcU)$ with respect to the set of maps $y^*\xfc(s)$, where $\xfc$ is a corolla and $s\in \mathbb S$ (see \cite[Proposition 5.5.4.2]{ht}).
	It follows that $\Segrep(\gcPSU)$ is presentable.
\end{proof}
This result implies the following.
\begin{cor}\label{cor PCSDFV pres}
If $\xV$ is a presentably symmetric monoidal
	\icat{}, then $\Segrep(\gcV)$ is a presentable \icat{}.
\end{cor}
\begin{proof}
According to \cite[Proposition 2.6.9]{ChuHaugseng} and \cite[Theorem 5.5.1.1]{ht}, there is an equivalence $\xV\simeq \PSU$ of symmetric monoidal $\infty$-categories, where $\xU$ is a small symmetric monoidal $\infty$-category admitting all $\kappa$-small colimits for some regular cardinal $\kappa$ and $\mathbb S$ is the set of maps of the form $\colim y\circ \phi \to y(\colim \phi)$ and $\phi$ runs over a set of representatives for equivalence classes of $\kappa$-small colimits in $\xU$. 
The previous corollary shows that this $\infty$-category $\Segrep(\gcV)\simeq \Segrep(\gcPSU)$ is presentable.
\end{proof}

\begin{theorem}
\label{theorem LcL Vrep}
If $\xV$ is a presentably symmetric monoidal $\infty$-category, then the map 
\[
	\Segrep(\levelV)\to\Segrep(\levelcV)
\]
from Proposition~\ref{proposition restriction restricts to segal stuff} is an equivalence.
\end{theorem}
A similar proof to the following will also give equivalences $\Segrep(\levelV_{\name{0-type}}) \simeq \Segrep(\levelV_{\name{sc}})$ and $\Segrep(\levelV_{\name{out}}) \simeq \Segrep(\levelV_{\name{out,c}})$.
\begin{proof}
Let $\mathbb{S}$ and $\xU$ be as in the proof of the previous corollary, with $\xV \simeq \PSU$.
We have a diagram
\[ \begin{tikzcd}
\SegS(\levelU) \dar[swap]{\simeq}  \rar{\simeq} & 
	\SegrepS(\levelg^{\Pre(\xU)}) \dar & 
	\Segrep(\levelPSU) \lar[swap]{\simeq} \dar \rar{\simeq} &
	\Segrep(\levelV) \dar \\
\SegS(\levelcU) \rar{\simeq} & 
	\SegrepS(\levelgconn^{\Pre(\xU)}) & 
	\Segrep(\levelcPSU) \lar[swap]{\simeq} \rar{\simeq} &
	\Segrep(\levelcV)
\end{tikzcd} \]
where the first two equivalences on each row are from Proposition~\ref{prop PSeg=PCS} and the vertical equivalence is from Proposition~\ref{prop LcL S}.
\end{proof}

\section{Algebras over enriched \texorpdfstring{$\infty$}{∞}-properads}\label{sec algebras over inf properads}

The purpose of this section is to introduce, for two $\xV$-enriched $\infty$-properads $\mathcal{Q}$ and $\mathcal{R}$, the \icat{} $\Alg^{\xV}_{\mathcal{Q}}(\mathcal{R})$ of  $\mathcal{Q}$-algebras in $\mathcal{R}$.
Often we have that $\xV$ is self-enriched, and we take $\mathcal{R}$ to be $\xV$, appropriately regarded as a $\xV$-enriched $\infty$-properad.

We begin with an extension of the notion of `inner anodyne map' to $\levelg$ and $\levelU$-presheaves; this is a key technical tool which has no analogue for $\bbY$-presheaves.
The core of the section is in \S\ref{sec tensor product}, where we show that $\Segrep(\levelV)$ is tensored over $\Seg(\simp)$.
This tensor product is a partial extension of a Boardman--Vogt-style tensor product of properads (see \cite[\S4]{hrybook}) in the unenriched case.
The adjoint functor theorem then provides us with our notion of algebras (and also cotensors).

Everything in this section is in the setting of algebraic $\xV$-enriched $\infty$-properads.
We will return to the complete case in \S\ref{subsec ffesc}.

\subsection{Inner horn inclusions and inner anodyne maps}\label{subsec inner anodyne}
In this subsection we generalize the definition of inner anodyne maps in $\Pre(\simp)$ to the setting of $\Pre(\levelg)$ and $\Pre(\levelV)$.
After some formal observations we then finally prove the main result is Proposition~\ref{propn:iaDFV} which says that inner anodyne maps in $\Pre(\levelV)$ are Segal equivalences.
Its proof is a properadic generalization of the operadic case treated in \cite[\textsection 2.7]{ChuHaugseng}.

\begin{definition}\label{def partical/lambda x}
Let $p\colon \levelg \to \simp$ denote the cartesian fibration from Definition~\ref{def levelg} and let $I\in \levelg$ be of height $n$.
For $0 \leq k \leq n$, define
$\Lambda^{n}_{k}I$ by the pullback
	\nolabelcsquare{\Lambda^{n}_{k}I}{
		I}{p^{*}\Lambda^{n}_{k}}{p^*(\Delta^n).}
Likewise, suppose that $\xV$ is a presentably symmetric monoidal \icat{}, and $q\colon \levelV \to \levelg$ denote the cartesian fibration from Definition~\ref{def LV_YV}.
If $\olI \in \levelV$ is of height $n$, we define $\Lambda^{n}_{k}\olI \coloneqq (pq)^{*}\Lambda^{n}_{k} \times_{(pq)^*(\Delta^n)} \olI$.
When $0 < k < n$, we call the inclusions $\Lambda^{n}_{k}I \hookrightarrow I$ and $\Lambda^{n}_{k}\olI \hookrightarrow \olI$ \emph{inner horn inclusions}.
The class of \emph{inner anodyne maps} in $\Pre(\levelg)$ or $\Pre(\levelV)$ is the \emph{weakly saturated class} generated by the inner horn inclusions, that is, the smallest class which both contains the inner horn inclusions and is closed under pushouts, transfinite compositions, and retracts.
\end{definition}

We make use of the following notion of `simple' morphism from \cite[Definition 2.7.11]{ChuHaugseng}.

\begin{defn}\label{def simple}
Let $p\colon \xE\to \xcc$ be a cartesian fibration between small $\infty$-categories.
For $F \in \Pre(\xE)$ and $Y \in \Pre(\xcc)$, we say a morphism $F \to p^*(Y)$ is \emph{simple} if for every map $\sigma\colon X \to Y$ from a representable object $X$, in the pullback
	\nolabelcsquare{F'}{F}{p^*(X)}{p^*(Y)}
the presheaf $F'$ is representable and the adjunct $p_{!}F' \to X$ is an equivalence.
\end{defn}

\begin{remark}\label{remark first simple properties}
It is immediate from the definition, and was pointed out in Remark 2.7.12 of \cite{ChuHaugseng}, that
\begin{enumerate}
\item if $F\in \xE$, then the counit map $F \to p^*p_! F \simeq p^*(pF)$ is simple, and \label{item counit simple}
\item the pullback of a simple map is simple. \label{item pullback simple}
\end{enumerate}
\end{remark}

We record a relative version of \eqref{item counit simple}, whose conception and proof are joint with Rune Haugseng.
\begin{lemma}\label{lemma: relative simple unit}
Let
\begin{equation}\label{eq pullback DEBC} \begin{tikzcd}
\mathcal{D} \rar{\beta} \dar{q} & \xE \dar{p} \\
\mathcal{B} \rar{\alpha} & \xcc
\end{tikzcd} \end{equation}
be a pullback of small $\infty$-categories with $p$ and $q$ cartesian fibrations.
For each $e\in \xE$, the unit map $\beta^*e \to q^*q_!\beta^*e$ is simple with respect to $q$.
\end{lemma}
Since the square is a pullback and $\xE^\op \to \xcc^\op$ is a cocartesian fibration, a Beck--Chevalley condition gives 
$
\alpha^* p_! \simeq q_! \beta^* \colon \Pre(\xE) \to \Pre(\mathcal{B});	
$
see Lemma~9.2.7 and Lemma~12.3.11 of \cite{RiehlVerity:EICT}.
We use this fact freely in the following proof.
We note also that $\beta^*e\to q^*q_! \beta^*e$ is equivalent to $\beta^*$ applied to the unit $e \to p^*p_!e$.
\begin{proof}
Fix $e\in \xE$ as in the statement, and let $f \colon b \to q_!\beta^*e \simeq \alpha^* p_! e$ be an arbitrary map of $\Pre(\mathcal{B})$ whose domain is representable. We need to show that the pullback of the unit map $\beta^*e \to q^*q_!\beta^*e$ along $q^*(f)$ is the unit map at a representable object $d\in \mathcal D$.
Let $f' \colon \alpha(b) \to p(e)$ be adjunct to $f$.
By \cite[Lemma 2.7.10]{ChuHaugseng}, there is a pullback 
\begin{equation}\label{dia simple}
 \begin{tikzcd}
e' \rar{\overline{f}'} \dar & e \dar \\
p^*(\alpha b) \rar{p^*(f')} & p^*(pe)
\end{tikzcd}
\end{equation}
where $\overline{f}'$ in $\xE$ is a $p$-cartesian lift of $f'$ and $p(e')\simeq \alpha b$.
Since \eqref{eq pullback DEBC} is a pullback, there exists an object $d\in \mathcal{D}$ given by the pair $(b,e')$. Suppose we define the morphism $\overline{f}\colon d\to \beta^* e$ to be adjunct to $\overline{f}'\colon e'\simeq \beta d\to e$ then we have a commutative diagram
\[ \begin{tikzcd}
d \rar{\overline f} \dar & \beta^* e \dar \\
q^*b \rar{q^*(f)} & q^*q_!\beta^*e
\end{tikzcd} \]
where the vertical maps are units. The statement of the lemma holds if we can show that this square is cartesian. In other words, it suffices to show that the image of this diagram under the functor $\Map_{\Pre(\mathcal{D})}(x,-)$ is a pullback of spaces for every $x\in \mathcal{D}$.

We first note that since the left adjoint $p_!$ preserves representable objects, for every $x\in \mathcal{D}$ the functor $\Map_{\Pre(\xE)}(\beta x, -)$ takes the pullback square (\ref{dia simple}) to the pullback diagram
\[ \begin{tikzcd}
\Map_{\xE}(\beta x, \beta d) \rar \dar & \Map_{\xE}(\beta x, e) \dar \\
\Map_{\xcc}(p\beta x, \alpha b) \rar & \Map_{\xcc}(p\beta x, pe ).
\end{tikzcd} \]
On the other hand, since \eqref{eq pullback DEBC} is a pullback, we also have a pullback square 
\[ \begin{tikzcd}
\Map_{\mathcal{D}}(x, d) \rar \dar & \Map_{\xE}(\beta x, \beta d) \dar \\
\Map_{\mathcal{B}}(q x, b) \rar & \Map_{\xcc}(\alpha q x, \alpha b).
\end{tikzcd} \]
Pasting these together, we have the left-displayed pullback
\[ \begin{tikzcd}
\Map_{\mathcal{D}}(x, d) \rar \dar & \Map_{\xE}(\beta x, e) \dar \\
\Map_{\mathcal{B}}(q x, b) \rar & \Map_{\xcc}(p\beta x, pe )
\end{tikzcd} 
\quad
\begin{tikzcd}
\Map_{\Pre(\mathcal{D})}(x, d) \rar \dar & \Map_{\Pre(\mathcal{D})}(x, \beta^*e) \dar \\
\Map_{\Pre(\mathcal{D})}(x, q^*b) \rar & \Map_{\Pre(\mathcal{D})}(x, q^*q_!\beta^*e )
\end{tikzcd}
 \]
which is equivalent to the right-hand square. Hence, the functor $\Map_{\Pre(\mathcal{D})}(x,-)$ indeed takes the commutative square (\ref{dia simple}) to a pullback.
\end{proof}

\begin{lemma}\label{lemma simple}
	Let $p^*\colon \Pre(\xcc)\to \Pre(\xE)$ be the functor induced by the composition with a cartesian fibration $p\colon \xE\to \xcc$.
	Suppose that $A,L\in \xcc$, and let $f\colon A\to L\times B$ and $g\colon L \times B\to L$ be two morphisms in $\Pre(\xcc)$ where $g$ is the projection map.
		Let  \[
		\begin{tikzcd}
		F \arrow{r} \arrow{d} & \overline L\times p^*B \arrow{r} \arrow{d} &
		\overline L \arrow{d}\\
		p^{*}A \arrow{r}{p^*f} & p^{*}(L\times B) \arrow{r}{p^*g} & p^{*}L.
		\end{tikzcd}
		\]
		be a commutative diagram in $\Pre(\xE)$ where the left-hand and the right-hand squares are pullbacks and the right vertical map is the adjunction unit.
		Then the presheaf $F$ is represented by the object $(gf)^*\overline L\in \xE$ given by the $p$-cartesian lift $(gf)^*\overline L\to \overline L$ of $gf\colon A\to L \times B\to L$.
\end{lemma}
\begin{proof}
The outer rectangle is a pullback, so \cite[Lemma 2.7.10]{ChuHaugseng} shows that $F\simeq (gf)^*\overline L$.
\end{proof}

\begin{propn}\label{propn:iaDFV}
	The inner anodyne maps in $\Pre(\levelV)$ are Segal equivalences.
\end{propn}
\begin{proof}
As the class of Segal equivalences in $\Pre(\levelV)$ is strongly saturated (Example~\ref{example seg equivs strongly saturated}), and inner anodyne maps are the weakly saturated class generated by the inner horn inclusions, it suffices to show that every inner horn inclusion $\bar \jmath \colon \Lambda^{n}_{k}\olI\to \olI$ in $\Pre(\levelV)$ is a Segal equivalence.
If $\mathbb{T}$ is a class of morphisms in a cocomplete $\infty$-category, write $\langle \mathbb{T} \rangle^r$ for the right-cancellative class generated by $\mathbb{T}$, that is, the smallest class of morphisms containing $\mathbb{T}$ and closed under finite compositions, pushouts, and right cancellations (i.e.\ $fg\in \langle \mathbb{T} \rangle^r$ and $g\in \langle \mathbb{T} \rangle^r$ implies $f\in \langle \mathbb{T} \rangle^r$).

Let $\mathbb{S}_1$ denote the set of spine inclusions $\Delta^m_\Seg\to \Delta^m$ in $\Pre(\simp)$; by the proof of Lemma 3.5 of \cite{JoyalTierney}, each inner horn inclusion $\Lambda_k^n \to \Delta^n$ is contained in $\langle \mathbb{S}_1 \rangle^r$.
Suppose we are given an inner horn inclusion 
\[ \begin{tikzcd}
\Lambda^{n}_{k}\olI \rar{\bar\jmath} \dar & \olI \dar \\
(pq)^*(\Lambda^{n}_{k}) \rar{j} & (pq)^*(\Delta^n).
\end{tikzcd} \]
Let $\mathbb{S}_2$ denote the class (depending on $\olI$) of maps $s_2$ in $\Pre(\levelV)$ appearing in a diagram of the form
\begin{equation}\label{eq etastar pq star S}
\begin{tikzcd}
A \rar{s_2} \dar & B \rar \dar & \olI \dar \\
(pq)^*(\Delta^m_\Seg) \rar{(pq)^*s_1} & (pq)^*\Delta^m \rar & (pq)^* \Delta^n
\end{tikzcd} \end{equation}
where both squares are pullbacks, $s_1 \in \mathbb{S}_1$, and the bottom right map is arbitrary.
Since inner horn inclusions $\Lambda^n_k \to \Delta^n$ are in $\langle \mathbb{S}_1 \rangle^r$, 
by Proposition 2.7.8 of \cite{ChuHaugseng}, we have that $\bar \jmath \colon \Lambda^{n}_{k}\olI\to \olI$ is contained in $\langle \mathbb{S}_2 \rangle^r$.
In particular, $\bar \jmath$ is contained in the strongly saturated class generated by $\mathbb{S}_2$.

As the right-hand map of \eqref{eq etastar pq star S} is a unit map it is simple, hence so too is the pullback $B \to (pq)^*\Delta^m$.
In particular, $B$ is representable and has height $m$.
It follows that $s_2$ is equivalent to a segmentation map, thus is also a Segal equivalence.
Since $\mathbb{S}_2$ is contained in the strongly saturated class of Segal equivalences, so too is $\bar \jmath$.
\end{proof}

\begin{corollary}\label{cor pb of inner anodyne}
Let $\xU$ be a small symmetric monoidal \icat{}, and let $\levelU \xrightarrow{q} \levelg \xrightarrow{p} \simp$ be the usual cartesian fibrations.
\begin{enumerate}
\item If $f\colon F\to (pq)^*K'$ is a simple map in $\Pre(\levelU)$ and $K\to K'$ is an inner anodyne map in $\Pre(\simp)$, then $(pq)^*K \times_{(pq)^*K'}F \to F$ in $\Pre(\levelU)$ is inner anodyne.
\item If $f\colon F \to q^*B$ is a simple map in $\Pre(\levelU)$ and $A\to B$ is an inner anodyne map in $\Pre(\levelg)$, then $q^*A \times_{q^*B} F \to F$ in $\Pre(\levelU)$ is inner anodyne.
\end{enumerate}
In both cases, the indicated map is also a Segal equivalence.
\end{corollary}
\begin{proof}
For the first item, let $\mathbb{S}$ denote the set of inner horn inclusions $\{\Lambda^n_k\to \Delta^n \mid 0<k<n \}$ in $\Pre(\simp)$. 
Applying \cite[Lemma 2.7.14]{ChuHaugseng}, we have that the map under consideration is in the weakly saturated class generated by the inner horn inclusions $\Lambda^n_k\olI\to \olI$ in $\Pre(\levelU)$.
This is the definition of being inner anodyne.
The proof of the second item is similar, except that $\mathbb{S}$ should be taken to be the set of inner horn inclusions $\{ \Lambda^n_k I \to I\}$ in $\Pre(\levelg)$.
\end{proof}

\subsection{Tensoring with Segal spaces}\label{sec tensor product}
In this subsection we prove that for a presentably symmetric monoidal \icat{} $\xV$, the \icat{} $\Segrep(\levelV)$ is a module over the symmetric monoidal \icat{} $\Seg(\simp)$ of Segal spaces.
At the end we will see that by applying the adjoint functor theorem to this tensoring functor we get an algebra functor which takes two objects in $\Segrep(\levelV)$ to the corresponding $\infty$-category of algebras between them.
Notice that we are no longer working with an arbitrary graph category $\gc$ (Notation~\ref{notation graph categories}), but rather just\footnote{Our techniques are slightly more general than this, and apply to $\levelg_{\name{out}}$ and $\levelg_{\name{0-type}}$ as mentioned in Remark~\ref{remark sec tensor product extension}.} with $\levelg$.
This shift is important for our proofs, though eventually we shall see (Corollary~\ref{cor equ enr properads}) that $\Segrep(\bbYV) \simeq \Segrep(\levelV)$, so our main results extend to that context (e.g., Proposition~\ref{cor:PCStensor}).

\begin{defn}\label{def tensor}
	Let $\xU$ be a small symmetric monoidal \icat{}, and let $p^{*}\colon \Pre(\simp) \to \Pre(\levelU)$ be the functor induced by $p \colon \levelU \to \simp$ given by the composition of the two cartesian fibrations $\levelU \to \levelg$ and $\levelg \to
	\simp$.
	The two presheaf categories are symmetric monoidal $\infty$-categories with respect to the cartesian product. Since $p^*$ is right adjoint to the functor $p_!$ given by left Kan extension, $p^*$ preserves products and hence, it is a morphism of commutative algebra objects in $\CatI$ (or even in the $\infty$-category $\PrL$ of presentable $\infty$-categories).
	Hence, by \cite[Corollary 3.4.1.7]{ha}, the functor 
\[\blank \times p^{*}(\blank)\colon \Pre(\levelU) \times \Pre(\simp) \to \Pre(\levelU).\]
exhibits $\Pre(\levelU)$ as a module over $\Pre(\simp)$.
	This functor then preserves colimits in each variable, because the cartesian product in $\xS$ does and $p^{*}$ is left adjoint to $p_*$ (given by right Kan extension).
\end{defn}

The main result of this section is the following theorem.
\begin{theorem}\label{theo tensor}
Let $\xU$ be a small symmetric monoidal \icat{} and let $L$ denote the localization $\Pre(\levelU) \to \Seg(\levelU)$.
The $\Pre(\simp)$-module structure on $\Pre(\levelU)$ induces a $\Seg(\simp)$-module structure on $\Seg(\levelU)$ and the tensor product
	\[\otimes \colon \Seg(\levelU) \times \Seg(\simp) \to \Seg(\levelU)\] 
	is given by $F \otimes K = L(F \times p^{*}K)$.
	In particular, the tensor product preserves colimits in each variable.
\end{theorem}

If follows from \cite[Proposition 2.2.1.9]{ha} that for the proof of Theorem~\ref{theo tensor} it suffices to show that the module structure on $\Pre(\levelU)$ is compatible with the Segal equivalences in the following sense:
\begin{proposition}\label{propn Seg tensor}
	Suppose $f \colon F \to F'$ is a Segal equivalence in $\Pre(\levelU)$ and $g \colon K \to K'$ is a Segal equivalence in $\Pre(\simp)$.
	Then $f \times p^*(g) \colon F \times p^*(K) \to F' \times p^{*}(K')$ is a Segal equivalence in $\Pre(\levelU)$.
\end{proposition}
As already mentioned in Definition~\ref{def tensor} the tensor functor $\blank \times p^*(\blank)$ preserves colimits in each variable.
This allows us to prove the proposition by reducing it to a few key special cases of Segal equivalences. We start with two easy cases.
\begin{propn}\label{propn prelim Seg tensor}
Given an object $\olI \in \levelU$ of height $n$ and a Segal equivalence $Z \to \Delta^{n}$, we write $f\colon \olI|_Z \to \olI$ for the map $\olI \times_{p^*(\Delta^n)}p^*(Z)\to \olI$.
For $g \colon K \to K'$ in $\Pre(\simp)$, consider the map
		\[f \times p^*(g) \colon  \olI|_{Z} \times
		p^*(K) \to \olI \times p^*(K').\] 
\begin{enumerate}
\item If $K\in \Pre(\simp)$ and $f \colon \olI|_{\Delta^n_\Seg} \to \olI$ is a segmentation map, then $f \times p^*(K)$ is a Segal equivalence in $\Pre(\levelU)$.\label{prelim Seg tensor one}
\item If $g \colon K\to K'$ is a Segal equivalence in $\Pre(\simp)$, then $\olI \times p^*(g)$ is a Segal equivalence in $\Pre(\levelU)$. \label{prelim Seg tensor two}
\end{enumerate}
\end{propn}
\begin{proof}
Suppose $Z \to \Delta^n$ and $K\to K'$ are inner anodyne maps of simplicial sets.
Two applications (first to $(Z\to \Delta^n, \varnothing \to K)$ and then to $(\varnothing \to \Delta^n, K \to K')$) of \cite[Corollary 2.3.2.4]{ht} gives that the product $Z\times K\to \Delta^n\times K'$ of inner horn inclusions is inner anodyne.
We have a commutative diagram
	\[
	\begin{tikzcd}
	\olI|_Z \times p^*(K) \arrow{r}{f \times p^*(g)} \arrow{d} & {\olI} \times p^* (K') \arrow{r} \arrow{d} &
	\olI \arrow{d}\\
	p^*(Z) \times p^* (K) \arrow{r} & p^*(\Delta^n) \times p^* (K')\arrow{r} & p^{*}(\Delta^n)
	\end{tikzcd}
	\]
consisting of two pullback squares. 
Remark~\ref{remark first simple properties} implies that the right two vertical maps are simple.
Corollary~\ref{cor pb of inner anodyne} gives that the left upper horizontal map $f \times p^*(g)$ is inner anodyne, hence a Segal equivalence.

Since spine inclusions are inner anodyne by \cite[Proposition 2.13]{Quadern45}, the first item \eqref{prelim Seg tensor one} follows immediately from the previous paragraph by taking $Z = \Delta^n_\Seg$ and $K\to K'$ an identity.

For the second statement, first observe that the class of maps $g \in \Pre(\simp)$ so that $\olI \times p^*(g)$ is a Segal equivalence is strongly saturated by \cite[Remark 5.5.4.10]{ht}.
In the first paragraph we showed that if $g$ is an inner horn inclusion then $\olI \times p^*(g)$ is inner anodyne, hence a Segal equivalence.
On the other hand, inner horn inclusions in $\Pre(\simp)$ generate the strongly saturated class of Segal equivalences (see \cite[Proposition 2.7.7]{ChuHaugseng}), so \eqref{prelim Seg tensor two} holds.
\end{proof}

For the proof of Proposition~\ref{propn Seg tensor} we need to understand explicitly the tensor product of a corolla with $\Delta^{1}$. 
For this purpose, it is convenient to introduce some notation:
\begin{defn}\label{def X pm}
	Given an object ${\olI} \in \levelU$ of height 1, write ${\olI}^{+},{\olI}^{-} \to {\olI}$ for the cartesian lifts of $s_{0},s_{1} \colon [2] \to [1]$, respectively.
	If $\olI$ lives over $I \in \levelg$ of the form $\mathbf{m} \overset{f}{\to} \mathbf{k} \overset{g}{\leftarrow} \mathbf{n}$, then the objects ${\olI}^{+},{\olI}^{-}$ lie over the height 2 level graphs (see Figure~\ref{figure I Iplus Iminus})
	\begin{align*}
	I^+ &= \left( \begin{tikzcd}[ampersand replacement=\&, row sep=tiny, column sep=tiny]
\mathbf{m} \arrow[dr, "\id"] \& \& \mathbf{m} \arrow[dl, "\id" swap] \arrow[dr, "f"] \&
\& \mathbf{n} \arrow[dl,"g" swap]
 \\
\& \mathbf{m} \& \& \mathbf{k} 
	\end{tikzcd} \right), & 
I^- &= \left( \begin{tikzcd}[ampersand replacement=\&, row sep=tiny, column sep=tiny]
\mathbf{m} \arrow[dr, "f"] \& \& \mathbf{n} \arrow[dl, "g" swap] \arrow[dr, "\id"] \&
\& \mathbf{n} \arrow[dl,"\id" swap]
 \\
\& \mathbf{k} \& \& \mathbf{n} 
	\end{tikzcd} \right).
	\end{align*}
	If $\olI$ is connected, that is, if $\olI = \xfc_{m,n}(v)$ then the corolla in $\olI^{+},\olI^{-}$ corresponding to $\xfc_{m,n}$ is labeled by $v$ while all other corollas are of the form $\xfc_{1,1}$ and labeled by the unit $\bbone\in \xU$; see Figure~\ref{figure example I I plus bar} for an illustration in the disconnected case. 
	We write $\xfc_{m,n}^{\pm}(v)$ for $\olI^{\pm}$ when $I$ is connected.
\end{defn}

\begin{figure}
\includegraphics[width=0.6\textwidth]{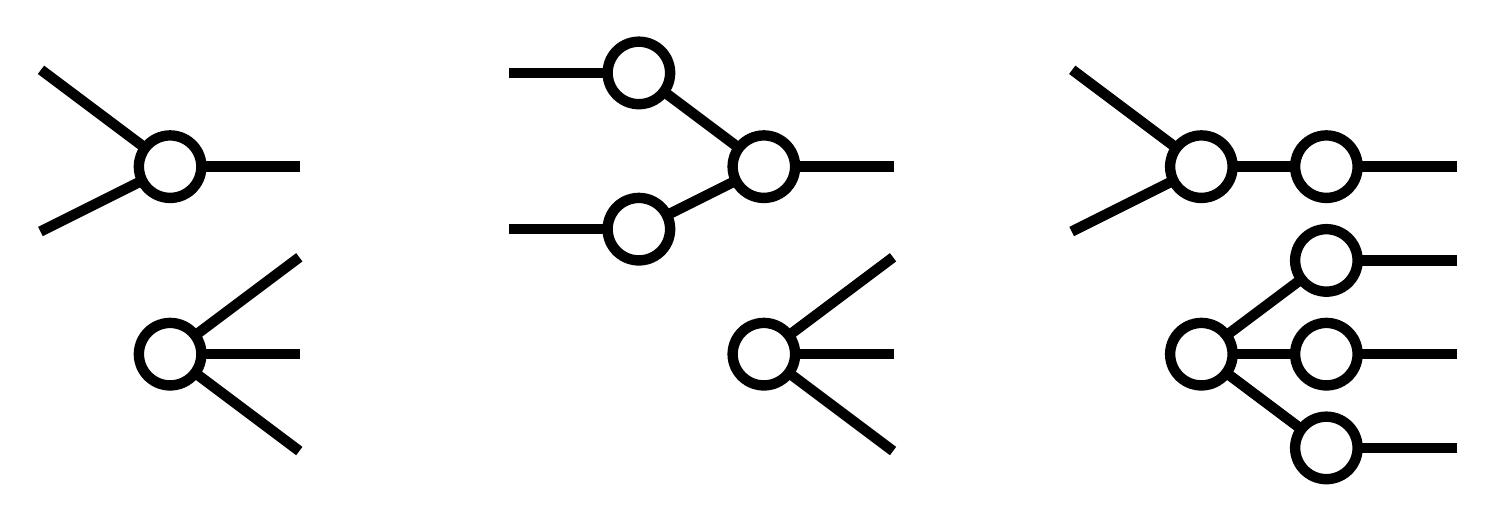}
\caption{An example $I$ (of the form $\mathbf{2} \to \mathbf{2} \leftarrow \mathbf{4}$), $I^+$, and $I^-$.}
\label{figure I Iplus Iminus}
\end{figure}

\begin{figure}[htb]
\labellist
\small\hair 2pt
 \pinlabel {$v_1$} [ ] at 49 94
 \pinlabel {$v_1$} [ ] at 220 94
 \pinlabel {$v_2$} [ ] at 49 40
 \pinlabel {$v_2$} [ ] at 220 40
 \pinlabel {$\bbone$} [ ] at 184 122
 \pinlabel {$\bbone$} [ ] at 184 76
\endlabellist
\centering
\includegraphics[scale=0.8]{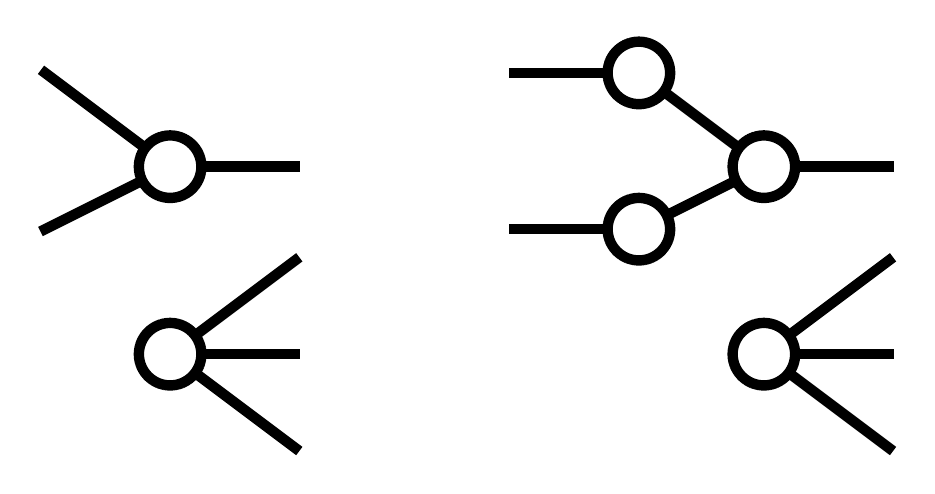}
\caption{An example of $\olI$ and $\olI^+$}
\label{figure example I I plus bar}
\end{figure}

\begin{lemma}\label{lem decomposing}
	For a height 1 object $\olI \in \levelU$, there is an equivalence
	\[ \olI^+ \amalg_{\olI} \olI^- \to
	\olI \times p^*(\Delta^{1}).\]
	In particular, in the connected case where $\olI= \xfc_{m,n}(v)$, we have an equivalence
	$\xfc_{m,n}^{+} (v)\amalg_{\xfc_{m,n}(v)}
	\xfc_{m,n}^{-} (v) \to \xfc_{m,n}(v) \times
	p^*(\Delta^{1}),$ natural in $v$.
\end{lemma}
\begin{proof}
First, let us introduce some notation for a simplicial subdivision of the square.
Namely, let 
	$\sigma^+, \sigma^- \colon \Delta^{2} \to \Delta^{1} \times \Delta^{1}$ denote the two non-degenerate 2-simplices of $\Delta^{1} \times \Delta^{1}$, taking $(0,1,2)$ to $((0,0), (0,1), (1,1))$ and $((0,0), (1,0), (1,1))$, respectively, and let $\delta \colon \Delta^{1} \to \Delta^{1} \times \Delta^{1}$ denote the diagonal map.
	Then the maps $\sigma^{\pm}$ and $\delta$ induce an equivalence $\Delta^{2} \amalg_{\Delta^{1}} \Delta^{2} \isoto \Delta^{1}\times
	\Delta^{1}$.

	Let $A^+$ and $A$ denote the pullbacks
\[ \begin{tikzcd}
A^+ \rar \dar & \olI \times p^*(\Delta^1) \dar & 
A \rar \dar & \olI \times p^*(\Delta^1) \dar
 \\
p^*(\Delta^2) \rar{p^*(\sigma^+)} & p^*(\Delta^1) \times p^*(\Delta^1) & 
p^*(\Delta^1) \rar{p^*(\delta)} & p^*(\Delta^1) \times p^*(\Delta^1)
\end{tikzcd} \]
and similarly for $A^-$.
We can extend these diagrams by projection onto the first factor, yielding in the first case
\[ \begin{tikzcd}
A^+ \rar \dar & \olI \times p^*(\Delta^1) \dar \rar &  \olI \dar
 \\
p^*(\Delta^2) \rar{p^*(\sigma^+)} \arrow[rr, bend right=20, "p^*(s^0)"' near end] & p^*(\Delta^1) \times p^*(\Delta^1) \rar & p^*(\Delta^1).
\end{tikzcd} \]
It follows from this diagram and Lemma~\ref{lemma simple} that $A^+ \simeq \olI^+$.
Using the corresponding diagrams for $A^-$ and $A$, one concludes also that $A^- \simeq \olI^-$ and $A\simeq \olI$.
Since pullbacks in $\Pre(\levelU)$ preserve colimits, we have a natural pullback square 
\[
\begin{tikzcd}
\olI^+ \amalg_{\olI} \olI^- \rar{\sim} \dar & 
\olI \times p^*(\Delta^1) \dar \\
p^*(\Delta^2) \amalg_{p^*(\Delta^1)} p^*(\Delta^2) \rar{\sim} & 
p^*(\Delta^1) \times p^*(\Delta^1),
\end{tikzcd}
\]
which completes the proof.
\end{proof}

\begin{proof}[Proof of Proposition~\ref{propn Seg tensor}]
	The map $f \times p^*(g) \colon F \times p^*(K) \to F' \times p^{*}(K')$ in the statement factors as $F \times p^*(K) \to F' \times p^*(K) 
\to F'\times p^{*}(K')$ and below we prove separately that both morphisms are Segal equivalences.

By Proposition~\ref{propn prelim Seg tensor}\eqref{prelim Seg tensor two} the map $\id\times p^*(g)\colon \overline L\times p^*(K)\to \overline L \times p^*(K')$ is a Segal equivalence for every object $\overline L\in \levelU$ and every Segal equivalence $g\colon K\to K'$ in $\Pre(\simp)$.
Since $\blank\times p^*(\blank)$ preserves colimits in each variable and Segal equivalences are closed under colimits, the map $F\times p^*(K)\to F \times p^*(K')$ is a Segal equivalence for every $F\in \Pre(\levelU)$ and every Segal equivalence $K\to K'$ in $\Pre(\simp)$.

In the remaining part of this proof we show that the first map
	$F \times p^*(K)\to F'\times p^*(K)$ is a Segal equivalence as well. 
As $\blank\times p^*(\blank)$ preserves colimits the map $F \times p^*(K)\to F'\times p^*(K)$ is a Segal equivalence if it is true in the case where $K$ is a simplex $\Delta^{n}$.
	Therefore, it suffices to show that the bottom map of the commutative diagram
	\nolabelcsquare{F \times
		p^{*}(\Delta^n_{\Seg})}{F' \times
		p^{*}(\Delta^n_{\Seg})}{F \times
		p^*(\Delta^n)}{F' \times p^*(\Delta^n)}
	is a Segal equivalence. The previous paragraph shows that the vertical morphisms are Segal equivalences and the definition of $\Delta^n_\Seg$ implies that the upper horizontal map is a colimit of maps of the form $f \times p^*(\Delta^{1})$. Hence, we only need to prove that these maps are a Segal equivalences.
	By Proposition~\ref{propn:SegDFVcond}, we can further reduce to the case where $f$ is either a segmentation map or a short Segal core inclusion (Definition~\ref{def bar Segal cores}).

	For the segmentation maps, the claim follows from Proposition~\ref{propn prelim Seg tensor}\eqref{prelim Seg tensor one}.
We now consider the case where $f \colon \coprod \overline{L}_{i} \to \overline{L}$ a short Segal core inclusion, $\overline L$ is of height $1$ and each $\overline{L}_{i}$ is connected.
Using that $\blank \times p^*(\Delta^1)$ commutes with colimits, 
Lemma~\ref{lem decomposing} tells us that to show $f \times p^{*}\Delta^{1}$ is a Segal equivalence, it suffices to show that
\[
\left( \coprod_i {\overline L}_i^+ \right) \amalg_{\left( \coprod_i {\overline L}_i \right)} \left( \coprod_i {\overline L}_i^- \right) \simeq
\coprod_i({\overline L}_i^+\amalg_{{\overline L}_i}  {\overline L}_i^-) \to {\overline L}^+\amalg_{\overline L} {\overline L}^-\]
is a Segal equivalence.
We know that $f\colon \coprod_i {\overline L}_i \to \overline{L}$ is a short Segal core inclusion, hence a generating Segal equivalence.
It remains to show that $\coprod_i {\overline L}_i^+ \to {\overline L}^+$ and $\coprod_i {\overline L}_i^- \to {\overline L}^-$ are Segal equivalences; we consider only the first case (the other is similar).
There is a commutative square
	\nolabelcsquare{\coprod_i {\overline L}_i^+|
		_{\Delta^2_\Seg}}{{\overline L}^+| _{\Delta^2_\Seg}}{\coprod_i
		{\overline L}_i^+}{{\overline L}^+,} where the vertical maps are Segal equivalences and the upper horizontal map is of the form
\[ \coprod_i \left(
{\overline {L}^{+,\{0,1\}}_i}
\amalg_{{{\overline L}^{+,\{1\}}_i}} 
{{\overline L}^{+,\{1,2\}}_i}
\right)
\to 
{\overline L}^{+,\{0,1\}}
\amalg_{{\overline L}^{+,\{1\}}} 
{\overline L}^{+,\{1,2\}}
.\]
To show that the bottom map in the square is a Segal equivalence, it suffices to show that the top map is such.
This top map is a pushout of the following:
\begin{align}
\coprod_i{\overline {L}^{+,\{0,1\}}_i} &\to 
	{\overline L}^{+,\{0,1\}}  \label{aligned for pushout one}\\
\coprod_i{\overline {L}^{+,\{1,2\}}_i} &\to 
	{\overline L}^{+,\{1,2\}} \label{aligned for pushout two} \\
\coprod_i{\overline L}^{+,\{1\}}_i &\to 
	{\overline L}^{+,\{1\}} \label{aligned for pushout three}
\end{align} 
Since we generally have $\olI^{+,\{1,2\}} = \olI$, map \eqref{aligned for pushout two} is a short Segal core inclusion.
If $m$ is the number of input edges of $L$, then there are commutative triangles
\[ \begin{tikzcd}[column sep=tiny]
\coprod\limits_{j\in \mathbf{m}} \xfc_{1,1}(\bbone) \arrow[rr] \arrow[dr] & & {\overline L}^{+,\{0,1\}} \\
& \coprod\limits_i{\overline {L}^{+,\{0,1\}}_i} \arrow[ur,"\eqref{aligned for pushout one}" description]
\end{tikzcd} 
\qquad
\begin{tikzcd}[column sep=tiny]
\coprod\limits_{j\in \mathbf{m}} \mathfrak{e} \arrow[rr] \arrow[dr] & & {\overline L}^{+,\{1\}} \\
& \coprod\limits_i{\overline {L}^{+,\{1\}}_i} \arrow[ur,"\eqref{aligned for pushout three}" description]
\end{tikzcd} 
\]
where the top maps are short Segal core inclusions, and the downward arrows are coproducts of such.
Hence \eqref{aligned for pushout one} and \eqref{aligned for pushout three} are Segal equivalences as well.
It follows that $f \times p^{*}\Delta^{1}$ is a Segal equivalence when $f$ is the Segal core inclusion into a height 1 graph.

It remains to consider the case where $f\colon \coprod_{i\in \mathbf n}\xfe\to \overline L$ is a short Segal core inclusion and $\overline L$ is of height $0$. 
As above, we must understand $\overline L \times p^*(\Delta^1)$, but unlike the height 1 case (Lemma~\ref{lem decomposing}), in the height 0 case this object is representable.
We first note that canonical equivalence $\sigma\colon\Delta^1\isoto \Delta^0\times \Delta^1$ induces a pullback square
	\csquare{X}{\overline L\times
		p^*(\Delta^1)}{p^*(\Delta^1)}{p^*(\Delta^0)\times
		p^*(\Delta^1)}{\sim}{}{}{p^*\sigma} where the horizontal maps are equivalences.
	Lemma~\ref{lemma simple} implies that the presheaf $X$ is represented by $s_0\overline L$ determined by the cartesian lift $s_0 \overline L\to \overline L$ of the projection $s^0\colon \Delta^1\to \Delta^0$.
	Hence, the presheaf $X$ is represented by $L'(\bbone_c)_{c\in \vertex_{\levelg}(L')}$, where $L'$ is the graph $\mathbf{n}\xto{\id}\mathbf{n}\overset{\id}{\leftarrow}\mathbf{n}$ and $\bbone$ denotes the unit in $\xU$.
	In particular, we have $\xfe\times p^*(\Delta^1)\simeq (\xfc_{1,1},\bbone)$.
	Therefore, the map $f\times p^*(\Delta^1)$ is given by 
\[ 
\left(\coprod_{i\in \mathbf{n}} \mathfrak{e}\right)\times p^*(\Delta^1)
	\simeq
\coprod_{i\in \mathbf{n}} \left( \mathfrak{e}\times p^*(\Delta^1) \right)
	\simeq
\coprod_{i\in \mathbf{n}} \xfc_{1,1}(\bbone)
	\to 
L'(\bbone_c)_{c\in \vertex_{\levelg}(L')}
	\simeq 
\overline L\times p^*(\Delta^1)
,\]
	which is a short Segal core inclusion.
\end{proof}

This completes the proof of Theorem~\ref{theo tensor}.
As a consequence, we get:
\begin{cor}\label{cor tensor S}
	Let $\xU$ be a small symmetric monoidal \icat{} and let $\mathbb{S}$ be a set of morphisms in $\Pre(\xU)$ compatible with the symmetric monoidal structure.
	Then the $\Seg(\simp)$-module structure on $\Seg(\levelU)$ induces a $\Seg(\simp)$-module structure on $\SegS(\levelU)$. Moreover, this tensor product preserves colimits in each variable.
\end{cor}
\begin{proof}
	By definition $\SegS(\levelU)$ is a localization of $\Seg(\levelU)$ with respect to maps of the form $y^*\xfc(s)$, $s\in \mathbb{S}$, according to \cite[Proposition 2.2.1.9]{ha} the $\Seg(\simp)$-module structure on $\Seg(\levelU)$ induces one on $\SegS(\levelU)$ if for every map $s \colon X \to Y$ in $\mathbb{S}$ and every $K\in\Seg(\simp)$, the map
	\[ y^{*}\xfc(s) \times p^{*}K \colon  y^{*}\xfc(X) \times p^{*}K \to
	y^{*}\xfc(Y) \times p^{*}K\]
 is an $\mathbb{S}$-Segal equivalence.
The proof of Proposition~\ref{propn Seg tensor} shows that it suffices to verify the this for the case $K = \Delta^{1}$.
		By Lemma~\ref{lem decomposing}, there is an equivalence $\olI^+ \amalg_{\olI} \olI^- \to
	\olI \times p^*(\Delta^{1})$ for every $\olI\in \levelU$ of height 1.
	Since $y^{*}\xfc(X)$ is a colimit of these objects, the map $y^{*}\xfc(s) \times p^{*}(\Delta^{1})$ is equivalent to
	\[ y^{*}\xfc^{+}(X) \amalg_{y^{*}\xfc(X)} y^{*}\xfc^{-}(X) \to y^{*}\xfc^{+}(Y) \amalg_{y^{*}\xfc(Y)} y^{*}\xfc^{-}(Y).\]
	It then suffices to show that the morphisms $y^{*}\xfc^{\pm}(X) \to y^{*}\xfc^{\pm}(Y)$ are both $\mathbb{S}$-Segal equivalences.
	We consider the case of $\xfc^{+}$; the proof for $\xfc^{-}$ is the same.
	If $\xfc=\xfc_{m,n}$ then the definition of $\xfc_{m,n}$ implies that the upper horizontal map of the commutative diagram
	\nolabelcsquare{y^*\xfc^{+}(X)_\Seg}{y^*\xfc^{+}(Y)_\Seg}{y^*\xfc^{+}(X)}{y^*\xfc^{+}(Y)}
	is given by $\left(\coprod_{m} \mathfrak{c}_{1,1}(\bbone)\right)
		\amalg_{\left(\coprod_{m} \mathfrak{e}\right)} y^*\xfc(X)\to\left(\coprod_{m} \mathfrak{c}_{1,1}( \bbone)\right)
		\amalg_{\left(\coprod_{m} \mathfrak{e}\right)} y^*\xfc(Y).$
	As a pushout of $y^{*}\xfc(X) \to y^{*}\xfc(Y)$, this map is an $\mathbb{S}$-Segal equivalence.
	The vertical maps are Segal equivalences by definition, hence, the bottom horizontal map is also an $\mathbb{S}$-Segal equivalence.
\end{proof}

\begin{corollary}\label{cor Alg(-,-)}
	Let $\xV$ be a presentably symmetric monoidal $\infty$-category.
	There exists a tensor product
	$\otimes\colon \Segrep(\levelV)\times \Seg(\simp)\to
	\Segrep(\levelV)$ and it induces
	\[\xAlg^\xV_{(\blank)}(\blank)\colon (\Segrep(\levelV))^\op\times
	\Segrep(\levelV) \to \Seg(\simp)\] such that
	\[ \Map_{\Seg(\simp)}(\xcc, \Alg^{\xV}_{\mathcal{Q}}(\mathcal{R}))
	\simeq \Map_{\Segrep(\levelV)}(\mathcal{Q} \otimes \xcc,
	\mathcal{R})\]
	and a cotensor product
	\[(\blank)^{(\blank)}\colon
	\Segrep(\levelV)\times (\Seg(\simp))^\op\to
	\Segrep(\levelV)\]
	such that
	\[\Map_{\Segrep(\levelV)}(\mathcal{Q}, \mathcal{R}^{\xcc}) \simeq
	\Map_{\Segrep(\levelV)}(\mathcal{Q} \otimes \xcc, \mathcal{R}).\]
	Moreover, both of these functors preserve limits in each variable.
\end{corollary}
\begin{proof}
	Since $\xV$ is presentable there exists a small $\infty$-category $\xU$ and a set of morphisms $\mathbb{S}$ in $\Pre(\xU)$ such that $\xV\simeq \Pre_{\mathbb S}(\xU)$. The existence of the tensor product then follows from Corollary~\ref{cor tensor S} and the equivalence $\Seg_\mathbb{S}(\levelU)\simeq \Segrep(\levelV)$ of Proposition~\ref{prop PSeg=PCS}. The remaining statements follow from the adjoint functor theorem.
\end{proof}

\begin{remark}\label{remark sec tensor product extension}
All proofs and statements in \S\ref{sec tensor product} hold equally well if $\levelg$ is replaced by $\levelg_{\name{out}}$ or $\levelg_{\name{0-type}}$ from Definition~\ref{def subcategories of levelg}.
This leads to the natural question about the compatibility of the various tensor products.
We will address this in Theorem~\ref{thm generic compatibility with tensor} below, where we show that they are related via left Kan extension.
\end{remark}

\section{Comparison of \texorpdfstring{$\levelgconn$ and $\bbY$}{Lconn and G} presheaves}\label{sec comparison}

In Theorem~\ref{theorem LcL Vrep} we have shown the canonical inclusion $\levelgconn^\xV\hookrightarrow \levelV$ induces an equivalence $\Segrep(\levelg^\xV)\simeq \Segrep(\levelgconn^\xV)$. 
In this section we introduce a functor $\bartau \colon \levelcV \to \bbY^\xV$ lying over $\tau\colon \levelgconn \to \bbY$ from Lemma~\ref{lem factorization int el}.
The main result is Theorem~\ref{thm:PSegeq}, which shows that $\bartau$ induces an equivalence $\bartau^{*} \colon \Segrep(\bbYV) \isoto \Segrep(\levelcV)$.
The key step is to show that these $\infty$-categories are the $\infty$-categories of algebras of the same monad. 
Once this is proven an $\infty$-categorical version of the Barr--Beck Theorem gives the desired equivalence.

\subsection{The equivalence and its consequences}
Let $\xV$ be a symmetric monoidal \icat{}.
Recall that the category of connected level graphs admits a vertex functor $\levelgconn \hookrightarrow \levelg \xrightarrow{\vertex_{\levelg}} \pfinset^\op$ (see Definition~\ref{def vertex levelg}), which was used in 
Definition~\ref{def LV_YV} to define $\oplevelcV$ as the pullback
	\csquare{{\oplevelcV}}{\xV^{\op,\otimes}}{ \levelgconn^\op }{\pfinsetskel.}{}{}{}{}
The functor $\tau\colon \levelgconn \to \bbY$ from Lemma~\ref{lem factorization int el} fits into the commutative diagram
\[ \begin{tikzcd}[column sep=small]
\levelg \arrow[dr, "\vertex_{\levelg}" swap]& \levelgconn \lar[hook'] \rar & \bbY \arrow[dl, "\vertex_{\bbY}"] \\
& \pfinset^\op 
\end{tikzcd} \]
(where $\vertex_{\bbY}$ is from Definition~\ref{definition bbY to Fin}).
It follows that there is a commutative diagram
\[ \begin{tikzcd}
\oplevelV  \dar &  \oplevelcV \lar \rar \dar \arrow[dr, phantom, "\lrcorner" very near start] \arrow[dl, phantom, "\llcorner" very near start] & \opbbYV \dar  \\
\levelg^\op & \levelgconn^\op \lar[hook'] \rar{\tau^\op} & \bbY^\op
\end{tikzcd} \]
in which both squares are pullbacks.
\begin{defn}
For a symmetric monoidal \icat{} $\xV$, write \[ \bartau^\op \colon \oplevelcV \to \opbbYV\] for the arrow appearing in the preceding pullback diagram.
\end{defn}

\begin{remark}\label{rmk bartauel eq}
The functor $\tau$ restricts to an isomorphism of categories $\levelel \cong \bbY_\xel$ (Lemma~\ref{lem factorization int el}), hence the restriction $\levelcVel \to \bbYVel$ of $\bartau$ is an equivalence.
\end{remark}

\begin{lemma}\label{lem LY}
For $\overline L\in \oplevelcV$ and $\overline E\in \oplevelcVel$, we have equivalences $(\oplevelcVel)_{\overline L/}\simeq (\opbbYVel)_{\bartau \overline L/}$ and 
$(\oplevelcV)_{\overline E/}\simeq \oplevelcV\times_{\opbbYV}(\opbbYV)_{\bartau \overline E/}$.
\end{lemma}
\begin{proof}	
The definitions of inert maps in $\levelg$ and $\bbY$ imply that $\tau$ induces an equivalence $\xMap_{\levelg_{\name{c}, \name{int}}}( E,  L)\isoto \xMap_{\hryint}(\tau E, \tau L)$. Since the images $L$ and $\tau L$ in $\pfinset$ coincide, the constructions of $\oplevelcV$ and $\opbbYV$ show that the map
\[
\xMap_{\levelg_{\name{c}, \name{int}}^{\op, \xV}}( \overline L, \overline E)\isoto \xMap_{\hryint^{\op,\xV}}(\bartau \overline L, \bartau \overline E)
\]
is an equivalence.
This shows that the fibres of the vertical maps of the commutative square
\csquare{(\oplevelcVel)_{\overline L/}}{(\opbbYVel)_{\bartau \overline L/}}{\oplevelcVel}{\opbbYVel}{}{}{}{\sim}
are equivalent and therefore, the upper horizontal map is an equivalence.

For the second equivalence we first observe that a map $\tau L\to \tau E$ in $\bbY$ is either active or factors through $\tau \xfe= {\downarrow}$. 
If $\tau L\to \tau E$ is active, then it is given by contracting all inner edges, i.e. edges bounded by two vertices. 
It follows that in both cases $\tau L\to \tau E$ lies in the image of $\tau$ and hence
$\xMap_{\levelg_{\name{c}}}(L, E)\isoto \xMap_{\bbY}(\tau L, \tau E)$ is an equivalence.
Once again this gives this equivalence
\[
\xMap_{\oplevelcV}(\overline E, \overline L)\isoto \xMap_{\opbbYV}(\bartau \overline E,\bartau \overline L)
\]
which proves that the fibres of the vertical maps of the commutative square
\nolabelcsquare{(\oplevelcV)_{\overline E/}}{(\opbbYV)_{\bartau \overline E/}}{\oplevelcV}{\opbbYV}
are equivalent. It follows that the this square is cartesian, giving the second desired equivalence.
\end{proof}

Our goal is to prove:
\begin{thm}\label{thm:PSegeq}
  Let $\xU$ be a small symmetric monoidal \icat{}.
The functor $\bartau^{*} \colon \Pre(\bbYU) \to \Pre(\levelU)$ given  by composition with $\bartau^\op$
  restricts to an equivalence
  \[ \Seg(\bbYU) \to \Seg(\levelcU).\]
\end{thm}

Before we give the proof, we want to derive some consequences from this theorem.
\begin{cor}\label{cor:PSSegeq}
 Given a small symmetric monoidal \icat{} $\xU$ and a small set
  $\mathbb{S}$ of morphisms in $\Pre(\xU)$ which is
  compatible with the symmetric monoidal structure.
Then the following hold:
  \begin{enumerate}[label=(\roman*), ref={\roman*}]
  \item The functor $\bartau \colon \levelcU \to \bbYU$ induces an equivalence
    \[\bartau^*\colon  \SegS(\bbYU) \isoto \SegS(\levelcU).\] \label{equiv bartau U}
  \item The functor $\bartau \colon
    \levelgconn^{\PSU} \to
    \bbY^{\PSU}$ induces an equivalence
    \[ \bartau^*\colon \Segrep(\bbYPSU) \isoto
    \Segrep(\levelcPSU).\] \label{equiv bartau U presheaf}
  \end{enumerate}
\end{cor}

\begin{proof}
By Definition~\ref{def S properads}, $\SegS(\bbYU)$ and $\SegS(\levelcU)$ are the full subcategories of the respective $\infty$-categories $\Seg(\bbYU)$ and $\Seg(\levelcU)$, spanned by objects which are local with respect to the same set of morphisms under the equivalence 
$\levelg_{\name{c,el}}^\xU \isoto \bbY^\xU_\xel$ 
of Remark~\ref{rmk bartauel eq}.
Hence, the equivalence $\bartau^{*} \colon \Seg(\bbYU) \isoto \Seg(\levelcU)$ of Theorem~\ref{thm:PSegeq} restricts to an equivalence $\bartau^*\colon  \SegS(\bbYU) \isoto \SegS(\levelU)$ of \eqref{equiv bartau U}.

The equivalence \eqref{equiv bartau U presheaf} follows by combining Proposition~\ref{prop PSeg=PCS} with \eqref{equiv bartau U}.
\end{proof}

\begin{corollary}\label{cor equ enr properads}
If $\xV$ is a presentably symmetric monoidal $\infty$-category, then
the following $\infty$-categories are equivalent 
$$\Segrep(\bbYV) \simeq \Segrep(\levelcV)\simeq \Segrep(\levelV).$$
\end{corollary}
\begin{proof}
	The first equivalence can be identified with $\Segrep(\bbYPSU) \isoto
	\Segrep(\levelcPSU)$ of Corollary~\ref{cor:PSSegeq}\eqref{equiv bartau U presheaf} for some small symmetric monoidal $\infty$-category $\xU$ and some set $\mathbb{S}$ and the second equivalence is given by Theorem~\ref{theorem LcL Vrep}. 
\end{proof}

\begin{proposition}\label{cor:PCStensor}
	For every presentably symmetric monoidal \icat{} $\xV$, the $\infty$-category $\Segrep(\bbY^\xV)$ has a $\Seg(\simp)$-module structure where the tensor product preserves colimits in each variable.
\end{proposition}
\begin{proof}
	The previous corollary shows that $\Segrep(\bbY^\xV)\simeq \Segrep(\levelV)$ which by Proposition~\ref{prop PSeg=PCS} can be identified with $\Seg_\mathbb{S}(\levelU)$. 
	This $\infty$-category has a $\Seg(\simp)$-module structure by Corollary~\ref{cor tensor S}.
\end{proof}

\begin{remark}\label{remark sec comparison extensions}
The functor $\tau \colon \levelgconn \to \bbY$ restricts to functors $\levelg_{\name{out,c}} \to \bbYout$ and $\levelg_{\name{sc}} \to \bbY_{\name{sc}}$ (from Definition~\ref{def subcategories of levelg} and Definition~\ref{def Y}), and the proofs of all of the above statements remain valid when applied to these restrictions.
For instance, the analogue of Theorem~\ref{thm:PSegeq} tells us that 
\[
	\Seg(\bbYout^\xU) \to \Seg(\levelg_{\name{out,c}}^\xU) \quad \& \quad \Seg(\bbY_{\name{sc}}^\xU) \to \Seg(\levelg_{\name{sc}}^\xU)
\]
are equivalences.
The proof of Corollary~\ref{cor equ enr properads} gives equivalences 
\[ \begin{gathered}
\Segrep(\bbY_{\name{sc}}^\xV) \simeq \Segrep(\levelg_{\name{sc}}^\xV)\simeq \Segrep(\levelg_{\name{0-type}}^\xV) \\
\Segrep(\bbYout^\xV) \simeq \Segrep(\levelg_{\name{out,c}}^\xV)\simeq \Segrep(\levelg_{\name{out}}^\xV).
\end{gathered} \]
Finally, the proof of Proposition~\ref{cor:PCStensor} along with Remark~\ref{remark sec tensor product extension} tells us that both $\Segrep(\bbY_{\name{sc}}^\xV)$ and $\Segrep(\bbYout^\xV)$ have appropriate $\Seg(\simp)$-module structures.
\end{remark}

\subsection{Towards the proof of Theorem~\ref{thm:PSegeq}}
\label{subsection comparison proof}
The proof of Theorem~\ref{thm:PSegeq} relies on certain intricate filtrations.
We build these decompositions in this subsection, and conclude with a proof of the theorem.

\begin{definition}
Suppose that $G\in\bbY$, $I$ is a connected, height $n$ level graph, and $\varphi \colon I \to \tau^{*}G$ is a morphism in $\Pre(\levelgconn)$.
\begin{itemize}
\item We say that $\varphi$ is \emph{non-degenerate} if, whenever $I \to J \to \tau^{*}G$ is a factorization of $\varphi$ with $I\to J$ in $\levelg$ lying over a surjective morphism in $\simp$, we have that $I\to J$ is an isomorphism.
\item We let $\varphi$ also denote the adjoint map $\tau I\to G$ in $\bbY$ and we say that $\varphi$ is \emph{admissible} if the following conditions hold:
\begin{enumerate}
\item If $v\in I_{n-1,n}$ is a bottom vertex, then $\varphi(v)$ is either an edge or a corolla.\label{enum where bottom layer goes}
\item There is a unique bottom vertex $v\in I_{n-1,n}$ so that $\varphi(v)$ is a corolla. \label{enum unique bottom}
\end{enumerate}
\end{itemize}
\end{definition}

\begin{proposition}\label{proposition about heights and factorizations}
Let $I,J \in \levelgconn$ be of height $m$ and $n$, respectively, and $G\in \bbY$ a graph with at least one vertex.
Suppose 
\[
\begin{tikzcd}[column sep=tiny]
I \arrow[rr] \arrow[dr,"\psi" swap] & & J \arrow[dl,"\varphi"] \\
& \tau^*G
\end{tikzcd}\qquad 
\begin{tikzcd}[column sep=tiny]
\tau I \arrow[rr,"h"] \arrow[dr,"g" swap] & & \tau J \arrow[dl,"f"] \\
& G
\end{tikzcd}
\]
are adjoint commutative diagrams in $\Pre(\levelgconn)$ and $\bbY$, where $I\to J$ lies over $\alpha \colon [m] \to [n]$.
If $g$ is active, $\varphi$ is admissible, and $\psi$ is non-degenerate and not admissible, then $n > m$.
\end{proposition}
\begin{proof}
Since $\psi$ is non-degenerate, the map $\alpha$ is injective, hence $n\geq m$.
Since $g$ is active and $G$ has at least one vertex, we have $m > 0$.

We first show that $\alpha(m) = n$.
Let $v \in J_{n-1,n}$ be the unique vertex so that $f(C_{v})$ is a corolla, and let $w\in \vertex(G)$ be the vertex in this corolla.
The diagram
\[
\begin{tikzcd}[column sep=tiny]
\vertex(\tau I)_+ & & \vertex(\tau J)_+ \arrow[ll,"\vertex(h)" swap] & \ni v \\
& \vertex(G)_+  \arrow[ur,"\vertex(f)" swap]  \arrow[ul,"\vertex(g)"] & \ni w \arrow[ur, mapsto]
\end{tikzcd}
\]
in $\pfinset$ commutes and $\vertex(g)$ is active (Example~\ref{ex: bbG}), so $\vertex(h)(v)$ is not the base point.
In other words, since $g$ is active, there exists a vertex $u \in \vertex(\tau I)$ with $w\in g(C_{u})$.
Since $u\in I_{k-1,k}$ and $v \in J_{n-1,n}$, then $\alpha(k-1) \leq n-1 < n \leq \alpha(k)$, which implies that $n = \alpha(k)$.
This implies that $\alpha(m) = n$.

By injectivity, we know that $\alpha(m-1) < \alpha(m) = n$; now suppose that $\alpha(m-1) = n - 1$. 
Then $h(v)$ is a corolla for each $v\in I_{m-1,m} \subset \vertex(\tau I)$.
In particular, if \eqref{enum where bottom layer goes} fails for $\psi$, then there exists a $v\in I_{m-1,m}$ with $g(v)$ contains more than one vertex.
But this implies that $\varphi$ also fails \eqref{enum where bottom layer goes} for the vertex $h(v) \in J_{n-1,n}$.
On the other hand, if \eqref{enum where bottom layer goes} holds for $\psi$, then non-degeneracy and the fact that \eqref{enum unique bottom} fails for $\psi$ implies that there are distinct vertices $v,v' \in I_{m-1,m}$ which are both sent to corollas by $\psi$.
But this can't happen, otherwise $\varphi$ would fail \eqref{enum unique bottom} for $h(v) \neq h(v') \in J_{n-1,n}$.

We conclude that $\alpha(m-1)$ cannot be $n-1$, hence $n-1$ is not in the image of the injective map $\alpha$ and the conclusion follows.
\end{proof}

\begin{lemma}\label{lemma one lower nonadmissible}
Suppose that $\varphi \colon I \to \tau^* G$ is non-degenerate and admissible.
If $I \in \levelgconn$ is of height $n > 1$, then the composite $d_{n-1} I \to I \to \tau^* G$ is not admissible.
\end{lemma}
\begin{proof}
Write $\varphi' \colon d_{n-1}I \to I \to \tau^*G$ for the induced map.
As $\varphi$ is admissible, there is a unique bottom vertex $v\in I_{n-1,n}$ with $\varphi(v)$ a corolla.
Consider the commutative square
\[ \begin{tikzcd}
{[1]} \rar{d^1} \dar{\beta} \arrow[dr,"\gamma" description] & {[2]} \dar{\alpha} \\
{[n-1]} \rar[swap]{d^{n-1}} & {[n]}
\end{tikzcd} \]
where $\alpha(t) = t + n - 2$ and $\beta(t) = t + n - 2$.
We have that $\alpha^* I \to I$ is a (possibly disconnected) height two subgraph, whose connected components are in bijection with the vertices of $\gamma^*I$.
Since $\varphi$ is non-degenerate, there exists at least one vertex $w \in I_{n-2,n-1}$ so that $\varphi(w)$ is not an edge.

If $w$ and $v$ are in the same component of $\alpha^*I$, then there is a bottom vertex $x$ of $d_{n-1}I$ so that both $w$ and $v$ are in the image of $x$. 
In this case, $\varphi'$ does not satisfy \eqref{enum where bottom layer goes} since $\varphi'(x)$ contains $\varphi(w)$ and $\varphi(v)$, hence contains more than one vertex.

If $w$ and $v$ are in different components of $\alpha^*I$, then there are distinct bottom vertices $y$ and $x$ of $d_{n-1}I$ with $w\mapsto y$ and $v \mapsto x$ under $\vertex(I)_+ \rightarrow \vertex(d_{n-1}I)_+$.
We have $\varphi(w) \subseteq \varphi'(y)$ and $\varphi(v) \subseteq \varphi'(x)$, so $\varphi'(y)$ and $\varphi'(x)$ are not edges.
If $\varphi'(y)$ and $\varphi'(x)$ are both corollas, then $\varphi'$ does not satisfy \eqref{enum unique bottom}, while if one of them is not a corolla, then $\varphi'$ does not satisfy \eqref{enum where bottom layer goes}.
\end{proof}

\begin{lemma}\label{lemma preservation of admissibility}
Suppose that $\varphi \colon I \to \tau^* G$ is admissible and $I \in \levelgconn$ is of height $n$.
If $0< k < n-1$, then the composite $d_k I \to I \to \tau^* G$ is admissible.
\end{lemma}
\begin{proof}
The map $d_k I \to I$ in $\levelg$ identifies the sets of bottom vertices.
\end{proof}

\begin{lemma}\label{lemma preservation of nondegeneracy}
Suppose that $\varphi \colon I \to \tau^* G$ is non-degenerate and $I \in \levelgconn$ is of height $n$.
If $0< k < n$, then the composite $d_k I \to I \to \tau^* G$ is non-degenerate.
\end{lemma}
\begin{proof}
Write $\psi'$ for this composite. 
Since $d_kI\to I$ restricts to an identity for vertices $v\in d_k I$ which are not in $(d_kI)_{k-1,k}$, we only need to exclude the possibility that $\psi'(v)$ is an edge for each vertex $v$ in $(d_kI)_{k-1,k} \cong I_{k-1,k+1} = I_{k-1,k} \amalg_{I_{k,k}} I_{k,k+1}$. 
If this were the case, then every vertex $w \in I_{k-1,k} \amalg I_{k,k+1}$ maps to some $v\in (d_kI)_{k-1,k}$ under this isomorphism and $\psi(w) \subset \psi'(v)$.
Thus each vertex of $I$ in level $k$ or $k+1$ must go to an edge, a contradiction.
\end{proof}

We now endeavor to show that there is a sufficient supply of non-degenerate, admissible maps $I \to \tau^*G$.
The next construction, which proves the existence of certain factorizations, is essential for the proof of Proposition~\ref{proposition exterior boundary inner anodyne}.
In the construction, we say that a vertex $v$ of a graph $G$ is a \emph{bottom vertex} if all of its outgoing edges are also outgoing edges for the graph, $\out(v) \subseteq \out(G)$.
Given a bottom vertex $v$, we can form a (possibly disconnected) graph $G\setminus v$ with $\vertex(G\setminus v) = \vertex(G) \setminus \{ v \}$, $\edge(G\setminus v) = \edge(G) \setminus \out(v)$ with the rest of the structure evident.
We have $C_v \amalg_{\inp(v)} G \setminus v \cong C_v \cup G\setminus v = G$, where the second expression is a union as subobjects of $G$.
If $H$ is a connected component of $G\setminus v$, then $H \in \sub(G)$ is a structured subgraph.

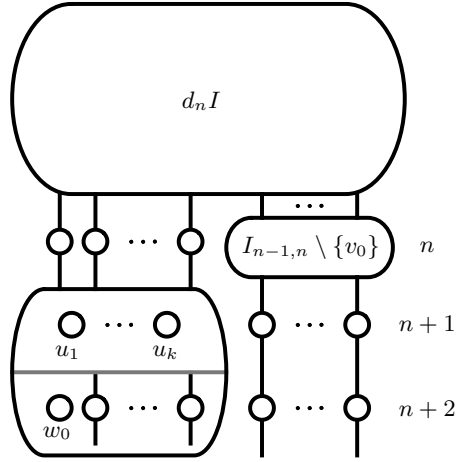
\begin{figure}[htb]
\labellist
\small\hair 2pt
 \pinlabel {$n$} [ ] at 327 171
 \pinlabel {$n+1$} [ ] at 327 112
 \pinlabel {$n+2$} [ ] at 327 48
 \pinlabel {$d_n I$} [ ] at 156 282
 \pinlabel {$I_{n-1,n} \setminus \{v_0\}$} [ ] at 239 171
 \pinlabel {$u_1$} [ ] at 55 91
 \pinlabel {$u_k$} [ ] at 129 91
 \pinlabel {$w_0$} [ ] at 47 30
\endlabellist
\centering
\includegraphics[width=0.5\textwidth]{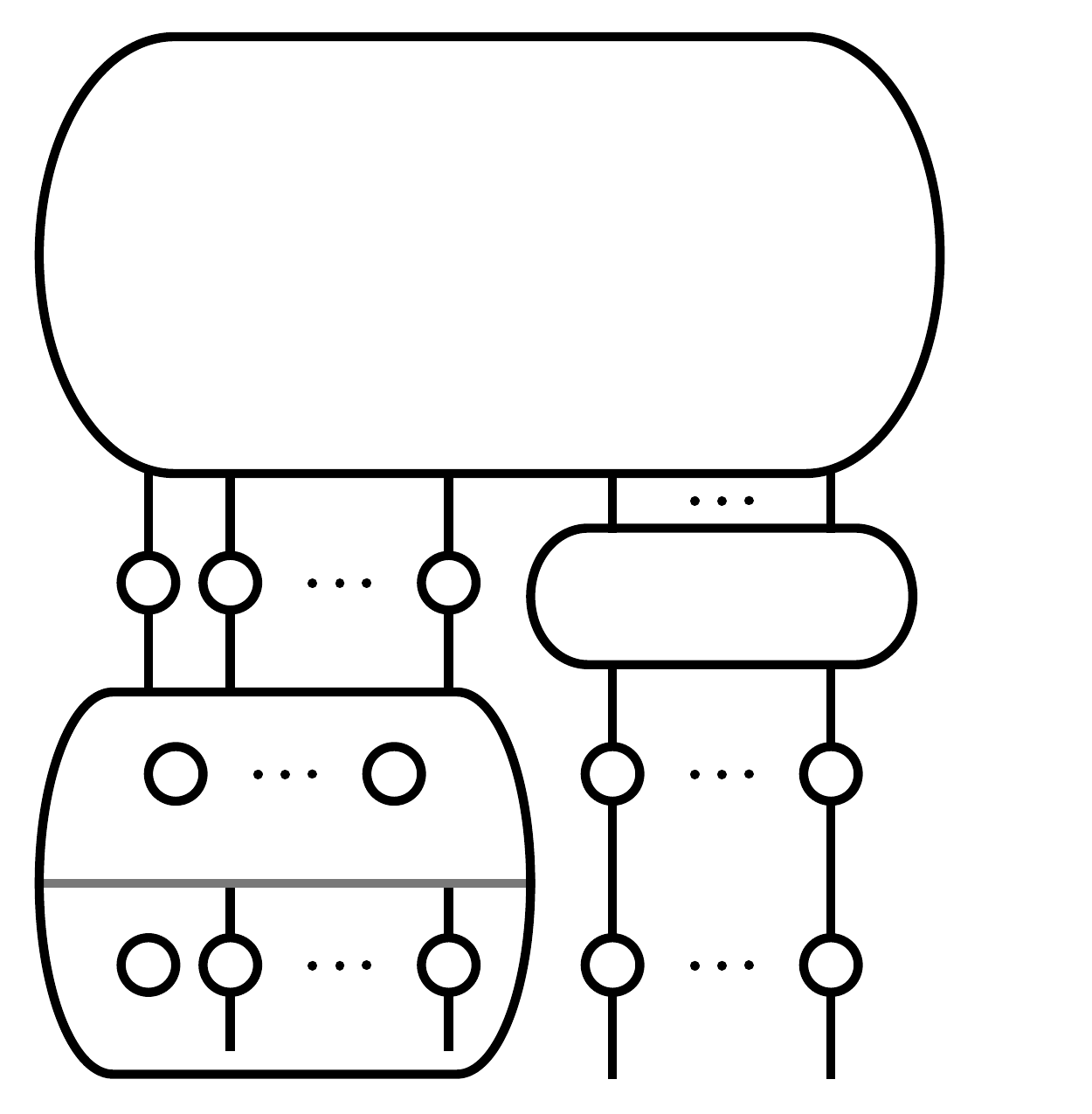}
\caption{Schematic of the level graph $J$}
\label{figure J schematic picture}
\end{figure}

\begin{construction}\label{construction: factorization through admissible}
Let $I \in \levelgconn$ be of height $n > 0$, $G \in \bbY$, and suppose that $\psi \colon  \tau I \to G$ is an active map whose adjoint $I \to \tau^*G$ is non-degenerate.
We assume that the adjoint of $\psi$ is not admissible, and construct a factorization $\tau I \to \tau J \to G$ with $I\to J$ in $\levelg$ so that $J \to \tau^*G$ is admissible and non-degenerate.
The level graph $J$ will have height $n+1$ or $n+2$, depending on $\psi$.
A schematic picture is given in Figure~\ref{figure J schematic picture}, where the $u_t$ and $w_0$ will be explained below.
The map $\tau J \to G$ will send all of the $1,1$ vertices in the picture to edges, and will behave as $\psi$ on $d_nI$ and on $I_{n-1,n} \setminus \{w_0\}$.
The schematic tells us that we should have $J_{k,k} = I_{k,k}$ and $J_{k-1,k} = I_{k-1,k}$ for $k \leq n-1$, so that the top of $J$ coincides with the (possibly disconnected) level graph $d_nI$ of height $n-1$.

Since the adjoint of $\psi$ is non-degenerate, there exists a vertex $v_0$ in level $n$ of $I$ (that is, $v_0 \in I_{n-1,n}$) with $\psi(v_0) \in \sub(G)$ not an edge.
Let $w_0$ be a bottom vertex of the graph $\psi(v_0)$, and let $H_1, \dots, H_k \in \sub(\psi(v_0)) \subseteq \sub(G)$ be the connected components of the (possibly disconnected) graph $\psi(v_0) \setminus w_0$.
For each $t=1,\dots, k$, the set $\inp(w_0) \cap \out(H_t)$ is nonempty, otherwise $\psi(v_0)$ would be disconnected.
Notice that 
\begin{equation}\label{input equation} \inp(\psi(v_0) \setminus w_0) = \inp(H_1 \amalg \dots \amalg H_k) \cong \inp(v_0) \subseteq I_{n-1,n-1}, \end{equation}
and let
\begin{equation}\label{output equation}
	\out(\psi(v_0) \setminus w_0) \cong X \amalg Y
\end{equation}
where $X = \inp(w_0)$ and $Y\subseteq \out(v_0) \subseteq I_{n,n}$ is the subset so that $\psi(Y) \amalg \out(w_0) = \out(\psi(v_0)) \cong \out(v_0)$.
Let $U = \{ u_1, \dots, u_k \}$ be a set of size $k$ (equal to the number of components of $\psi(v_0) \setminus w_0$).
Define a level graph $K \colon \scriptyell_0^3 \to \Set$ of height 3 which is of the form
\[
\resizebox{\textwidth}{!}{
 \begin{tikzcd}[column sep=-1.4cm, ampersand replacement=\&]
I_{n-1,n-1} \arrow[dr] \& \& \inp(v_0) \amalg (I_{n,n} \setminus \out(v_0))  \arrow[dl] \arrow[dr]
\& \& X \amalg Y \amalg (I_{n,n} \setminus \out(v_0))  \arrow[dr] \arrow[dl] \& \& I_{n,n} \arrow[dl]
 \\
\& \inp(v_0) \amalg (I_{n-1,n} \setminus \{v_0\})  \& \& U \amalg (I_{n,n} \setminus \out(v_0)) \& \& \{w_0\} \amalg Y \amalg (I_{n,n} \setminus \out(v_0))
\end{tikzcd}
}\]
We must of course define these maps.
All of these functions are defined to be the identity when possible (for example, the two maps on the left are the identity on $\inp(v_0) \subseteq I_{n-1,n-1}$), and otherwise the two maps on the left are induced from the maps in $I$.
The third function involves a map $\inp(v_0) \to U$, which sends $a\in \inp(v_0)$ to $u_t$ if $a$ corresponds to an input of $H_t$ under the isomorphism \eqref{input equation}.
Likewise, the map $X\amalg Y \to U$ in the fourth map comes from sending $a$ to $u_t$ if $a$ corresponds to an output of $H_t$ under the isomorphism \eqref{output equation}.
The fifth map sends $X$ to $w_0$ and the final map sends those elements of $I_{n,n}$ which map to $\out(w_0)$ under $\psi$ to $w_0$.

Let us now define $f \colon \tau K \to G$ on vertices.
At level one, $f$ is given by 
\[
	\inp(v_0) \amalg (I_{n-1,n} \setminus \{ v_0 \}) \hookrightarrow \edge(\tau I) \amalg \vertex(\tau I) \xrightarrow{\psi} \sub(G).
\]
The binary vertices labeled by $I_{n,n} \setminus \out(v_0)$ in levels two and three are sent to the appropriate elements of $\edge(G) \subseteq \sub(G)$ using $\psi$.
The remainder of level two is the set $U$, and we declare that $f(u_t) = H_t$.
At level three, we define $f(w_0) = C_{w_0}$ and $f(y) = \psi(y) \in \edge(G)$ for $y\in Y$.

As constructed, it is possible that the adjoint of $f \colon \tau K \to G$ is degenerate.
This will happen either when $\psi(v_0) = C_{w_0}$, in which case $f$ sends every vertex in level two to edges, or when there is a \emph{unique} vertex $v_0 \in I_{n-1,n}$ with $\psi(v_0)$ not an edge, in which case $f$ sends all vertices in level one to edges.
Notice that we cannot have both of these situations occur simultaneously, otherwise the adjoint of $\psi$ would have been admissible already.
We thus let $f' \colon \tau K' \to G$ be a non-degenerate factorization, where $K'$ is of height two or three and $K \to K'$ lives over $s^0\colon [3] \to [2]$, $s^1 \colon [3] \to [2]$, or $\id\colon [3] \to [3]$.
We do not say that the adjoint of $f'$ is admissible and non-degenerate, but only because those terms were only defined when the domain is a \emph{connected} level graph.
Let \[ J \coloneqq d_nI \coprod_{I_{n-1,n-1}} K'\]
be the graph of height $n+1$ or $n+2$, and $\psi' \colon \tau J \to G$ be induced by $d_n\psi$ and $f'$.
Note that $J$ is connected, $\psi'$ is active, and the adjoint of $\psi'$ is admissible and non-degenerate.
Further, we either have $I = d_{n} J$ or $I=d_{n} d_{n+1} J$, and $\tau I \to \tau J \to G$ is just $\psi$.
\end{construction}

\begin{defn}[External boundary]\label{def ext} \quad
\begin{itemize}
\item Suppose that $G\in \bbY$ is a graph, and let $\name{Sub}(G)$ be the full subcategory of $(\bbY_{\xint})_{/G}$ spanned by the non-invertible morphisms.
We let $\Dext G$ be the colimit of the composition
\[
	\name{Sub}(G) \to \bbY \to \Pre(\bbY),
\]
which we call the \emph{external boundary} of $G$.
\item There is an analogous functor $\name{Sub}(G) \to \Pre(\bbY)$ which takes $H$ to $H_{\Seg}$.
We write $(\Dext G)_{\Seg}$ for the colimit of this functor, which is the same as 
\[
	\colim_{H\in \name{Sub}(G)} H_{\Seg} \simeq \colim_{H\in \name{Sub}(G)} (H \times_{G} G_{\Seg}) \simeq (\Dext G) \times_{G} G_\Seg.
\]
\item Analogously, suppose we are given an object $\overline G \in \bbYV$.
	We write $\name{Sub}(\overline G)$ for the full subcategory of $(\bbYVint)_{/\overline G}$ spanned by non-equivalence morphisms and we define $\Dext \overline G$ to be the colimit of the composition $\name{Sub}(\overline G) \to	\bbYV \to \Pre(\bbYV)$.
	We call $\Dext \overline G$ the \emph{external boundary} of $\overline G$.
\end{itemize}
\end{defn}

\begin{propn}\label{proposition exterior boundary inner anodyne}
For every object $G$ in $\bbY$ with at least two vertices, the map $\tau^{*} (\Dext G)  \to \tau^{*} G$ is an inner anodyne map in $\Pre(\levelgconn)$.
\end{propn}
\begin{proof}
The presheaf $\tau^{*}(\Dext G)$ is a subpresheaf of $\tau^{*}G$, which is a presheaf of sets.
For $n\geq 0$, we define the presheaf $F_{n}$ by declaring that $F_n(I)$ is the union of $\tau^{*}(\Dext G)(I)$ with the set of maps $I \to \tau^{*}G$ which factor through an admissible and non-degenerate morphism $J \to \tau^{*}G$ with $J$ of height $\leq n$ and whose adjoint $\tau I \to G$ is active. Note that then $\tau J\to G$ has to be active too. To see this let $\tau J \to K\to G$ be the active-inert factorization of $\tau J\to G$  and let $\tau I \to L \to K$ the active-inert factorization of $\tau I\to \tau J \to K$. Since $\tau I \to L$ is active, $L \to K \to G$ is inert and their composition equals the active morphism $\tau I\to G$, the uniqueness of the factorization implies that $L\to G$ is an equivalence. Now, the equivalence $\bbY\simeq \hryGamma$ of Theorem~\ref{theorem upsilon equivalent definitions} shows that $\bbY$ is a generalized Reedy category by \cite[Theorem 6.70]{hrybook}. According to \cite[Lemma 6.65]{hrybook} inert morphisms such as $G\simeq L\to K $ and $K\to G$ preserve or raise degrees. Since their composite is an equivalence, it preserves the degree, hence, $G\to K $ and $K\to G$ also preserve the degree and are isomorphisms by the definition of generalized Reedy categories \cite[Definition 1.1]{bm_reedy}. In particular, we see that $\tau J\to G$ is active if $\tau I\to G$ is.

Each morphism $I \to \tau^*G$ whose adjoint is not active is automatically contained in $\tau^{*}(\Dext G)$.
Thus each map $\mathfrak{e} \to \tau^{*}G$  (where $\mathfrak{e}$ is the unique height zero object of $\levelgconn$) is in $\tau^{*}(\Dext G)$.
On the other hand, $\mathfrak{e} \to G$ is never active when $G$ has at least one vertex, hence $F_0 = \tau^{*}(\Dext G)$.
By Construction~\ref{construction: factorization through admissible}, every active map $\tau J \to G$ with $J \in \levelgconn$ necessarily factors through a non-degenerate, admissible map, hence the filtration
\[ F_0 = \tau^{*}(\Dext G) \subseteq F_1 \subseteq \cdots \subseteq \colim_{n \to \infty} F_{n} = \tau^{*}G
\]
is exhaustive.

It suffices to show that each inclusion $F_{n-1} \hookrightarrow F_{n}$ is inner anodyne.
Let $S_{n}$ denote the set of isomorphism classes of non-degenerate, admissible, active maps $\varphi \colon \tau(I) \to G$ where $I \in \levelgconn$ is has height $n$.
Since $G$ has at least two vertices, there is no map $\varphi \colon \xfc \to \tau^*G$  which is admissible and has active adjoint, so $S_1 = \varnothing$.
It follows that $F_0 = F_1$.
In general, for $\varphi \in F_n$ ($n \geq 2$) the following hold:
	 \begin{itemize}
	 	\item For $i\in \{0,n\}$, the assumption that $\varphi$ is non-degenerate and the definition of $\tau^{*}(\Dext G)$ imply that the faces $d_{i}I \to I \to \tau^{*}G$
	 	factor through $\tau^{*}(\Dext G)$, and thus through
	 	$F_{n-1}$.
	 	\item For $0 < i < n-1$, admissibility and non-degeneracy of $\varphi$ implies the same about the faces $d_{i}I \to I \to \tau^{*}G$ (by Lemma~\ref{lemma preservation of admissibility} and Lemma~\ref{lemma preservation of nondegeneracy}) which thus factor through $F_{n-1}$.
	 	\item The face $d_{n-1}I \to I \to \tau^{*}G$ is \emph{not}
	 	admissible by Lemma~\ref{lemma one lower nonadmissible}.
	 	Further, it is non-degenerate (by Lemma~\ref{lemma preservation of nondegeneracy}) and its adjoint is active, so Proposition~\ref{proposition about heights and factorizations} applies that it can only factor through an admissible map whose domain has at least height $n$.
	 	In particular, $d_{n-1}I \to I \to \tau^{*}G$ cannot lie in $F_{n-1}$.
	 \end{itemize}
This shows that there exists a pushout diagram
	 \nolabelcsquare{\coprod_{S_{n}}
	 	\Lambda^{n}_{n-1}I}{F_{n-1}}{\coprod_{S_{n}} I}{F_{n}.}  
Since the left vertical morphism is inner anodyne by definition, so is the right vertical map.
\end{proof}

\begin{propn}\label{prop preserving segal equivalences}
	Let $\xU$ be a symmetric monoidal $\infty$-category.
  The functor $\bartau^{*} \colon \Pre(\bbYU) \to \Pre(\levelcU)$
  preserves Segal equivalences.
\end{propn}
\begin{proof}
  It suffices to show that $\bartau^*$ preserves Segal core inclusions, which generate Segal equivalences.
In other words, for every $\overline G\in\bbYU$, the map
  $\bartau^{*}\overline G_{\Seg} \to \bartau^{*}\overline G$ needs to be a Segal
  equivalence in $\Pre(\levelcU)$.
We prove this by inducting on the
  number of vertices of $\overline G$.
Since the statement is
  vacuous if $\overline G$ has zero or one vertices, we assume that $\overline G$ has at least two vertices and lies over $G\in\bbY$.
Let $\pi^*\colon \Pre(\bbY)\to \Pre(\bbYU)$ be given by composition with the opposite of the natural projection $\pi\colon \bbYU \to \bbY$. We then define $\Dext \overline G$ to be the presheaf given by the pullback square
  \nolabelcsquare{\Dext \overline G}{\overline G}{\pi^* \Dext
    G}{\pi^* G.}  
We let $(\Dext \overline G)_{\Seg}$ be defined analogously (see Definition~\ref{def ext}).
By applying the idea of the proof of \cite[Lemma 5.8]{ChuHaugsengHeuts} to our case, we see that the canonical map $(\Dext G)_{\Seg} \to G_{\Seg}$ is an equivalence. 

Since $\Pre(\bbYU)$ is an $\infty$-topos, pullbacks commute with colimits, and the map $(\Dext \overline G)_{\Seg} \isoto \overline G_{\Seg}$, given by the pullback of $\pi^*(\Dext G)_{\Seg} \isoto \pi^*G_{\Seg}$, is an equivalence.
We now have a
  commutative square
  \csquare{(\Dext \overline G)_{\Seg}}{\Dext
    \overline G}{\overline G_{\Seg}}{\overline G.}{}{\wr}{}{}
The upper horizontal morphism is a colimit indexed by $\name{Sub}(G)$ of generating Segal
  equivalences for graphs with fewer vertices than $\overline G$, and
  is therefore mapped to a Segal equivalence in $\Pre(\levelcU)$ by the
  inductive hypothesis.
By the $2$-of-$3$ property the claim follows if
  $\bartau^{*}\Dext \overline G \to \bartau^{*}\overline G$ is a Segal equivalence.
Let $p\colon \levelcU \to
  \levelgconn$ be the projection.
Since $\tau p\simeq \pi \bartau$ and the right adjoint $\bartau^{*}$  preserves pullbacks we have a cartesian
  square
  \nolabelcsquare{\bartau^{*}\Dext
    \overline G}{\bartau^{*}\overline G}{p^* \tau^* \Dext
    G}{p^* \tau^* G.}
  The previous proposition shows that
  $\tau^{*} \Dext G  \to \tau^{*} G$ is an inner anodyne map in
  $\Pre(\levelg)$.
The right-hand map $\bartau^*\overline G \to p^*\tau^* G$ is equivalent to the unit $\bartau^* \overline G \to p^*p_! \bartau^* \overline G$, so by Lemma~\ref{lemma: relative simple unit} is simple.
By Corollary~\ref{cor pb of inner anodyne}, the top map is inner anodyne, hence a Segal equivalence.
\end{proof}

\begin{proposition}
\label{proposition right kan extension of bartau restricts}
	Let $\xU$ be a small symmetric monoidal $\infty$-category.
The adjunction 
$\bartau^*\colon \xFun(\opbbYU,\xS)\rightleftarrows \xFun(\oplevelcU, \xS) \cocolon \bartau_*$ restricts to an adjunction
\[ \bartau^* \colon \Seg(\bbYU)\rightleftarrows \Seg(\levelcU) \cocolon \bartau_*.\]
\end{proposition}
\begin{proof}
The claim is saying that the functors $\bartau_*$ and $\bartau^*$ preserve local objects which is equivalent to requiring that the corresponding left adjoints $\bartau^*$ and $\bartau_!$ preserve Segal equivalences. 
The functor $\bartau^*$ preserves Segal equivalences by the previous proposition and it remains to prove that the same is true for $\bartau_!$. 
Since the left adjoint $\bartau_!$ preserves colimits and $\bartau_! \overline K$ is represented by $\bartau \overline K$ we have \[\bartau_!(\overline L_\Seg)\simeq \colim_{\overline E\in (\levelg^{\op,\xU}_\xel)_{\overline L/}} \bartau\overline E. \] 
The equivalence $\levelcUel \to \bbYUel$ and
$(\oplevelcUel)_{\overline L/}\simeq (\opbbYUel)_{\bartau \overline L/}$ of Remark~\ref{rmk bartauel eq} and Lemma~\ref{lem LY} implies that \[ \bartau_!(\overline L_\Seg)\simeq \colim_{\overline G\in (\opbbYUel)_{\bartau \overline L/}} \overline G.\] 
Hence, $\bartau_!$ takes a generating Segal equivalence $\overline L_\Seg\to \overline L$ to a generating Segal equivalence $(\bartau \overline L)_{\Seg}\to \bartau \overline L$. 
\end{proof}

\begin{proof}[Proof of Theorem~\ref{thm:PSegeq}]
We have a commutative square 
\[
\begin{tikzcd}
\Seg(\bbYU) \arrow{d}{\bar \imath^*_{\bbY}} \arrow{r}{\bartau^*} & \Seg(\levelcU) \arrow{d}{\bar\imath_{\levelg}^*} \\
\Seg(\bbYUint) \arrow{r}{\bartauint^*}  & \Seg(\levelcUint) 
\end{tikzcd}
\]
of right adjoints (where $\bartauint \colon  \levelcUint \to \bbYUint$ is the restriction of $\bartau$), where the vertical morphisms induced by precompositions of the canonical inclusions $\bar\imath_{\bbY}\colon \bbYUint \to \bbYU$, $\bar\imath_{\levelg}\colon \levelcUint \to \levelcU$ are monadic according to \cite[Corollary 8.2]{patterns1}. 
Let $F_{\bbY}$ and $F_\levelg$ denote the corresponding left adjoints. By the discussion in the paragraph following Definition~\ref{def algpatt}, the bottom horizontal map $\bartauint^*$ can be identified with $\Pre(\bbYUel)\to \Pre(\levelcUel)$ which is an equivalence by Remark~\ref{rmk bartauel eq}. Then \cite[Corollary 4.7.3.16]{ha} implies that $\bartau^{*}$ is an
equivalence if the canonical natural
transformation $F_{\levelg} \to
\bartau^{*} F_{\bbY} (\bartauint^{*})^{-1}$ is an equivalence, which is the same as requiring that the corresponding
transformation of right adjoints $\bartauint^{*} \bar\imath_{\bbY}^{*}\bartau_{*} \to \bar\imath_{\levelg}^{*}$ is an equivalence. In other words it suffices to show that for every $F \in \Seg(\levelcU)$ and $\overline E\in \levelg_{\name{c}, \xint}^{\xU, \xel}$, the canonical map $\bartauint^{*}\bar\imath_{\bbY}^{*}\bartau_{*}F(\overline E) \simeq
\bar\imath_{\levelg}^{*}\bartau^{*}\bartau_{*}F(\overline E) \to \bar\imath_{\levelg}^{*}F(\overline E)$ is an
equivalence.

The description of right Kan extension allows us to identify the domain of this map with
\begin{equation}\label{equation of limit over LcopV tau E}
(\bartau_{*}F)(\bartau \overline E) \simeq \lim F(\overline L)
\end{equation}
where the limit is over $\oplevelcU\times_{\opbbYU}(\opbbYU)_{\bartau \overline E/}$.
The equivalence \[ \oplevelcU\times_{\opbbYU}(\opbbYU)_{\bartau \overline E/} \simeq (\levelgconn^{\op, \xU})_{\overline E/}\] of Lemma~\ref{lem LY} implies that this $\infty$-category has an initial object $\id_{\overline E}$, hence the limit in \eqref{equation of limit over LcopV tau E} is equivalent to $F(\overline E)$.
\end{proof}

\section{Completeness and enriched \texorpdfstring{$\infty$}{∞}-properads}\label{sec ffes}
Two ordinary properads are equivalent if there is a fully faithful and essentially surjective functor between them. 
These two notions have natural generalizations in the $\infty$-categorical setting, and the $\infty$-category of $\xV$-enriched $\infty$-properads should be given by a localization of the algebraic model $\Segrep(\bbYV)$ for $\xV$-enriched $\infty$-properads with respect to fully faithful and essentially surjective functors. 
In \cite{Rezk}, Rezk introduced the notion of completeness and proved that complete Segal spaces model $\infty$-categories. 
This idea was generalized to enriched $\infty$-operads in \cite{ChuHaugseng}.
After defining completeness for objects in $\Segrep(\levelV)\simeq\Segrep(\bbYV)$ we show that the full subcategory of complete objects is the correct $\infty$-category of $\xV$-enriched $\infty$-properads, in the sense that it is given by localizing $\Segrep(\bbYV)$ with respect to fully faithful and essentially surjective functors. 

There are analogues of the statements in this section to dioperads and output properads, which will be discussed in Remark~\ref{remark sec ffes extensions}.
We begin with a necessary detour comparing Segal objects for presheaves over various graph categories.

\subsection{Free functors}\label{section free functors}
Consider any of the functors appearing in the following diagram.
\begin{equation}\label{eq big diagram of graph cats}
\begin{tikzcd}[column sep=small]
& & \simp \arrow[drr] \arrow[dll] \dar \\
\bbO \dar & & \simp_\finsetskel^1 \dar \arrow[rr] \arrow[ll]& & \simp_\finsetskel \dar \\
\bbYout \arrow[dd] & & \levelg_{\name{out,c}} \arrow[rr] \arrow[ll]\arrow[dd] & & \levelg_{\name{out}} \arrow[dd] \\[-1.5em]
& \bbY_{\name{sc}} \arrow[from=uul, bend left, crossing over]\arrow[dl] & & \levelg_{\name{sc}} \arrow[from=uul, bend left, crossing over]\arrow[dl] \arrow[rr, crossing over] \arrow[ll, crossing over]& & \levelg_{\name{0-type}} \arrow[from=uul, bend left, crossing over]\arrow[dl] \\[-1.5em]
\bbY & & \levelgconn \arrow[rr] \arrow[ll]& & \levelg 
\end{tikzcd} \end{equation}
Write $i \colon \gcone \to \gctwo$ for such a functor, and, for a small symmetric monoidal $\infty$-category $\xU$, write $\bar\imath \colon \gcUone \to \gcUtwo$.
We have previously shown that we have restrictions
\begin{equation}\label{eq restrictions of middle functor}
\begin{tikzcd}[column sep=small]
\Seg(\gctwo) \dar[dashed,"i^*"] \rar & \Pre(\gctwo) \dar{i^*} & \Segrep(\gcUtwo) \dar[dashed, "\bariustar"] \rar & \Seg(\gcUtwo) \dar[dashed, "\bariustar"] \rar & \Pre(\gcUtwo) \dar{\bariustar} \\
\Seg(\gcone) \rar & \Pre(\gcone) & \Segrep(\gcUone) \rar & \Seg(\gcUone) \rar & \Pre(\gcUone)
\end{tikzcd} 
\end{equation}
of $i^*$ and $\bariustar$ (see Proposition~\ref{proposition restriction restricts to segal stuff} and Proposition~\ref{proposition right kan extension of bartau restricts} along with Remark~\ref{remark sec comparison extensions}).
These should generally be considered as forgetful functors.
For example, the map $\Segrep(\levelU) \to \Segrep(\levelg_{\name{sc}}^\xU)$ takes an enriched $\infty$-properad to its underlying enriched $\infty$-dioperad.

By \cite[Proposition 4.7]{patterns1}, each dashed arrow in \eqref{eq restrictions of middle functor} admits a left adjoint, given by first taking the left Kan extension and then localizing.\footnote{In the case of the horizontal functors, which induce equivalences, this left adjoint can instead be described as the restriction of the \emph{right} Kan extension.}
We generically denote this by $Li_! \dashv i^*$ (resp.\ $L\barishriek \dashv \bariustar$).
Our main goal is to prove Theorem~\ref{thm generic compatibility with tensor}, which says that the various $\Seg(\simp)$-module structures are compatible via these left adjoints.

For several of the functors in \eqref{eq big diagram of graph cats}, we can actually get away with just using the raw left Kan extension, rather than composing with localization.
For example, the free $\infty$-properad generated by an $\infty$-operad is just given by left Kan extension at the presheaf level.
This will be convenient in the next section.

\begin{lemma}
\label{lemma: restriction of lke}
Suppose $i \colon \gcone \to \gctwo$ is one of the fully-faithful functors appearing in the diagram
\[ \begin{tikzcd}[column sep=small]
& & \simp \arrow[dll] \dar \\
\bbO \dar & & \simp_\finsetskel^1 \dar  & & \simp_\finsetskel \dar \\
\bbYout \arrow[dd] & & \levelg_{\name{out,c}}  \arrow[dd] & & \levelg_{\name{out}} \arrow[dd] \\[-1.5em]
& \bbY_{\name{sc}} \arrow[from=uul, bend left, crossing over] & & \levelg_{\name{sc}} \arrow[from=uul, bend left, crossing over] & & \levelg_{\name{0-type}} \arrow[from=uul, bend left, crossing over] \\[-1.5em]
\bbY & & \levelgconn  & & \levelg 
\end{tikzcd} \]
and let $\xU$ be a small symmetric monoidal $\infty$-category.
Then we have restrictions
\[ \begin{tikzcd}
\Seg(\gcone) \dar[dashed] \rar & \Pre(\gcone) \dar{i_!} \\
\Seg(\gctwo) \rar & \Pre(\gctwo)
\end{tikzcd} 
\quad
\begin{tikzcd}
\Segrep(\gcUone) \dar[dashed] \rar & \Seg(\gcUone) \dar[dashed] \rar & \Pre(\gcUone) \dar{\barishriek} \\
\Segrep(\gcUtwo) \rar & \Seg(\gcUtwo) \rar & \Pre(\gcUtwo)
\end{tikzcd} 
\]
which we also call $i_!$ and $\barishriek$.
These restrictions are left adjoint to the restrictions $i^*$ and $\bariustar$.
\end{lemma}
\begin{proof}
We give the proof for the version without `$\xU$,' the other is similar.
We know that $i_!$ takes representable presheaves to representable presheaves, and, since $i_!$ preserves colimits, that $i_!(K_\Seg \to K) = K_\Seg \to K$ for any $K\in \gcone$.
Further, since $i$ is fully-faithful, the unit of the adjunction $i_! \dashv i^*$ is an equivalence.
Suppose that $A$ is an object of $\Seg(\gcone)$.
As $A \simeq i^*i_!A$ is again Segal, we have, for each $K\in \gcone$, the right equivalence in the following natural diagram.
\[ \begin{tikzcd}[column sep=small]
\Map_{\Pre(\gctwo)} (K, i_!A)\rar[no head, "\simeq"]\dar & \Map_{\Pre(\gctwo)} (i_!K, i_!A)\rar[no head, "\simeq"]\dar & \Map_{\Pre(\gcone)} (K, i^*i_!A) \dar["\simeq"] \\
\Map_{\Pre(\gctwo)} (K_\Seg, i_!A)\rar[no head, "\simeq"] & \Map_{\Pre(\gctwo)} (i_!(K_\Seg), i_!A)\rar[no head, "\simeq"] & \Map_{\Pre(\gcone)} (K_\Seg, i^*i_!A)
\end{tikzcd} \]
It follows that the left-hand map is also an equivalence.

It remains to show that the map $\Map_{\Pre(\gctwo)} (K, i_!A) \to \Map_{\Pre(\gctwo)} (K_\Seg, i_!A)$ is an equivalence for $A\in \Seg(\gcUone)$ and $K\notin \gcone$.
If $G\to H$ is a morphism in $\gctwo^\op$ with $G\in \gcone$, then $H$ is also in $\gcone$ (Lemma~\ref{lem comparison between levelgraph cats}\eqref{item one comparison between levelgraph cats} and Lemma~\ref{lem comparison between dirgraph cats}).
We thus have $(\gcone^\op)_{/K}$ is empty whenever $K\notin \gcone$, hence $\Map_{\Pre(\gctwo)} (K, i_!A)= (i_!A)(K)  =\varnothing$.
But by Lemma~\ref{lem comparison between levelgraph cats}\eqref{item two comparison between levelgraph cats} or Definition~\ref{def Y} we see that $K\notin \gcone$ implies that there exists an inert map $K\to \xfc$ of $\gctwo^\op$ with $\xfc$ a corolla in $\gctwo^\op$.
As $(i_!A)(\xfc) = \varnothing$, it follows that $\Map_{\Pre(\gctwo)} (K_\Seg, i_!A) = \varnothing$.
Thus $i_!A$ is a Segal presheaf.
\end{proof}

By adjointness, we have that $i_!$ commutes with the localization functors
\[
\begin{tikzcd}
\Pre(\gcone) \dar{i_!} \rar & \Seg(\gcone) \dar{i_!} \\
\Pre(\gctwo) \rar & \Seg(\gctwo)
\end{tikzcd} 
\]
(and likewise for $\barishriek$), hence is a retract of $i_! \colon \Pre(\gcone) \to \Pre(\gctwo)$.

\begin{example}\label{example: failure of lke restriction}
The proof Lemma~\ref{lemma: restriction of lke} fails for $i \colon \levelg_{\name{sc}} \to \levelgconn$, since even if $G$ is not in $\levelg_{\name{sc}}$, all its elementary subgraphs will be.
But the result also does not hold. 
Indeed, if $F\in \Seg(\levelg_{\name{sc}})$ and $L$ is the graph 
$\begin{tikzcd}[cramped, sep=small]
{} \rar[no head, end anchor=center] & \bullet \rar[bend left=30, no head, start anchor=center, end anchor=center] \rar[bend right=30, no head, start anchor=center, end anchor=center] & \bullet \rar[no head, start anchor=center] & {}
\end{tikzcd}$
then, as in the proof of Lemma~\ref{lemma: restriction of lke}, $(i_!F)(L) = \varnothing$.
On the other hand, \[ \Map(L_{\Seg},i_!F) \simeq F(\xfc_{1,2}) \times_{F(\xfe)^2} F(\xfc_{2,1}), \]
which is often not equivalent to $\varnothing$.
\end{example}

The distinction between the two situations should not be too surprising.
To pass from an operad $O$ to the free properad $P$ it generates, we have
\[
	P(a_1, \dots, a_m; b_1, \dots, b_n) = \begin{cases}
		\varnothing & n \neq 1 \\
		O(a_1, \dots, a_m; b_1) & n=1.
	\end{cases}
\]
On the other hand, to go from a dioperad $D$ to the free properad $P$ it generates, we likely have that $P(a_1, \dots, a_m; b_1, \dots, b_n)$ and $D(a_1, \dots, a_m; b_1, \dots, b_n)$ are both inhabited, but very different.
The properad $P$ has many operations formally generated from those in $D$.

\begin{theorem}\label{thm generic compatibility with tensor}
Let $\xU$ be a small symmetric monoidal $\infty$-category and let $i \colon \gcone \to \gctwo$ be a composite of fully-faithful functors appearing in Lemma~\ref{lemma: restriction of lke}.
Then for $A \in \Seg(\gcUone)$ and $\xcc \in \Seg(\simp)$, we have
\[
	\barishriek (A \otimes \xcc) \simeq (\barishriek A) \otimes \xcc
\]
in $\Seg(\gcUtwo)$.
If instead $i$ appears in the diagram \eqref{eq big diagram of graph cats}, then $L\barishriek (A \otimes \xcc) \simeq (L\barishriek A) \otimes \xcc$.
\end{theorem}
The tensorings here are those from Theorem~\ref{theo tensor}, Remark~\ref{remark sec tensor product extension}, and \cite{ChuHaugseng}.

\begin{proof}
It is enough to prove the result for functors between categories of possibly disconnected level graphs (that is, in the rightmost part of \eqref{eq big diagram of graph cats}), as the other tensorings are defined along the compatible equivalences induced by the horizontal functors.
We prove the statement for $\barishriek$, as the proof for $L\barishriek$ is nearly identical.

Write $q$ and $p$ for the relevant composite cartesian fibrations from $\gcUone$ and $\gcUtwo$ to $\simp$, that is, the following diagram commutes:
\[ \begin{tikzcd}[row sep=tiny]
\gcUone \arrow[dd,"\bar \imath"] \rar \arrow[rrd,bend left, "q"] & \gcone \arrow[dd,"i"] \arrow[dr] \\
& & \simp\\
\gcUtwo \rar \arrow[rru, bend right, "p"'] & \gctwo \arrow[ur] 
\end{tikzcd} \]
Write $L_1 \colon \Pre(\gcUone) \to \Seg(\gcUone)$ and $L_2 \colon \Pre(\gcUtwo) \to \Seg(\gcUtwo)$ for the localization functors.

For $B\in \Seg(\gcUtwo)$ and $\xcc \in \Seg(\simp)$, we have that $B\otimes \xcc$ is defined to be $L_2(B \times p^*\xcc)$.
We apply $\bariustar$ to the localization map $B \times p^*\xcc \to B\otimes \xcc$ in $\Pre(\gcUtwo)$.
As $\bariustar$ preserves Segal objects, $\bariustar(B\otimes \xcc)$ is Segal and we have the indicated factorization
\[ \begin{tikzcd}
\bariustar(B \times p^*\xcc) \dar{\simeq} \rar & 
	\bariustar (B\otimes \xcc)  \\
(\bariustar B) \times q^*\xcc \rar & (\bariustar B)\otimes \xcc \uar[dashed]
\end{tikzcd} \]
in $\Pre(\gcUone)$.
We thus have constructed a map 
\[
	(\bariustar B)\otimes \xcc \to \bariustar (B\otimes \xcc)
\]
natural in $B\in \Seg(\gcUtwo)$ and $\xcc \in \Seg(\simp)$, hence we have
\[
	A \otimes \xcc \to (\bariustar \barishriek A)\otimes \xcc \to \bariustar ((\barishriek A)\otimes \xcc)
\]
natural in $A \in \Seg(\gcUone)$ and $\xcc \in \Seg(\simp)$.
By adjointness, this amounts to 
\begin{equation}\label{eq desired map}
	\barishriek(A \otimes \xcc) \to (\barishriek A)\otimes \xcc
\end{equation}
where both sides commute with colimits in each variable.

First note that objects $A$ in $\Seg(\gcUone)$ are colimits over corollas $\xfc_{m,n}(v)$ where $\xfc_{m,n}$ ranges over corollas of $\gcone$.
As usual it is true that any presheaf is a colimit of representables $\overline{G}$, and applying $L_1$ each of these representables splits into a colimit of elementary representables.
Moreover, $\xfe$ is a retract of $\xfc_{1,1}(\bbone)$.
Finally, $A$ was already assumed to be Segal, so the result follows for $A$ and not just $L_1(A)$.
Similarly, every object of $\Seg(\simp)$ is a colimit of $\Delta^1$.

By cocontinuity, to see that \eqref{eq desired map} is an equivalence it suffices to show that $\barishriek(\xfc(v) \otimes \Delta^1) \to (\barishriek \xfc(v))\otimes \Delta^1$ is an equivalence for each corolla $\xfc \in \gcone$.
This follows more or less by Lemma~\ref{lem decomposing}.
Specifically, the lemma gives that the left hand-side is 
\begin{align*}
\barishriek L_1 (\xfc(v) \times q^* \Delta^1) & \simeq
\barishriek L_1 (\xfc^+(v) \amalg_{\xfc(v)} \xfc^-(v)) \\ & \simeq  
(\barishriek L_1 \xfc^+(v) \amalg_{\barishriek L_1 \xfc(v)} \barishriek L_1 \xfc^-(v)) \\ & \simeq 
\xfc^+(v) \amalg_{\xfc(v)} \xfc^-(v),
\end{align*}
using that representables lying over simply-connected graphs are already local, and that $\barishriek$ sends representables to representables.
Likewise, the right-hand side is 
\begin{align*}
(\barishriek \xfc(v)) \otimes \Delta^1 &\simeq \xfc(v) \otimes \Delta^1 \\ &\simeq 
L_2 (\xfc(v) \times p^* \Delta^1) \\ &\simeq 
L_2 (\xfc^+(v) \amalg_{\xfc(v)} \xfc^-(v)) \\ &\simeq 
L_2\xfc^+(v) \amalg_{L_2\xfc(v)} L_2\xfc^-(v) \\ & \simeq
\xfc^+(v) \amalg_{\xfc(v)} \xfc^-(v).
\qedhere
\end{align*}
\end{proof}

We can extend these results to the presentable case in the usual way, by choosing a suitable subcategory of $\kappa$-compact objects.
This gives the following result.

\begin{corollary}
\label{cor adjoint rep}
Suppose that $i \colon  \gcone\to\gctwo$ is one of the maps in \eqref{eq big diagram of graph cats}.
Let $\xV$ be a presentably symmetric monoidal \icat{} and let $\bar\imath \colon \gcVone \to \gcVtwo$.
Then
	\begin{enumerate}[label=(\roman*), ref={\roman*}]
		\item There is an adjunction \label{enum first adjunction}
		\[ \barinotshriek\colon\Segrep(\gcVone) \rightleftarrows
		\Segrep(\gcVtwo)\cocolon\bariustar.\]
		\item The left adjoint $\barinotshriek$ is compatible with the tensoring with $\Seg(\simp)$ in the sense that there is a natural equivalence $\barinotshriek(X \otimes \xcc) \simeq (\barinotshriek X) \otimes \xcc$ for $X \in \Segrep(\gcVone)$ and $\xcc \in \Seg(\simp)$.\label{second natural equiv}
		\item There is a natural equivalence
		$\bariustar (A^{\xcc}) \simeq
		(\bariustar A)^{\xcc}$ for $A \in
		\Segrep(\gcVtwo)$ and $\xcc \in \Seg(\simp)$.
		\label{third natural equiv}
	\end{enumerate}
\end{corollary}
\begin{proof}
	Note that there exists a small symmetric monoidal $\infty$-category $\xU$ and a set $\mathbb{S}$ such that $\xV\simeq \PSU$.
The claims \eqref{enum first adjunction} and \eqref{second natural equiv} then follow from $\Segrep(\gc^\xV)\simeq \SegS(\gc^\xU)$ and the previous theorem.
In particular, $\barinotshriek$ is induced along equivalences from either $\barishriek$ or $L\barishriek$ from $\SegS(\gcone^\xU)$ to $\SegS(\gctwo^\xU)$ (depending on the domain and codomain of $i$).

We show only \eqref{third natural equiv}; let $Q \in \Segrep(\gcVone)$ be arbitrary.
On the one hand we have (see Proposition~\ref{cor Alg(-,-)}) 
\[
	\Map(Q, (\bariustar A)^{\xcc}) \simeq \Map(Q \otimes \xcc, \bariustar A) \simeq \Map(\barinotshriek (Q \otimes \xcc), A), 
\]
while on the other we have 
\[
	\Map(Q, \bariustar(A^\xcc)) \simeq \Map(\barinotshriek Q , A^\xcc) 
	\simeq \Map((\barinotshriek Q) \otimes \xcc, A).
\]
These coincide by \eqref{second natural equiv}.
\end{proof}

Notice that we have written $\barinotshriek$ for the left adjoint, rather than $\barishriek$, as the functor is not literally given by left Kan extension even in the situation of Lemma~\ref{lemma: restriction of lke}.
This contrasts with the conventions of \cite{ChuHaugseng}, see Warning~2.9.7 there.
This notation will only be used again at the end of \S\ref{subsec ffesc}.
 
\subsection{Fully faithfulness, essential surjectivity and completeness}\label{subsec ffesc}
In this section, we generalize the definition of the $\infty$-category $\name{Opd}^\xV_\infty$ of $\xV$-enriched $\infty$-operads from \cite[\S3.2]{ChuHaugseng} to give $\infty$-categories of $\xV$-enriched $\infty$-properads, $\infty$-output-properads, and $\infty$-dioperads.
In each case, these are full subcategories of representable presheaves on the objects whose underlying simplicial presheaf is complete.

Given a graph category $\gc$, there is a functor $\gc \to \gcV$ which on objects sends an object $G$ to $G(\bbone_c)_{c\in\vertex(G)}$.
We thus have functors, both denoted by $u$,
\[
	u \colon \simp \to \bbYV \qquad u \colon \simp \to \levelcV.
\]
By similar techniques to above, one can show that $u^* \colon \Pre(\bbYV) \to \Pre(\simp)$ and $u^* \colon \Pre(\levelcV) \to \Pre(\simp)$ restrict to Segal objects, and we have:
\begin{definition}[Underlying $\infty$-category]
Given $F$ in $\Segrep(\levelcV)$ or $\Segrep(\bbYV)$, we write call 
\[
	u^*F \in \Seg(\simp)
\]
the \emph{underlying $\infty$-category of} $F$.
\end{definition}
 
\begin{remark}[Another description of the underlying $\infty$-category]
Let $\varphi \coloneqq \xMap_\xV(\bbone, \blank)\colon \xV\to \xS$ denote the lax monoidal functor which is right adjoint to the unique colimit-preserving functor $\xS\to \xV$ taking $*$ to the unit $\bbone$.
Let $\varphi_*\colon \Segrep(\bbYV) \leftrightarrows  \Segrep(\bbY^\xS)$ be the right adjoint of Corollary~\ref{cor:laxmonftr} induced by $\varphi$.
Given an object $F\in \Segrep(\bbYV)$, we can use the equivalence $\Segrep(\bbY^\xS)\simeq \Seg(\bbY)$ of Corollary~\ref{cor hry equ} to consider $\varphi_*F \in \Seg(\bbY)$.
This object restricts further, along $i \colon \simp \to \bbY$, and we have $i^*\varphi_*F \simeq u^*F$ is the underlying enriched $\infty$-properad associated to $F$.
\end{remark}

\begin{definition}\label{def ffes}
	Let $\xMap_F(x_{1},\ldots, x_{m};y_1,\ldots, y_n)$ be as defined in Definition~\ref{def cts Seg psh}. We say a morphism $f\colon F\to F'$ in $\Segrep(\levelV)$ is
	\begin{itemize}
		\item \emph{fully faithful} if it induces an equivalence
		$$\xMap_F(x_{1},\ldots, x_{m};y_1,\ldots, y_n)\isoto \xMap_{F'}(fx_{1},\ldots, fx_{m};fy_1,\ldots,f y_n)$$ in $\xV$ for every corolla $\xfc_{m,n}$ and every $\underline{xy} \in F(\xfe)^{m+n}$, and
		\item \emph{essentially surjective} if the induced functor $u^*(f)\colon u^*(F)\to  u^*(F')$ of the underlying Segal space is essentially surjective.
	\end{itemize}
\end{definition}

\begin{definition}\label{def En}
We write $E^{n}$ for the indiscrete category (viewed as a Segal space) with $n+1$ objects, i.e.\ it has a unique morphism between any pair of objects.
\end{definition}
\begin{remark}
	The category $E^{1}$ is the ``generic isomorphism,'' so giving a morphism of Segal spaces $E^{1} \to X$ corresponds to giving two objects of $X$ and an equivalence between them.
	Similarly, giving a map $E^n\to X$ is equivalent to specifying $n+1$ equivalent objects in $X$.
\end{remark}

\begin{remark}
	Following \cite{Rezk}, the correct space of objects of a Segal space $\xcc$ should be given by $\iota \xcc$ defined as the colimit of the simplicial space $\xMap_{\Seg(\simp)}(E^{(\blank)},\xcc) $. Using this notation, the functor $f$ of Definition~\ref{def ffes} is essentially surjective if and only if 
	$\pi_0(\iota u^*(f))\colon \pi_0(\iota u^*(F))\to \pi_0(\iota u^*(F'))$ is a surjection of sets.
\end{remark}

The following extends Definition 3.2.1 and Definition 4.4.12 of \cite{ChuHaugseng} to the setting of properads.

\begin{defn}\label{def completness}
	Let $\xV$ be a presentably symmetric monoidal $\infty$-category.
	\begin{enumerate}
		\item We say a Segal space $F$ is \emph{complete} if the map
		\[F([0])\simeq \Map(E^0, F) \to \Map(E^{1}, F)\]
		induced by the map $s^0\colon E^{1} \to E^{0} \simeq \{[0]\}$, is an equivalence of spaces. 
We write $\CatI$ for the full subcategory of $\Seg(\simp)$ spanned by the complete Segal spaces.
\item We say an object $F\in \Segrep(\levelcV)$ is \emph{complete} if its underlying Segal space $u^*F$ is complete and we let $\Segrepc(\levelcV)$ denote the full subcategory of $\Segrep(\levelcV)$ spanned by the complete objects.
\item 
A similar definition holds for objects of $\Segrep(\bbYV)$, and we write $\Segrepc(\bbYV)$ for the full subcategory of $\Segrep(\bbYV)$ spanned by complete objects.
The equivalence $\Segrep(\levelcV)\simeq \Segrep(\bbYV)$ of Corollary~\ref{cor equ enr properads} induces an equivalence $\Segrepc(\levelcV) \simeq \Segrepc(\bbYV)$.
\item A \emph{$\xV$-enriched $\infty$-properad} is a complete, fibrewise representable $\bbYV$-Segal space or $\levelcV$-Segal space. 
\end{enumerate}
\end{defn}
\begin{remark}
	It follows from the previous definition and \cite[Definition 3.2.1]{ChuHaugseng} that an object in $\Segrep(\levelcV)\simeq \Segrep(\bbYV)$ is complete if and only if the underlying object in $\Segrep(\simp^{1,\xV}_{\finsetskel})\simeq \Segrep(\bbO^\xV)$ is complete.
\end{remark}

By replacing $\simp_{\finsetskel}^{\xV}$ with $\levelcV$ the proofs of \cite[\textsection 3.5]{ChuHaugseng} give the following result.
\begin{theorem}\label{theo cp obj are loc}
	There is a completion functor 
	\[\Segrep(\levelcV) \to \Segrepc(\levelcV)
	\]
	which takes every object in $\Segrep(\bbYV)$ to a complete one and is left adjoint to the inclusion $\Segrepc(\levelcV)\hookrightarrow \Segrep(\levelcV)$ of the full subcategory. Moreover, it exhibits $\Segrepc(\levelcV)$ as a localization of $\Segrep(\levelcV) $ with respect to the class of fully faithful and essentially surjective functors. \qed
\end{theorem}

Of course a similar theorem holds if $\levelcV$ is replaced by $\bbYV$.
This theorem says that the $\infty$-categories $\Segrepc(\levelcV)\simeq \Segrepc(\bbYV)$ are the correct $\infty$-category of $\xV$-enriched $\infty$-properads.
We now introduce a new notation, similar to $\name{Opd}^\xV_\infty$ for $\xV$-enriched $\infty$-operads from Notation 3.2.2 of \cite{ChuHaugseng}.

\begin{notation}\label{no PrpdV}
	We write $\nmproperad_\infty^\xV$ for $\Segrepc(\levelcV)$ or $\Segrepc(\bbYV)$ when we do not want to emphasize the specific implementation of $\infty$-category of $\xV$-enriched $\infty$-properads.
\end{notation}

It follows from the definition that a lax symmetric monoidal functor $F\colon \xV\to \xW$ induces a functor $\Alg_{\levelgconn^{\op}/\xS}(\xV)\to \Alg_{\levelgconn^{\op}/\xS}(\xW)$ of $\infty$-categories. 
This functor can be identified with $\Segrep(\levelcV)\to \Segrep(\levelgconn^\xW)$ under the equivalence $\Segrep(\levelcV)\isoto \Alg_{\levelgconn^{\op}/\xS}(\xV)$ of Theorem~\ref{thm:PCSisAlgLT}. 
We then obtain the next proposition by localizing $\Segrep(\levelcV)\to \Segrep(\levelgconn^\xW)$ with respect to fully faithful and essential surjective functors.
\begin{propn}
	The \icat{} $\nmproperad_\infty^\xV$ is functorial in $\xV$ with respect to lax symmetric monoidal functors.
	Moreover, if $F \colon \xV
	\to \xW$ is a colimit-preserving symmetric monoidal functor then $F_{*} \colon \nmproperad_\infty^\xV \to \nmproperad_\infty^{\xW}$ preserves colimits; thus $\nmproperad_\infty^{(\blank)}$ defines a functor $\name{CAlg}(\PrL) \to \PrL$ where $ \PrL$ is the $\infty$-category of presentable $\infty$-categories. \qed
\end{propn}

The proof of \cite[Proposition 3.4.9]{ChuHaugseng} gives the following proposition.
\begin{proposition}\label{prop tensor}
	The tensor product
	$\otimes\colon\Segrep(\levelV)\times \Seg(\simp)\to
	\Segrep(\levelV)$ of Proposition~\ref{cor:PCStensor} restricts to a tensor product functor
	\[\otimes\colon  \nmproperad_\infty^\xV\times \xCat_\infty\to
	\nmproperad_\infty^\xV\]
	of presentable $\infty$-categories	which preserves colimits in each variable. \qed
\end{proposition}
The adjoint functor theorem gives the following corollary.
\begin{corollary}\label{corollary algebras cotensors complete}
	Let $\xcc\in \CatI$ and $\mathcal{Q},\mathcal{R}\in \nmproperad_\infty^\xV$.
	The tensor product $\otimes\colon \nmproperad_\infty^\xV \times \CatI \to \nmproperad_\infty^\xV$ induces
	\[\xAlg^\xV_{(\blank)}(\blank)\colon (\nmproperad_\infty^\xV)^\op\times \nmproperad_\infty^\xV \to
	\CatI\] such that
	\[ \Map_{\CatI}(\xcc, \Alg^{\xV}_{\mathcal{Q}}(\mathcal{R}))
	\simeq \Map_{\nmproperad_\infty^\xV}(\mathcal{Q} \otimes \xcc,
	\mathcal{R})\]
	and a cotensor product
	\[(\blank)^{(\blank)}\colon
	\nmproperad_\infty^\xV \times \CatI^\op\to \nmproperad_\infty^\xV\]
	such that
	\[\Map_{\nmproperad_\infty^\xV}(\mathcal{Q}, \mathcal{R}^{\xcc}) \simeq
	\Map_{\nmproperad_\infty^\xV}(\mathcal{Q} \otimes \xcc, \mathcal{R}).\]
	Moreover, both of these functors preserve limits in each variable. \qed
\end{corollary}

Though we are using the same notation for algebras and cotensors as we did in Corollary~\ref{cor Alg(-,-)}, we do not know that these two notions coincide. 
Specifically, we do not know that the objects produced by Corollary~\ref{cor Alg(-,-)} are complete even when the inputs are.

\begin{remark}\label{remark sec ffes extensions}
As is usual, the theorems and definitions of this section can be carried out using other graph categories.
We write 
\[
	\name{DOpd}^\xV_\infty \simeq \Segrepc(\levelg_{\name{sc}}^\xV) \simeq \Segrepc(\bbY_{\name{sc}}^\xV) 
\]
for the \icat{} of $\xV$-enriched $\infty$-dioperads
and 
\[ \nmproperad^{{\name{out}},\xV}_\infty \simeq \Segrepc(\levelg_{\name{out,c}}^\xV) \simeq \Segrepc(\bbYout^\xV) 
\]
for the \icat{} of $\xV$-enriched $\infty$-output-properads.
These are defined as full subcategories of the appropriate $\Segrep(\gcV)$ on the complete objects (where completeness is created by $\Segrep(\gcV) \to \Seg(\simp)$), and then shown to be localizations at the fully-faithful and essentially surjective functors (as in Theorem~\ref{theo cp obj are loc}).
Similar statements also hold for $\name{Opd}^\xV_\infty$, and previously appeared in \cite{ChuHaugseng}.
\end{remark}

\begin{proposition}\label{proposition adjunctions}
There are adjunctions
\[ 
\begin{tikzcd}
\name{DOpd}^\xV_\infty 
	\arrow[r, shift right=2] \arrow[r, phantom, "\text{\footnotesize$\bot$}"] 
& 
\name{Opd}^\xV_\infty 
	\arrow[r, shift left=2] \arrow[r, phantom, "\text{\footnotesize$\bot$}"] 
	\arrow[l, shift right=2] 
& 
\nmproperad^{{\name{out}},\xV}_\infty 
	\arrow[l, shift left=2] 
	\arrow[r, shift left=2] \arrow[r, phantom, "\text{\footnotesize$\bot$}"] 
& 
\nmproperad_\infty^\xV 
	\arrow[l, shift left=2]
\end{tikzcd}
\]
restricted from those in Corollary~\ref{cor adjoint rep}\eqref{enum first adjunction}.
The left adjoints preserve tensors.
\end{proposition}
\begin{proof}
For concreteness, we only prove the statement about $\nmproperad^{{\name{out}},\xV}_\infty  \rightleftarrows \nmproperad_\infty^\xV$ and we do so in terms of $\bbY_{\name{out}}^\xV$ and $\bbY^\xV$-presheaves.
Write $\bar \imath \colon \bbY_{\name{out}}^\xV \to \bbY^\xV$.
The key fact that we use about this situation, which falls under the setting of Lemma~\ref{lemma: restriction of lke}, is that the unit $1 \to \bariustar \barinotshriek$ of the adjunction $\barinotshriek\colon\Segrep(\bbYV_{\name{out}}) \rightleftarrows \Segrep(\bbYV)\cocolon\bariustar$ of Corollary~\ref{cor adjoint rep}\eqref{enum first adjunction} is an equivalence. 
Indeed, this follows from the corresponding fact for the equivalent adjunction $\SegS(\bbY_{\name{out}}^\xU) \rightleftarrows \SegS(\bbYU)$ (for an appropriately chosen $\xU$ and $\mathbb{S}$), using that $\bbY_{\name{out}}^\xU \to \bbYU$ is fully-faithful.

We name two additional functors
\[ \begin{tikzcd}
\simp \rar{u'} \arrow[dr, "u"'] & \bbY_{\name{out}}^\xV \dar{\overline{\imath}} \\
& \bbYV. 
\end{tikzcd} \]
An object $F \in \Segrep(\bbYV)$ is complete if and only if $u^*F = (u')^*\bariustar F \in \Segrep(\simp)$ is complete if and only if $\bariustar F \in \Segrep(\bbY_{\name{out}}^\xV)$ is complete. 
Thus $\bariustar$ restricts to $\Segrepc(\bbYV) \to \Segrepc(\bbY_{\name{out}}^\xV)$.
Now suppose that $A\in \Segrep(\bbY_{\name{out}}^\xV)$ is complete.
As
\begin{equation}\label{eq using unit equiv}
	u^* \barinotshriek A \simeq (u')^*\bariustar \barinotshriek A \simeq (u')^* A,
\end{equation}
it follows that $\barinotshriek A \in \Segrep(\bbYV)$ is also complete.
Thus $\barinotshriek \dashv \bariustar$ restricts to an adjunction $\barinotshriek \colon \Segrepc(\bbY_{\name{out}}^\xV) \rightleftarrows \Segrepc(\bbYV) \cocolon \bariustar$.
Using the two tensorings (from Proposition~\ref{prop tensor} and Remark~\ref{remark sec ffes extensions}) we know that $\barinotshriek (A\otimes \xcc) \simeq (\barinotshriek A) \otimes \xcc$ by Corollary~\ref{cor adjoint rep}\eqref{second natural equiv}.
\end{proof}

\begin{remark}
The proof of the preceding proposition can be modified to show that if $i \colon \bbY_{\name{sc}} \to \bbY$ is the inclusion, then $\bariustar \colon \Segrep(\bbYV) \to \Segrep(\bbYV_{\name{sc}})$ restricts to $\nmproperad_\infty^\xV \to \name{DOpd}^\xV_\infty$.
By Theorem~\ref{theo cp obj are loc} this functor will have a left adjoint, though it is not clear whether or not it is restriction of the functor $\barinotshriek \colon \Segrep(\bbYV_{\name{sc}}) \to \Segrep(\bbYV)$ from Corollary~\ref{cor adjoint rep}.
Indeed, $\barinotshriek$ will typically significantly enlarge the underlying $\Seg(\simp)$ object, meaning that the equivalence corresponding to \eqref{eq using unit equiv} will not hold.
This is in the same spirit as Example~\ref{example: failure of lke restriction}, since $\barinotshriek$ requires a localization. 

That said, \cite[\S4]{HackneyRobertsonYau:SMIP} provides evidence that the adjunction \[ \barinotshriek \colon \Segrep(\bbYV_{\name{sc}}) \rightleftarrows \Segrep(\bbYV) \cocolon \bariustar\] restricts to $\name{DOpd}^\xV_\infty \rightleftarrows \nmproperad_\infty^\xV$.
Suppose that $F$ is the functor from simplicially-enriched dioperads to simplicially-enriched properads, left adjoint to the forgetful functor.
One consequence of Proposition 4.4 of \cite{HackneyRobertsonYau:SMIP} is that if $D$ is a nice-enough dioperad, then the underlying simplicially-enriched categories of $D$ and $FD$ will have the same classes of equivalences.
The main idea is that the new operations in $FD$ are obtained by decorating graphs having nonzero first betti number by operations in $D$.
Composing operations in $FD$ is additive on first betti numbers of said graphs, so in particular the newly added operations are never invertible.
It would be interesting to see if a similar argument can be carried out in the present situation, but such a detailed argument is beyond what we are trying to achieve here.
Thus we leave this as Conjecture~\ref{conjecture equivalences} below.
Notice that if $A$ and $\barinotshriek A$ have the same objects and same equivalences, then $A$ will be complete if and only if $\barinotshriek A$ is complete.
This tells us that the rest of the proof of Proposition~\ref{proposition adjunctions} carries through without change.
\end{remark}

\begin{conj}\label{conjecture equivalences}
If $A \in \Segrep(\bbYV_{\name{sc}})$, then the equivalences in $\barinotshriek A \in \Segrep(\bbYV)$ coincide with the equivalences in $A$.
Consequently, $\barinotshriek$ restricts to $\Segrepc(\bbYV_{\name{sc}}) \to \Segrepc(\bbYV)$ which preserves tensors.
\end{conj}

\section{Rectification theorems}\label{sec alg}
The aim of this section is to understand whether the homotopy theory of enriched $\infty$-properads is equivalent to a Dwyer--Kan-type homotopy theory for ordinary enriched properads.
In \S\ref{subsec approx} we show that $\bbY$ (resp.\ $\bbY_{\name{sc}}$, $\bbY_{\name{out}}$) and the operads governing properads (resp.\ dioperads, output properads) induce the same $\infty$-category of algebras.
Then, in \S\ref{subsec rect} we turn to the question of rectifying enriched $\infty$-properads.
This we can do only over very particular bases (see Theorem~\ref{theorem non-rect}), though for enriched $\infty$-dioperads and enriched $\infty$-output-properads rectification holds quite generally (Theorem~\ref{thm:rect}).

\subsection{Operads governing properads}\label{subsec approx}
In this subsection we first recall, for a set $S$, the operad whose algebras in a symmetric monoidal $\infty$-category are enriched $S$-colored properads.
The main result of this subsection is that the $\infty$-category $\bbY^\op_S$ is an ``approximation'' to this operad in the sense of \cite[\S 2.3.3]{ha}.
This observation immediately implies that an enriched $\infty$-properad in our sense is indeed equivalent to an enriched $\infty$-properad defined as an algebra over the operad for properads.

\begin{definition}
\label{ordered graph definition}
Let $Z$ be a finite set.
For our purposes, a $Z$-graph will consist of a connected, acyclic graph $G$ together with total orderings on each of the sets $\inp(G)$, $\out(G)$, $\inp(v)$ (for each $v\in \vertex(G)$), and $\out(v)$, as well as a bijection $Z \isoto \vertex(G)$.
Likewise, an \emph{$S$-colored $Z$-graph} will additionally come with a function $\edge(G) \to S$.
We say that two $Z$-graphs are \emph{strictly isomorphic} if there is a graph isomorphism preserving all of the structure.
\end{definition}
Note that there is at most one strict isomorphism between any two $Z$-graphs.

We now recall, from \cite[\S 14.1]{YauJohnson:FPAM}, a colored operad $\bfproperad_S$ which controls $S$-colored properads.
Before we define a colored operad $\bfproperad_S$, we first introduce the special case where $S$ is the terminal set.

\begin{definition}\label{def Prd}
Let $Z$ be a finite set.
Suppose that $\nsqelt{k}$ and $\nsqelt{k}_z$ (indexed over $z\in Z$) are elements of $\mathbb{N}\times\mathbb{N}$.
Define
\[
	\bfproperad(\{\nsqelt{k}_z\}_{z\in Z}; \nsqelt{k})
\]
to be the set of strict isomorphism classes of $Z$-graphs $G$ with $(|\inp(v_z)|, |\out(v_z)|) = \nsqelt{k}_z$ for each $z \in Z$ and $(|\inp(G)|, |\out(G)|) = \nsqelt{k}$.
This forms an $\mathbb{N}\times \mathbb{N}$-colored operad $\bfproperad$ with operadic composition given by graph substitution.
As all of the sets $\inp(v)$, $\inp(G)$ and so on are totally ordered, we use the unique order-preserving isomorphisms as our graph substitution data.
\begin{itemize}
	\item 
	For $Z$ a one-element set, the identity element in $\bfproperad(\{\nsqelt{k}\}; \nsqelt{k})$ is a corolla $C$ with vertex $v$ so that the two orderings on $\inp(C) = \inp(v)$ agree, and likewise for $\out(C) = \out(v)$.
	\item If $\sigma \colon  Z \to Z'$ is a bijection, there is an isomorphism
	\[
		\bfproperad(\{\nsqelt{k}_{z'}\}_{z'\in Z'}; \nsqelt{k}) \to \bfproperad(\{\nsqelt{k}_{\sigma(z)}\}_{z\in Z}; \nsqelt{k})
	\]
	given on a $Z'$-graph $G$ by precomposing the bijection $Z' \to \vertex(G)$ with $\sigma$.
\end{itemize}
\end{definition}

A special case of this definition is when $Z$ is the empty set.
As there is a unique graph $G$ in $\bbY$ which does not have any vertices, we have
\[
	\bfproperad(\{\,\, \}; \nsqelt{k}) = \begin{cases}
		\{ \xfe \} & \text{if $\nsqelt{k} = (1,1)$, and} \\
		\varnothing & \text{otherwise.}
	\end{cases}
\]

\begin{remark}
Note that the operad from \cite[\S 14.1]{YauJohnson:FPAM} is actually the `skeletal' version of this one, that is, is only indexed on the finite sets $Z = \mathbf{n}$.
As is customary, we will write the corresponding set of operations as 
\[
	\bfproperad(\nsqelt{k}_1, \dots, \nsqelt{k}_n; \nsqelt{k})
\]
since $\mathbf{n}$ has a natural total order.
\end{remark}

We now want to extend $\bfproperad$ to a colored operad $\bfproperad_S$.

\begin{definition}\label{def PrdS}
Let $S$ be a nonempty set.
The operations in $\bfproperad_S$ as are strict isomorphism classes of $S$-colored $Z$-graphs (where strict isomorphism is interpreted to mean strict isomorphism preserving the coloring functions).
The set of colors is $(\coprod_{n\geq 0} S^{\times n})^{\times 2}$, the set of pairs of ordered lists of elements of $S$.
Given an $S$-colored $Z$-graph, the vertex $v_z$ has an associated pair of lists of elements of $S$ using the two functions $\inp(v_z) \to S$ and $\out(v_z) \to S$, as does the whole graph using $\inp(G) \to S$ and $\out(G) \to S$.
This determines the profile where this operation lives.
Otherwise, the structure is very similar to that of $\bfproperad$.
\end{definition}

\begin{lemma}[Section 14.1 of \cite{YauJohnson:FPAM}]\label{YJ Section 14.1}
	For a set $S$, the $\bfproperad_S$-algebras in a symmetric monoidal category $\xcc$ are $\xcc$-enriched properads with $S$ as the set of colors.
\end{lemma}

Let us now look at some special cases of $\bfproperad_S$.
Notice that there is an operad map $\bfproperad_S \to \bfproperad$ which forgets the edge colorings on operations, and whose color map is induced from $S\to *$:
\[
	\left(\coprod_{n\geq 0} S^{\times n}\right) \times \left(\coprod_{n\geq 0} S^{\times n}\right) \to \left(\coprod_{n\geq 0} *\right) \times \left(\coprod_{n\geq 0} *\right) \cong \mathbb{N} \times \mathbb{N}.
\]

\begin{example}\label{examples colored operads}
\leavevmode
	\begin{enumerate}
		\item For $S=*$, $\bfproperad_*$ coincides with $\bfproperad$ of Definition~\ref{def Prd}.
		\item Let $\mathbf{Cat}_S$ denote the full suboperad of $\bfproperad_S$ with color set the pullback of $ \{1\} \times \{1 \} \hookrightarrow \mathbb{N} \times \mathbb{N}$.
		Algebras over $\mathbf{Cat}_S$ are categories with $S$ as the set of objects (see \cite[2.1]{enriched}). Likewise, letting $\mathbf{Opd}_S$ denote the full suboperad of $\bfproperad_S$ with color set the pullback of $\mathbb{N}\times \{1 \} \hookrightarrow \mathbb{N} \times \mathbb{N}$, we recover the operad whose algebras are $S$-colored operads (see \cite[Definition 5.1.5]{ChuHaugseng}).
		\item Let $\bfproperad^{\out}_S$ denote the full suboperad of $\bfproperad_S$ with color set the pullback of $\mathbb{N} \times \{1, 2, 3, \dots \} = \mathbb{N} \times \mathbb{N}_+ \hookrightarrow \mathbb{N} \times \mathbb{N}$.
		Algebras over $\bfproperad^{\out}_S$ are $S$-colored properads in which every operation has at least one output color. Likewise, there is a colored operad $\bfproperad^{\inp}_S$ with color set living over $\mathbb{N}_+ \times \mathbb{N}$.
		\item Let $\mathbf{DOpd}_S$ be the suboperad of $\bfproperad_S$ with the same color set, with the requirement that the underlying graph of any operation is simply-connected.
		Algebras over $\mathbf{DOpd}_S$ are $S$-colored dioperads (see \cite[\S 11.5]{YauJohnson:FPAM}).
	\end{enumerate}
\end{example}

\begin{remark}\label{remark not sigma free}
The operad $\bfproperad_S$ is not $\Sigma$-free.
Indeed, consider the left graph from Example~\ref{example of etale but not mono} with the output orderings at $u_i$ and the input orderings at $v_i$ are given from left to right, and with vertex ordering $u_0, u_1, v_0, v_1$.
With this convention, the graph represents an element of $\bfproperad((0,2),(0,2),(2,0),(2,0); (0,0))$.
This element is fixed by the group element $(12)(34) \in \Sigma_4$.
This issue is intrinsic, that is, any other operad governing properads will also have such a fixed point.
In particular, this means that our conception of properads is different from that of \cite[10.4]{batanin-berger} which arises as algebras over a finitary polynomial monad in $\Set$, as a polynomial monad always describes a $\Sigma$-free colored operad (see Section 6 of \cite{batanin-berger}).

The other colored operads from Example~\ref{examples colored operads} are $\Sigma$-free.
Indeed, for each of the other types of graphs, the only strict automorphisms are identities.
For $\mathbf{DOpd}_S$, this fact is \cite[Proposition 4.14]{YauJohnson:FPAM}, while for the others ($\bfproperad^{\out}_S$, $\mathbf{Opd}_S$, and so on) this is \cite[Lemma 4.8]{YauJohnson:FPAM}.
\end{remark}

For the reader's convenience we now recall the definition of the $\infty$-operad associated to a symmetric operad introduced in \cite[Construction 2.1.1.7]{ha}.
\begin{definition}
	For a symmetric operad $\mathbf{O}$, we define its associated $\infty$-operad $\xxO\to \pfinsetskel$ to be the functor determined by the following:
	\begin{enumerate}
		\item The objects in $\xxO$ are finite sequences $(x_1,\ldots, x_m)$ of colors in $\mathbf{O}$.
		\item For two objects $(x_1,\ldots, x_m), (y_1,\ldots, y_n)$ in $\xxO$, we define the set of morphisms
		$$\xxO((x_1,\ldots, x_m), (y_1,\ldots, y_n)) \coloneqq \coprod_{\alpha\in \xHom(\langle m\rangle, \langle n\rangle)}\prod_{1\leq j\leq n}\mathbf{O}(\{x_i\}_{\alpha(i)=j}, y_j).$$
		\item The composition in $\xxO$ is induced by that of $\pfinsetskel$ and $\mathbf{O}$.
		\item The map $\xxO\to \pfinsetskel$ is the obvious projection map.
	\end{enumerate}
\end{definition}
\begin{notation}
		We let $\calproperad_S \to \pfinsetskel$ denote the $\infty$-operad associated to the simplicial operad $\bfproperad_S$.
\end{notation}

If $S$ is a set and $\xV$ is a symmetric monoidal $\infty$-category, then we can consider the \icat{} $\xAlg_{\calproperad_S}(\xV)$ of algebras over $\calproperad_S$.
We should consider objects of this \icat{} as $S$-colored, $\xV$-enriched $\infty$-properads.
Our next main goal is to show this is reasonable, by proving in Corollary~\ref{cor:OOpSeq} that this \icat{} is equivalent to the \icat{} $\xAlg_{\bbY_S^\op}(\xV)$ from Definition~\ref{def alg}.
The stated equivalence utilizes a relationship between $\calproperad_S$ and $\bbY_S^\op$.
Our initial task is to define a family of functors $\Theta_S \colon \bbY^\op_S \to \calproperad_S$.
Before doing so, we explain in the next remark why we do not use the category $\kockgraphs$, and then we replace $\bbY$ with a convenient, equivalent category.

\begin{remark}\label{remark no vertex map Gr}
In Proposition~\ref{prop active inert bbY}, we studied a functor $\vertex_\bbY\colon \bbY\to \pfinset^\op$ which took a graph to its set of vertices.
There is no corresponding functor from $\kockgraphs$ (Definition~\ref{definition kockgraphs}), or even from $\kockint$.
Indeed, consider the graphs from Example~\ref{example of etale but not mono}. 
In $\kockint$ there is an \'etale map from $G$ to $K$ which takes each $u_i$ to $u$ and each $v_i$ to $v$, so it is unclear how one should construct a meaningful base point preserving function
\[
	\vertex(K)_+ = \{u,v, *\} \rightarrow \{u_0, u_1, v_0, v_1, *\} = \vertex(G)_+
\]
in the same manner.
Certainly the rule \eqref{definition of partial map} from Definition~\ref{definition bbY to Fin} is not single-valued.
As there is no meaningful vertex functor $\kockgraphs \to \pfinset^\op$, there will not be a meaningful functor $\kockgraphs^\op \to \calproperad$.
\end{remark}

A variant of the following notion appeared in Appendix~\ref{appendix comparison with HRY}.
\begin{notation}[Ordered variant of $\bbY$]
For the remainder of this section, we will work with the category whose objects are graphs together with orderings on the sets $\inp(v)$, $\out(v)$, and also on $\vertex(G)$ (but not on $\inp(G)$ or $\out(G)$).
Morphisms are just morphisms in $\bbY$, that is, they ignore this extra structure.
This category is equivalent to the usual $\bbY$. 
To avoid clutter, we will simply write $\bbY$ for this category until the end of \S\ref{subsec approx}.
\end{notation}

\begin{definition}
\label{definition of Theta}
Define a functor $\Theta \colon \bbY^\op \to \calproperad$ as follows.
Each graph in $\bbY$ comes equipped with a chosen ordering of the vertices.
On objects, send a graph $G$ to $(\nsqelt{k}_1, \dots, \nsqelt{k}_n)$ where $\nsqelt{k}_a = (|\inp(v_a)|, |\out(v_a)|)$.
Now suppose that $f \colon H \to G$ is a morphism in $\bbY$.
We already know there is a functor $\vertex_\bbY\colon \bbY^{\op} \to \pfinset\simeq \pfinsetskel$ from just above Proposition~\ref{prop active inert bbY}, which is one part of the morphism $\Theta(f)$. 
That is, we have a morphism $q \colon  \langle n \rangle \to \langle m \rangle$, so that $q(a) = b$ means that $v_a \in \vertex(G)$ is a vertex of the structured subgraph $f(C_{w_b}) \in \sub(G)$.
On the other hand, for $b = 1, \dots, m$, we write $G_b$ for the structured subgraph $f(C_{w_b}) \in \sub(G)$ together with the following data:
\begin{itemize}
	\item the induced bijection
	\[ \begin{tikzcd}[sep=small]
	q^{-1}(b) \dar["\subseteq"] \rar["\cong"] & \vertex(f(C_{w_b})) \dar["\subseteq"] \\
	\langle n \rangle \rar["\cong"] & \vertex(G),
	\end{tikzcd} \]
	\item for each $v\in \vertex(f(C_{w_b}))$, an ordering on $\inp(v)$ and $\out(v)$ from the corresponding orderings in $G$, and
	\item an ordering on $\inp(f(C_{w_b}))$ and $\out(f(C_{w_b}))$ from the corresponding orderings on $\inp(w_b)$ and $\out(w_b)$.
\end{itemize}
In this way, $G_b$, for $b=1,\dots, m$ is considered as an element of 
\[
	\bfproperad( \{\nsqelt{k}_a\}_{a \in q^{-1}(b)}; \nsqelt{\jmath}_b).
\]
Declare the value of $\Theta(f)$ to be
\[
	(G_1, \dots, G_m) \in \prod_{b=1}^m \bfproperad(\{\nsqelt{k}_a\}_{q^{-1}(b)}; \nsqelt{\jmath}_b) \subseteq_{q} \calproperad(\nsqelt{k}_1, \dots, \nsqelt{k}_n; \nsqelt{\jmath}_1, \dots, \nsqelt{\jmath}_m).
\]
\end{definition}

\begin{lemma}
$\Theta$ is a functor.
\end{lemma}
\begin{proof}
For composition, suppose we have
\[ \begin{tikzcd}[row sep=tiny]
K \rar{g} & H \rar{f} & G & \text{in $\bbY$} \\
\langle \ell \rangle & \langle m \rangle \lar["p" swap] & \langle n \rangle \lar["q" swap] & \text{in $\pfinsetskel$}
\end{tikzcd} \]
with 
\[
	(H_1, \dots, H_\ell) \in \prod_{c=1}^\ell \bfproperad(\{ \nsqelt{\jmath}_b \}_{p^{-1}(c)}; \nsqelt{\imath}_c) \subseteq_{p} \calproperad (\nsqelt{\jmath}_1, \dots, \nsqelt{\jmath}_m; \nsqelt{\imath}_1, \dots, \nsqelt{\imath}_\ell)
\]
equaling $\Theta(g)$, 
\[
	(J_1, \dots, J_\ell) \in \prod_{c=1}^\ell \bfproperad(\{\nsqelt{k}_a\}_{(pq)^{-1}(c)}; \nsqelt{\imath}_c) \subseteq_{pq} \calproperad (\nsqelt{k}_1, \dots, \nsqelt{k}_n; \nsqelt{\imath}_1, \dots, \nsqelt{\imath}_\ell)
\]
equaling $\Theta(fg)$, and $\Theta(f) = (G_1, \dots, G_m)$ as given in Definition~\ref{definition of Theta}.
Note that $J_c$ is the graph $f(g(C_{u_c})) \in \sub(G)$ together with the bijection $(pq)^{-1}(c) \to \vertex(f(g(C_{u_c})))$ and the above indicated orderings on inputs and outputs. 

On the other hand, the composition $\Theta(g)\Theta(f)$ is also in the $pq$ component of $\calproperad (\nsqelt{k}_1, \dots, \nsqelt{k}_n; \nsqelt{\imath}_1, \dots, \nsqelt{\imath}_\ell)$.
Its $c$th projection is given by applying the operadic composition
\[\begin{tikzcd}
\bfproperad(\{ \nsqelt{\jmath}_b \}_{p^{-1}(c)}; \nsqelt{\imath}_c) \times \prod\limits_{b\in p^{-1}(c)} \bfproperad(\{\nsqelt{k}_a\}_{q^{-1}(b)}; \nsqelt{\jmath}_b) \dar \\ 
\bfproperad(\{\nsqelt{k}_a\}_{(pq)^{-1}(c)}; \nsqelt{\imath}_c)
\end{tikzcd}
\]
to $H_c, \{ G_b \}_{b\in p^{-1}(c)}$.
This graph substitution $H_c \{ G_b \}_{b\in p^{-1}(c)}$ is isomorphic to $J_c$ since
\[ \begin{tikzcd}
\coprod\limits_{p^{-1}(c)} q^{-1}(b) \rar{\cong} \dar{\cong} & (pq)^{-1}(c) \dar{\cong} \\
\coprod\limits_{\vertex(g(C_{u_c}))} \vertex(f(C_{w_b})) \rar{\cong} &  \vertex(f(g(C_{u_c})))
\end{tikzcd} \]
commutes, and the input / output orderings are induced from the same places.
Thus $\Theta(fg) = \Theta(g)\Theta(f)$.
\end{proof}

\begin{definition}\label{def Y_S}
Given a set $S$, we let $\bbY_S$ be the category where an object consists of a graph $G$ in $\bbY$ together with a function $\edge(G) \to S$.
Morphisms should respect the coloring function.
In other words, if $\edge \colon \bbY \to \Set$ is the functor which takes $G$ to its set of edges $\edge(G)$, then $\bbY_S $ is the comma category $\edge \downarrow S$. Analogously we define the categories $\bbY_{\name{out},S}$ and $\bbY_{\name{sc}, S}$.
\end{definition}

The category $\bbY_S$ is given by applying the construction of Notation~\ref{no Y_X} to the special case where $\gc = \bbY$ and $S$ is a discrete space.
We will not need such generality here, and our $\bbY^\op_S$ (for $S$ a set) will be an ordinary category, rather than an \icat. Moreover, notice that $\bbY_S\cong \bbY$ if $S=*$.

\begin{notation}\label{no Theta_S}
	For a set $S$, the functor $\Theta \colon \bbY^\op \to \calproperad$ from Definition~\ref{definition of Theta} naturally extends to a functor $\bbY^\op_S \to \calproperad_S$ which we denote by $\Theta_S$.
\end{notation}

We want to prove that the functor $\Theta_S$ is an
\emph{approximation} in the sense of \cite[Definition 2.3.3.6]{ha}:
\begin{definition}\label{def approximation}
Given an $\infty$-operad $p\colon \xxO \to \pfinsetskel$ and an $\infty$-category $\xcc$.
	We say a functor $f\colon \xcc\to \xxO$ is an \emph{approximation to $\xxO$}, if it satisfies the following conditions:
	\begin{enumerate}
		\item Suppose $p'=p\circ f$, $c\in \xcc$ is an object and $p'(c)=\langle n\rangle$.
		For every $1\leq i\leq n$, the inert map $\rho^i\colon\langle n\rangle\to \langle 1\rangle$ has a locally $p'$-cocartesian lift $\tilde \rho^i\colon c\to c_i $ in $\xcc$ such that $f(\tilde \rho^i)$ in $\xxO$ is inert.
		\item Every active morphism $\alpha\colon x\to f(c)$ in $\xxO$ has an $f$-cartesian lift $\tilde \alpha\colon \tilde x\to c$ in $\xcc$.
	\end{enumerate}
\end{definition}

\begin{propn}\label{prop approx}
	The functor $\Theta_{S} \colon \bbY_S^\op \to \calproperad_{S}$ is an approximation to $\calproperad_{S}$.
\end{propn}
\begin{proof}
For simplicity, we restrict the proof to the case where $S$ is a point.
The general case is similar.
\begin{enumerate}
\item 
Let $p\colon \calproperad\to \pfinsetskel$ be the structure map and let $p'= p\circ \Theta$.
Then it follows from the definition of $p$ and $\Theta$ that for each $G$ with $p'(G) = \langle n \rangle$,  every inert map $\rho^i \colon \langle n \rangle \to \langle 1 \rangle$ has a unique inert lift $g \colon G \to C_i$ in $\bbY^\op$.
Notice that $\Theta(g)$, which is 
represented by 
\[
	\id_{\nsqelt{k}_i} \in \bfproperad(\nsqelt{k}_i; \nsqelt{k}_i) \subseteq_{\rho^i} \calproperad(\nsqelt{k}_1, \dots, \nsqelt{k}_n; \nsqelt{k}_i) 
\]
 is $p$-cocartesian, hence it is inert in $\calproperad$. 
It remains to show that $g$ is locally $p'$-cocartesian.
That is, given the pullback
\[ \begin{tikzcd}
\bbY^{\op}_i \rar \dar{q} & \bbY^{\op} \dar{p'} \\
\Delta^1 \rar{\rho^i} & \pfinsetskel
\end{tikzcd} \]
we must show that the morphism $(g,0\to 1)$ in $\bbY^{\op}_i$ is $q$-cocartesian.
	Hence, for every corolla $C'\in \bbY^\op$, we have to show that the commutative diagram	
	\nolabelcsquare{\xMap_{\bbY^\op_i}( (C,1),  (C',1))}{\xMap_{\bbY^\op_i}( (G,0),  (C',1))}{\xMap_{\Delta^1}(1,1)}{\xMap_{\Delta^1}(0,1)}
	is a pullback square.
	This is automatic as the horizontal maps are equivalences.
\item 
Let $G$ be an object of $\bbY$ and let 
\begin{equation*}\label{eq map we start with}
	(\nsqelt{\jmath}_1, \dots, \nsqelt{\jmath}_m) \to \Theta(G) = (\nsqelt{k}_1, \dots, \nsqelt{k}_n) \tag{$\heartsuit$}
\end{equation*}
be an active morphism of $\calproperad$ lying over $\alpha \colon \langle m \rangle \to \langle n \rangle$, which is exhibited by
\[
	(H_a)_a \in \prod_{a=1}^n \bfproperad(\{ \nsqelt{\jmath}_b\}_{b \in \alpha^{-1}(a)}; \nsqelt{k}_a).
\]
Suppose that $G'$ is $G$ with some ordering of $\inp(G)$ and $\out(G)$ and that $\nsqelt{k}$ is $(|\inp(G)|, |\out(G)|)$.
Let $H'$ be the image of $(G', (H_a)_a)$ under the operadic composition
\[
	\bfproperad(\nsqelt{k}_1, \dots, \nsqelt{k}_n; \nsqelt{k}) \times \prod_{a=1}^n \bfproperad(\{ \nsqelt{\jmath}_b\}_{b \in \alpha^{-1}(a)}; \nsqelt{k}_a) \to \bfproperad(\nsqelt{\jmath}_1, \dots, \nsqelt{\jmath}_m; \nsqelt{k}) 
\]
and let $H \in \bbY$ be the graph obtained by forgetting the orderings on inputs and outputs of $H'$.
We then have an active map $f \colon H \to G$ of $\bbY^{\op}$ which sends $H_a \in \sub(H)$ to $C_a \in \sub(G)$ (using \cite[Theorem 6.50]{hrybook} since $G\{H_a\} = H$), and the image of this map under $\Theta$ is \eqref{eq map we start with}.
The map $f$ is our proposed $\Theta$-cartesian lift of \eqref{eq map we start with}.

Suppose that we are in the situation of having $g \colon K \to G$ in $\bbY^{\op}$ and $t \colon \Theta(K) \to \Theta(H)$ in $\calproperad$ satisfying $\Theta(f) t = \Theta(g)$.
\begin{equation*}\label{eq original cartesian lifting problem} \begin{tikzcd}
K \arrow[drr, bend left=15, "g" near end] \arrow[dd, no head, dashed] \arrow[dr, dotted, "\exists ! \tilde t" swap] \\
& H \arrow[dd, no head, dashed] \rar["f"]& G \arrow[dd, no head, dashed] \\[-2em]
\Theta(K) \arrow[drr, bend left=15, "\Theta(g)" near end]
\arrow[dr, "t" swap]
\\
& \Theta(H) \rar["\Theta(f)"]& \Theta(G)
\end{tikzcd} \tag{$\diamondsuit$}
\end{equation*}
Our goal is to show there exists a unique $\tilde t \colon K \to H$ so that $\Theta(\tilde t) = t$ and $f\tilde t = g$.

We first reduce to the active case.
Observe that if $g = \bar{g} \mathring{g} \colon K \to L \to G$ is a decomposition of $g$ with $\mathring{g}$ inert and $\bar{g}$ active, then $\Theta(\mathring{g}) =: \mathring{t}$ is part of a similar decomposition $t = \bar{t} \mathring{t}$ since $\Theta(f)$ is active.
\[ \begin{tikzcd}
K \arrow[dd, no head, dashed] \rar{\mathring{g}} & L \arrow[drr, bend left=15, "\bar{g}" near end] \arrow[dd, no head, dashed] \\
& & H \arrow[dd, no head, dashed] \rar["f"]& G \arrow[dd, no head, dashed] \\[-2em]
\Theta(K) \rar{\mathring{t}} & \Theta(L) \arrow[drr, bend left=15, "\Theta(\bar{g})" near end]
\arrow[dr, "\bar{t}" swap]
\\
& & \Theta(H) \rar["\Theta(f)"]& \Theta(G)
\end{tikzcd} \]
If we knew that there is a unique $s \colon L \to H$ so that $f s = \bar{g}$ and $\Theta(s) = \bar{t}$, then $s \mathring{g}$ gives existence of a solution to \eqref{eq original cartesian lifting problem}.
Further, $s$ must be active since $f$ and $\bar{g}$ are active.
Suppose that $q$ is some other solution satisfying $\Theta(q) = t$ and $f q = g$, and write $q = \bar{q} \mathring{q}$ for an inert-active factorization $K \to J \to H$.
There are unique isomorphisms $z \colon L \to J$, $r \colon \Theta(L) \to \Theta(J)$, and $w \colon  \Theta(L) \to \Theta(J)$ making the following diagrams commute:
\[ 
\begin{tikzcd}[column sep=small, row sep=scriptsize]
& J \arrow[dr, "f\bar{q}"] \\
K \arrow[ur,"\mathring q"] \arrow[dr, "\mathring g" swap] & & G \\
& L \arrow[uu, "z"] \arrow[ur, "\bar{g}" swap] 
\end{tikzcd} 
\begin{tikzcd}[column sep=small, row sep=scriptsize]
& \Theta(J) \arrow[dr, "\Theta(\bar{q})"] \\
\Theta(K) \arrow[ur,"\Theta(\mathring q)"] \arrow[dr, "\mathring t" swap] & & \Theta(H) \\
& \Theta(L) \arrow[uu, "r"] \arrow[ur, "\bar{t}" swap] 
\end{tikzcd} 
\begin{tikzcd}[column sep=small, row sep=scriptsize]
& \Theta(J) \arrow[dr, "\Theta(f\bar{q})"] \\
\Theta(K) \arrow[ur,"\Theta(\mathring q)"] \arrow[dr, "\Theta(\mathring g)" swap] & & \Theta(G) \\
& \Theta(L) \arrow[uu, "w"] \arrow[ur, "\Theta(\bar{g}) = \Theta(f) \bar{t}" swap] 
\end{tikzcd} 
\]
As replacing $w$ in the third diagram by either $\Theta(z)$ or $r$ makes it commute, we have that $r = \Theta(z)$.
Since $f \bar{q} z = \bar{g} = f s$ and $\Theta(\bar{q} z) = \Theta(\bar{q}) r = \bar{t}$, we see that $\bar{q} z = s$.
It follows that $s \mathring g = \bar{q} z \mathring{g} = \bar{q} \mathring{q} = q$.
Hence we have showed that $s \mathring{g}$ is the \emph{unique} solution to \eqref{eq original cartesian lifting problem}.

It remains to show that \eqref{eq original cartesian lifting problem} has a unique solution in the case when $g$ is active.
Let the morphism $t \colon \Theta(K) = (\nsqelt{\imath}_1, \dots, \nsqelt{\imath}_\ell) \to \Theta(H)$ lying over $\beta \colon  \langle \ell \rangle \to \langle m \rangle$
be exhibited by
\[
	(K_b) \in \prod_{b=1}^m \bfproperad(\{ \nsqelt{\imath}_c\}_{c \in \beta^{-1}(b)}; \nsqelt{\jmath}_b).
\]
Commutativity of the bottom triangle of \eqref{eq original cartesian lifting problem} is simply the assertion if $L_a$ is the image of $H_a, \{K_b\}_{b\in \alpha^{-1}(a)}$ under operadic composition
\[
	\bfproperad(\{ \nsqelt{\jmath}_b\}_{b \in \alpha^{-1}(a)}; \nsqelt{k}_a) \times \prod_{b\in \alpha^{-1}(a)} \bfproperad(\{ \nsqelt{\imath}_c\}_{c \in \beta^{-1}(b)}; \nsqelt{\jmath}_b) \to \bfproperad( \{ \nsqelt{\imath}_c\}_{c \in \beta^{-1}\alpha^{-1}(a)}; \nsqelt{k}_a),
\]
then 
\[
	(L_a) \in \prod_{a=1}^n \bfproperad( \{ \nsqelt{\imath}_c\}_{c \in \beta^{-1}\alpha^{-1}(a)}; \nsqelt{k}_a).
\]
exhibits $\Theta(g) \colon \Theta(K) \to \Theta(G)$.
But as graphs, 
\begin{equation}\label{showing g is tilde t f}
g(C_a) \cong L_a \cong H_a\{K_b\}_{b\in \alpha^{-1}(a)},
\end{equation} so since $g(C_a)$ is a structured subgraph of $K$, so is $K_b$.
Define $\tilde t \colon K \to H$ in $\bbY^\op$ by setting $\tilde t (C_b)  = K_b$.
This is a graphical map by \cite[Theorem 6.50]{hrybook} since its image under $H\{K_b\} \cong G\{H_a\} \{K_b\} \cong G \{ H_a \{ K_b\}_{b\in \alpha^{-1}(a)} \} \cong G\{ L_a \} \cong K$ is a structured subgraph of $K$.
By construction, $\Theta(\tilde t) = t$, and $f \tilde t = g$ by \eqref{showing g is tilde t f}.
Finally, $\Theta$ is a faithful functor, so $\tilde t$ is unique. \qedhere
\end{enumerate}
\end{proof}
\begin{corollary}\label{cor on other subcategories and theta}
	Let $\calproperad^{\out}_S$ be the $\infty$-operad associated to the operad $\bfproperad^{\out}_S$ defined in Example~\ref{examples colored operads}. Then $\Theta_S\colon \bbY^\op_S\to  \calproperad_{S}$ restricts to an approximation $\bbY^{\op}_{\out,S}\to \calproperad^{\out}_S$.
\end{corollary}
\begin{proof}
	According to \cite[Remark 2.3.3.9]{ha} the pullback of the approximation $\Theta_S\colon \bbY^\op_S\to  \calproperad_{S}$ along the morphisms $\calproperad^{\out}_S\to \calproperad_S$ induced by the canonical inclusions $\bfproperad^{\out}_S\to \bfproperad_S$ is again an approximation and the construction of $\Theta_S$ implies that the pullback $\bbY^\op_S\times_{\calproperad_{S}}\calproperad^{\out}_S$ coincides with $\bbY^{\op}_{\out,S}$ introduced in Definition~\ref{def Y_S}.
\end{proof}
	Let $\mathcal{DO}pd_S$ be the $\infty$-operad associated to the operad $\mathbf{DOpd}_S$ defined in Example~\ref{examples colored operads}. Contrary to the previous corollary $\bbY^{\op}_{\name{sc}, S}$ is not given by the pullback of $\Theta_S$ along the canonical inclusion $\mathcal{DO}pd_S\to \calproperad_{S}$. Nevertheless, a small adaptation of the proof of  Proposition~\ref{prop approx} yields the following result. 
\begin{lemma}\label{lem DOpd app}
The functor $\Theta_S\colon \bbY^\op_S\to  \calproperad_{S}$ restricts to an approximation $\bbY^{\op}_{\name{sc}, S}\to \mathcal{DO}pd_S$. 
\end{lemma}
\begin{proof}
	We assume that $S=*$ as the general case can be proven analogously.
	After replacing $\mathbf{Prpd}$ with $\mathbf{DOpd}$ in the proof of Proposition~\ref{prop approx} we see that the first part of the proof is still valid since the graph corresponding to $\id_{\nsqelt{k}_i} $ is a corolla and in particular simply-connected. The second part of the proof is also not affected by the change because $\bbY^{\op}_{\name{sc}}$ is closed under graph substitutions ($\bbY^{\op}_{\name{sc}}$ is a full subcategory of $\bbY^\op$).
\end{proof}

The following corollary is an easy application of Proposition~\ref{prop approx}.
\begin{cor}\label{cor:OOpSeq}
For every symmetric monoidal \icat{} $\xV$, the precomposition with $\Theta_{S}$ induces an equivalence
\[\Theta_{S}^{*} \colon \Alg_{\calproperad_{S}}(\xV) \isoto	\Alg_{\bbY^\op_{S}}(\xV).\] 
\end{cor}
\begin{proof}
	Since $\Theta_{S}$ obviously restricts to an equivalence $\bbY^\op_{S}\times_{\mathbb{F}_*}\{\langle 1\rangle\}\isoto \calproperad_{{S},\langle 1\rangle}$ of fibres over $\langle 1\rangle$, by
	\cite[Theorem 2.3.3.23]{ha} the functor $\Theta_{S}^{*}$ is an equivalence.
\end{proof}

A similar statement holds in the $\bbY_{\name{out},S}^\op$ / $\calproperad^{\out}_S$ and $\bbY_{\name{sc}, S}^\op$ / $\mathcal{DO}pd_S$ contexts as well, by
using Corollary~\ref{cor on other subcategories and theta} and Lemma~\ref{lem DOpd app} instead of Proposition~\ref{prop approx}.
Indeed, the four remaining items from this subsection have analogues in both of these contexts.

\begin{defn}
We write $\AlgPrdSet(\xV) \to \Set$ for the cartesian fibration corresponding to the functor $\Set^{\op} \to \CatI$ taking $S$ to $\Alg_{\calproperad_{S}}(\xV)$, and we let $\Alg_{\bbY^\op/\Set}(\xV) \to \Set$ denote the pullback of the cartesian fibration $\Alg_{\bbY^\op/\xS}(\xV)	\to \xS$ along the inclusion $\Set \hookrightarrow \xS$.
Since the functors $\Theta_{S}$ are natural in $S$, they induce a functor for which we write
\[\Theta^{*} \colon
\AlgPrdSet(\xV) \to
\Alg_{\bbY^\op/\Set}(\xV)\]
of cartesian fibrations over $\Set$.
\end{defn}

\begin{cor}\label{cor:AlgOpdeq}
	The functor \[\Theta^{*} \colon
	\AlgPrdSet(\xV) \to
	\Alg_{\bbY^\op/\Set}(\xV)\] is an equivalence.
\end{cor}
\begin{proof}
	The functor $\Theta^*$ is an equivalence because it is one at each fibre by Corollary~\ref{cor:OOpSeq}. 
\end{proof}

\begin{propn}\label{propn:FFESSet}
	For every presentably symmetric monoidal \icat{} $\xV$,	the inclusion $\Alg_{\bbY^\op/\Set}(\xV) \hookrightarrow
	\Alg_{\bbY^\op/\xS}(\xV)$ induces an equivalence
	\[ \Alg_{\bbY^\op/\Set}(\xV)[\name{FFES}^{-1}] \isoto
	\Alg_{\bbY^\op/\xS}(\xV)[\name{FFES}^{-1}]\]
	after localizing with respect to the class $\name{FFES}$ of fully faithful and essential surjective functors.
\end{propn}
\begin{proof}
It can be proven in as in \cite[Theorem 5.3.17]{enriched}.
\end{proof}

\begin{cor}\label{cor AlgPrd=PCSbbY}
	For every presentably symmetric monoidal \icat{} $\xV$, there is equivalence of $\infty$-categories
	\[ \AlgPrdSet(\xV)[\name{FFES}^{-1}]\simeq \Segrep(\bbYV)[\name{FFES}^{-1}]\simeq \nmproperad_\infty^\xV.\]
\end{cor}
\begin{proof}
The first equivalence is induced by that of Corollary~\ref{cor:AlgOpdeq}, Proposition~\ref{propn:FFESSet} and Theorem~\ref{thm:PCSisAlgLT}, while the second is given by Theorem~\ref{theo cp obj are loc} and Notation~\ref{no PrpdV}.
\end{proof}

\subsection{Rectification}\label{subsec rect}
In this subsection we compare our $\infty$-categorical definition of $\xV$-enriched $\infty$-properads with the strict notion of properad enriched in a symmetric monoidal model category $\mathbf{V}$.
One of our main findings, Theorem~\ref{theorem non-rect}, is that it is not always possible to rectify an enriched $\infty$-properad to a strict one.
We show in Theorem~\ref{theorem rational rectification for properads} that it is possible to perform such rectification when working over a field of characteristic zero.
The situation is very different for dioperads and for output properads, where we prove a rectification result over an arbitrary base in Theorem~\ref{thm:rect}.

\begin{defn}\label{def admissible}
	Let $\mathbf{V}$ be a simplicial symmetric monoidal model category. 
	We call an operad $\mathbf{O}$ \emph{admissible in $\mathbf{V}$} (alternatively, $\mathbf{V}$ is admissible for $\mathbf{O}$) if there is a model structure on $\Alg_{\mathbf{O}}(\mathbf{V})$ such that the weak equivalences and fibrations are those maps whose underlying maps in $\mathbf{V}$ are weak equivalences and fibrations, respectively.
\end{defn}
\begin{examples}\label{exs:admissible}
	According to \cite[\S
	7]{PavlovScholbachSymm} following model categories are admissible for \emph{all} operads:
	\begin{enumerate}[(i)]
		\item the category of simplicial sets, equipped with the Kan--Quillen model structure,
		\item the category of compactly generated weak Hausdorff spaces, equipped with the usual model structure,
		\item the category of chain complexes of $k$-vector spaces, where $k$ is a field of characteristic $0$ (or more generally a commutative ring containing $\mathbb{Q}$), equipped with the projective model structure, and
		\item the category of symmetric spectra, equipped with the positive stable model structure.
	\end{enumerate}
	Moreover, by \cite{NikolausSagave} we know that for any presentably symmetric monoidal \icat{} $\xV$ there exists a symmetric monoidal simplicial combinatorial model category modeling $\xV$ for which all (simplicial) operads are admissible.
\end{examples}

We are of course especially interested in the operad $\bfproperad$.
The relevant model structure for properads in chain complexes over a field of characteristic 0 was first constructed in the appendix of \cite{MerkulovVallette:DTRP}.

\begin{definition}\label{def canmor}
	If $\mathbf{V}$ is a simplicial symmetric monoidal model category and an operad $\mathbf{O}$ which is admissible for $\mathbf{V}$, and let $\xV$, $\xxO$ denote the associated $\infty$-category and $\infty$-operad, respectively. We refer to the map \[\Alg_{\mathbf{O}}(\mathbf{V})\to \Alg_{\xxO}(\xV) \]
	as the \emph{canonical map}.  
	Here, we regard $\Alg_{\mathbf{O}}(\mathbf{V})$ as the $\infty$-category associated to the model category $\Alg_{\mathbf{O}}(\mathbf{V})$, i.e.\ we implicitly identify it with the nerve of the localization of the full subcategory of $\Alg_{\mathbf{O}}(\mathbf{V})$ spanned by cofibrant objects with respect to weak equivalences.
\end{definition}

\begin{theorem}\label{theo:AlgOpScomp}
If $\mathbf{O}$ is either of $\mathbf{Prpd}^{\name{out}}_S$ or $\mathbf{DOpd}_S$ and is admissible in $\mathbf{V}$, then the canonical map
\[\Alg_{\mathbf{O}}(\mathbf{V})\to \Alg_{\xxO}(\xV) \]
is an equivalence.
\end{theorem}
\begin{proof}
	The proof of the theorem is based on the verification of the assumptions of \cite[Theorem 7.11]{PavlovScholbach} which gives a necessary and sufficient condition for under which the canonical map is an equivalence. 
By Remark~\ref{remark not sigma free}, the operads $\mathbf{Prpd}^{\name{out}}_S$ and $\mathbf{DOpd}_S$ are $\Sigma$-free which implies that the symmetric flatness condition of \cite[Theorem 7.11]{PavlovScholbach} is satisfied. 
This gives the equivalence $\Alg_{\mathbf{O}}(\mathbf{V})\isoto \Alg_{\xxO}(\xV)$.
\end{proof}

Theorem~\ref{theo:AlgOpScomp} is a key component in the proof of our Rectification Theorem~\ref{thm:rect} below.
Unfortunately it is a statement about dioperads and output (or input) properads, rather than about general properads.
This is simply because it does not hold in generality:

\begin{theorem}\label{theorem non-rect}
	Let $\xsSet$ denote the model category of simplicial sets equipped with the Kan--Quillen model structure, and suppose that $S$ is a nonempty set. 
	Then the canonical map of Definition~\ref{def canmor}
	\[\Alg_{\mathbf{Prpd}_S}(\xsSet)\to \Alg_{\calproperad_S}(\xS) \] is \emph{not} an equivalence.
\end{theorem}
\begin{proof}
Let $\mathbf{P} = \mathbf{Prpd}_S$ and let $\phi \colon \mathbf{O} \to \mathbf{P}$ be a $\Sigma$-cofibrant replacement of $\mathbf{P}$ in the category of $T$-colored simplicial operads (by taking, for instance, a suitable product of $\mathbf{P}$ with the Barratt--Eccles operad), where $T$ is the color set of $\mathbf{P}$ (see Definition~\ref{def PrdS}).
We let $x \in \mathbf{P}(a,a,b,b;c)$ be the left graph from Example~\ref{example of etale but not mono}, where all edges are colored by some fixed $s\in S$ (so $a = (\,\,; s, s)$, $b=(s,s;\,\,)$, and $c=(\,\,;\,\,)$ are elements of $T$).
Let $V= \{ e, (12), (34), (12)(34) \} \subseteq \Sigma_4$ be the group of permutations fixing the profile $a,a,b,b$ and let $\Sigma_2 = \{e, (12)(34)\} \subseteq V$.
As mentioned in Remark~\ref{remark not sigma free}, the stabilizer group of $x$ is $\Sigma_2$.
Let $X \subseteq \mathbf{O}(a,a,b,b;c)$ be the fiber over $x$, which is a summand of $\mathbf{O}(a,a,b,b;c)$ (since $\mathbf{P}$ is a discrete colored operad).
We have that $X = \phi^{-1}(x) \to \{x\}$ is a weak equivalence, and using that $\mathbf{O}$ is $\Sigma$-cofibrant, one can show that the $\Sigma_2$-action on $X$ is free.

Let $F_{\mathbf{P}} \colon \xsSet^T \to \Alg_{\mathbf{P}}(\xsSet)$ be the free $\mathbf{P}$-algebra functor, and likewise for $F_{\mathbf{O}}$.
Letting $\ast_T$ be the terminal $T$-colored object, we have a commutative diagram
\[ \begin{tikzcd}
X \rar \dar & \{ x\} \dar \\
\mathbf{O}(a,a,b,b;c) \rar \dar & \mathbf{P}(a,a,b,b;c) \dar \\
F_{\mathbf{O}}(\ast_T)_c \rar & F_{\mathbf{P}}(\ast_T)_c
\end{tikzcd} \]
where the vertical composites are surjective onto summands.
More specifically, $\mathbf{O}(a,a,b,b;c) / V$ is naturally a summand of $F_{\mathbf{O}}(\ast_T)_c$, and $\phi^{-1}(x \cdot V) \subseteq \mathbf{O}(a,a,b,b;c)$ is a $V$-summand, so 
\[
	X / \Sigma_2 \cong \phi^{-1}(x \cdot V) / V
\]
is a summand of $\mathbf{O}(a,a,b,b;c) / V$.

As the left vertical map factors through $X / \Sigma_2 \simeq B\Sigma_2$, the map $F_{\mathbf{O}}(\ast_T)_c \to F_{\mathbf{P}}(\ast_T)_c$ has a summand of the form $B\Sigma_2 \to \{x\} / \Sigma_2 = \ast$.
It follows that $F_{\mathbf{O}}(\ast_T) \to \phi^*F_{\mathbf{P}}(\ast_T)$ is not a weak equivalence of $\mathbf{O}$-algebras.

We are now in a position to see that 
\[
\phi_! \colon \Alg_{\mathbf{O}}(\xsSet) \rightleftarrows \Alg_{\mathbf{P}}(\xsSet) \colon \phi^*
\]
is not a Quillen equivalence.
Both model structures are right-induced from the forgetful functors to $\xsSet^T$, hence $\phi^*$ creates weak equivalences and fibrations.
In particular, $\phi^*$ is right Quillen and reflects weak equivalences between fibrant objects, so by \cite[Corollary 1.3.16]{Hovey:MC}, $\phi_! \dashv \phi^*$ is a Quillen equivalence if and only if $A \to \phi^* (\phi_!A)^f$ is a weak equivalence for all cofibrant $A$ in $\Alg_{\mathbf{O}}(\xsSet)$.
But the fibrant replacement $\phi_!A \to (\phi_!A)^f$ is a weak equivalence and $\phi^*$ preserves weak equivalences, so by 2-of-3, $\phi_! \dashv \phi^*$ is a Quillen equivalence if and only if $A \to \phi^*\phi_!A$ is a weak equivalence for all cofibrant $A$ in $\Alg_{\mathbf{O}}(\xsSet)$.

The algebra $A = F_{\mathbf{O}}(\ast_T)$ is cofibrant.
Since the diagram
\[ \begin{tikzcd}[column sep=tiny]
\Alg_{\mathbf{O}}(\xsSet)\arrow[dr] & & \arrow[ll, "\phi^*" swap] \Alg_{\mathbf{P}}(\xsSet) \arrow[dl]\\
& \xsSet^T
\end{tikzcd} \]
commutes, we have $\phi_!A = \phi_! F_{\mathbf{O}}(\ast_T) \cong F_{\mathbf{P}}(\ast_T)$.
But as we saw, $A = F_{\mathbf{O}}(\ast_T) \to \phi^* \phi_! F_{\mathbf{O}}(\ast_T) = \phi^* F_{\mathbf{P}}(\ast_T) $ is not a weak equivalence, so $\phi_! \dashv \phi^*$ is not a Quillen equivalence.

Since this is not a Quillen equivalence, \cite[Theorem 7.5]{PavlovScholbach} shows that $\phi$ is not symmetric flat in $\xsSet$.
Using their terminology, $\mathbf{O} = Q\mathbf{P}$, so the fact that $\phi \colon \mathbf{O} \to \mathbf{Prpd}_S$ is not symmetric flat implies, by \cite[Theorem 7.11]{PavlovScholbach},
that $\Alg_{\mathbf{Prpd}_S}(\xsSet)\to \Alg_{\calproperad_S}(\xS)$ is not an equivalence.
\end{proof}

\begin{definition}
\label{def prpdV}
	Given a symmetric monoidal model category $\mathbf{V}$ for which the operads $\bfproperad_{S}$ are admissible for all sets $S$. 
	Let $\xV$ denote the associated $\infty$-category. We write $\nmproperad(\mathbf{V}) \to \Set$ for the Grothendieck fibration (and opfibration) which corresponds to the functor $\Set\to \xCat$ taking $S$ to $\Alg_{\bfproperad_{S}}(\mathbf{V})$. 
	It follows from \cite[Proposition 4.25]{enrcomp} or \cite[Theorem 3.0.12]{HarpazPrasma} that the model structure on $\Alg_{\bfproperad_{S}}(\mathbf{V})$ induces one on $\nmproperad(\mathbf{V}) $ where weak equivalences are those morphisms that are bijective on objects and weak equivalences on all multimorphism objects. 
	We write $\nmproperad(\xV)$ for the $\infty$-category associated to the model category $\nmproperad(\mathbf{V})$. 
	Similarly, we define $\infty$-categories $\name{Prpd}^{\name{out}}(\xV)$ and $\name{DOpd}(\xV)$ when $\mathbf{V}$ is a symmetric monoidal model category so that the operads $\mathbf{Prpd}^{\name{out}}_{S}$ (respectively, $\mathbf{DOpd}_{S}$) are admissible for all sets $S$.
\end{definition} 

We should emphasize here that the model structures $\nmproperad(\mathbf{V})$ and so on are only \emph{intermediate} model structures, and not of independent interest for us.
The reader should also be careful to distinguish between this model structure and others that may exist on the same underlying category.
For example (with some restriction on $\mathbf{V}$) there is a model structure on $\nmproperad(\mathbf{V})$ whose weak equivalences are the Dwyer--Kan equivalences from Definition~\ref{def DK equiv} below (see \cite{HackneyRobertsonYau:SMIP} for the case $\mathbf{V} = \xsSet$ and \cite{Yau:DKHTAOC} for certain other $\mathbf{V}$).

\begin{cor}\label{cor:OpdWeq}
Let $\mathbf{V}$ be a symmetric monoidal model category.
If $\mathbf{Prpd}^{\name{out}}_{S}$ is admissible in $\mathbf{V}$ for all sets $S$, then there is an equivalence
\[\nmproperad^{\name{out}}(\xV) \isoto
	\Alg_{\calproperad^{\name{out}}/\name{Set}}(\mathbf{\xV})\]
	over $\Set$.
Likewise, if $\mathbf{DOpd}_{S}$ is admissible in $\mathbf{V}$ for all sets $S$, then there is an equivalence
\[\name{DOpd}(\xV) \isoto
	\Alg_{\mathcal{DO}pd/\name{Set}}(\mathbf{\xV})\]
	over $\Set$.
\end{cor}
\begin{proof}
We prove the first statement, as the second is entirely analogous.
According to \cite[Corollary 4.22]{enrcomp} or
	\cite[Proposition 2.1.4]{hinloc}, $\nmproperad^{\name{out}}(\mathbf{V}) \to \Set$ is the cartesian (and cocartesian) 
	fibration corresponding to the functor taking $S$ to the $\infty$-category $\Alg_{\calproperad^{\name{out}}_{S}}(\xV)$ associated to $\Alg_{\bfproperad^{\name{out}}_{S}}(\mathbf{V})$.
	This shows that the functor
	\[\nmproperad^{\name{out}}(\xV) \to
	\Alg_{\calproperad^{\name{out}}/\name{Set}}(\xV)\]
	over $\Set$ is a functor between cartesian fibrations that preserves cartesian morphisms which is then an equivalence as it is one on each fibre by Theorem~\ref{theo:AlgOpScomp}.
\end{proof}

Suppose that $\mathbf{V}$ is a symmetric monoidal model category.
The functor $\mathbf{V} \to h\mathbf{V}$ to the homotopy category is symmetric monoidal, so to a $\mathbf{V}$-enriched operad $\mathbf{P}$ we can associate an $h\mathbf{V}$-enriched operad $h\mathbf{P}$.

\begin{defn}\label{def DK equiv}
If $\mathbf{V}$ is a symmetric monoidal model category, we say a morphism $F \colon \mathbf{P} \to \mathbf{P}'$ of $\mathbf{V}$-enriched properads is a \emph{Dwyer--Kan equivalence} if:
	\begin{enumerate}[(1)]
		\item The map 
		\[ \mathbf{P}(x_{1},\ldots,x_{m}; y_1,\ldots, y_n) \to
		\mathbf{P}'(F(x_{1}), \ldots, F(x_{m}); F(y_m),\ldots, F(y_n))\] 
		is a weak equivalence in $\mathbf{V}$ for all $x_{1},\ldots,x_{},y_1,\ldots, y_n$ in $\mathbf{P}$.
		\item  The induced functor of $h\mathbf{V}$-enriched operads $h F \colon h \mathbf{P} \to h\mathbf{P}'$ is essentially surjective (i.e.\ its underlying functor of enriched categories is essentially surjective).
	\end{enumerate}
Dwyer--Kan equivalences for $\mathbf{V}$-enriched dioperads are defined similarly.
\end{defn}

We are now ready to prove our main rectification result.

\begin{thm}\label{thm:rect}
Suppose $\mathbf{V}$ is a symmetric monoidal model category for which the operads $\bfproperad^{\name{out}}_{S}$ are admissible for all sets $S$.
Then the $\infty$-category $\nmproperad^{{\name{out}},\xV}_\infty$ (Notation~\ref{no PrpdV}) can be identified with the localization of $\nmproperad^{\name{out}}(\xV)$ at the class of Dwyer--Kan equivalences:
\[
	\nmproperad^{{\name{out}},\xV}_\infty \simeq \nmproperad^{\name{out}}(\xV)[\name{DK}^{-1}].
\]
Likewise, if $\mathbf{DOpd}_S$ is admissible in $\mathbf{V}$ for all sets $S$, then the $\infty$-category $\name{DOpd}^\xV_\infty$ of $\xV$-enriched $\infty$-dioperads can be identified with the localization of $\name{DOpd}(\xV)$ at the Dwyer--Kan equivalences.
\end{thm}
\begin{proof}
	We prove the first statement, as the second is similar.
	Under the equivalence of Corollary~\ref{cor:OpdWeq} the class of Dwyer--Kan equivalences corresponds to the class of fully faithful and essentially surjective morphisms in $\Alg_{\calproperad^{\name{out}}/\name{Set}}(\mathbf{\xV})$. 
	Hence, the localization of $\nmproperad^{\name{out}}(\xV)$ with respect to Dwyer--Kan equivalences is equivalent to $\Alg_{\calproperad^{\name{out}}/\name{Set}}(\mathbf{\xV})[\name{FFES}^{-1}]$ which can be identified with $\nmproperad^{{\name{out}},\xV}_\infty$ by the appropriate variation of Corollary~\ref{cor AlgPrd=PCSbbY}.
\end{proof}

In light of Theorem~\ref{theorem non-rect}, the method of proof used in this theorem breaks down if one tries to prove that $\nmproperad_\infty^\xV$ is equivalent to $\nmproperad(\xV)[\name{DK}^{-1}]$ in generality.
However, we are able to adapt this proof when working over very special bases.

\begin{theorem}\label{theorem rational rectification for properads}
Let $k$ be a commutative ring containing $\mathbb{Q}$, let $\mathbf{Ch}_k$ be the category of unbounded chain complexes equipped with the projective model structure, and let $\mathcal{C}h_k$ denote the $\infty$-category associated to $\mathbf{Ch}_k$.
Then there is an equivalence (see Notation~\ref{no PrpdV})
\[
	\nmproperad(\mathcal{C}h_k)[\name{DK}^{-1}] \simeq \nmproperad_\infty^{\mathcal{C}h_k}
\]
after localizing at the class of Dwyer--Kan equivalences.
\end{theorem}
\begin{proof}
It was observed in \cite[\S7.4]{PavlovScholbachSymm} that $\mathbf{Ch}_k$ is symmetric flat, so the proof of Theorem~\ref{theo:AlgOpScomp} can be adapted to show that the canonical map
\[\Alg_{\mathbf{Prpd}_S}(\mathbf{Ch}_k)\to \Alg_{\calproperad_S}(\mathcal{C}h_k) \]
is an equivalence for all $S$.
A variation on the proof of Corollary~\ref{cor:OpdWeq} gives that
\[\nmproperad(\mathbf{Ch}_k) \isoto \AlgPrdSet(\mathcal{C}h_k)\] 
is an equivalence.
The remainder of the proof follows that of Theorem~\ref{thm:rect}.
\end{proof}

\begin{remark}\label{remark symm spectra}
The key to the proof of the previous theorem is that the model structure on $\mathbf{Ch}_k$ is symmetric flat.
In particular, it is also true that 
\[
	\nmproperad(\xV)[\name{DK}^{-1}] \simeq \nmproperad_\infty^{\xV}
\]
when $\mathbf{V}$ is the positive stable model structure on symmetric spectra.
Indeed, \cite[Proposition 3.5.1]{PavlovScholbachSp} shows that this $\mathbf{V}$ is symmetric flat (actually this holds for spectra over more general `nice' bases as in \cite[Definition 2.3.1]{PavlovScholbachSp}), so the proof of Theorem~\ref{theorem rational rectification for properads} is readily adapted.
\end{remark}

\begin{question}
In Theorem~\ref{thm:rect}, we showed that the $\infty$-category of $\infty$-dioperads, $\name{DOpd}^\xS_\infty$, is equivalent to $\name{DOpd}(\xS)[\name{DK}^{-1}]$.
On the other hand, \cite[Theorem 5.4]{HackneyRobertsonYau:SMIP} shows that the category $\name{DOpd}(\xsSet)$ (see Definition~\ref{def prpdV}) admits a model structure whose weak equivalences are the Dwyer--Kan equivalences from Definition~\ref{def DK equiv}.
We do not know if the $\infty$-category presented by this model category is equivalent to $\name{DOpd}(\xS)[\name{DK}^{-1}]$ or not, but it would be interesting to explore.
One difficulty in addressing this in the present context is simply that the model structure from \cite[Theorem 5.4]{HackneyRobertsonYau:SMIP} is not a Bousfield localization of the model structure in Definition~\ref{def prpdV}. 
In fact, more is true: the identity functor on this category is not a Quillen functor between the two model structures.
\end{question}

Similar questions can be raised for other ground symmetric monoidal categories $\mathbf{V}$ and for properads instead of dioperads.

\appendix

\renewcommand{\thetheorem}{\thesection.\arabic{theorem}}

\section{Equivalence of \texorpdfstring{$\bbY$}{G} with the properadic graphical category}
\label{appendix comparison with HRY}

This section is devoted to a proof of the following theorem.

\begin{theorem}\label{theorem upsilon equivalent definitions}
The category $\bbY$ is equivalent to the category $\hryGamma$ from \cite{hrybook}.
\end{theorem}

One difference between the setup of \cite{hrybook} and that of the present paper is that the graphs in that book always come equipped with orderings of $\inp(v)$, $\out(v)$, $\inp(G)$, and $\out(G)$.
As such, we first replace $\bbY$ with an equivalent category $\bbY'$ whose objects are graphs together with orderings on the sets $\inp(v)$, $\out(v)$, $\inp(G)$, and $\out(G)$.
Morphisms ignore this extra structure entirely.

We will show that $\bbY'$ is isomorphic to the category $\hryGamma$ from \cite{hrybook}.
Both categories have the same set of objects.

Given a morphism $f : G \to K$ in $\bbY'$, we define a corresponding properadic graphical map $f^\gamma : G \to K$ in $\hryGamma$.
This map has $f^\gamma_0 = f_0$, while $f^\gamma_1 (v) \coloneqq f_1(C_v)$ (see \cite[Lemma 5.19]{hrybook}).
At the moment, this is just a map of the corresponding colored properads.
\begin{lemma}\label{computing the image of falpha at a graph}
	If $H\in \sub(G)$, then $f^\gamma(H)$ is equal to $f_1(H)$.
\end{lemma}
\begin{proof}
We induct on $\deg(H) = \lvert\vertex(H)\rvert$. 
The result is either automatic for $\deg(H)\in\{ 0, 1\}$.
Suppose that $\deg(H) \geq 2$.
By \cite[Corollary 2.76]{hrybook}, $H$ has an almost isolated vertex $v$.
We have $H = C_v \strcup H'$, where $H' \in \sub(H) \subseteq \sub(G)$\footnote{Using Proposition~\ref{strsub and ordsub coincide}.} has $\vertex(H) \setminus \{v \}$ as its set of vertices (see \cite[Definition 2.60 \& \S 6.1.2]{hrybook}). 
Then by the induction hypothesis, we have
\begin{equation}\label{breaking H into pieces}
f_1(H) = f_1 (C_v) \strcup f_1(H') = f^\gamma(C_v) \strcup f^\gamma(H').
\end{equation}
If $f^\gamma(C_v) = {\downarrow_e}$, then $f^\gamma(H') = f^\gamma(H)$ and we are done.
Otherwise, the subgraph from \eqref{breaking H into pieces} is an open subgraph with the same set of vertices as $f^\gamma(H)$ (see \cite[Definition 6.40]{hrybook}), hence must be equal to $f^\gamma(H)$.
\end{proof}
\begin{proof}[Proof of Theorem~\ref{theorem upsilon equivalent definitions}]
As $\bbY'$ and $\hryGamma$ have the same set of objects it suffices to show that there is a bijection of morphisms which respects compositions and identities.
The previous lemma shows $f^\gamma(G) = f_1(G)$ is a structured subgraph of $K$, so  that the assignment $f\mapsto f^\gamma$ takes a morphism in $\bbY'$ to a properadic graphical map in $\hryGamma$ introduced in \cite[Definition 6.46]{hrybook}.

In the reverse direction, suppose that $f: G \to K$ is a properadic graphical map in $\hryGamma$.
Define $f^\upsilon$ by $f^\upsilon_0 = f_0$ and $f^\upsilon_1(H) = f(H)$, which we know is a structured subgraph by \cite[Theorem 6.50]{hrybook}.
The fact that Definition~\ref{definition complete morphism}\eqref{definition complete morphism preservation of ports} holds follows from the fact that $f$ is a map between the corresponding colored properads, as in the second part of \cite[Lemma 5.19]{hrybook}.
We have that
\[
	\vertex(f(H)) \cong \coprod_{v\in \vertex(H)} \vertex(f_1(v))
\]
for any $H\in \sub(G)$ by \cite[Definition 6.40 \& Remark 2.42(1)]{hrybook}, thus
\begin{equation*}
\vertex(f^\upsilon_1(H_1 \strcup H_2)) = \vertex(f^\upsilon_1 (H_1)) \cup \vertex(f^\upsilon_1 (H_2)) = \vertex( f^\upsilon_1 (H_1) \ordcup f^\upsilon_1 (H_2) ).
\end{equation*}
Further, we have $f^\upsilon_1 (H_1) \ordcup f^\upsilon_1 (H_2)$ is an open subgraph of $f^\upsilon_1(H_1 \strcup H_2)$ since $f^\upsilon_1$ preserves the partial order $\strsub$.
Here we have two open subgraphs which have the same set of vertices. If that vertex set is non-empty, then the containment becomes an equality.

We now show that the operations $f\mapsto f^\gamma$ and $f\mapsto f^\upsilon$ are inverse to each other. First suppose that $f$ is a morphism in $\bbY'$, we wish to show that $(f^\gamma)^\upsilon = f$. Since this is true by definition on edge sets, we must show $(f^\gamma)^\upsilon_1 = f_1$.
Let $H\in \sub(G)$. Then we have 
\[
	(f^\gamma)^\upsilon_1(H) = f^\gamma(H) = f_1(H)
\]
where the first equality is the definition and the second equality is given by Lemma~\ref{computing the image of falpha at a graph}.

Likewise, if $f$ is a properadic graphical map in $\hryGamma$, we wish to show that $(f^\upsilon)^\gamma = f$. 
But \[
	(f^\upsilon)^\gamma_0 = f^\upsilon_0 = f_0
\]
hence $(f^\upsilon)^\gamma = f$ by \cite[Corollary 6.62]{hrybook}. Thus we have established a bijective correspondence between morphisms of $\bbY'$ and properadic graphical maps of $\hryGamma$.

It remains to show that this bijection constitutes a functor between the two categories in question.
But if $f, g$ are two composable morphisms in $\bbY'$, then $f^\gamma g^\gamma = (fg)^\gamma$ using \cite[Corollary 6.62]{hrybook}, since
\[
	(f^\gamma g^\gamma)_0 = f^\gamma_0 g^\gamma_0 = f_0 g_0 = (fg)_0 = (fg)^\gamma_0.
\]
Likewise, $(\id_G)^\gamma_0 = (\id_G)_0$ is an identity, hence $(\id_G)^\gamma$ must be an identity.
\end{proof}
\begin{remark}
By Theorem~\ref{theorem upsilon equivalent definitions}, the category $\bbY$ is equivalent to the category $\hryGamma$ which can be identified with a wide subcategory of Kock's category $\kockgraphs$ (from Remark~\ref{definition kockgraphs}) containing fewer inert morphisms and the same active morphisms (see \cite[2.4.14]{Kock_Properads}).
\end{remark}

\section{Proof of Proposition~\ref{proposition functor l to y}}\label{appendix proof l to y}

In this section, \emph{graph} will always mean \emph{acyclic} graph.
The following is a variation of the construction in Definition~\ref{definition kockgraphs}, which was only about \emph{connected} graphs.
\begin{definition}\label{def properad genned by graph}
If $G$ is a (possibly disconnected) graph, then there is an $\edge(G)$-colored properad $\mathfrak{P}(G)$ generated by the vertices of $G$.
Each connected, open subgraph $H$ of $G$ determines (up to a choice of orderings on inputs and outputs) an element in $\mathfrak{P}(G)$.
\end{definition} 
In general, there are many other elements of $\mathfrak{P}(G)$ that do not come from connected, open subgraphs.
The reader concerned about disconnected graphs can either define $\mathfrak{P}(\coprod_i G_i) \coloneqq \coprod_i \mathfrak{P}(G_i)$ where each $G_i$ is connected, or note that Definition 5.3 and 5.7 of \cite{hrybook} make no real use of connectedness of $G$.
The properads acquired in these two ways will be isomorphic.

Many examples of this construction, as well as for properad maps $\mathfrak{P}(G) \to \mathfrak{P}(H)$, are given for connected graphs in \cite[Chapter 5]{hrybook}.
Thus we will give a single example for a disconnected graph here.

\newcommand{\coloredone}{{\color[HTML]{006666}1}}
\newcommand{\coloredtwo}{{\color[HTML]{999933}2}}
\newcommand{\coloredthree}{{\color[HTML]{3333CC}3}}
\newcommand{\coloredfour}{{\color[HTML]{990099}4}}
\newcommand{\coloredfive}{{\color[HTML]{336600}5}}
\newcommand{\coloredsix}{{\color[HTML]{990000}6}}

\begin{figure}[htb]
\labellist
\small\hair 2pt
 \pinlabel {\coloredone} [ ] at 28 111
 \pinlabel {\coloredtwo} [ ] at 2 52
 \pinlabel {\coloredthree} [ ] at 75 52
 \pinlabel {\coloredfour} [ ] at 112 111
 \pinlabel {\coloredfive} [ ] at 99 18
 \pinlabel {\coloredsix} [ ] at 141 18
\endlabellist
\centering
\includegraphics[scale=0.75]{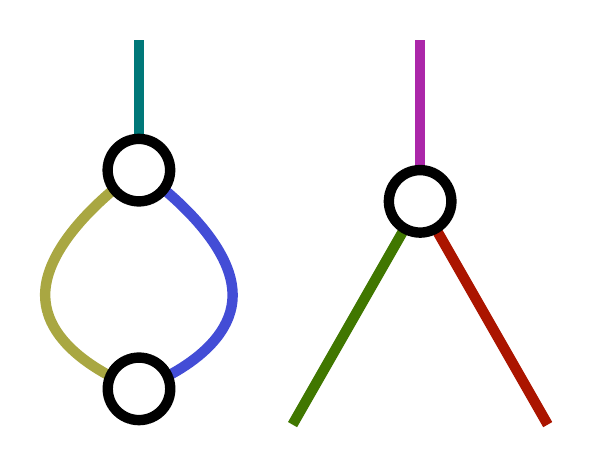}
\caption{The graph $G$ from Example~\ref{example PG disconnected}}
\label{figure example PG disconnected}
\end{figure}
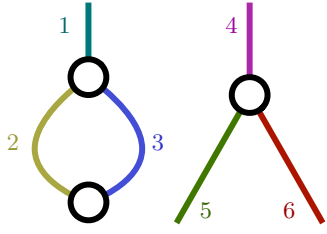

\begin{example}\label{example PG disconnected}
Consider the disconnected graph $G$ from Figure~\ref{figure example PG disconnected} with edge set $\{\coloredone,\coloredtwo,\coloredthree,\coloredfour,\coloredfive,\coloredsix\}$.
Up to orderings, $\mathfrak{P}(G)$ has elements precisely in the following profiles (where $n\geq 1$):
\begin{align*}
&\mathfrak{P}(G)(\coloredone; \coloredtwo, \coloredthree)  && \mathfrak{P}(G)(\coloredtwo, \coloredthree; \,\,) \\
&\mathfrak{P}(G)(\coloredone, \overset{n}{\cdots}, \coloredone; \,\,)  && \mathfrak{P}(G)(\coloredone, \overset{n}{\cdots}, \coloredone, \coloredtwo; \coloredtwo) \\
& \mathfrak{P}(G)(\coloredone, \overset{n}{\cdots}, \coloredone, \coloredtwo, \coloredthree;\,\,) && \mathfrak{P}(G)(\coloredone, \overset{n}{\cdots}, \coloredone, \coloredthree; \coloredthree) \\
& \mathfrak{P}(G)(\coloredone, \overset{n}{\cdots}, \coloredone; \coloredtwo, \coloredthree)  && \mathfrak{P}(G)(\coloredfour; \coloredfive, \coloredsix).
\end{align*}
See Figure~\ref{figure example PG disconnected elements} for pictures of some of these elements.
\end{example}

\begin{figure}[htb]
\includegraphics[scale=0.75]{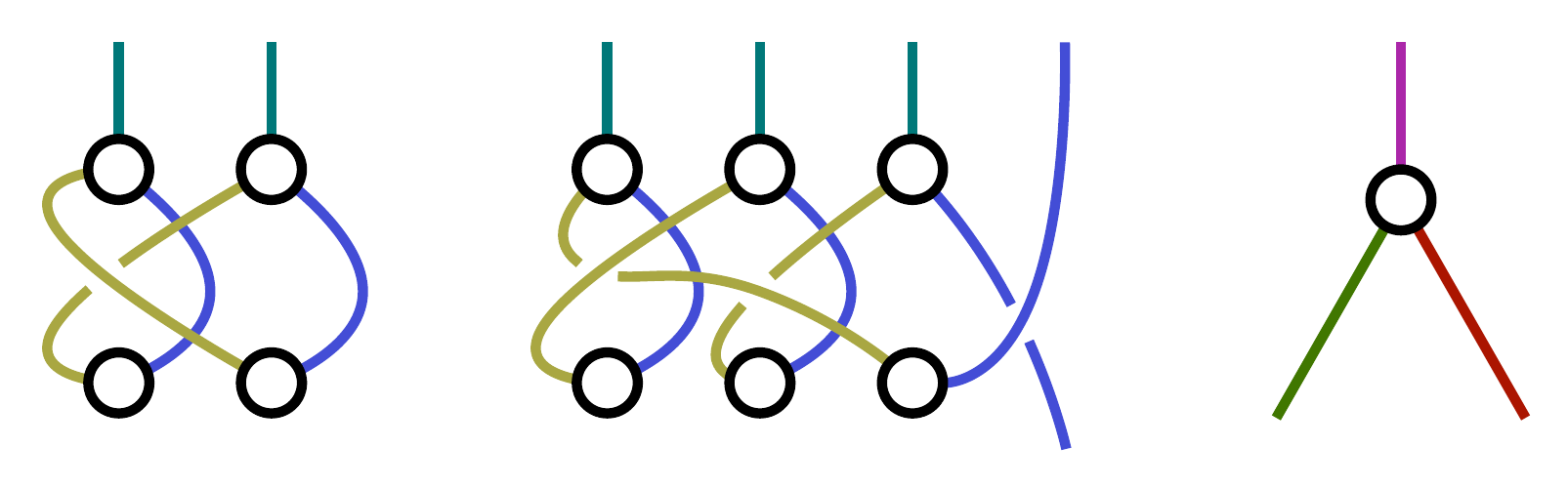}
\caption{Elements of $\mathfrak{P}(G)(\coloredone, \coloredone; \,\,)$, $\mathfrak{P}(G)(\coloredone, \coloredone, \coloredone, \coloredthree; \coloredthree)$, and $\mathfrak{P}(G)(\coloredfour; \coloredfive, \coloredsix)$ in the graph $G$ from Figure~\ref{figure example PG disconnected}.}
\label{figure example PG disconnected elements}
\end{figure}

If $(\alpha, \eta) : G \to H$ is in $\levelg$, then there is a well-defined map of properads $\mathfrak{P}(G) \to \mathfrak{P}(H)$ sending $v\in G_{i,i+1}$ to $\eta_{i,i+1}(v) \in H_{\alpha(i),\alpha(i+1)} \hookrightarrow \mathfrak{P}(H)$. 
The goal of this section is to show, when $G \to H$ is in $\levelgconn$, that the map $f : \mathfrak{P}(G) \to \mathfrak{P}(H)$ is actually a `graphical map.'
That is, the \emph{functor} from $\levelgconn \to \properads$ factors through the (non-full) subcategory $\hryGamma$.
Thus we are using Theorem~\ref{theorem upsilon equivalent definitions} in an essential way in this section.

Each $\mathfrak{C}$-colored properad $P$ has an underlying bimodule (see \S3.6.1 of \cite{hrybook}), that is, for each pair of lists of colors $c_1,\dots, c_n$ and $c_1', \dots, c_m'$, a set
\[
	P(c_1,\dots, c_n; c_1', \dots, c_m'),
\]
along with isomorphisms 
\[ \begin{tikzcd}
P(c_1,\dots, c_n ; c_1', \dots, c_m') \rar["\cong","(\sigma';\sigma)" swap] & P(c_{\sigma'(1)},\dots, c_{\sigma'(n)} ; c_{\sigma^{-1}(1)}', \dots, c_{\sigma^{-1}(m)}')
\end{tikzcd} \]
for $\sigma' \in \Sigma_n$ and $\sigma \in \Sigma_m$.
These give a right $\Sigma_n$ action and a left $\Sigma_m$ action on 
\[
	P(n;m) \coloneqq \coprod_{\mathfrak{C}^n \times \mathfrak{C}^m} P(c_1,\dots, c_n ; c_1', \dots, c_m').
\]
\begin{definition}
If $P$ is a $\mathfrak{C}$ colored properad (in $\Set$), let $\overline{P}$ denote the bigraded set
\[
	\overline{P}(n;m) \coloneqq P(n;m)/(\Sigma_n^{\op} \times \Sigma_m)
\]
along with the induced functions
\[
	\overline{P}(n;m) \to (\mathfrak{C}^n/\Sigma_n) \times (\mathfrak{C}^m / \Sigma_m).
\]
\end{definition}
Another way of phrasing this is that if $M\mathfrak{C}$ is the free commutative monoid on $\mathfrak{C}$, then $\overline{P}$ is a set equipped with a function $\overline{P} \to M\mathfrak{C} \times M\mathfrak{C}$.
As we've taken the underlying object of $P$ and removed all symmetry, the object $\overline{P}$ will not typically admit a properad structure.

Why do we consider this structure?
Each vertex $v \in \vertex(G)$ with $n$ inputs and $m$ outputs will generate $n!m!$ different morphisms in the properad $\mathfrak{P}(G)$ generated by a graph $G$, one for each choice of ordering on $\inp(v)$ and $\out(v)$.
For our current purposes, this extra data is a distraction, hence we have use for $\overline{\mathfrak{P}(G)}$, which admits a well-defined function
\[
	\vertex(G) \hookrightarrow \overline{\mathfrak{P}(G)}
\]
sending $v$ to the set of morphisms it generates.

\begin{definition}\label{definition vertex lift}
Let us say that a map of properads $f : \mathfrak{P}(H) \to \mathfrak{P}(G)$ \emph{has a vertex lift} if the dotted arrow $\widetilde f$ exists in the diagram
\begin{equation}\label{eq vertex lifts}
\begin{tikzcd}
\vertex(H) \dar[hook] \rar[dashed, "\widetilde f"] & \vertex(G) \dar[hook] \\
\overline{\mathfrak{P}(H)} \rar{\overline{f}} & \overline{\mathfrak{P}(G)}.
\end{tikzcd} \end{equation}
\end{definition}

\begin{lemma}
\label{lemma inj colors to etale mono}
Let $H$ be a connected graph, and let $f  : \mathfrak{P}(H) \to \mathfrak{P}(G)$ be a map of properads.
If $f$ has a vertex lift, then $f$ comes from an \'etale map $H \to G$.
If, in addition, $f$ is injective on colors (that is, if $\edge(H) \to \edge(G)$ is injective), then $f$ comes from an \'etale monomorphism $H\to G$.
\end{lemma}
\begin{proof}
Let
\[
	x \in \mathfrak{P}(H)(e_1, \dots, e_n; e_1', \dots, e_m')
\]
be a representative for a vertex $v$.
This implies that $\inp(v) = \{ e_1, \dots, e_n \}$ is an $n$-element subset of $\edge(H)$ and $\out(v) = \{ e_1', \dots, e_m' \}$ is an $m$-element subset of $\edge(H)$.
We know
\[
	fx \in  \mathfrak{P}(G)(fe_1, \dots, fe_n; fe_1', \dots, fe_m')
\]
and that this element represents the vertex $\widetilde f (v) \in \vertex(G)$.
As is the case whenever we look at an element representing a vertex in $\mathfrak{P}(G)$, we have that $\inp(\widetilde f (v)) = \{ fe_1, \dots, fe_n \}$ is an $n$-element set and $\out(\widetilde f (v)) = \{ fe_1', \dots, fe_m' \}$ is an $m$-element set.
Define a na\"ive morphism $f' : H\to G$ which is $f|_{\edge(H)}$ on edges and $\widetilde f$ on vertices.
We have seen that $f'$ preserves numbers of inputs and outputs, hence $f'$ is an \'etale map (Definition~\ref{definition kockint}).

With the first statement proved, let us now suppose that $\edge(H) \to \edge(G)$ is injective.
For the second statement, we must show that $\widetilde f : \vertex(H) \to \vertex(G)$ is injective.
Suppose $\widetilde f(v_1) = \widetilde f(v_2)$, from which it follows that $\inp(\widetilde f(v_1)) = \inp(\widetilde f(v_2))$ and $\out(\widetilde f(v_1)) = \out(\widetilde f(v_2))$.
If $\inp(\widetilde f(v_k)) = \varnothing = \out(\widetilde f(v_k))$, then we also have that $\inp(v_k) = \varnothing = \out(v_k)$.
Since $H$ is connected, this implies that $\vertex(H)$ has a unique element, so $v_1 = v_2$.

We may thus assume that one of the two sets $\inp(\widetilde f(v_k))$ or $\out(\widetilde f(v_k))$ is nonempty; without loss of generality, assume $\inp(\widetilde f(v_1)) = f(\inp(v_1)) = f(\inp(v_2))$ is nonempty.
Since $f|_{\edge(H)}$ is injective, this set is inhabited by an element $f(e)$ with $e\in \inp(v_1) \cap \inp(v_2)$, so $v_1=v_2$.
Thus $\widetilde f$ is injective, so the \'etale map $f'$ given in the first paragraph is a monomorphism.
\end{proof}

We now turn to the case when $f$ is not injective on colors.

\begin{lemma}\label{lemma mix up each ways}
Let $H$ be a connected graph, and let $f  : \mathfrak{P}(H) \to \mathfrak{P}(G)$ be a map of properads.
Suppose that $f$ has a vertex lift $\widetilde f : \vertex(H) \to \vertex(G)$ which is injective.
If $e_1 \neq e_2$ are edges of $H$ so that $f(e_1) = f(e_2)$, then $(e_1,e_2)$ or $(e_2,e_1)$ is an element of $\inp(H) \times \out(H)$.
\end{lemma}
\begin{proof}
Suppose $e_1 \neq e_2$ and $f(e_1) = f(e_2)$.
It suffices to show that $(e_1,e_2)\notin \inp(H) \times \out(H)$ implies $(e_2,e_1)\in \inp(H) \times \out(H)$. 
For this we first prove that $(e_1,e_2)\notin \inp(H) \times \out(H)$ implies $e_1\notin \inp(H) $ and $e_2\notin \out(H)$. 

The existence of $e_1 \neq e_2$ shows that $H\not\cong {\downarrow}$, hence, if $e_1 \notin \inp(H)$ then $e_1 \in \out(v)$ for some vertex $v$.
If $e_2 \in \out(H)$ then there exists a vertex $u$ with $e_2 \in \out(u)$.
On the one hand, the injectivity of $f|_{\out(v)}$ implies that $u \neq v$.
On the other, $f(e_1) = f(e_2)$ can be an output for at most one vertex. 
Thus, $\widetilde f(u) = \widetilde f(v)$ and the injectivity of $\widetilde f$ gives $u = v$.
This contradiction proves that $e_1 \notin \inp(H)$ implies $e_2 \notin \out(H)$ and a symmetric argument shows that the reverse implication also holds. 
Therefore, if $(e_1,e_2)\notin \inp(H) \times \out(H)$ then $e_1\notin \inp(H) $ and $e_2\notin \out(H)$. 

Now let us assume that $e_1 \in \out(v)$ and $e_2 \in \inp(w)$ for some vertices $v, w$. 
To show that $(e_2, e_1) \in \inp(H) \times \out(H)$ it is necessary to exclude that possibility that $e_1 \in \inp(u)$ for some $u$. 
If such a vertex $u$ exists, then $f(e_1) \in \inp(\widetilde f(u))$ equals $f(e_2) \in \inp(\widetilde f(w))$, so we must have $\widetilde f(u) = \widetilde f(w)$ which implies that $u=w$.
But $f|_{\inp(w)}$ is a monomorphism, so this would imply $e_1 = e_2$, contrary to our hypothesis.
Thus no such $u$ exists and $e_1 \in \out(H)$.
The symmetric argument establishes that $e_2 \in \inp(H)$, so $(e_2, e_1) \in \inp(H) \times \out(H)$.
\end{proof}

The following example, which appeared as Example 5.25 of \cite{hrybook}, shows that the behavior of Lemma~\ref{lemma mix up each ways} actually occurs.

\begin{example}\label{ex etale not inj edges}
There is an evident \'etale map from the graph $H$ on the left to the graph $G$ on the right which is not injective on edge sets.
The induced map of properads $f : \mathfrak{P}(H) \to \mathfrak{P}(G)$ has an injective vertex lift, but does not satisfy the hypothesis of the next corollary.
\begin{center}
\begin{tikzpicture}[new set=import nodes]
\begin{scope}[nodes={set=import nodes}] 
\node (u0) at (1.2,.8) {};
\node [circle,draw,very thick] (u1) at (2,2) {$u$};
\node (v0) at (1.2,1.2) {};
\node [circle,draw,very thick] (v1) at (2,0) {$v$};
\node [circle,draw,very thick] (u) at (5,2) {$u$};
\node [circle,draw,very thick] (v) at (5,0) {$v$};
\end{scope}
\graph [edges={thick}]
{ (import nodes);
u0 -> v1;
u1 -> v0;
u1 -> [bend left] v1;
u -> [bend right] v; u -> [bend left] v;
}; \end{tikzpicture}
\end{center}
\end{example}

\begin{corollary}\label{cor vertex inclusion in out inclusion}
Let $H$ be a connected graph, and let $f  : \mathfrak{P}(H) \to \mathfrak{P}(G)$ be a map of properads.
Suppose that $f$ has a vertex lift $\widetilde f : \vertex(H) \to \vertex(G)$ which is injective.
If $f|_{\inp(H) \cup \out(H)}$ is injective, then $f$ comes from an \'etale monomorphism.
\end{corollary}
\begin{proof}
Lemma~\ref{lemma mix up each ways} shows that $f$ is injective on edges, so Lemma~\ref{lemma inj colors to etale mono} applies.
\end{proof}

\begin{lemma}
\label{lemma in and out determination}
Suppose $H_1$ and $H_2$ are connected open subgraphs of a connected graph $G$.
If $\inp(H_1) = \inp(H_2)$ and $\out(H_1) = \out(H_2)$, then $H_1=H_2$.
\end{lemma}
\begin{proof}
Notice that if $H$ is an open \emph{connected} subgraph of $G$, then $\inp(H) \cap \out(H) \neq \varnothing$ if and only if $H$ is an edge.
Moreover, if this holds then $H$ is uniquely determined by $\inp(H)$, a one element set.
Thus it suffices to consider the case when $\inp(H_k) \cap \out(H_k) = \varnothing$ ($k=1,2$).
Forgetting the direction, the two inclusions become embeddings of undirected graphs in the sense of \cite{HackneyRobertsonYau:GCHMO}. 
The conclusion follows of the lemma follows by applying \cite[Proposition 1.25]{HackneyRobertsonYau:GCHMO} to these two inclusions (in their notation we have $\eth(f_k) = \inp(H_k) \amalg \out(H_k)$, $k=1,2$).
\end{proof}

\begin{theorem}\label{theorem on levelgconn to etale mono}
Let $\varphi = (\alpha, \eta): G \to H$ be a morphism of $\levelgconn$.
Suppose that $x\in G_{i,j}$ is an $(i,j)$-level subgraph with associated level graph $K$.
Let $\vertex(K) \cong \{v_1, \dots, v_n\} \subseteq \coprod_{p = i}^{j-1} G_{p,p+1}$ be the set of vertices that map to $x$, and let $H_\ell$ be the graph associated to $\eta(v_\ell) \in H_{\alpha(p),\alpha(p+1)}$.
Then there is an \'etale monomorphism $f' : K\{H_1, \dots, H_n\} \to H$ whose image is $\eta(x)$.
\end{theorem}

\begin{proof}
We know that
\[
	\vertex(K\{H_1, \dots, H_n\}) = \coprod_\ell \vertex(H_\ell)
\]
and we can use this to define a properad map
\[
	f : \mathfrak{P}(K\{H_1, \dots, H_n\}) \to \mathfrak{P}(H)
\]
whose restriction to the generators in $\vertex(H_\ell)$ is just the inclusion $\vertex(H_\ell) \subseteq \vertex(H)$.
Since $\vertex_{\levelg}(\varphi)(w) = v_\ell$ when $w$ is in $H_\ell$, we see that the graphs $H_\ell$ have disjoint sets of vertices.
Thus the properad map $f$ induces a monomorphism
\[
	\widetilde f : \vertex(K\{H_1, \dots, H_n\}) \to \vertex(H).
\]
Moreover, we have
\begin{gather*}
	\inp(K\{H_1, \dots, H_n\}) = \inp(K) \\
	\out(K\{H_1, \dots, H_n\}) = \out(K)
\end{gather*}
and the map $f|_{\inp(K\{H_1, \dots, H_n\})}$ is just the restriction of $\eta_{i,i}$ and likewise for outputs.
Since $\eta_{i,i}$ and $\eta_{j,j}$ are monomorphisms, this gives that $f|_{\inp \cup \out}$ is a monomorphism except in the case when $\alpha(i) = \alpha(j)$.
If $\alpha(i) = \alpha(j)$ then $\eta(v_\ell) = \eta(x)$ is always an edge, as is $K\{H_1, \dots, H_n\}$.
Thus the requirements of Corollary~\ref{cor vertex inclusion in out inclusion} are satisfied, and we have that $f' : K\{H_1, \dots, H_n\} \to H$ is an \'etale monomorphism.

We will now show that the image of this monomorphism is $\eta(x)$.
The set $\inp(K\{H_1, \dots, H_n \}) \cong \inp(K)$ maps to the inputs (in $H_{\alpha(i),\alpha(i)}$) of $\eta(x) \in H_{\alpha(i),\alpha(j)}$. 
A similar statement holds for outputs.
By Lemma~\ref{lemma in and out determination}, the result follows.
\end{proof}

\begin{proof}[Proof of Proposition~\ref{proposition functor l to y}]
By Theorem~\ref{theorem upsilon equivalent definitions}, it is sufficient to show that this becomes a morphism in $\hryGamma$.
Let $f : \mathfrak{P}(G) \to \mathfrak{P}(H)$ be the properad morphism associated to a morphism $(\alpha, \eta) : G\to H$ in $\levelgconn$, where $G$ has height $k$.
Apply Theorem~\ref{theorem on levelgconn to etale mono} to the unique element $x\in G_{0,k}$, representing $G$ itself.
As $\eta(x)$ represents a structured subgraph of $H$ by Lemma~\ref{lemma: level subgraph}, it follows that $f(G) \cong K\{H_1, \dots, H_n\}$ can be considered as a structured subgraph of $H$.
Thus we have verified the condition from \cite[Definition 6.46]{hrybook}, so $f$ is a morphism of $\hryGamma$.
\end{proof}

\bibliographystyle{amsalpha}
\bibliography{biblio}

\end{document}